\documentclass[11pt]{amsart}
\usepackage[T1]{fontenc}
\usepackage{txfonts}
\usepackage{bbm}
\usepackage{amssymb,verbatim}
\usepackage[all]{xy}
\usepackage{subfig}
\usepackage{tikz}
\usepackage{color}
\usepackage{a4wide}
\usepackage{mathrsfs}
\usepackage{hyperref}
\usepackage{wrapfig}
\usetikzlibrary{arrows,snakes}
\DeclareMathAlphabet{\mathpzc}{OT1}{pzc}{m}{it}

\newcommand{\R}{\mathbb{R}}
\newcommand{\C}{\mathbb{C}}

\newcommand\Z{\mathbb{Z}}
\newcommand{\N}{\mathbb{N}}
\newcommand{\Q}{\mathbb{Q}}
\renewcommand{\S}{\mathbb{S}}

\newcommand{\Ab}{\mathbf{A}}

\newcommand{\Bb}{\mathbf{B}}

\newcommand{\Mb}{\mathbf{M}}

\newcommand{\Pb}{\mathbb{P}}

\newcommand{\Ub}{\mathbf{U}}
\newcommand{\Vb}{\mathbf{V}}

\newcommand{\fracp}{\mathfrak{p}}
\newcommand{\frakp}{\mathfrak{p}}

\newcommand{\ee}{\mathbf{e}}

\newcommand{\hh}{\mathbf{h}}

\newcommand{\pp}{\mathbf{p}}
\newcommand{\qq}{\mathbf{q}}

\newcommand{\uu}{\mathbf{u}}

\newcommand{\xx}{\mathbf{x}}
\newcommand{\yy}{\mathbf{y}}

\newcommand{\Ocal}{\mathcal{O}}
\newcommand{\Pcal}{\mathcal{P}}
\newcommand{\Qcal}{\mathcal{Q}}

\newcommand{\Hbb}{\mathbb{H}}

\newcommand{\OO}{\mathscr{O}}

\newcommand{\id}{\mathrm{id}}
\newcommand{\Id}{\mathrm{Id}}

\newcommand{\SL}{{\rm SL}}
\newcommand{\GL}{{\rm GL}}

\newcommand{\Aa}{\textrm{Area}}

\newcommand{\vol}{{\rm vol}}
\newcommand{\Vol}{{\rm Vol}}

\newcommand{\End}{\mathrm{End}}

\newcommand{\Jac}{\mathrm{Jac}}

\renewcommand{\Im}{\mathrm{Im}}
\renewcommand{\Re}{\mathrm{Re}}
\newcommand{\Prym}{\mathrm{Prym}}

\renewcommand{\div}{\mathrm{div}}
\newcommand{\Span}{\mathrm{Span}}
\newcommand{\odd}{\mathrm{odd}}

\newcommand{\hol}{\mathrm{hol}}

\newcommand{\cyl}{\mathrm{cyl}}
\newcommand{\s}{\sigma}
\newcommand{\res}{\mathrm{res}}
\newcommand{\reg}{\mathrm{reg}}

\newcommand{\ord}{\mathrm{ord}}
\newcommand{\area}{\mathrm{area}}

\newcommand{\XD}{\mathcal{X}_D}

\newcommand{\hXD}{\hat{\mathcal{X}}_D}
\newcommand{\CD}{\mathcal{C}_D}

\newcommand{\hCD}{\hat{\mathcal{C}}_D}
\newcommand{\tCD}{\tilde{\mathcal{C}}_D}

\newcommand{\hxx}{\hat{\mathbf{x}}}
\newcommand{\hyy}{\hat{\mathbf{y}}}
\newcommand{\hpp}{\hat{\mathbf{p}}}

\newcommand{\vide}{\varnothing}
\newcommand{\Sig}{\Sigma}
\newcommand{\sig}{\sigma}

\newcommand{\eps}{\epsilon}

\newcommand{\ol}{\overline}
\newcommand{\ul}{\underline}

\def\supp{\mathrm{supp}}

\newcommand{\cA}{\mathcal{A}}
\newcommand{\cB}{\mathcal{B}}
\newcommand{\cC}{\mathcal{C}}
\newcommand{\cD}{\mathcal{D}}
\newcommand{\cE}{\mathcal{E}}

\newcommand{\cH}{\mathcal{H}}

\newcommand{\cL}{\mathcal{L}}
\newcommand{\cM}{\mathcal{M}}
\newcommand{\cN}{\mathcal{N}}
\newcommand{\cO}{\mathcal{O}}

\newcommand{\cQ}{\mathcal{Q}}
\newcommand{\cR}{\mathcal{R}}
\newcommand{\cS}{\mathcal{S}}
\newcommand{\cT}{\mathcal{T}}
\newcommand{\cU}{\mathcal{U}}
\newcommand{\cV}{\mathcal{V}}
\newcommand{\cW}{\mathcal{W}}
\newcommand{\cX}{\mathcal{X}}
\newcommand{\cY}{\mathcal{Y}}

\newtheorem{Theorem}{Theorem}[section]
\newtheorem{Corollary}[Theorem]{Corollary}
\newtheorem{Lemma}[Theorem]{Lemma}
\newtheorem{Proposition}[Theorem]{Proposition}

\newtheorem{Claim}[Theorem]{Claim}

\theoremstyle{remark}
\newtheorem{Remark}[Theorem]{Remark}
\begin{document}
\title[Siegel-Veech constants for eigenform loci]{Intersection theory and Siegel-Veech constants for Prym eigenform loci in $\Omega\cM_3(2,2)^\odd$}
\author{Duc-Manh Nguyen}
\address{Université de Tours, Université d'Orléans, CNRS, IDP, UMR 7013, Parc de Grandmont, 37200 Tours, France}
\email[D.-M.~Nguyen]{duc-manh.nguyen@univ-tours.fr}

\date{\today}
\begin{abstract}
We compute  the Siegel-Veech constants associated to saddle connections with distinct endpoints on Prym eigenforms for real quadratic orders with  non-square discriminant in $\Omega \cM_3(2,2)^\odd$.
\end{abstract}

\maketitle

\section{Introduction}\label{sec:intro}
\subsection{Statement of the main result}\label{subsec:statement}
Siegel-Veech constants are dynamical invariants associated with $\GL^+(2,\R)$-orbit closures in moduli space of translation surfaces. Let $\cN$ be a $\GL^+(2,\R)$-orbit closures in a stratum $\Omega\cM_g(\kappa)$ of translation surfaces. It follows from the works of Eskin-Mirzakhani~\cite{EM18} and Eskin-Mirzakhani-Mohammadi~\cite{EMM13} that the subset $\cN_1 \subset \cN$ of surfaces with unit area in $\cN$ is the support of an ergodic $\SL(2,\R)$-invariant probability measure $\nu$. Given any configuration $\cC$ of saddle connections on surfaces in $\cN$, the corresponding Siegel-Veech transform of  any  integrable function  with compact support $\varphi$ on $\R^2$ is the following function 
$$
\begin{array}{cccc}
\widehat{\varphi}: & \Omega_1\cM_g(\kappa) & \to & \R \\	
   & \widehat{\varphi}(M)  & \mapsto  &\sum_{\gamma}\varphi(\hol_M(\gamma))
\end{array}
 $$ 
where $\Omega_1\cM_g(\kappa)$ is the set of surfaces of unit area in $\Omega\cM_g(\kappa)$, $\gamma$ runs through the set of  saddle connections in configuration $\cC$ on  $M$, and $\hol_M(\gamma)$ is the  holonomy vector (equivalently, the period) of $\gamma$. In \cite{Vee98} Veech showed that for all $\varphi$ we have
\begin{equation}\label{eq:S-V:constant:def}
	\int_{\R^2}\varphi d\lambda_{\rm Leb} =c_{\cC}(\nu)\int_{\Omega_1\cM_g(\kappa)}\widehat{\varphi}d\nu.
\end{equation}
where $c_{\cC}(\nu)$ is a constant depending only on $\nu$.
It was proved in \cite{EMM13} that  $c_{\cC}(\nu) >0$ for all ergodic $\SL(2,\R)$-invariant probability measure on $\Omega_1\cM_g(\kappa)$ (for the case $\nu$ is the Masur-Veech volume, this was proved in \cite{EM01}). In fact $c_{\cC}(\nu)$ is the average asymptotic of the number of saddle connections in configuration $\cC$ on every surface $M$ whose $\GL(2,\R)$-orbit  closure equals $\cN$ (cf. \cite[Th. 2.12]{EMM13}). This asymptotic is particularly relevant in applications to  billiards in rational polygons.

If one considers holonomy vectors of regular closed geodesics weighted by the area of the cylinder formed by geodesics in the same homotopy class,  the corresponding constant is called the area Siegel-Veech constant and denoted by $c_{\area}$. Those constants are of interest since they are involved in the computation of the Lyapunov exponents of the Hodge bundle over the Teichm\"uller geodesic flow on moduli space (see \cite{EKZ14}).

%

Getting explicit values of Siegel-Veech constants is a challenging problem of the field.
For lattice surfaces (which generate closed $\GL^+(2,\R)$-orbits), and branched covers of lattice surfaces, such constants were computed by works of  Veech~\cite{Vee89}, Gutkin-Judge~\cite{GJ01}, Eskin-Masur-Schmoll~\cite{EMS03}, and  Eskin-Marklof-Morris~\cite{EMM06}.
For strata of translation surfaces and strata of quadratic differentials equipped with Masur-Veech measures, those  constants were computed by Eskin-Masur-Zorich~\cite{EMZ03}, Masur-Zorich\cite{MZ08}, and Goujard~\cite{Gouj16}. 
The $c_{\area}$ Siegel-Veech constants have also been investigated by many authors see for instance \cite{CM12, EKZ14, Gouj15, CMS23, CGM25}. 


The Prym eigenform loci are affine invariant submanifolds in moduli spaces of translation surfaces  discovered by McMullen~\cite{McM:prym}. Up to now these are the only known infinite family of primitive  invariant submanifolds 
(in the sense that they are not obtained from covering constructions) in fixed genus. It is worth noticing that those loci only exist in genus $2,3,4,5$, and that their affine definition field is (real) quadratic over  $\Q$.
The Siegel-Veech constants of Prym eigenforms in genus two have been computed by Bainbridge~\cite{Bai:GT,Bai:GAFA}.

The main aim of this paper is to compute Siegel-Veech constants for Prym eigenform loci in the stratum $\Omega\cM_3(2,2)^\odd$ of genus three translation surfaces  with two double zeros and odd spin.  Those loci are three-dimensional suborbifolds of  $\Omega\cM_3(2,2)^\odd$.
To state our main result, let us recall some basic facts about Prym eigenform loci in genus two and genus three. Let $D$ be a positive integer such that $D>1$ and $D\equiv 0,1 [4]$. 
We denote by $\cO_D$ the real quadratic order of discriminant $D$.   Let $\Omega E_D(\kappa)$ denote the locus of Prym eigenforms for real multiplication by $\cO_D$ in $\Omega E_D(\kappa)$ (see \textsection\ref{subsec:setting} for  more details on Prym eigforms). 
By a result of McMullen~\cite{McM:prym}, the loci $\Omega E_D(2)$ and $\Omega E_D(4)$ consist of finitely many $\GL^+(2,\R)$-closed orbits in  $\Omega \cM_2(2)$ and $\Omega \cM_3(4)$ respectively. Let $W_D(2)$ (resp. $W_D(4)$) denote the image of $\Omega E_D(2)$ (resp. $\Omega E_D(4)$) in $\Pb\Omega \cM_2$ (resp. $\Pb\Omega \cM_3$). Then $W_D(2)$ (resp. $W_D(4)$) consists of finitely many Teichm\"uller curves. The classifications of the components of $W_D(2)$ and of $W_D(4)$ are obtained respectively by McMullen~\cite{McM:spin} and by Lanneau-Nguyen~\cite{LN:H4}.
 

 By the results of \cite{LN:components}, for all $D\geq 8$, $\Omega E_D(2,2)^\odd$  is non-empty if only if $D\equiv 0,1, 4 \; [8]$. Moreover, $\Omega E_D(2,2)^\odd$ is connected if $D\equiv 0,4 \; [8]$, and  has two connected components, denoted by $\Omega E_{D+}(2,2)^\odd$ and $\Omega E_{D-}(2,2)^\odd$, in the case $D\equiv 1 \; [8]$. 
 
 In \textsection\ref{subsec:triple:tori:ef:def} we will introduce the notion of triple of tori Prym eigenform, which is a generalization of Prym eigenforms to disconnected Riemann surfaces. For each discriminant $D$, the space of triples of tori Prym eigenforms for $\cO_D$ will be denoted by $\Omega E_D(0^3)$. Let $W_D$ be the quotient of $\Omega E_D(0^3)$ by $\C^*$. We will see that $W_D$ is a finite cover of the modular curve $\Hbb/\SL(2,\Z)$ whose Euler characteristic can be computed explicitly (cf. \textsection\ref{sec:triple:tori:comput}).

 In the case $D\equiv 0 \, [4]$, for $k \in\{1,2,3\}$, let $c^{SV}_k(D)$ denote  the Siegel-Veech constant associated with saddle connections with multiplicity $k$ joining the two singularities on surfaces in $\Omega E_D(2,2)^\odd$.
 For $D\equiv 1 \, [8]$, we denote by $c^{SV}_k(D\pm)$ the similar Siegel-Veech constant for $\Omega E_{D\pm}(2,2)^\odd$.
 The main result of this paper is the following 
 \begin{Theorem}\label{th:Siegel:Veech}
 	Let $D \equiv 0, 1, 4 \;  [8], \, D >9$, be a non-square discriminant. 	In what follows $\chi(.)$ designates the Euler characteristic.
 	\begin{itemize} 
 		\item[$\bullet$] If $4 \, | \, D$,  then we have
 		\begin{align*}
 			c_1^{SV}(D) & = \frac{15\chi(W_D(4))}{\chi(W_D(2))+b_D \chi(W_{D/4}(2))+9\chi(W_D(0^3))}\\
 			c_2^{SV}(D) &= \frac{9\left(\chi(W_{D}(2)) + b_D\chi(W_{D:4}(2)) \right)}{\chi(W_D(2))+b_D \chi(W_{D/4}(2))+9\chi(W_D(0^3))} \\
 			c_3^{SV}(D) & = \frac{3\chi(W_D(0^3))}{\chi(W_D(2))+b_D \chi(W_{D/4}(2))+9\chi(W_D(0^3))}
 		\end{align*}
 		with
 		$$
 		b_D= \left\{
 		\begin{array}{ll}
 			0 & \text{ if } D/4 \equiv 2,3 \, [4] \\
 			4 & \text{ if } D/4 \equiv 0 \, [4] \\
 			3 & \text{ if } D/4 \equiv 1 \, [8] \\
 			5 & \text{ if } D/4 \equiv 5 \, [8]. \\
 		\end{array}
 		\right.
 		$$
 		\item[$\bullet$] If $D\equiv 1 \; [8]$, then
 		\begin{align*}
 			c_1^{SV}(D+) = c_1^{SV}(D-) & = \frac{15\chi(W_{D}(4))}{2\chi(W_D(2))+9\chi(W_D(0^3))}\\
 			c_2^{SV}(D+) = c_2^{SV}(D-) &= \frac{18\chi(W_{D}(2))}{2\chi(W_{D}(2))+9\chi(W_D(0^3))} \\
 			c_3^{SV}(D+) = c_3^{SV}(D-) & = \frac{3\chi(W_{D}(0^3))}{2\chi(W_{D}(2))+9\chi(W_D(0^3))}.
 		\end{align*}
 	\end{itemize}	
 \end{Theorem}
 
 The values of $\chi(W_D(2))$ have been calculated by Bainbridge in \cite{Bai:GT} for all discriminants $D$, and the values of $\chi(W_D(4))$ have been calculated by M\"oller in \cite{Mo14} for non-square discriminants. In \textsection~\ref{sec:triple:tori:comput}, we provide explicit formulas computing the Euler characteristic of  $W_D(0^3)$.  The values of $\chi(W_D(.))$ for $D\leq 50$, $D \equiv 0,1 \; [4]$ non-square, are recorded in Table~\ref{tab:Euler:char:values} below (note that $W_D(4)$ and $W_D(0^3)$ do not exist if $D \equiv 5 \; [8]$). 
 \begin{table}[htbp]
 	\centering
 	\begin{tabular}{|  p{0.5cm}|  p{2 cm}|  p{2 cm} | p{2 cm} || p{0.5 cm} | p{2cm}| p{2cm}| p{2cm}|}
 		\hline
 		 $D$ & $-\chi(W_D(4))$ & $-\chi(W_D(2))$ & $-\chi(W_D(0^3))$ & $D$ & $-\chi(W_D(4))$ & $-\chi(W_D(2))$ & $-\chi(W_D(0^3))$ \tabularnewline
 		 \hline
 		 5   & \centering   -  &  \centering 3/10   &  \centering - & 29 & \centering -  &  \centering 9/2  & \centering -   \tabularnewline
 		 \hline  
 		 8   & \centering   12/5  &  \centering 3/4   &  \centering 1/6 & 32 & \centering 5 &  \centering 6  & \centering 2   \tabularnewline
 		 \hline 
 		 12   & \centering   5/6  &  \centering 3/2   &  \centering 1/3 & 33  & \centering 10 &  \centering 9  & \centering 4   \tabularnewline
 		 \hline
 		 13   & \centering   -  &  \centering 3/2   &  \centering - & 37 & \centering -  &  \centering 15/6  & \centering  -  \tabularnewline
 		 \hline
 		 17   & \centering   10/3  &  \centering 3   &  \centering 4/3 & 40 & \centering 35/6  &  \centering 21/2  & \centering 7/3   \tabularnewline
 		 \hline
 		 20   & \centering   5/2  &  \centering 3   &  \centering 1 & 41 & \centering 40/3  &  \centering 12  & \centering 16/3   \tabularnewline
 		 \hline
 		 21   & \centering  -  &  \centering 3   &  \centering - & 44 & \centering  35/6 &  \centering 21/2  & \centering 7/3   \tabularnewline
 		 \hline 
 		  24   & \centering   5/2  &  \centering 9/2   &  \centering 1 & 45 & \centering - &  \centering 6  & \centering  - \tabularnewline
 		 \hline
 		  28  & \centering   10/3  &  \centering 6   &  \centering 4/3 & 48 & \centering 10 &  \centering 12  & \centering 4   \tabularnewline
 		 \hline
 	\end{tabular}  
 	
 	\caption{Euler characteristics of Teichm\"uller curves in $\partial\Pb\ol{\Omega} E_D(2,2)^\odd$}
 	\label{tab:Euler:char:values}
 \end{table}
 
 For $D\equiv 1 \; [8]$, since $c_k^{SV}(D+)=c_k^{SV}(D-)$, let us denote the common value by $c_k^{SV}(D)$.
 Surprisingly,  for all checked values of $c_k^{SV}(D)$ we always have
 \begin{equation*}
 	c_1^{SV}(D) =\frac{25}{9}, \quad c_2^{SV}(D)=3, \quad c_3^{SV}(D)=\frac{2}{9}.
 \end{equation*}
 By definition, all of the loci $\Omega E_D(2,2)^\odd$ are contained in the locus $\tilde{\cQ}(4,-1^4)$ of  canonical double covers of quadratic differentials in the stratum $\cQ(4,-1^2)$. It follows from the main result of~\cite{AN16} that $\tilde{\cQ}(4,-1^4)$ contains a unique proper rank two invariant suborbifods $\tilde{\cH}(2)$ consisting of unramified double covers of surfaces in $\Omega \cM_2(2)$. Since $\Omega E_D(2,2)^\odd$ is clearly not contained in $\tilde{\cH}(2)$ for any $D$, it follows from the results of~\cite{EMM13} (see also \cite{Doz19}) that as $D \to +\infty$ the $\SL(2,\R)$-invariant probability measure supported on $\Omega_1E_D(2,2)^\odd$ equidistributes to the one supported on $\tilde{\cQ}_1(4,-1^4)$ ($\tilde{\cQ}_1(4,-1^4)$ is the space of surfaces of unit area in $\tilde{\cQ}(4,-1^4)$). As a consequence, as $D \to \infty$, the sequence $c_k^{SV}(D)$ converges to the corresponding Siegel-Veech constant of $\tilde{\cQ}(4,-1^4)$ that we denote   by $\tilde{c}_k^{SV}(4,-1^4)$. Following the strategy of Eskin-Masur-Zorich~\cite{EMZ03} (see also \cite{Gouj15, Gouj16}), one can compute $\tilde{c}_k^{SV}(4,-1^4)$ from the Masur-Veech volumes of $\tilde{\cQ}_1(4,-1^4)$ and its boundary strata.  It turns out that we have
 $$
 \tilde{c}_1^{SV}(4,-1^4) =\frac{25}{9}, \quad \tilde{c}_2^{SV}(4,-1^4) =3, \quad \tilde{c}_3^{SV}(4,-1^4) =\frac{2}{9}.
 $$
 In a forthcoming work \cite{Ng26}, we will prove that the constants $c^{SV}_k(D)$ is indeed independent of $D$ and has the expected value.  Interestingly,  in genus two, Bainbridge~\cite{Bai:GAFA} also showed that the Siegel Veech constants of the loci $\Omega E_D(1,1)$ are actually the same for all $D$.

 \subsection{Strategy}\label{subsec:strategy}
 It has been known since pioneer work of Eskin-Masur-Zorich~\cite{EMZ03} that Siegel-Veech constants can be computed from the volumes of  invariant suborbifolds. Our first task is to define a suitable volume form on $\Omega E_D(2,2)^\odd$.  In \textsection\ref{sec:vol:form}, we give a construction of volume forms for Prym eigenform loci in all strata. By pushing forward, we obtain a volume form $d\mu$ on the  space $\Pb\Omega E_D(\kappa):=\Omega E_D(\kappa)/\C^*$.  
 The core of the current paper is the computation of the volume of $\Pb\Omega E_D(2,2)^\odd$ with respect to $d\mu$. 
 
 We will compute $\mu(\Pb\Omega E_D(2,2)^\odd)$ by intersection theory in a compact complex orbifold. To this purpose we first  need a convenient compactification of $\Pb\Omega E_D(2,2)^\odd$. 
 By definition, every element $(X,\omega)$ of $\Omega E_D(2,2)^\odd$ admits an involution $\tau$ which has $4$ fixed points and exchanges the two zeros of $\omega$. The quotient $X/\langle \tau\rangle$ is an elliptic curve with five marked points, four of which are the images of the fixed points of $\tau$, the fifth one is the image of the zeros of $\omega$.  
%
In the literature, the Riemann surface $X$ is called a {\em  bielliptic curve}. 
In view of this, we consider the space $\cB_{4,1}$ of smooth curves of genus three admitting a ramified double cover over an elliptic curve (there must be  $4$ branched points), together with a pair of points that are permuted by the deck transformation. 
It is well known that $\cB_{4,1}$ admits an orbifold compactification $\ol{\cB}_{4,1}$ consisting of stable curves that are {\em admissible double covers}  of curves in $\ol{\cM}_{1,5}$. 
Let $\Omega\ol{\cB}_{4,1}$ denote the Hodge bundle over $\ol{\cB}_{4,1}$. By definition, every curve $C \in \ol{\cB}_{4,1}$ comes equipped with an involution $\tau_C$. Denote by $\Omega(C)^-$ the space of Abelian differentials on $C$ (that is holomorphic sections of the dualizing sheaf $\omega_C$) that are anti-invariant under $\tau_C$. We have $\dim_\C \Omega(C)^-=2$, and $\Omega(C)^-$ is in fact the fiber over $C$ of a rank two holomorphic vector bundle $\Omega'\ol{\cB}_{4,1}\to \ol{\cB}_{4,1}$. 

Let $\Omega'\cB_{4,1}$ be the restriction of $\Omega'\ol{\cB}_{4,1}$ to $\cB_{4,1}$, and $\Omega'\cB_{4,1}(2,2)$ the set of pair $(C,\xi)$ in $\Omega'\cB_{4,1}$ such that $\xi$ has double zeros at the pair of marked points permuted by $\tau_C$. 
Let $\Omega \cX_D$ denote the preimage of $\Omega E_D(2,2)^\odd$ in $\Omega'\cB_{4,1}(2,2)$, and $\cX_D$ the projection of $\Omega\cX_D$ in $\Pb\Omega'\cB_{4,1}$. 
By definition $\Omega\cX_D$ is the complement of the zero section in the total space of the tautological line bundle over $\cX_D$.
We  have a covering $\hat{\rho}_2: \cX_D \to \Pb\Omega E_D(2,2)^\odd$ of degree $4!=24$. 
Denote by $d\mu$ the pullback of the volume form on $\Pb\Omega E_D(2,2)^\odd$ to $\cX_D$. Our goal now is to compute $\mu(\cX_D)$. 
  
Let $\ol{\cX}_D$ be the closure of $\XD$ in $\Pb\Omega'\ol{\cB}_{4,1}$. In general, $\ol{\cX}_D$ is a singular surface. We will show that  the normalization $\hat{\cX}_D$ of $\ol{\cX}_D$ is an orbifold. 
One can coarsely partition the boundary of $\hat{\cX}_D$ into two parts: $\partial_1\hat{\cX}_D$   consists of Abelian differentials which have no simple poles, and $\partial_\infty \hat{\cX}_D$ consists of differentials with simple poles (on singular curves). We will show that $\partial_1\hat{\cX}_D$ is a finite union of the complex curves each of which is a finite cover of one of the curves in $\{W_D(4), W_D(2), W_{D/4}(2), W_D(0^3)\}$. Moreover, points in $\partial_1\hat{\cX}_D$ are smooth points of $\hat{\cX}_D$, while $\partial_\infty \hat{\cX}_D$ contains all the singular points of $\ol{\cX}_D$.

Let $\ol{\cC}_D$ (resp. $\cC_D$) be the universal curve over $\ol{\cX}_D$ (resp. over $\XD$), and $\hat{\cC}_D$ be the pullback of  $\ol{\cC}_D$ to $\hat{\cX}_D$.  
By construction, we have an involution $\hat{\tau}$ on $\hat{\cC}_D$ which restricts to the Prym involution on each fiber of the map $\hat{\pi}: \hat{\cC}_D \to \hat{\cX}_D$.
Note that $\hat{\cC}_D$ is a three-dimensional variety which is singular in general.
Applying some slight modification to $\hat{\cC}_D$, we obtain an orbifold $\tilde{\cC}_D$ together with a projection $\tilde{\pi}: \tilde{\cC}_D \to \hat{\cX}_D$ verifying the followings
\begin{itemize}
	\item[$\bullet$] the fibers of $\tilde{\pi}$ are semi-stable curves,
	
	\item[$\bullet$] the tautological sections associated to the marked points in $\hat{\cC}_D$ lift to sections of $\tilde{\pi}$,
	
	\item[$\bullet$] the boundary of $\tilde{\cC}_D$ is a normal crossing divisor,
	
	\item[$\bullet$] the involution $\hat{\tau}$  on $\hat{\cC}_D$ extends to an involution $\tilde{\tau}$ of $\tilde{\cC}_D$ preserving each fiber of $\tilde{\pi}$.  
\end{itemize}  
We will show that there is a smooth closed $(2,2)$-form $\Theta$ on $\cC_D$  which satisfies
$$
\mu(\cX_D)=\int_{\XD}d\mu=\frac{1}{2}\cdot\int_{\Sigma_5\cap\cC_D}\Theta
$$ 
where $\Sigma_5$ is the divisor in $\tilde{\cC}_D$ associated to the zeros of the differentials parametrized by $\cX_D$.  
The key of our approach is that $\Theta$ defines a closed current on $\tilde{\cC}_D$ with the following properties
\begin{itemize}
	\item[(a)] for any divisor $\cD \subset \hat{\cX}_D$, $\langle [\Theta],[\tilde{\pi}^*\cD]\rangle = 8\pi\cdot c_1(\OO(-1))\cdot [\cD]$, where $[\Theta]$ and $[\cD]$ are the cohomology classes of $\Theta$ and $\cD$ respectively, and $\OO(-1)$ is the tautological line bundle over $\hXD$,
	
	\item[(b)] if $\Sigma \subset \tilde{\cC}_D$ is a section of  $\tilde{\pi}$ which intersects fibers of $\tilde{\pi}$ at smooth points, then we have
	$$
	\langle [\Theta],  [\Sigma]\rangle =\int_{\Sigma\cap \cC_D}\Theta,
	$$

	\item[(c)] for any irreducible component $\cT$ of $\partial_\infty \tilde{\cC}_D:=\tilde{\pi}^{-1}(\partial_\infty\hXD)$, we have $\langle [\Theta], [\cT] \rangle =0$. 
\end{itemize}
Moreover we have
\begin{equation}\label{eq:vol:n:Theta}
	\mu(\cX_D)=
	 \frac{-\pi}{24}\cdot \langle[\Theta],[\omega_{\tilde{\cC}_D/\hat{\cX}_D}]\rangle 
\end{equation} 
where $\omega_{\tilde{\cC}_D/\hat{\cX}_D}$ is the relative dualizing sheaf of $\tilde{\pi}$.

To compute $\langle[\Theta],[\omega_{\tilde{\cC}_D/\hat{\cX}_D}]\rangle $, we look for  a convenient expression of $[\omega_{\tilde{\cC}_D/\hat{\cX}_D}]$. By construction, the quotient $\tilde{\cC}_D/\langle \tilde{\tau}\rangle$ gives a family $\tilde{\cE}_D$ of semi-stable  curves of genus one and $5$ marked points over $\hat{\cX}_D$. Forgetting the first four marked points and passing to the stable model, we obtain a family $\varpi: \cE_D \to \hat{\cX}_D$ of $1$-pointed stable curve of genus one.
It is not difficult to compute the difference between $\omega_{\tilde{\cC}_D/\hXD}$ and the pullback of $\omega_{\cE_D/\hXD}$ to $\tilde{\cC}_D$. 
Using the induced morphism $\hat{\cX}_D \to \ol{\cM}_{1,1}$ and the  fact that $\omega_{\ol{\cC}_{1,1}/\ol{\cM}_{1,1}}$ is the pullback of a $\Q$-divisor in $\ol{\cM}_{1,1}$, we can express $[\omega_{\tilde{\cC}_D/\hXD}]$ as a combination of divisors with support in $\partial \tCD$. The fundamental properties of $[\Theta]$ then allow us to compute $\langle [\Theta], [\omega_{\tCD/\hXD}]\rangle$ in terms of the Euler characteristics of the curves in $\{W_D(2), W_{D/4}(2), W_{D}(0^3)\}$. 
The derivation of the Siegel-Veech constants from the volume of $\Omega E_D(2,2)^\odd$ follows from standard arguments.

 \subsection{Remarks and related works}\label{subsec:comment} \hfill
\begin{itemize} 
	\item[(i)] An analogue of the $(2,2)$-form $\Theta$ can be defined on the universal curve over any (projectivized) invariant suborbifold $\cM$ which has rel one, that is the leaves of the kernel foliation in $\cM$ have dimension one.  It can be shown that \eqref{eq:vol:n:Theta} still holds in this case. Thus, in principle, we have a method to compute the volume of such invariant suborbifolds. However, to get the explicit values, it is necessary to have an adequate expression of the cohomology class of the  relative dualizing sheaf.

	\item[(ii)] In \cite{McM:annals} McMullen  defined an $\SL(2,\R)$-invariant measure on the loci $\Omega_1 E_D(1,1)$ of Prym eigenforms with unit area in the stratum $\Omega \cM_2(1,1)$. It can be shown that the induced measure on $\Pb \Omega E_D(1,1)$ coincides with the volume form $d\mu$ constructed in this paper up to a constant.  
	
	\item[(iii)] The volumes of $\Omega_1 E_D(1,1)$ have been computed  by Bainbridge~\cite{Bai:GT, Bai:GAFA}. An essential ingredient of Bainbridge's approach is the identification of $\Pb\Omega E_D(1,1)$ with open dense subsets of  Hilbert modular surfaces. In our situation, even though there is a map from $\Pb \Omega E_D(2,2)^\odd$ onto an open dense subsets of a  version of Hilbert modular surfaces (see~\cite{Mo14}), the author is not aware of any result on the degree of this map in the literature.
	
	\item[(iv)] A natural compactification of $\Pb \Omega E_D(2,2)^\odd$ is its closure in $\Pb\Omega \cM_3$. However, information about the Prym involution, which is essential to the study of  Prym eigenforms, might be lost in the boundary of this closure. For this reason, the compactification of the lift of $\Pb\Omega E_D(2,2)^\odd$ in the anti-invariant Hodge bundle $\Pb\Omega'\ol{\cB}_{4,1}$ seems to be more relevant.
	
	\item[(v)] Another important invariant of $\GL(2,\R)$-orbit closures of translation surfaces is the Siegel-Veech constant $c_{\cyl}$ associated with the counting of cylinders. Unfortunately, the results of this paper do not allow us to compute this constant for $\Omega E_D(2,2)^\odd$. 
	
	\item[(vi)] In view of the results in this paper, here are some open questions:  How to compute the Siegel-Veech constants associated to cylinders on Prym eigenforms? Can the method of this paper be generalized to other Prym eignform loci for instance $\Omega E_D(2,1,1)$, or to the case $D$ is a square? 
	
\end{itemize}

 \subsection{Outline} The paper is organized as follows: in \textsection\ref{sec:vol:form} we recall some basic properties of Prym eigenforms in general. We then give a construction of a volume form $d\vol$ on any loci $\Omega E_D(\kappa)$ and define the induced measure $\mu$ on $\Pb \Omega E_D(\kappa)$. It turns out that $\mu$ is the measure associated with a volume form $d\mu$. The main result of this section is Theorem~\ref{th:vol:form:proj:diff:express} which provides an explicit local expression of $d\mu$. 
 
 In \textsection\ref{sec:eigen:form:g3} we recall some geometric characteristics of Prym eigenforms in $\Omega E_D(2,2)^\odd$.
 We emphasize on the facts that the surfaces in $\Omega E_D(2,2)^\odd$ are completely periodic, and their cylinder diagrams are parametrized by a finite set.
 
 In \textsection\ref{sec:adm:covers}, we introduce the space of bielliptic curve $\cB_{4,1}$ and its closure $\ol{\cB}_{4,1}$. We define $\Omega \cX_D$ (resp. $\XD$) as the preimage of $\Omega E_D(2,2)^\odd$ (resp. $\Pb\Omega E_D(2,2)^\odd$) in the anti-invariant Hodge bundle $\Omega' \ol{\cB}_{4,1}$ (resp. in $\Pb \Omega'\ol{\cB}_{4,1}$). We close this section by showing that the projection $\cX_D \to \Pb\Omega E_D(2,2)^\odd$ has degree $24$.
 
 In \textsection\ref{sec:classify:bdry:str:XD}, we classify the (projectivized) differentials contained in the boundary of the closure $\ol{\cX}_D$ of $\cX_D$ in $\Pb\Omega'\ol{\cB}_{4,1}$. The complete classification is given in Theorem~\ref{th:bdry:eigen:form:H22}. Since the proof of this theorem has no significant connection with the rest  of the  paper, it will be provided in Appendix \textsection\ref{sec:prf:bdry:egein:form:H22}. 
 The geometry of $\ol{\cX}_D$ in the neighborhood of every point in its boundary is analyzed in \textsection\ref{sec:geometry:bdry:XD}. An immediate consequence of the results in \textsection\ref{sec:geometry:bdry:XD} is that the normalization $\hXD$ of $\ol{\cX}_D$ is an orbifold. 
 
 Let $\hat{\pi}: \hCD \to \hXD$ be the universal curve over $\hXD$. In \textsection \ref{sec:normal:univ:curve}, we show that $\hCD$ admits a modification $\tCD$ (obtained by blowing up finitely many points) which is an orbifold such that the projection $\tilde{\pi}: \tCD \to \hXD$ is a family of semi-stable curves which has essentially the same properties as $\hat{\pi}$ (cf. Proposition~\ref{prop:univ:curves:orb}).

 In preparation to the computation of $\mu(\XD)$, in \textsection\ref{sec:div:relations:in:CD} we prove some crucial relations of tautological divisors in $\tCD$. In particular, in Proposition~\ref{prop:rel:cotangent:class}, we prove a formula which expresses the class $[\omega_{\tCD/\hXD}]$ as a combination of divisors supported in the boundary of $\tCD$ and tautological sections of $\tilde{\pi}$.
 
 In \textsection\ref{sec:curv:curent:vol} we introduce the $(2,2)$-form $\Theta$ on $\cC_D$ and show that it defines a closed current in $\tCD$. To prove the latter, among other things, one needs a detailed description of the neighborhood of every point in the boundary of $\tCD$ as well as an explicit local section of the relative dualizing sheaf. In particular, the constructions in \textsection \ref{sec:geometry:bdry:XD} play an important role in the proof. 
 
 In \textsection\ref{sec:prop:Theta} we prove the fundamental properties of the current  $[\Theta]$. As a consequence, in \textsection\ref{sec:vol:XD:n:Theta} we obtain a formula expressing the volume of $\XD$ as intersection number of $[\Theta]$ and some boundary divisors in $\tCD$ (cf. Theorem~\ref{th:vol:XD}). It turns out that the divisors involved in the computation of $\mu(\XD)$ project to strata of $\partial \hXD$ that are finite covers of the curves $W_D(2), W_{D/4}(2), W_D(0^3)$. In \textsection\ref{sec:triple:tori:comput} and \textsection\ref{sec:W:Teich:curves} we show that the intersection of $[\Theta]$ and the divisors mentioned above can be computed from the Euler characteristics of  $W_D(2), W_{D/4}(2), W_D(0^3)$. For this, it is necessary to determine the degree of the map from some strata of $\partial \hXD$ onto $W_D(2)$ and $W_{D/4}(2)$, as well as the degree of natural projections from $W_D(0^3)$ onto the modular curve $\Hbb/\SL(2,\Z)$.
 
 Once the intersections of $[\Theta]$ and the divisors of $\tCD$ are computed, one immediately deduces the volumes of $\XD$ and of $\Pb\Omega E_D(2,2)^\odd$. Details of the calculations are given in \textsection\ref{sec:vol:Omg:E:D:2:2}. Finally, in \textsection\ref{sec:Siegel:Veech}, we give the proof of Theorem~\ref{th:Siegel:Veech}.

 \subsection{Notation and convention:} Throughout this paper, 
 \begin{itemize}
 	\item[$\bullet$]  $D$ will be an integer such that $D\geq 4$, and $D\equiv 0,1,4 \; [8]$,
 	
 	\item[$\bullet$] $\Delta=\{z\in \C, \; |z| < 1\}$ is the unit disc in $\C$,
 	
 	\item[$\bullet$]  for all $\eps\in \R_{>0}$, $\Delta_\eps = \{z \in \C, \; |z| <  \eps\}$. 
 \end{itemize}

 \subsection{Acknowledgement}
 The author thanks D.~Zvonkin and A.~Page for the helpful discussions.

\section{Volume form on Prym eigenform loci}\label{sec:vol:form}
\subsection{Prym eigenform}\label{subsec:setting}
A real quadratic order is a ring isomorphic to $\Z[x]/(x^2+bx+c)$, with $b,c\in \Z$ such that $D:=b^2-4c >0$. The number $D$ is called the discriminant of the order. 
A quadratic order is determined up to isomorphism  by its discriminant. For all $D \in \N, D\equiv 0,1 \; [4]$, we will denote by $\cO_D$ the real quadratic order of discriminant $D$.

Let $A$ be a polarized Abelian surface. We say that $A$ admits a real multiplication by $\cO_D$ if there exists a faithful ring morphism $\rho: \cO_D \to \End(A)$ such that
\begin{itemize}
	\item the image of $\rho$ consists of self-adjoint endomorphisms with respect to the polarization of $A$.
	
	\item $\rho$ is proper, meaning that if $f \in \End(A)$, and for some $n \in \Z\setminus\{0\}$, we have $nf\in \rho(\cO_D)$, then $f \in \rho(\cO_D)$.
\end{itemize}


Consider a Riemann surface $X$  admitting an involution $\tau$. Let $\Omega(X)^-$  be the eigenspace of the eigenvalue $-1$ for the action of $\tau$ on $\Omega(X)$.
Define  $H_1(X,\Z)^-:=\{c\in H_1(X,\Z), \; \tau_*c=-c\}$. 
The {\em Prym variety} of the pair $(X,\tau)$ to defined to be
$$
\Prym(X,\tau):=(\Omega(X)^-)^*/H_1(X,\Z)^-.
$$
This is an Abelian subvariety of $\Jac(X)$ with polarisation being the restriction of the polarisation on $\Jac(X)$.
Let $\omega$ be a non-trivial holomorphic $1$-form on $X$. The pair $(X,\omega)$ is called a {\em translation surface}. 
Following McMullen~\cite{McM:prym}, we will call an element $(X,\omega)$  a {\em Prym eigenform for real multiplication by $\cO_D$} if we have

\begin{itemize}
	\item $\dim_\C\Prym(X,\tau)=2$, and  $\Prym(X,\tau)$  admits a real multiplication by $\cO_D$,
	
	\item as an element of $\Omega(\Prym(X,\tau))$, $\omega$ is an eigenvector for the action of $\cO_D$ on $\Omega(\Prym(X,\tau))$.
\end{itemize}

Let $g$ be the genus of $X$. Then the pair $(X,\omega)$ is an element of the Hodge bundle $\Omega \cM_g$ over the moduli space $\cM_g$. 
The locus of Prym  eigenform for real multiplication by $\cO_D$ in $\Omega\cM_g$ is denoted by $\Omega E_D$. 
The condition $\dim \Prym(X,\tau)=2$ means that $g(X/\langle\tau\rangle)=g(X)-2$, where $g(.)$ is the genus. It then follows from the Hurwitz formula that we must have $2\leq g \leq 5$
Thus $\Omega E_D$ only exists for $g\in\{2,3,4,5\}$.

The Hodge bundle $\Omega\cM_g$ is naturally stratified as
$$
\Omega\cM_g=\underset{\substack{\kappa=(k_1,\dots,k_n)\\ k_1+\dots+k_n=2g-2}}{\bigsqcup}\Omega\cM_g(\kappa).
$$
where $k_1,\dots,k_n$ are positive integers, and $\Omega\cM_2(k_1,\dots,k_n)$ is the set of Abelian differentials having exactly $n$ zeros with orders $(k_1,\dots,k_n)$. Each $\Omega\cM_g(\kappa)$ is called a stratum of $\Omega\cM_g$.
The intersection of $\Omega E_D$ with a stratum $\Omega\cM_g(\kappa)$ will be denoted by $\Omega E_D(\kappa)$.

It is a well known fact that there is an action of $\GL^+(2,\R)$ on $\Omega \cM_g$ preserving its stratification.
It is shown by McMullen~\cite{McM:prym} that $\Omega E_D(\kappa)$ is a closed suborbifold of $\Omega\cM_g(\kappa)$ which is invariant under the action of $\GL^+(2,\R)$.
If $D$ is  not square then $\Omega E_D(\kappa)$ is primitive in the sense that $\Omega E_D(\kappa)$ does not arise from 
a $\GL^+(2,\R)$-invariant  suborbifold of another space $\Omega\cM_{g'}$ with $g'<g$ by a covering construction.
In particular, it is shown in \cite{McM:prym} that if non-empty, the Prym eigform locus $\Omega E_D(2g-2)$ in the minimal stratum $\Omega \cM_g(2g-2)$ for $g=2,3,4$ consists of finitely many primitive closed $\GL^+(2,\R)$-orbits (their projections into $\cM_g$ are called {\em Teichm\"uller curves}).  
To the author knowledge, the loci $\Omega E_D(\kappa)$, $D$ non-square,  constitute the only known examples of infinite families of primitive $\GL^+(2,\R)$-invariant suborbifolds of $\Omega\cM_g$ for a given $g\geq 2$.

\subsection{Affine structure}\label{sec:affine:structure}
 We first give a description of a neighborhood  of an eigenform $(X,\omega)$ in $\Omega E_D(\kappa)$.
Let $x_1, \dots, x_n$ be the zeros of $\omega$ where $x_i$ has order $k_i$.
Then $\omega$ defines an element  of $H^1(X,\{x_1,\dots,x_n\};\C)$.
 By definition, for any cycle in $H_1(X,\{x_1,\cdots,x_n\};\Z)$ represented by a $C^1$-piecewise path $c$, one has
$$
\omega(c):=\int_c\omega.
$$
If $(X',\omega')\in \Omega \cM_g(\kappa)$ is close enough to $(X,\omega)$, then $H_1(X',\{x'_1,\dots,x'_n\};\Z)$, where $x'_1,\dots,x'_n$ are the zeros of $\omega'$, can be identified with 
$H_1(X,\{x_1,\cdots,x_n\};\Z)$. We thus have a map $\Phi: \cU \to H^1(X,\{x_1,\dots,x_n\},\C)$ defined on a neighborhood $\cU$ of $(X,\omega)$ in $\Omega\cM_g(\kappa)$. 
This map can be defined in more concrete terms as follows: fix a basis $\{\gamma_1,\dots,\gamma_{2g+n-1}\}$ of $H_1(X,\{x_1,\dots,x_n\};\Z)$. Then $\Phi$ is given by
$$
\begin{array}{cccc}
\Phi: & \cU & \to & \C^{2g+n-1} \\
      & (X,\omega) & \mapsto & (\int_{\gamma_1}\omega,\dots,\int_{\gamma_{2g+n-1}}\omega)
\end{array}
$$
The map $\Phi$ is called the {\em period mapping}.
It is a well known fact that period mappings are local biholomorphisms, thus can be used to define an atlas of $\Omega\cM_g(\kappa)$.
Transition maps of this atlas correspond to changing the basis of $H_1(X,\{x_1,\dots, x_n\};\Z)$.

Let $\wp: H^1(X,\{x_1,\dots,x_n\};\C) \to H_1(X,\C)$ be the natural projection. For all any $\eta\in H^1(X,\{x_1,\dots,x_n\};\C)$, $\wp(\eta)$ is the restriction of $\eta$ to the (absolute) cycles in $H_1(X,\C)$.
Define
$$
W:=\Span(\Re(\omega),\Im(\omega)) \subset H^1(X,\C)^-, \quad  \text{ and} \quad W_\R:=W\cap H^1(X,\R)^-.
$$
In \cite{McM:prym}, McMullen proved the following
\begin{Proposition}[McMullen]\label{prop:affine:coord}
The period mapping $\Phi$ identifies a neighborhood of $(X,\omega)$ in $\Omega E_D(\kappa)$ with an open subset of the linear subspace
$$
V:=\wp^{-1}(W)\cap H^1(X,\{x_1,\dots,x_n\};\C)^- \subset H^1(X,\{x_1,\dots,x_n\};\C)^-.
$$
\end{Proposition}

\subsection{Volume form on $\Omega E_D(\kappa)$}
In this section, we introduce a  construction of volume forms on Prym eigenform loci in general. This construction actually works for all rank one invariant sub-orbifolds in $\Omega \cM_g(\kappa)$. We will eventually  compute the total volume of $\Pb\Omega E_D(2,2)^\odd$ with respect to this volume form and derive from this the formulas computing the Siegel-Veech constants in Theorem~\ref{th:Siegel:Veech}.
Throughout this section $(X,\omega)$ is a Prym eigenform in some locus  $\Omega E_D(\kappa) \subset \Omega \cM_g(\kappa)$.

A zero of $\omega$ is either fixed or exchanged  by $\tau$ with another zero. Let $x_1,\dots,x_r$ be the zeros that are fixed by $\tau$ and $x_{r+1},\dots,x_{r+2s}$ be the remaining ones where $x_{r+j}$ and $x_{r+s+j}$ are exchanged by $\tau$.

\begin{Lemma}\label{lm:dual:basis:ker:abs:proj}
For $j=1,\dots,s$, let $c_j$ be a path from $x_{r+j}$ to $x_{r+s+j}$. Then the map
$$
\begin{array}{cccc}
\phi: & V\cap\ker\wp &  \to &  \C^s\\
      & v            & \mapsto & (v(c_1),\dots,v(c_s))
\end{array}
$$
is an isomorphism
\end{Lemma}
\begin{proof}[Sketch of proof]
Since $V=\wp^{-1}(W)\cap H^1(X,\{x_1,\dots,x_n\};\C)^-$ and $\ker\wp \subset\wp^{-1}(W)$, we get
$$
V\cap\ker\wp=H^1(X,\{x_1,\dots,x_n\};\C)^-\cap\ker\wp.
$$
We have the following exact sequence in cohomology
\begin{equation}\label{eq:exact:seq:coh}
0 \to H^0(X,\C) \to H^0(\{x_1,\dots,x_n\},\C) \overset{\delta}{\to} H^1(X,\{x_1,\dots,x_n\};\C) \overset{\wp}{\to} H^1(X;\C) \to 0.
\end{equation}
Since $\tau$ acts equivariantly on the terms of this exact sequence, by restricting to the eigenspaces of the eigenvalue $-1$, we get the following exact sequence
\begin{equation}\label{eq:exact:seq:coh:eval:-1}
0 \to H^0(\{x_1,\dots,x_n\},\C)^- \overset{\delta}{\to} H^1(X,\{x_1,\dots,x_n\};\C)^- \overset{\wp}{\to} H^1(X;\C)^- \to 0.
\end{equation}
Elements of $H^0(\{x_1,\dots,x_n\};\C)$ are $\C$-valued functions on the set $\{x_1,\dots,x_n\}$.
By definition, $\delta(f)\in H^1(X,\{x_1,\dots,x_n\};\C)$ is a $\C$-linear form on $H_1(X,\{x_1,\dots,x_n\};\C)$ which associates to a path $c:[0;1] \to X$ with $\partial c \subset \{x_1,\dots,x_n\}$ the number $f(c(1))-f(c(0))$.
Clearly, $f\in H^0(\{x_1,\dots,x_n\};\C)^-$ if and only if
\begin{itemize}
\item[$\bullet$] $f(x_i)=0$, for all $i=1,\dots,r$,

\item[$\bullet$] $f(x_{r+j})=-f(x_{r+s+j})$, for all $j=1,\dots,s$.
\end{itemize}
It follows that the family of paths $\{c_1,\dots,c_s\}$ is basis of $\delta(H^0(\{x_1,\dots,x_n\},\C)^-)^*$, and the lemma follows.
\end{proof}
%
%

Let $\langle.,.\rangle$ denote the intersection form on $H_1(X,\Z)$. By a slight abuse of notation ,we will also denote by $\langle.,.\rangle$ the intersection form on $H^1(X,\R)$. 
We extend $\langle.,.\rangle$ to $H^1(X,\C)$ by $\C$-linearity, and define the  Hermitian form $(.,.)$ on $H^1(X,\C)$ by
$$
(\eta,\xi)=\frac{\imath}{2}\langle \eta,\bar{\xi}\rangle
$$
where $\bar{\xi}$ is the complex conjugate of $\xi$. The restriction of $(.,.)$ to $\Omega^{1,0}(X,\C)$ is positive definite, while  the restriction to $\Omega^{0,1}(X,\C)$ is negative definite.
Since $\{\omega, \ol{\omega}\}$ is a $\C$-basis of $W$, the restriction of $(.,.)$ to $W$ has signature $(1,1)$. In particular, $(.,.)_{|W}$ is non-degenerate. Therefore the imaginary part of $(.,.)$, denoted by $\vartheta$, gives a symplectic form on $W$.

Recall that a neighborhood of $(X,\omega)$ in $\Omega E_D(\kappa)$ is identified with an open subset of $V=\wp^{-1}(W)\cap H^1(X,\{x_1,\dots,x_n\}; \C)^-$.
By a slight abuse of notation, we denote by $\vartheta$ the pullback of the imaginary part of $(.,.)$ to $V$.
Let $\{c_1,\dots,c_s\}$ be the paths in Lemma~\ref{lm:dual:basis:ker:abs:proj}.  We consider the $c_j$'s as elements of $(H^1(X,\{x_1,\dots,x_n\};\C))^*$.

\begin{Proposition}\label{prop:vol:def}
Let
$$
\Xi:=\left(\frac{\vartheta^2}{2}\right)\wedge\left( \frac{\imath}{2}\right)^{s}\cdot c_1\wedge\bar{c}_1\wedge\dots\wedge c_s\wedge \bar{c}_s \in \Lambda^{s+2,s+2}(H^1(X,\{x_1,\dots,x_n\},\C)).
$$
Then the restriction of $\Xi$ to $V$ is a non-trivial volume form, which does not depend on the choice of the paths $\{c_1,\dots,c_s\}$.  As a consequence, $\Xi_{|V}$ gives rise to a well defined volume form on $\Omega E_D(\kappa)$.
\end{Proposition}
\begin{proof}
Let $L\subset H_1(X,\R)^-$ be the subspace generated by the dual of $\Re(\omega)$ and $\Im(\omega)$ in $H_1(X,\R)^-$.
Let $L'$ be the orthogonal complement of $L$ with respect to the intersection form on $H_1(X,\R)^-$. Since the restriction of the intersection to $L$ is non-degenerate,  we have
$\dim L = \dim L'=2$, and $H_1(X,\R)^-=L\oplus L'$. 

We can choose a basis $\{a,b\}$ of $L$ and $\{a',b'\}$ such that $\langle a,b\rangle = \langle a', b'\rangle =1$. Note that $\{a,b,a',b'\}$ is a basis of $H_1(X,\R)^-$.   Using this basis, the intersection form on $H^1(X,\R)^-$ is given by $a\wedge b + a'\wedge b'$, that is
$$
\langle \alpha,\beta\rangle =\alpha(a)\beta(b)-\beta(a)\alpha(b)+\alpha(a')\beta(b')-\beta(a')\alpha(b'), \quad \forall \alpha,\beta\in H^1(X,\R).
$$
We now consider $a,b,a',b'$ as complex linear forms on $H^1(X,\C)$.
By definition, for all $c \in H_1(X,\C)$,  $\bar{c}$ is the $\C$-valued linear form on $H^1(X,\C)$ defined by $\bar{c}(\eta)=\ol{\eta(c)}$.
The Hermitian form $(.,.)$ on $H^1(X,\C)^-$ is then given by $\imath(a\otimes\bar{b}-b\otimes\bar{a}+a'\otimes \bar{b}'-b'\otimes\bar{a}')$, and therefore
$$
\vartheta=\frac{\imath}{2}\left(a\wedge\bar{b} - b\wedge\bar{a}+a'\wedge\bar{b}'-b'\wedge\bar{a}'\right).
$$
Since $a'$ and $b'$ vanish on $W$, we get $\vartheta_{\left|W\right.}=\frac{\imath}{2}\left(a\wedge \bar{b} - b \wedge \bar{a}\right)$.
Thus
$$
\vartheta^2_{\left|W \right.}=\frac{-1}{2}a\wedge \bar{a}\wedge b\wedge\bar{b}.
$$
In particular, $\vartheta^2$ restricts to a volume form on $W$.

It follows from Lemma~\ref{lm:dual:basis:ker:abs:proj} that $c_1  \wedge \bar{c}_1 \wedge \dots \wedge c_s \wedge \bar{c}_s$ restricts to a volume form on $\ker\wp\cap H^1(X,\{x_1,\dots,x_n\};\C)^-$. Since the spaces $V, W$, and $\ker\wp\cap H^1(X,\{x_1,\dots,x_n\};\C)^-$ fit into the following exact sequence
$$
0 \to \ker\wp\cap  H^1(X,\{x_1,\dots,x_n\};\C)^- \to V \overset{\wp}{\to} W \to 0,
$$
we conclude that $\Xi$ is a volume form on $V$.
It remains to shows that $\Xi$ does not depend on the choice of the paths $c_1,\dots,c_s$. Let $c'_j$ be a path with the same endpoints as $c_j$. Then as elements of $H_1(X,\{x_1,\dots,x_n\};\Z)^-$, we can write
$$
c'_j=c_j+x_j
$$
where $x_j$ is an absolute cycle, that is an element of $H_1(X,\Z)$. We consider $x_j$ as an element of $H^1(X,\C)^*$. Since $\left(\vartheta^2\wedge x_j\right)_{\left|W\right.}=\left(\vartheta^2\wedge \bar{x}_j\right)_{\left|W\right.}=0$ (because $\vartheta^2_{\left|W\right.}$ is a volume form on $W$). 
As a consequence
$$
\vartheta^2\wedge  c_1\wedge\bar{c}_1\wedge\dots\wedge c_j \wedge\bar{c}_j\wedge\dots\wedge c_s\wedge \bar{c}_s=
\vartheta^2\wedge  c_1\wedge\bar{c}_1\wedge\dots \wedge c'_j \wedge\bar{c}'_j\wedge \dots \wedge c_s\wedge \bar{c}_s
$$
and the proposition is proved.
\end{proof}
\begin{Remark}\label{rk:vol:form:Omega:E:D:express}
	The restriction $\Xi_{|V}$ can be given in more concrete terms as follows: let $a,b, c_1,\dots,c_s$ be as in the proof of Proposition~\ref{prop:vol:def}. Then a neighborhood of $(X,\omega)\in \Omega  E_D(\kappa)$ is identified with an open subset of $\C^{s+2}$ via the period mapping
	$$
	\Phi: (X,\omega) \mapsto \left(\int_a\omega,\int_b\omega,\int_{c_1}\omega,\dots,\int_{c_s}\omega \right)
	$$
	Let $(z_1,z_2,w_1,\dots,w_s)$ be the coordinates on $\C^{2+s}$. Then $\Xi_{|V}$ is the pullback by $\Phi$ of the volume form
	$$
	\frac{-1}{2}\cdot\left(\frac{\imath}{2}\right)^sdz_1d\bar{z}_1dz_2d\bar{z}_2dw_1 d\bar{w}_1\dots dw_s d\bar{w}_s=2\lambda_{2(2+s)},
	$$
	where $\lambda_{2(2+s)}$ is the Lebesgue measure on $\C^{2+s}\simeq \R^{2(2+s)}$.
	
\end{Remark}

Denote by $d\vol$ the volume form on $\Omega E_D(1,1)$ induced by $\Xi_{|V}$.
Recall that for all $(X,\omega) \in \Omega\cM_g$, the {\em Hodge norm} of $\omega$ is defined to  be
$$
||\omega||^2:=(\omega,\omega)=\frac{\imath}{2}\cdot\int_X\omega\wedge\bar{\omega} =\Aa(X,|\omega|),
$$
where $|\omega|$ denote the flat metric defined by $\omega$.
Define
$$
\Omega_1 E_D(\kappa):= \{(X,\omega) \in \Omega E_D(\kappa), \; \Aa(X,|\omega|)=1\},
$$
and
$$
\Omega_{\leq 1} E_D(\kappa):=\{(X,\omega) \in \Omega E_D(\kappa), \; \Aa(X,|\omega|)\leq 1\}.
$$
Note that $\Omega_1E_D(\kappa)$ is an  $\SL(2,\R)$-invariant closed subset of $\Omega \cM_g(\kappa)$.
There is a natural projection from $\Omega_{\leq 1} E_D(\kappa)$ onto $\Omega_1 E_D(\kappa)$ by rescaling.
The volume form $d\vol$ on $\Omega E_D(\kappa)$ defines a measure on $\Omega_{\leq 1} E_D(\kappa)$. The pushforward of this measure on $\Omega_1 E_D(\kappa)$ will be denoted by $d\vol_1$.

In the case $\kappa=(1,1), g=2$, McMullen \cite{McM:annals} defined a measure on $\Omega_1 E_D(1,1)$ which differs from $d\vol_1$ by a multiplicative constant using the foliation of $\Omega_1 E_D(1,1)$ by $\SL(2,\R)$-orbits  (see also \cite[\textsection 4]{Bai:GAFA}). 

\subsection{Volume form on the space of projectivized differentials}\label{subsec:vol:on:proj:diff}
Let $\Pb\Omega\cM_g$ be the projective bundle associated with the Hodge bundle $\Omega\cM_g$.
Let $\Omega\cM_g^*$ denote the complement of the zero section in $\Omega \cM_g$.
For any Abelian diffrerential $(X,\omega)\in \Omega\cM_g^*$, denote by $(X,[\omega])$ its pojection in $\Pb\Omega\cM_g$.
For any subvariety $\cM \subset \Omega\cM_g^*$ which is invariant under the $\C^*$-action, we denote by $\Pb\cM$ its image in $\Pb\Omega\cM_g$.

Consider now the projectivization $\Pb\Omega E_D(\kappa)$ of some Prym eigenform locus $\Omega E_D(\kappa)$. 
We have seen that $\Omega E_D(\kappa)$ can be endowed with a volume form $d\vol$. Let $\mu$ denote measure on $\Pb\Omega E_D(\kappa)$ which is the pushforward of the restriction of $d\vol$ to $\Omega_{\leq 1}E_D(\kappa)$. This means that for all open subset $B$ of $\Pb\Omega E_D(\kappa)$, let $C(B)\subset \Omega E_D(\kappa)$ be the cone over $B$ and $C_1(B):=C(B)\cap\Omega_{\leq 1}E_D(\kappa)$, then we have
$$
\mu(B)=\int_{C_1(B)}d\vol =: \vol(C_1(B)).
$$
One of the interests of considering $\Pb\Omega E_D(\kappa)$ instead of $\Omega_1 E_D(\kappa)$ is that  $\Pb\Omega E_D(\kappa)$ is an algebraic  complex orbifold. Therefore, we can use tools from algebraic and complex analytic geometry to compute the volume of $\Pb\Omega E_D(\kappa)$.

%
It is not difficult to see that $\mu$ is actually the measure associated with a volume form on $\Pb\Omega E_D(\kappa)$.
To give a concrete expression of this volume form, let us consider the following situation: let $V$ be a $\C$-vector space of dimension $d$ equipped with a Hermitian form $H$ of rank $k$.
Let $\Omega$ be the imaginary part of $H$. Let $\{\xi_1,\dots,\xi_s\}$, where $s=d-k$, be an independent family in $V^*$ such that the $(d,d)$-form
$$
d\vol:=\left(\frac{\imath}{2}\right)^s\cdot\frac{\Omega^k}{k!}\wedge\xi_1\wedge\bar{\xi}_1\wedge\dots\wedge\xi_s\wedge\bar{\xi}_s
$$
is non-zero. Let $\vol$ denote the measure on $V$ obtained by integrating $d\vol$.
Define
$$
V^+:=\{v\in V, \; | \; H(v,v) >0\}.
$$
Let $\Pb V^+$ be the image of $V^+$ in the projective space $\Pb V$. Note that $\Pb V^+$ is an open subset of $\Pb V$.
By definition, $H$ gives a Hermitian metric on the tautological line bundle $\OO(-1)_{\Pb V^+}$ over $\Pb V^+$.
The measure $\vol$ on $V^+$ induces a measure $\mu$ on $\Pb V^+$ as follows:
for all open $U\subset \Pb V^+$, let $C_1(U):=\{v\in V^+, \; H(v,v) <1, \; \C\cdot v \in U\}$, then $\mu(U):=\vol(C_1(U))$.

\begin{Proposition}\label{prop:push:meas:proj:sp}
The measure $\mu$ is the one obtained by integrating a volume form $d\mu$ on $\Pb V^+$. 
Let $\xx$  be a point in $\Pb V^+$ and $\sigma$ a holomorphic section of the tautological line bundle $\OO(-1)_{\Pb V^+}$ on a neighborhood $U$ of $\xx$. Let $h(\xx'):=H(\sigma(\xx'),\sigma(\xx'))$ for all $\xx'\in U$. We then have
\begin{equation}\label{eq:vol:form:proj:sp:express}
d\mu=\frac{\pi}{d}\cdot\frac{(-1)^{k-1}}{2^{k-1}(k-1)!}\cdot\left(\frac{\imath}{2}\right)^s\cdot \left(-\imath\partial\bar{\partial}\ln h\right)^{k-1} \wedge \partial\bar{\partial}\left(  \frac{|\xi_1\circ\sigma|^2}{h}\right) \wedge\dots\wedge
\left(\partial\bar{\partial} \frac{|\xi_s\circ\sigma|^2}{h}\right).
\end{equation}
\end{Proposition}
\begin{Remark}\label{rk:vol:independ:section}
The right hand side of \eqref{eq:vol:form:proj:sp:express} does not depend on the choice of the section $\sigma$.
\end{Remark}
\begin{proof}
Since $\xx\in \Pb V^+$, we have $\xx=\langle v_0 \rangle$ for some $v_0$ such that $h(v_0)=1$. By choosing an appropriate basis, we can identify $V$ with $\C^d$ in such a way that
\begin{itemize}
\item[$\bullet$] $v_0=(1,0,\dots,0)$,

\item[$\bullet$] if $v=(z_0,z_1,\dots,z_{d-1})$ then $H(v,v)=\sum_{i=0}^{p-1}|z_i|^2-\sum_{i=p}^{k-1}|z_i|^2$ ($p\geq 1$).
\end{itemize}
In these coordinates, we have
$$
\Omega=\frac{\imath}{2}\cdot\left( \sum_{i=0}^{p-1} dz_i\wedge d\bar{z}_i -\sum_{i=p}^{k-1} dz_i\wedge d\bar{z}_i \right).
$$
Thus
$$
\frac{\Omega^k}{k!} = \left(\frac{\imath}{2}\right)^k\cdot(-1)^{k-p}\cdot dz_0\wedge d\bar{z}_0\wedge\dots\wedge dz_{k-1}\wedge d\bar{z}_{k-1}.
$$
Since the $(d,d)$-form $\Omega^k\wedge \xi_1\wedge\bar{\xi}_1\wedge\dots\wedge\xi_s\wedge\bar{\xi}_s$ is non-zero, we can adjust the basis of $V$ such that $\xi_i=dz_{k+i-1}+\sum_{j=0}^{k-1}\lambda_{i,j}dz_j$ for all $i=1,\dots,s$.
In the corresponding coordinate system, we have
$$
d\vol=\left(\frac{\imath}{2}\right)^d\cdot(-1)^{k-p}\cdot dz_0\wedge d\bar{z}_0\wedge\dots\wedge dz_{d-1}\wedge d\bar{z}_{d-1}.
$$
Let $\eps=(\eps_1,\dots,\eps_{d-1})$ be a coordinate system on $U$.
\begin{Claim}\label{clm:vol:form:push:loc:form}
The measure $\mu$ is the one associated with the volume form
\begin{equation}\label{eq:vol:form:push:loc:form}
d\mu=\left(\frac{\imath}{2}\right)^{d-1}\cdot (-1)^{k-p}\cdot\frac{\pi}{d}\cdot\frac{1}{h^{d}(\eps)} \cdot d\eps_1d\bar{\eps}_1\dots d\eps_{d-1}d\bar{\eps}_{d-1}
\end{equation}
on $U$.
\end{Claim}
\begin{proof}
We have a natural section $\sigma$ of $\OO(-1)_{\Pb V^+}$ over $U$ given by $\sigma(\eps)=(1,\eps_1,\dots,\eps_{d-1})$, for all $\eps=(\eps_1,\dots,\eps_{d-1})\in U$. 
We then have
$$
 h(\eps):= H(\sigma(\eps),\sigma(\eps))= 1+|\eps_1|^2+\dots+|\eps_{p-1}|^2-(|\eps_p|^2+\dots+|\eps_{k-1}|^2).
$$
The cone $C(U)$ over $U$ can be parametrized by $\S^1\times]0,+\infty[\times U$  via the map
$$
\begin{array}{cccc}
\phi: & \S^1\times]0,+\infty[\times U & \to & V \\
      & (\theta,t,\eps) & \mapsto & e^{\imath\theta}\cdot t \cdot\sigma(\eps).
\end{array}
$$
We have $\phi^{-1}(C_1(U))=\{(\theta,t,\eps) \in \S^1\times]0,\infty[\times U, \;  \; t < \frac{1}{\sqrt{h(\eps)}} \}$
and
$$
\phi^*d\vol=\left(\frac{\imath}{2}\right)^{d-1}\cdot (-1)^{k-p}\cdot t^{2d-1}\cdot d\theta\wedge dt\wedge d\eps_1\wedge\bar{\eps}_1\wedge\dots\wedge d\eps_{d-1}\wedge d\bar{\eps}_{d-1}.
$$
It follows that
\begin{align*}
\vol(C_1(U)) & = \int_{C_1(U)}d\vol = \int_{\phi^{-1}(C_1(U))}\phi^*d\vol\\
& = \left(\frac{\imath}{2}\right)^{d-1}\cdot (-1)^{k-p}\cdot \int_0^{2\pi}d\theta\cdot\int_U\left(\int_0^{\frac{1}{\sqrt{h(\eps)}}}t^{2d-1}dt \right)d\eps_1d\bar{\eps}_1\dots d\eps_{d-1}d\bar{\eps}_{d-1}\\
                      & = \left(\frac{\imath}{2}\right)^{d-1}\cdot (-1)^{k-p}\cdot\frac{\pi}{d}\cdot\int_{U}\frac{1}{h^{d}(\eps)}d\eps_1d\bar{\eps}_1\dots d\eps_{d-1}d\bar{\eps}_{d-1}.\\
\end{align*}
By definition, we have $\mu(U)=\vol(C_1(U))$. Thus, $\mu$ is the measure associated with the volume form
$$
d\mu=\left(\frac{\imath}{2}\right)^{d-1}\cdot (-1)^{k-p}\cdot\frac{\pi}{d}\cdot\frac{1}{h^{d}(\eps)} \cdot d\eps_1d\bar{\eps}_1\dots d\eps_{d-1}d\bar{\eps}_{d-1}
$$
\end{proof}

It remains to show that $d\mu$ coincides with the right hand side of \eqref{eq:vol:form:proj:sp:express}.
We first notice that 
$$
\partial\bar{\partial}\ln h  = \partial \left(\frac{\bar{\partial} h}{h}\right) = \frac{\partial\bar{\partial}h}{h} -\frac{\partial h \wedge \bar{\partial} h}{h^2}.
$$
Now
$$
\partial\bar{\partial}h = \sum_{i=1}^{p-1}d\eps_i d\bar{\eps}_i-\sum_{i=p}^{k-1}d\eps_i d\bar{\eps}_i \quad \text{ and } \quad
\partial h = \ol{(\bar{\partial}h)} = \sum_{i=1}^{p-1}\bar{\eps}_i d\eps_i - \sum_{i=p}^{k-1}\bar{\eps}_i d\eps_i.
$$
imply
\begin{align*}
\left(\partial\bar{\partial}\ln h\right)^{k-1} & = \frac{\left(\partial\bar{\partial}h\right)^{k-1}}{h^{k-1}} - (k-1)\frac{\left(\partial\bar{\partial}h\right)^{k-2}\wedge\partial h\wedge \bar{\partial} h}{h^{k}}\\
& = (k-1)!\cdot(-1)^{k-p}\cdot \frac{h-\left(\sum_{i=1}^{p-1}|\eps_i|^2-\sum_{i=p}^{k-1}|\eps_i|^2\right)}{h^k} \cdot d\eps_1d\bar{\eps}_1\dots d\eps_{k-1}d\bar{\eps}_{k-1} \\
& = (k-1)!\cdot(-1)^{k-p}\cdot \frac{d\eps_1d\bar{\eps}_1\dots d\eps_{k-1}d\bar{\eps}_{k-1}}{h^k}.
\end{align*}
Since
$$
\left(\partial\bar{\partial}\ln h\right)^{k-1}\wedge d\eps_i = \left(\partial\bar{\partial}\ln h \right)^{k-1}\wedge d\bar{\eps}_i=0,
\; \quad \hbox{for all $i=1,\dots,k-1$}
$$
it follows 
$$
\left(\partial\bar{\partial}\ln h\right)^{k-1}\wedge\partial h = \left(\partial\bar{\partial}\ln h\right)^{k-1}\wedge\bar{\partial} h= \left(\partial\bar{\partial}\ln h\right)^{k-1}\wedge \partial\bar{\partial} h =0.
$$
For all $i=1,\dots,s$, let $h_i(\eps):= |\xi_i(\sigma(\eps))|^2/h$. We then have
$$
\partial \bar{\partial} h_i = \frac{d\eps_{k+i-1}\wedge d\bar{\eps}_{k+i-1}}{h} + \zeta_i,
$$
where $\zeta_i \in \Lambda^{1,1}(U)$ satisfies  $\left(\partial\bar{\partial}\ln h\right)^{k-1}\wedge\zeta_i=0$.
We thus have
$$
\left(-\imath\partial\bar{\partial}\ln h\right)^{k-1}\wedge(\frac{\imath}{2}\partial\bar{\partial}h_1)\wedge\dots\wedge (\frac{\imath}{2} \partial\bar{\partial}h_s)= (k-1)!\cdot\frac{(-1)^{p+1}}{2^s}\cdot\imath^{d-1} \cdot \frac{d\eps_1d\bar{\eps}_1\dots d\eps_{d-1}d\bar{\eps}_{d-1}}{h^{d}}
$$
which implies
$$
d\mu = \frac{\pi}{d}\cdot\frac{(-1)^{k-1}}{2^{k-1}(k-1)!}\cdot\left(-\imath\partial\bar{\partial}\ln h\right)^{k-1} \wedge(\frac{\imath}{2}\partial\bar{\partial}h_1)\wedge\dots\wedge (\frac{\imath}{2}\partial\bar{\partial}h_s)
$$
and \eqref{eq:vol:form:proj:sp:express} follows.
\end{proof}

Consider now a point $\xx:=(X,[\omega])]$ in $\Pb\Omega E_D(\kappa)$.
Recall that the zeros of $\omega$ are denoted by $\{x_1,\dots,x_n\}$, where $x_1,\dots,x_r$ are fixed, and $x_{r+i}$ and $x_{r+s+i}$ are permuted by the Prym involution.
Let $\sigma: \cU \to \Omega E_D(\kappa)$ be a section of the tautological line bundle over a neighborhood $\cU$ of $\xx$ in $\Pb\Omega E_D(\kappa)$.
Let us write $\sigma(\uu):=(X_{\uu},\omega_{\uu})$ for all $\uu \in \cU$.
Define $$
h(\uu):=||\omega_\uu||^2=\frac{\imath}{2}\cdot\int_{X_\uu}\omega_\uu\wedge\ol{\omega}_\uu.
$$
For each $i\in \{1,\dots,s\}$, we choose a path $c_i$ from $x_{r+i}$ to $x_{r+s+i}$.
If $\cU$ is small enough,  $c_i$ determines a path in $X_\uu$ (up to isotopy)  joining two zeros of $\omega_\uu$ that are permuted by the Prym involution of $X_\uu$. We abusively denote this path on $X_\uu$ again by $c_i$, and define a function $h_i:\cU\to \R^+$ by
$$
h_i(\uu):=\frac{\left|\int_{c_i}\omega_\uu\right|^2}{||\omega_\uu||^2}.
$$
As a consequence of Proposition~\ref{prop:push:meas:proj:sp} we get

\begin{Theorem}\label{th:vol:form:proj:diff:express}
The measure $\mu$ on $\Pb\Omega E_D(\kappa)$ is the one associated with a volume form $d\mu$. In a neighborhood of $\xx$ we have
\begin{equation}\label{eq:vol:form:proj:diff:express}
d\mu=\frac{\pi}{2(s+2)}\cdot(-\imath \partial\bar{\partial}\ln h)\wedge\left( \frac{\imath}{2}\cdot\partial\bar{\partial}h_1\right)\wedge\dots\wedge\left( \frac{\imath}{2}\cdot\partial\bar{\partial}h_s\right).
\end{equation}
\end{Theorem}
\begin{proof}
By Proposition~\ref{prop:affine:coord}, $\Omega E_D(\kappa)$ is locally modeled on the space $V=\wp^{-1}(W)\cap H^1(X,\{x_1,\dots,x_n\};\C)^-$, where $\dim_\C W=2$ and the restriction of $(.,.)$ to $W$ is non-degenerate. It follows that the rank of the Hermitian form defined by $(.,.)$ on $V$ is equal to $2$.
Let $\xi_i, \; i=1,\dots,s$, denote the element of $(H^1(X,\{x_1,\dots,x_n\};\C))^*$ defined by $c_i$. We can now apply  Proposition~\ref{prop:push:meas:proj:sp}, with $H=(.,.)$, $k=2$, and $d=s+2$ to conclude.
\end{proof}

Theorem~\ref{th:Siegel:Veech} will be derived from the following result, whose proof is given in \textsection~\ref{sec:vol:Omg:E:D:2:2}
\begin{Theorem}\label{th:vol:Omg:E:D:2:2}
	Let $D\in \N, \; D >4$, be an integer such that $D\equiv 0,1,4 \; [8]$ and $D$ is not a square.  If $4 \, | \, D$ then we have
	\begin{equation}\label{eq:vol:Omg:E:D:2:2:formula}
		\mu(\Pb\Omega E_D(2,2)^\odd)=\frac{\pi^2}{36} \left(\chi(W_D(2))+ b_D\chi(W_{D/4}(2))+ 9\chi(W_{D}(0^3))\right)
	\end{equation}
	where
	$$
	b_D= \left\{
	\begin{array}{ll}
		0 & \text{ if } D/4 \equiv 2,3 \, [4] \\
		4 & \text{ if } D/4 \equiv 0 \, [4] \\
		3 & \text{ if } D/4 \equiv 1 \, [8] \\
		5 & \text{ if } D/4 \equiv 5 \, [8]. \\
	\end{array}
	\right.
	$$
	
	If $D \equiv 1 \, [8]$ then we have
	\begin{equation}\label{eq:vol:Omg:E:D:2:2:D:odd}
		\mu(\Pb\Omega E_{D+}(2,2)^\odd) = \mu(\Pb\Omega E_{D-}(2,2)^\odd)=\frac{\pi^2}{72}\left( 2\chi(W_D(2))+9\chi(W_D(0^3))\right).
	\end{equation}
\end{Theorem}
\section{Prym eigenforms in genus three}\label{sec:eigen:form:g3}
\subsection{Generalities}\label{subsec:Prym:e:form:g3:general}
We now focus in the case where $(X,\omega)$ is a Prym  eigenform in $\Omega E_D(2,2)^\odd$.
Let $Y:=X/\langle \tau \rangle$, where $\tau$ is the Prym involution of $X$. Then we have $g(Y)=g(X)-2=1$. The Riemann-Hurwitz formula implies that the projection $X \to Y$ is branched over 4 points. This means that $\tau$ has exactly $4$ fixed points.
Since $\tau^*\omega=-\omega$, the zero set of  $\omega$ is invariant by $\tau$. It is not difficult to  see that $(X,\omega)\in \Omega \cM_3(2,2)^\odd$ if and only if the zeros of $\omega$ are permuted by $\tau$ (see \cite{LN:components}). It follows from Proposition~\ref{prop:affine:coord} and Lemma~\ref{lm:dual:basis:ker:abs:proj} that $\dim_\C \Omega E_D(2,2)^\odd=3$.     
The classification of the components of $\Omega E_D(2,2)^{\odd}$ is obtained in \cite{LN:components}. In particular, we  have that $\Omega E_D(2,2)^\odd$ is connected if $D \equiv 0 \mod 4$, and in the case $D\equiv 1 \mod 8$, $\Omega E_D(2,2)^\odd$ has two connected components denoted by $\Omega E_{D\pm}(2,2)^\odd$.

\begin{Lemma}\label{lm:Prym:var:1:2:polarization}
Let $(X,\omega)$ be a Prym eigenform in genus $3$ with Prym involution $\tau$. Then the intersection form on $H^1(X,\Z)^-$ is of type $(1,2)$, that is there is a basis $(a_1,b_1,a_2,b_2)$ of $H^1(X,\Z)^-$ in which the intersection form is given by the matrix $\left(\begin{smallmatrix} 0 & 1 & 0 & 0 \\ -1 & 0 & 0 & 0\\ 0 & 0 & 0 & 2\\ 0 & 0 & -2 & 0 \end{smallmatrix} \right)$.
\end{Lemma}
\begin{proof}
Let $p_1,\dots,p_4$ be the fixed points of the Prym involution, and $q_1,\dots,q_4$ their image in $Y=X/\langle \tau \rangle$ ($q_i$ is the image of $p_i$).  Then the restriction of the projection $\pi: X \to Y$ to $X-\{p_1,\dots,p_4\}$ is a covering map of degree $2$ from $X':=X\setminus \{p_1,\dots,p_4\}$ onto $Y'=Y-\{q_1,\dots,q_4\}$. Such a covering is determined up to homeomorphism by the image of $\pi_1(X')$ in $\pi_1(Y')$.
In this case, $\pi_1(X')$ is the kernel of a group morphism $\chi: \pi_1(Y')\to \Z/2\Z$, which sends the boundary of a small disc about $q_i$ to $1\in \Z/2\Z$, for all $i=1,\dots,4$.

Consider now a topological torus $S$ with $4$ marked points $s_1,\dots,s_4$. Denote by $S'$ the punctured surface $S-\{s_1,\dots,s_4\}$. Let $\chi, \chi': \pi_1(S')\to \Z/2\Z$ be two groups morphisms that map the boundary of a small disc about $s_i$ to $1$, for all $i=1,\dots,4$.
We claim that there always exist a homeomorphism $\varphi$ of $S$ fixing the set $\{s_1,\dots,s_4\}$ pointwise such that $\chi'=\chi\circ\varphi$.
To see this, we first remark that $\chi$ and $\chi'$ factor through some morphisms from $H_1(S',\Z)$ to $\Z/2\Z$.
One can always find a pair of simple closed curves $\{a,b\}$ (resp. a pair of simple closed curves $\{a',b'\}$) in $S'$  such that $(a,b)$ (resp. $(a',b')$) is a basis of $H_1(S,\Z)$ and $\chi(a)=\chi(b)=0$ (resp. $\chi'(a')=\chi'(b')=0$).
The complements of $a\cup b$ and of $a'\cup b'$ in $S$ are both topological disc that contains the points $\{s_1,\dots,s_n\}$ in their interior.
We deduce that there  exists a homeomorphism $\varphi: S \to S$ that fixes each of the points in $\{s_1,\dots,s_4\}$ and satisfies $\varphi(a)=a', \varphi(b)=b'$, which proves the claim.

The previous claim means that if $(X,\omega)$ and $(X',\omega')$ are two Prym eigenforms in genus $3$ then $H^1(X,\Z)^-\simeq H^1(X',\Z)^-$.  In \cite[\textsection 4]{LN:H4}, the statement of the lemma was shown for the case $(X,\omega)\in \Omega E_D(4)$. Thus, the same holds true for all Prym eigenform in genus $3$.
\end{proof}

The following lemma follows from direct calculations.
\begin{Lemma}\label{lm:matrix:generator:O:D}
Let $T\in \Mb_{4}(\Z)$ is a self-adjoint matrix with respect to the skew-symmetric form $J=\left(\begin{smallmatrix} 0 & 1 & 0 & 0 \\ -1 & 0 & 0 & 0\\ 0 & 0 & 0 & 2\\ 0 & 0 & -2 & 0 \end{smallmatrix} \right)$. Then we have $T=\left(
\begin{array}{cc}
e\cdot \Id_2 & 2B \\

B^* & f\cdot\Id_2
\end{array}
\right)
$, where $e,f\in \Z$, $B\in \Mb_2(\Z)$, and $\left(\begin{smallmatrix} a & b \\ c & d \end{smallmatrix}\right)^*=\left(\begin{smallmatrix} d & -b \\ -c & a \end{smallmatrix}\right)$.
\end{Lemma}

In what follows, given two complex numbers $\alpha$ and $\beta$, we define
$$
\alpha\wedge \beta :=\det\left(\begin{array}{cc}
	\Re(\alpha) & \Re(\beta) \\
	\Im(\alpha) & \Im(\beta)
\end{array} \right) =\Im(\bar{\alpha}\beta) \in \R.
$$
\begin{Proposition}\label{prop:eigen:form:per:eq}
Let $(X,\omega) \in \Omega \cM_3$ be a Prym eigenform for a quadratic  order $\cO_D$ in genus $3$. Let $\{a_1,b_1,a_2,b_2\}$ be a symplectic basis of $H_1(X,\Z)^-$, where $\langle a_1,b_1\rangle=1$ and $\langle a_2,b_2\rangle =2$.
Assume that $D$ is not a square.
Then there exists a generator $T$ of $\cO_D$ such that
\begin{itemize}
\item[(a)] the matrix of $T$ in the basis $\{a_1,b_1,a_2,b_2\}$ has the form $T=\left(
\begin{array}{cc}
e\Id_2 & 2B \\
B^* & 0_2 
\end{array}
\right)$, where $B=\left(\begin{smallmatrix} a & b \\ c & d \end{smallmatrix} \right)\in \Mb_2(\Z)$ satisfies $\gcd(a,b,c,d,e)=1$ and $D=e^2+8\det(B)$,

\item[(b)] $T^*\omega=\lambda\cdot\omega$, where $\lambda$ is a positive root of the polynomial $X^2-eX -2\det(B)$,

\item[(c)] $(\omega(a_2) \; \omega(b_2))=\frac{2}{\lambda}\cdot(\omega(a_1) \; \omega(b_1))\cdot B$.
\end{itemize}
As a consequence, for a given $D$, if $\omega(a_1)\wedge\omega(b_1)>0$ and $\omega(a_2) \wedge\omega(b_2)>0$, then the ratio $\omega(a_2)\wedge\omega(b_2)/\omega(a_1)\wedge\omega(b_1)$ belongs to a finite set.
\end{Proposition}
\begin{proof}
Let $T \in \End(\Prym(X,\tau))$ be a generator  of $\cO_D$. Since the action of $T$ on $H_1(X,\Z)^-$ is self-adjoint with respect to the intersection form $\langle.,.\rangle$,  by Lemma~\ref{lm:matrix:generator:O:D} it is given by a matrix of the form
$\left(
\begin{array}{cc}
e\cdot \Id_2 & 2B \\
B^* & f\cdot\Id_2
\end{array}
\right)
$, with $B=\left(\begin{smallmatrix} a & b \\ c & d \end{smallmatrix} \right) \in \Mb_2(\Z)$ in the basis $\{a_1,b_1,a_2,b_2\}$.
By replacing $T$ by $T-f$, we can assume that $f=0$. The condition that the subring of $\End(\Prym(X,\tau))$ generated by $T$ is proper means that $\gcd(e,a,b,c,d)=1$. Note that $T$ satisfies
$$
T^2=eT+2\det(B)\Id_4.
$$
Since $T$ generates $\cO_D$, we must have $D=e^2+8\det(B)$.
By assumption, there is a real number $\lambda$ such that $T^*\omega=\lambda\cdot\omega$. Thus we have
\begin{equation}\label{eq:eigenform:matrix:form}
(\omega(a_1) \; \omega(b_1) \; \omega(a_2) \; \omega(b_2))\cdot T=\lambda\cdot (\omega(a_1) \; \omega(b_1) \; \omega(a_2) \; \omega(b_2)).
\end{equation}
Note that $\lambda$ must be a root of the polynomial $P(X)=X^2- eX-2\det(B)$. If $D$ is not a square then $\det(B) \neq 0$ and $\lambda \neq 0$. Replacing $T$ by $-T$ if necessary, we can always suppose that $\lambda >0$.  
Equality \eqref{eq:eigenform:matrix:form} implies that
\begin{equation}\label{eq:eigenform:rel:periods:basis}
(\omega(a_2) \; \omega(b_2))=\frac{2}{\lambda}\cdot(\omega(a_1) \; \omega(b_1))\cdot B.
\end{equation}
It follows that 
$$
\omega(a_2)\wedge \omega(b_2)=\left(\frac{2}{\lambda}\right)^2\cdot\det(B)\cdot\omega(a_1)\wedge\omega(b_1).
$$
If $\omega(a_1)\wedge\omega(b_1)>0$ and $\omega(a_2)\wedge\omega(b_2)>0$, then 
$\det(B) >0$. Since $\det(B) < D$, it follows that $\det(B)$ belongs to a finite set. As a consequence $e$ also belongs to a finite set. Since $\lambda$ is the positive root of the polynomial $X^2 - eX-2\det(B)$, we conclude that
$$
\frac{\omega(a_2)\wedge\omega(b_2)}{\omega(a_1)\wedge\omega(b_1)} =\left( \frac{2}{\lambda}\right)^2\cdot\det(B)
$$
belongs to a finite set.
\end{proof}

Let $K_D=\Q(\sqrt{D})$.  
Considering homology with rational coefficients, we have the following
\begin{Lemma}\label{lm:D:no:square:hol:map:inj}
	Let $(a_1,b_1,a_2,b_2)$ be a basis of $H_1(X,\Q)^-$ such that $\langle a_i,b_i\rangle=1$. 
	Define $\hol: H_1(X,\Q)^-\to \C, \; c \mapsto \omega(c)$. 
	If $D$ is not a square then $\hol$ realizes an isomorphism of $\Q$-vector spaces from $H_1(X,\Q)^-$ and $K_D\cdot\omega(a_1)+K_D\cdot\omega(b_1) \subset \C$.
\end{Lemma}
\begin{proof}
 By the same arguments as in Proposition~\ref{prop:eigen:form:per:eq}, there is a generator of $\cO_D$ which is given in the basis $(a_1,b_1,a_2,b_2)$ by a matrix $T$ of the form $T=\left(\begin{smallmatrix} e\Id_2 & B \\ B^* & 0 \end{smallmatrix} \right)$, for some $B\in \Mb_2(\Q)$ satisfying $\det B \neq 0$, such that $T^*\omega=\lambda\omega$ with $\lambda\in \R_{>0}$. 
As a consequence, we  have
\begin{equation}\label{eq:sympl:basis:per:rel}
(\omega(a_2), \omega(b_2)) = (\omega(a_1), \omega(b_1))\cdot B',
\end{equation} 
where $B'=\frac{1}{\lambda}\cdot B \in \Mb_2(K_D )$.
	
We claim that $\omega(a_1)\wedge \omega(b_1) \neq 0$. To see this we remark that
\begin{align*}
\Aa(X,|\omega|) =\frac{\imath}{2}\int_X\omega\wedge\ol{\omega} &=\Im(\ol{\omega}(a_1)\omega(b_1)) + \Im(\ol{\omega}(a_2)\omega(b_2))\\ 
& = \omega(a_1)\wedge\omega(b_1)+ \omega(a_2)\wedge\omega(b_2)\\
& = (1+\det B')\omega(a_1)\wedge \omega(b_1).
\end{align*}
Since $\Aa(X,|\omega|) >0$, we must have $\omega(a_1)\wedge \omega(b_1) \neq 0$.

For all $c\in H_1(X,\Q)^-$, let $V(c) \in \Q^4$ be the coordinates of $c$ in the basis $(a_1,b_1,a_2,b_2)$. 
It follows from \eqref{eq:sympl:basis:per:rel} that
$$
\hol(c)=\omega(c)=(\omega(a_1), \; \omega(b_1))\cdot \left(\Id_2 \;  B'\right)\cdot V(c).
$$
Thus it suffices to shows that the $\Q$-linear map $A: \Q^4 \to \Q(\sqrt{D})^2, \, v \mapsto (\Id_2 \; B')\cdot v$ is an isomorphism.
Since $\dim_\Q(K_D\cdot\omega(a_1)+K_D\cdot \omega(b_1))=4$, we only need to show that $A$ is injective. Since $B'=B/\lambda$, where $B\in\Mb_2(\Q)$, $\det B\neq 0$, and $\lambda\not\in \Q$, we get the desired conclusion.
\end{proof}

\subsection{Periodicity and cylinder decompositions}\label{subsec:commplete:per:n:cyl:dec}
A translation surface is said to be {\em completely periodic} if it satisfies the following condition: for any direction $\theta\in \R\Pb^1$, if there is a regular closed geodesic in direction $\theta$, all trajectories in the same direction are either saddle connections or regular closed geodesics.
If the latter occurs, the surface is then decomposed into a union of finitely many cylinders in direction $\theta$. Throughout this paper, by a {\em cylinder diagram} we will mean the combinatorial data associated with such decompositions. In particular, given two surfaces $(X,\omega)$ and $(X',\omega')$, where $(X,\omega)$ has a cylinder decomposition in direction $\theta$, while $(X',\omega')$ has a cylinder decomposition in direction $\theta'$, we say that $X$ and $X'$ have the same cylinder diagram if there is a homeomorphism from $X$ to  $X'$ mapping a saddle connection in direction $\theta$ of $X$ onto a saddle connection in the direction $\theta$ on $X'$ and respecting the orders of the zeros. Such a map must send a cylinder in direction $\theta$ on $X$ onto a cylinder in direction $\theta'$ on $X'$.

Prym eigenform loci are examples of $\GL_2(\R)$-orbit closures of rank $1$, that is the $\Omega E_D(\kappa)$ are locally parametrized (via the period mappings) by some vector subspaces of $H^1(X,\{x_1,\dots,x_n\};\C)$ whose projection in $H^1(X,\C)$ are two-dimensional.
It is proved in \cite{W:cyldef} that all surfaces in a rank one orbit closure are completely periodic (see also \cite{Calta, LN:finite} for the case of Prym eigenforms).

If $(X,\omega)$ has a cylinder $C$ then this cylinder persists on every surface (in the same stratum) close enough to $(X,\omega)$. This means that any surface in a neighborhood of $(X,\omega)$ has a cylinder corresponding to $C$.
In the case  $(X,\omega)$ belongs to  a rank one orbit closure, this property implies that whenever $X$ admits a cylinder decomposition in some direction $\theta\in \R\Pb^1$, we have a corresponding cylinder decomposition in all surfaces close enough in the same orbit closure.
The cylinder decomposition  on $X$  is then said to be {\em stable} if the corresponding cylinder decomposition on all surfaces nearby has the same diagram  (see \cite{LN:finite, LNW:finite}). In the case of $\Omega E_D(2,2)^\odd$ a cylinder decomposition  is stable if and only if each saddle connection in the direction of the cylinders joins a zero to itself. This notion of stability is of interest since we have

\begin{Proposition}~\label{prop:stable:cyl:dec}
Let $(X,\omega)$ be a surface in some Prym eigenform locus $\Omega E_D(\kappa)$. If $(X,\omega)$ admits a cylinder decomposition in some direction $\theta \in \R\Pb^1$, then for all $(X',\omega')$ in an open dense subset of a neighborhood of $(X,\omega)$ in $\Omega E_D(\kappa)$, the corresponding cylinder decomposition on $(X',\omega')$ is stable.
\end{Proposition}
\begin{proof}
See \cite[\textsection 4]{LN:finite}.
\end{proof}
\begin{Remark}\label{rem:stable:cyl:dec}
If the cylinder decomposition on $(X,\omega)$ is stable, then by definition, the corresponding cylinder decompositions on nearby surfaces are also stable and have the same diagram. Otherwise, the neighborhood of $(X,\omega)$ in $\Omega E_D(\kappa)$ is partitioned into several regions, the corresponding cylinder decompositions in each region are stable and have the same diagram.
\end{Remark}

\subsection{Prototypes and stable cylinder diagrams}\label{subsec:cyl:dec:prototype}
Every Prym eigenform in $\Omega\cM_3(2,2)^\odd$ is the canonical double cover of a quadratic differential in the stratum $\Qcal(4,-1^4)$. 
If $(X,\omega) \in \Omega \cM_3(2,2)^\odd$ is horizontally periodic, and the associated cylinder diagram is stable (that is each horizontal saddle connection joins a zero of $\omega$ to itself), then $(X,\omega)$ must have four horizontal cylinders. By inspecting the cylinder diagrams with $4$ cylinders which admit an involution exchanging the two zeros and having exactly $4$ fixed points (the latter condition means that the involution fixes two cylinders and exchanges the two remaining ones), one obtains the following

\begin{Proposition}\label{prop:stable:cyl:dec:models:H22}
There are $4$ stable diagrams for cylinder decompositions of translation surfaces that are canonical double covers of half-translation surfaces in $\Qcal(4,-1^4)$. Those diagrams are shown in Figure~\ref{fig:stable:diag:H:2:2:odd}.
By convention, in all diagrams, the cylinders $C_1$ and $C_2$ are fixed, while the cylinders $C_3$ and $C_4$ are exchanged by the Prym involution.
In Case I.A and Case I.B, all cylinders have  distinct zeros on their top and bottom boundary. In Case II.A and Case II.B, there is a pair of homologous cylinders which are exchanged by the Prym involution.
\begin{figure}[htb]
\begin{minipage}[t]{0.25\linewidth}
\centering
\begin{tikzpicture}[scale=0.40]
\draw (0,0)--(0,4)--(2,4)--(2,2)--(4,2)--(4,4)--(6,4)--(6,0)--(4,0)--(4,-2)--(2,-2)--(2,0)--cycle;
\foreach \x in {(0,0),(0,4),(2,4),(4,4),(6,4),(6,0),(4,0),(2,0)} \filldraw[fill=black] \x circle (3pt);
\foreach \x in {(2,2),(4,2),(4,-2),(2,-2),(0,2),(6,2)} \filldraw[fill=white] \x circle (3pt);
\draw(1,4) node[above] {\tiny $1$};
\draw(3,2) node[above] {\tiny $2$};
\draw(5,4) node[above] {\tiny $3$};
\draw(1,0) node[below] {\tiny $1$};
\draw(3,-2) node[below] {\tiny $2$};
\draw(5,0) node[below] {\tiny $3$};

\draw(3,1) node {\tiny $C_2$};
\draw(3,-1) node {\tiny $C_1$};
\draw(5,3) node {\tiny $C_3$};
\draw(1,3) node {\tiny $C_4$};

\draw(3,-3.5) node {\small {\rm Case I.A}};
\end{tikzpicture}
\end{minipage}
\begin{minipage}[t]{0.25\linewidth}
\centering
\begin{tikzpicture}[scale=0.40]
\draw (0,0)--(0,2)--(-2,2)--(-2,6)--(0,6)--(0,8)--(2,8)--(2,6)--(2,4)--(4,4)--(4,0)--cycle;
\foreach \x in {(0,0),(0,8),(2,8),(2,4),(4,4),(4,0),(-2,4),(2,0)} \filldraw[fill=black] \x circle (3pt);
\foreach \x in {(-2,6),(-2,2), (0,6),(0,2), (2,6), (4,2)} \filldraw[fill=white] \x circle (3pt);
\draw(-1,6) node[above] {\tiny $1$};
\draw(1,8) node[above] {\tiny $2$};
\draw(3,4) node[above] {\tiny $3$};
\draw(-1,2) node[below] {\tiny $1$};
\draw(1,0) node[below] {\tiny $2$};
\draw(3,0) node[below] {\tiny $3$};

\draw (-3,7) node {\tiny $C_1$} (-3,5) node {\tiny $C_3$} (-3,3) node {\tiny $C_2$} (-3,1) node {\tiny $C_4$};

\draw (1,-1.5) node {\small {\rm Case I.B}};
\end{tikzpicture}
\end{minipage}
\begin{minipage}[t]{0.2\linewidth}
\centering
\begin{tikzpicture}[scale=0.40]
\draw (0,0)--(0,8)--(4,8)--(4,6)--(2,6)--(2,4)--(4,4)--(4,2)--(2,2)--(2,0)--cycle;

\foreach \x in {(0,0),(0,8),(2,8), (4,8),(4,2),(2,2),(2,0),(0,2)} \filldraw[fill=black] \x circle (3pt);
\foreach \x in {(4,6),(2,6),(2,4),(4,4),(0,4),(0,6)} \filldraw[fill=white] \x circle (3pt);

\draw(1,8) node[above] {\tiny $1$};
\draw(3,4) node[above] {\tiny $3$};
\draw(3,8) node[above] {\tiny $2$};
\draw(1,0) node[below] {\tiny $1$};
\draw(3,2) node[below] {\tiny $2$};
\draw(3,6) node[below] {\tiny $3$};

\draw(0,1) node[left] {\tiny $C_4$};
\draw(0,3) node[left] {\tiny $C_2$};
\draw(0,5) node[left] {\tiny $C_3$};
\draw(0,7) node[left] {\tiny $C_1$};

\draw (1,-1.5) node {\small {\rm Case II.A}};
\end{tikzpicture}
\end{minipage}
\begin{minipage}[t]{0.2\linewidth}
\centering
\begin{tikzpicture}[scale=0.40]
\draw (0,0)--(0,8)--(4,8)--(4,6)--(2,6)--(2,4)--(4,4)--(4,2)--(2,2)--(2,0)--cycle;
\foreach \x in {(0,0),(0,8),(2,8), (4,8),(4,6),(2,6),(2,0), (0,6)} \filldraw[fill=black] \x circle (3pt);
\foreach \x in {(0,4),(0,2),(2,4),(2,2),(4,4),(4,2)} \filldraw[fill=white] \x circle (3pt);
\draw(1,8) node[above] {\tiny $1$};
\draw(3,4) node[above] {\tiny $3$};
\draw(3,8) node[above] {\tiny $2$};
\draw(1,0) node[below] {\tiny $1$};
\draw(3,2) node[below] {\tiny $3$};
\draw(3,6) node[below] {\tiny $2$};

\draw(0,1) node[left] {\tiny $C_1$};
\draw(0,3) node[left] {\tiny $C_3$};
\draw(0,5) node[left] {\tiny $C_2$};
\draw(0,7) node[left] {\tiny $C_4$};

\draw (1,-1.5) node {\small {\rm Case II.B}};
\end{tikzpicture}
\end{minipage}

\caption{Stable cylinder diagrams of double covers of surfaces in $\Qcal(4,-1^4)$}
\label{fig:stable:diag:H:2:2:odd}
\end{figure}
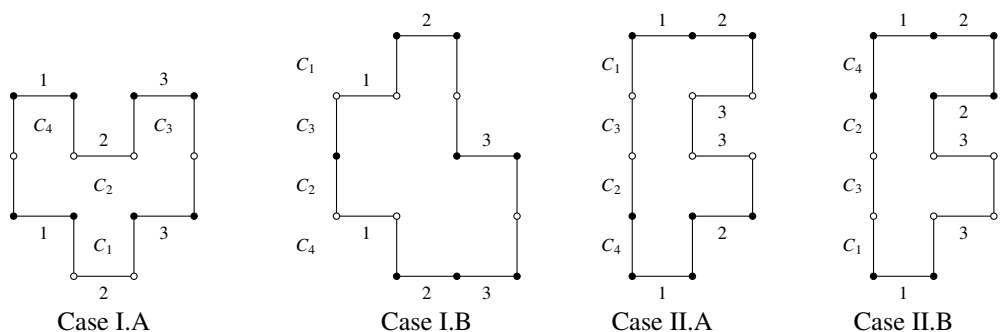
\end{Proposition}

Given a discriminant $D\in \N, \; D\equiv 0,1, 4 \mod 8$, we will
call a quadruple $\frakp=(a,b,d,e)\in \Z^4$ a {\em cylinder prototype} of discriminant $D$ if $\frakp$ satisfies the followings
$$
(\Pcal_{D,\cyl}) \qquad \left\{
\begin{array}{l}
a>0, \;  d >0, 0 \leq b < \gcd(a,d),\\
D=e^2+8ad,\\
\gcd(a,b,d,e)=1.
\end{array}
\right.
$$
The set of cylinder prototypes for a discriminant $D$ is denoted by $\Pcal_{D,\cyl}$.
For each $\frakp \in \Pcal_{D,\cyl}$, we define
$$
\lambda(\frakp):=\frac{e+\sqrt{D}}{2}.
$$
Consider a surface $(X,\omega) \in \Omega E_D(2,2)^\odd$, which admits a stable cylinder decomposition in the horizontal direction. By Proposition~\ref{prop:stable:cyl:dec:models:H22}, the corresponding cylinder diagram of $(X,\omega)$ is given by one of the four cases in Figure~\ref{fig:stable:diag:H:2:2:odd}. We will label the horizontal cylinders of $(X,\omega)$ by $C_1,\dots,C_4$ following the models shown in Figure~\ref{fig:stable:diag:H:2:2:odd}. For each $i\in \{1,\dots,4\}$, the circumference (width) and the height of $C_i$ are denoted by $\ell_i$ and $h_i$ respectively. We have

\begin{Proposition}\label{prop:cyl:dec:length:ratios:22}
Assume that $(X,\omega) \in \Omega E_D(2,2)^\odd$ admits a stable cylinder decomposition in the horizontal direction.
Then there is a prototype $\fracp=(a,b,d,e) \in \Pcal_{D,\cyl}$ such that
\begin{itemize}
\item[(i)] if  the corresponding cylinder diagram  is as in Case I.A, then
\begin{itemize}
\item[$\bullet$]  $\displaystyle \frac{\ell_3}{\ell_1}=\frac{\ell_4}{\ell_1}=\frac{a}{\lambda}$,

\item[$\bullet$] $\displaystyle \frac{h_2+h_4}{h_1+h_2}=\frac{h_2+h_3}{h_1+h_2}=\frac{d}{\lambda}$
\end{itemize}
where $\lambda:=\lambda(\fracp)$.

\item[(ii)] If  the corresponding cylinder diagram  is as in Case I.B, then
\begin{itemize}
\item[$\bullet$] $\displaystyle \frac{\ell_3-\ell_1}{\ell_1}=\frac{\ell_4-\ell_1}{\ell_1}=\frac{a}{\lambda}$,

\item[$\bullet$] $\displaystyle \frac{h_2+h_3}{h_1+h_2+h_3+h_4}=\frac{h_2+h_4}{h_1+h_2+h_3+h_4}=\frac{d}{\lambda}$.
\end{itemize}

\item[(iii)] If  the corresponding cylinder diagram  is as in Case II.A, then
\begin{itemize}
\item[$\bullet$] $\displaystyle \frac{\ell_3}{\ell_1}=\frac{\ell_4}{\ell_1}=\frac{a}{\lambda}$,

\item[$\bullet$] $\displaystyle \frac{h_3}{h_1+h_2}=\frac{h_4}{h_1+h_2}=\frac{d}{\lambda}$.
\end{itemize}

\item[(iv)] If the corresponding cylinder diagram  is as in Case II.B, then
\begin{itemize}
\item[$\bullet$] $\displaystyle \frac{\ell_3}{\ell_1}=\frac{\ell_4}{\ell_1}=\frac{a}{\lambda}$,

\item[$\bullet$] $\displaystyle \frac{h_3}{h_1+h_2}=\frac{h_4}{h_1+h_2}=\frac{d}{\lambda}$.
\end{itemize}
\end{itemize}
\end{Proposition}
\begin{proof}
The idea is to look for a symplectic basis $\{a_1,b_1,a_2,b_2\}$ of $H_1(X,\Z)^-$, where $a_1$ and $a_2$ are combinations of core curves of the horizontal cylinders.  We only give the proof for  Case I.A. Recall that in this case the Prym involution fixes $C_1, C_2$ and permutes $C_3$ with $C_4$. As a consequence, we have $\ell_1=\ell_3, \, h_1=h_3$.

Let $a_1$  be a core curve of $C_1$ and $b_1$  a simple closed curve composed by a segment that crosses $C_1$ and  segment crossing $C_2$ which is disjoint from $C_3$ and $C_4$.
Let $a'_2$ be a core curve of $C_3$ and $a''_2$ a core curve of $C_4$.
Let $b'_2$ be a simple closed curve composed by a segment that crosses $C_3$ and a segment that crosses $C_2$ (and disjoint from the cylinders $C_1$ and $C_4$). Similarly, let $b''_2$ be a simple closed curve which is composed by a segment that crosses $C_4$ and a segment that crosses $C_2$.
Define $a_2:=a'_2+a''_2$ and $b_2=b'_2+b''_2$. Then $\{a_1,b_1,a_2,b_2\}$ is a symplectic basis of $H_1(X,\Z)^-$ satisfying
$$
\langle a_1,b_1 \rangle =1, \quad \langle a_2,b_2 \rangle =2, \quad \langle a_1,a_2\rangle=\langle b_1, b_2 \rangle=\langle a_1,b_2\rangle=\langle a_2, b_1\rangle=0.
$$
We have
$$
\left\{\begin{array}{l}
\omega(a_1)=\ell_1, \quad \omega(a_2)=\ell_3+\ell_4=2\ell_3, \\
\Im(\omega(b_1))=h_1+h_2, \quad \Im(\omega(b_2))= 2(h_2+h_3).
\end{array}
\right.
$$
Rescaling $\omega$ by using $\GL_2^+(\R)$, we can assume that $\ell_1=1$ and $h_1+h_2=1$. Let us write $\omega(a_2)=x+ \imath y$, and $\omega(b_2)=z+\imath t$. Since $a'_2$ and $a''_2$ are core curves of horizontal cylinders, we must have $y=0$.

Let $T  \in \End(\Prym(X,\tau))$ be the generator of $\cO_D$ in Proposition~\ref{prop:eigen:form:per:eq}.
The matrix of $T$ in the basis $\{a_1,b_1,a_2,b_2\}$ is of the form $\left(\begin{smallmatrix}e\cdot\Id_2 & 2B \\ B^* & 0 \end{smallmatrix} \right)$, with $B=\left(\begin{smallmatrix} a & b \\ c & d \end{smallmatrix} \right) \in \Mb_2(\Z)$. By assumption, we have
\begin{equation}\label{eq:e:form:condition:mat:prod}
\left(
\begin{array}{cccc}
1 & 0 & x & z \\
0 & 1 & 0 & t
\end{array}
\right)\cdot \left(
\begin{array}{cc}
e\cdot \Id_2 & 2B\\
B^* & 0 \\
\end{array}
\right)=\lambda\cdot\left(
\begin{array}{cccc}
1 & 0 & x & z \\
0 & 1 & 0 & t
\end{array}
\right)
\end{equation}
which is equivalent to
$$
\left(\begin{array}{cc} e & 0 \\ 0 & e \end{array}\right) + \left(\begin{array}{cc} x & z \\ 0 & t \end{array}\right)\cdot \left(\begin{array}{cc} d & -b \\ -c & a \end{array}\right) = \lambda\cdot \Id_2, \quad \text{ and } \quad
2\left(\begin{array}{cc} a & b \\ c & d \end{array}\right) = \lambda\cdot\left(\begin{array}{cc} x & z \\ 0 & t\end{array}\right).
$$
Recall that $\lambda \in \R_{>0}$. It follows that $c=0, x=\frac{2a}{\lambda}$,  and  $t=\frac{2d}{\lambda}$.
Since $x=\omega(a_2)>0$, and $t=\Im(\omega(b_2))>0$, $a$ and $d$ must be positive integers.
Note that the cycles $b_1$  (resp. $b_2$) are only determined up to a multiple of $a_1$ (resp. a multiple of $a_2$).
Replacing $b_1$ by $b_1+ma_1$ and $b_2$ by $b_2+na_2$ amounts to change the tuple $(a,b,d,e)$ into $(a,b-na+md,d,e)$. Thus we can always choose a basis $(a_1,b_1,a_2,b_2)$ such that $0\leq b < \gcd(a,d)$.
By Proposition~\ref{prop:eigen:form:per:eq}, we have $\gcd(a,b,d,e)=1$, $D=e^2+8\det(B)=e^2+8ad$, and $\lambda$ is the positive root of the polynomial $x^2=ex+2ad$, that is $\lambda=\frac{e+\sqrt{D}}{2}$.
In particular, we have $(a,b,d,e) \in \Pcal_{D,\cyl}$.

Recall that we have $\omega(a_1)=\ell_1=1, \; \omega(a_2)=2\ell_3=x$,  $\Im(\omega(b_1))=h_1+h_2=1, \; \Im(\omega(b_2))=2(h_2+h_3)=t$. Therefore, we get
$$
\frac{\ell_3}{\ell_1}=\frac{x}{2}=\frac{a}{\lambda}, \quad \text{ and } \quad \frac{h_2+h_3}{h_1+h_2}=\frac{t}{2}=\frac{d}{\lambda}
$$
as desired.
\end{proof}


\section{Admissible covers}\label{sec:adm:covers}
To apply tools from complex analytic geometry, one needs  ``good" compactifications of Prym eigenform loci.  A  natural compactification of $\Pb\Omega E_D(2,2)^\odd$ is its closure in the projectivized Hodge bundle $\Pb\Omega\cM_3$. However, information about the Prym involution, which is essential to the definition of Prym eigenforms, may be lost in the boundary of this compactification. For this reason, it is more convenient to compactify those loci in the moduli space of admissible double covers.
Here below, we will provide some essential properties of objects parametrized by this moduli space.
For a comprehensible  introduction to  the  notion of admissible covers, we refer to \cite{HM:Invent82, HM98} and \cite[Chap. XVI]{ACG11}.

Let $(X,\omega)$ be a Prym eigenform in $\Omega E_D(2,2)^\odd$.  Then the Prym involution $\tau$ has four fixed points and permutes the pair of zeros of $\omega$.  The quotient $Y=X/\langle\tau\rangle$ is an elliptic curve with $4$ marked points $y_1,\dots,y_4$ that are the images of the fixed points of $\tau$.  In addition, we have another marked point $y_5$ coming from  the pair of zeros permuted by $\tau$. Thus, each $(X,\omega)$ corresponds to an element  $(Y,y_1,\dots,y_4,y_5)$ of $\cM_{1,5}$. By construction, $X$ is a double cover of $Y$ that is ramified over the points $y_1,\dots,y_4$. To get an adequate compactification of $\Pb\Omega E_D(2,2)^\odd$, one needs to extend the construction of the associated double covers to the boundary points of $\ol{\cM}_{1,5}$

Let $(E,q_1,\dots,q_4,q_5)$ be a pointed stable curve representing a point in $\ol{\cM}_{1,5}$.
An {\em admissible double cover} of $(E,q_1,\dots,q_5)$ with profile $(4,1)$ is a stable curve $(C,p_1,\dots,p_4,p_5,p'_5\})$ together with a map $f: C \to E$ such that
\begin{itemize}
\item[$\bullet$] $f^{-1}(\{q_i\})=\{p_i\}, \; i=1,\dots,4$,

\item[$\bullet$] $f^{-1}(\{q_5\})=\{p'_5,p'_5\}$,

\item[$\bullet$] the restriction of $f$ to the smooth part of $C\setminus\{p_1,\dots,p_4\}$ is a covering map of degree $2$.

\item[$\bullet$] $f$ maps the nodes of $C$ to the nodes of $E$.
\end{itemize}
Denote by $\ol{\cB}_{4,1}$ the moduli space of such admissible double covers. One can alternatively define  $\ol{\cB}_{4,1}$ as the moduli space of stable pointed curve  $(C,p_1,\dots,p_5,p'_5)$ of genus $3$ together with an involution $\tau$ such that
\begin{itemize}
\item[$\bullet$] $\tau(p_i)=p_i$, for all  $i=1,\dots,4$, and no other smooth point of $C$ is fixed by $\tau$,

\item[$\bullet$] $\tau(p_5)=p'_5$,

\item[$\bullet$] at any node of $C$ fixed by $\tau$, each local component through this node is mapped to itself.
\end{itemize}
Note that the fixed points of $\tau$ on $C$ are numbered globally, but the pair of points that are permuted by $\tau$ are not.
Let $\cB_{4,1}$ denote the subset of $\ol{\cB}_{4,1}$ consisting of tuples $(C,p_1,\dots,p_5,p'_5,\tau)$ where $C$ is smooth.
It is well known that $\cB_{4,1}$ is an open dense subset of $\ol{\cB}_{4,1}$, and both $\cB_{4,1}, \ol{\cB}_{4,1}$ are complex orbifolds (see for instance~\cite{ACV03} or \cite[Chap. XVI]{ACG11}).

By construction, one has two natural maps: $\rho_1: \ol{\cB}_{4,1}\to \ol{\cM}_{1,5}$  is the map which associates to $\xx:=(C,p_1,\dots,p_5,p'_5,\tau)$ the pointed curve $(E,q_1,\dots,q_5)$ where $E:=C/\langle \tau \rangle$, and $q_i$ is the image of $p_i$. The map $\rho_2: \ol{\cB}_{4,1} \to \ol{\cM}_3$ is the one which associates to $\xx$ the stable model of the curve obtained from $C$ without the marked points.

Let us denote by $\Omega\ol{\cB}_{4,1}$ the pullback of the Hodge bundle over $\cM_3$ to $\ol{\cB}_{4,1}$ by $\rho_2$. The fiber of $\Omega \ol{\cB}_{4,1}$ over $\xx \sim (C,p_1,\dots,p_5,p'_5,\tau)$ can be identified with $H^0(C,\omega_C)$, where $\omega_C$ is the dualizing sheaf of $C$.

For all $\xx \in \ol{\cB}_{4,1}$, let $\Omega^-(C,\tau)$ denote the space $\{\eta\in H^0(C,\omega_C), \; \tau^*\omega=-\omega\}$.
Note that we have $\dim_\C\Omega^-(C,\tau)=2$.
Let $\Omega'\ol{\cB}_{4,1}$ denote the subbundle of $\Omega \ol{\cB}_{4,1}$ whose fiber over $\xx$ is $\Omega^-(C,\tau)$.
Then $\Omega'\ol{\cB}_{4,1}$ is a rank two holomorphic subbundle of $\Omega\ol{\cB}_{4,1}$.
Let $\Pb\Omega' \ol{\cB}^-_{4,1}$ denote the projective bundle associated to $\Omega' \ol{\cB}_{4,1}$.

Now, given a positive integer $D>1, \; D \equiv 0,1,4 \mod 8$, we denote by $\Omega\cX_D$ the subset of $\Omega' \cB_{4,1}$ consisting of tuples $(C,p_1,\dots,p_5,p'_5,\tau,\omega)$, where $\xx=(C,p_1,\dots,p_5,p'_5,\tau) \in \cB_{4,1}$ and $\omega \neq 0$ is an element of $\Omega^-(C,\tau)$ satisfying the followings
\begin{itemize}
\item[$\bullet$] $\div(\omega)=2p_5+2p'_5$,

\item[$\bullet$] $\End(\Prym(C,\tau))$ contains a self-adjoint proper subring isomorphic to $\Ocal_D$ for which $\omega$ is an eigenform.
\end{itemize}
The closure of $\Omega\XD$ in $\Omega'\ol{\cB}_{4,1}$ is denoted by $\Omega\ol{\cX}_D$. The images of $\Omega \ol{\cX}_D$ and $\Omega\cX_D$ in $\Pb\Omega'\ol{\cB}_{4,1}$ are denoted by $\XD$ and $\ol{\cX}_D$ respectively.

\begin{Proposition}\label{prop:adm:cov:e:forms:degree}
Let $\hat{\rho}_2: \Pb\Omega\ol{\cB}_{4,1} \to \Pb\Omega\cM_3$ be the map induced by $\rho_2$. Then for all discriminant $D \geq 9, \; D\equiv 0,1,4 \, [8]$, we have $\hat{\rho}_2(\cX_D)=\Pb\Omega E_D(2,2)^\odd$ and $\deg (\hat{\rho}_{2\left|\cX_D\right.})=4!=24$.
\end{Proposition}
\begin{proof}
It is clear from the definition that $\hat{\rho}_2(\cX_D)=\Pb\Omega E_D(2,2)^\odd$.

Assume that $D\neq 9$. Let $(X, [\omega])$ be an element of  $\Pb\Omega E_D(2,2)^\odd$ (here $\omega \in \Omega^-(X)\setminus\{0\}$ and $[\omega]$ denotes the complex line generated by $\omega$ in $\Omega(X)$).
It follows from \cite[Th. 3.1]{LN:components} that the Prym involution $\tau$, which is implicitly involved in the definition of $\Omega E_D(2,2)^\odd$, is unique.
The preimage of $(X, [\omega])$ by $\hat{\rho}_2$ consists of tuples $(X,x_1,\dots,x_5,x'_5,\tau,[\omega])$, where $\{x_1,\dots,x_4\}$ is the set of fixed points of $\tau$ and $\{x_5,x'_5\}$ are the zeros of $\omega$ (that are permuted by $\tau$). It is clear that $\{x_5,x'_5\}$ is uniquely determined by $[\omega]$, while the set $\{x_1,\dots,x_4\}$ is determined by $\tau$. Since $\tau$ is unique, different points in the preimage corresponds to different numberings of the fixed points of $\tau$. Thus the preimage contains $4!=24$ points.

If $D=9$ then $\tau$ is not unique. However, the arguments of \cite[Th. 3.1]{LN:components} actually show all the different Prym involutions are conjugate by automorphisms of $X$. Therefore, we get the same conclusion.
\end{proof}

By a slight abuse of notation, we will denote by $d\mu$ the pullback of the volume forms on $\Pb\Omega E_D(2,2)^\odd$ to $\cX_D$.
It follows from Proposition~\ref{prop:adm:cov:e:forms:degree} that we have
\begin{Corollary}\label{cor:vol:XD:n:POmgE:D:2:2} The volumes of $\cX_D$ and $\Pb\Omega E_D(2,2)^\odd$ are related by
\begin{equation}\label{eq:vol:XD:n:POmgE:D:2:2}
	\mu(\cX_D)=24\mu(\Pb\Omega E_D(2,2)^\odd).
\end{equation}
\end{Corollary}

\section{Stratification of the boundary of $\ol{\cX}_D$}\label{sec:classify:bdry:str:XD}
Define $\partial\ol{\cX}_D:=\ol{\cX}_D-\cX_D$.
We have naturally a stratification of $\partial\ol{\cX}_D$ where each stratum contains Abelian differentials on stable curves with the same topology. 
Theorem~\ref{th:bdry:eigen:form:H22} here below  gives the exhaustive list of strata of $\partial\ol{\cX}_D$. 
These strata will be labeled according to the topology of  the quotient by the Prym involution of the underlying curves 
(the quotient is a stable pointed curve in $\ol{\cM}_{1,5}$).
More precisely, we will label of each stratum by $\cS_{x,y}^\alpha$, where $x$ (resp. $y$) is the number of separating (resp. non-separating) nodes on the quotient, and $\alpha$ is a letter which is added to distinguish different strata whose corresponding curves in $\ol{\cM}_{1,5}$ have the same topology. The letter $\alpha$ is omitted in the case there is only one stratum for which the quotient curve has $x$ separating nodes and $y$ non-separating nodes.

\begin{Theorem}\label{th:bdry:eigen:form:H22}
Assume that $D$ is not a square. 
Let $\pp=(C,p_1,\dots,p_5, p'_5,\tau,[\xi])$ be a point in $\partial\ol{\cX}_D$.
Then $\partial\ol{\cX}_D$ consists of the following strata 

\begin{enumerate}
\item $\cS_{1,0}$ is the stratum containing $\pp$ such that $C$ has two irreducible components, denoted $C'$ and $C''$ meeting at one node  such that
\begin{itemize}
	\item[.] $C'$ is isomorphic to $\Pb^1$ and contains $\{p_5,p'_5\}$ and one point in $\{p_1,\dots,p_4\}$,
	
	\item[.] $C''$ is a Riemann surface of genus three containing three points in $\{p_1,\dots,p_4\}$, 
	
	\item[.] the differential $\xi$ vanishes identically on $C'$ and $\xi'':=\xi_{\left|C''\right.} \in \Omega \cM_3(4)$, the unique zero of $\xi''$ is located at the node between $C''$ and $C'$.
\end{itemize}   

\item $\cS_{2,0}^a$ is the stratum  where $C$ has four irreducible components, denoted $C'_1, C'_2$, $C''_1, C''_2$, such that
\begin{itemize}
\item[.] $C'_1$ is an elliptic curve, $C'_2$ is isomorphic to $\Pb^1$, $C''_1$ and $C''_2$ are two isomorphic elliptic curves,
\item[.] $C'_1$ contains $3$ points in $\{p_1,\dots,p_4\}$, $C'_2$ contains one point in $\{p_1,\dots,p_4\}$ and $\{p'_5, p''_5\}$,
\item[.] $C'_2$ meets each of $C'_1$, $C''_1$, and $C''_2$ at one node,
\item[.] $\xi$ vanishes identically on $C'_2$ and is nowhere vanishing on $C'_1\cup C''_1\cup C''_2$.
\end{itemize}

\item $\cS_{2,0}^b$ is the stratum where $C$ has three irreducible components, denoted by $C'_1, C'_2$, and $C''$, such that
\begin{itemize}
\item[.] $C'_1$ (resp. $C'_2$) is isomorphic to $\Pb^1$ and contains two points in $\{p_1,\dots,p_4\}$,
\item[.] $C''$ is an elliptic curve which contains $\{p'_5, p''_5\}$,
\item[.] $C'_1$ (resp. $C'_2$) intersects $C''$ at two nodes,
\item[.] $\xi$ is non-trivial on all irreducible components, and has simple poles at all of the nodes.
\end{itemize}


\item $\cS_{1,1}$ is the stratum where  $C$ has two irreducible components denoted by $C'$ and $C''$, where $C'$ is isomorphic to $\Pb^1$, $C''$ is a genus two curve with two nodes such that
\begin{itemize}
\item[.] $C'$ contains two points in $\{p_1,\dots,p_4\}$,
\item[.] $C''$ contains $\{p'_5, p''_5\}$ and two points in $\{p_1,\dots,p_4\}$,
\item[.] there are two nodes between $C'$ and $C''$, and
\item[.] $\xi$ has simple poles at all of the nodes of $C$.
\end{itemize}

\item $\cS_{0,2}$ is the stratum where  $C$ has two irreducible components denoted by $C'$ and $C''$, where $C'$ is a Riemann surface of genus 2, $C''$ is isomorphic to $\Pb^1$ such that
 \begin{itemize}
    \item[.] $C'$ contains $\{p_1,\dots,p_4\}$, $C''$ contains $\{p'_5, p''_5\}$,
    \item[.] $C'$ and $C''$ intersect at two nodes both of which are fixed by $\tau$,
    \item[.] $(C',\xi_{\left|C'\right.})\in \Omega\cM_2(2)$, and $\xi_{\left|C''\right.}\equiv 0$.
\end{itemize}

\item $\cS_{2,1}^a$ is the stratum where  $C$ has three irreducible components denoted by $C'_1, C'_2$, and $C''$, such that
\begin{itemize}
\item[.] $C'_1$ and $C'_2$ are both isomorphic to $\Pb^1$,  $C''$ is a genus two curve  with two nodes that are exchanged by $\tau$,

\item[.] $C'_1$ contains $\{p'_5, p''_5\}$ and one point in $\{p_1,\dots,p_4\}$,

\item[.] $C'_2$ contains two points in $\{p_1,\dots,p_4\}$,

\item[.] $C'_1$ intersects $C''$ at one node, $C'_2$ intersects $C''$ at two
nodes

\item[.] $\xi_{\left|C'_1\right.}\equiv 0$, $\xi_{\left|C''\right.}$ has a zero of order $4$ at the node between $C''$ and $C'_1$, and has simple poles at all the other nodes of $C$.
\end{itemize}

\item $\cS_{2,1}^b$ is the stratum where  $C$ has four irreducible components $C'_1, C'_2, C''_1, C''_2$, all of which are isomorphic to $\Pb^1$,  such that
\begin{itemize}
\item[.] each of $C'_1$ and $C'_2$ contains two points in $\{p_1,\dots,p_4\}$.

\item[.] each of $C''_1$ and  $C''_2$ contains one point in $\{p'_5, p''_5\}$,

\item[.] $C'_1$ and $C'_2$ are disjoint, while $C''_1$ and $C''_2$  intersect each other at two nodes,

\item[.] $C'_1$ (resp. $C'_2$) intersects both $C''_1$ and $C''_2$,

\item[.] $\xi$  has simple poles at all the nodes.
\end{itemize}

\item $\cS_{2,1}^c$ is the stratum where  $C$ has four irreducible components $C'_1, C'_2, C''_1, C''_2$, such that
\begin{itemize}
\item[.] $C'_1$ and $C'_2$ are both isomorphic to $\Pb^1$, each of $C''_1, C''_2$ is a genus one curve with one node,

\item[.] $C'_1$ (resp. $C'_2$) contains two points in $\{p_1,\dots,p_4\}$, $C'_1, C'_2$ are disjoint.

\item[.] $C''_1$ (resp. $C''_2$) contains one point in $\{p'_5, p''_5\}$, $C''_1, C''_2$ are disjoint.

\item[.] $C'_1$ (resp. $C'_2$) intersects each of $C''_1$ and $C''_2$ at one node,

\item[.] $\xi$  has simple poles at all the nodes.
\end{itemize}

\item $\cS_{3,1}$ is the stratum where  $C$ has $5$ irreducible components denoted by $C'_i, \; i=1,2,3$, and $C''_j, \, j=1,2$, such that
\begin{itemize}
\item[.] $C'_i, \; i=1,2,3$, is isomorphic to $\Pb^1$, $C''_j, \, j=1,2$, is a genus $1$ curve with one node,
\item[.] $C'_1$ contains two points in $\{p_1,\dots,p_4\}$, $C'_2$ contains one point in $\{p_1,\dots,p_4\}$, $C'_1$ intersects $C'_2$ at two nodes,
\item[.] $C'_3$ contains one point in $\{p_1,\dots,p_4\}$ and the pair $\{p'_5, p''_5\}$, $C'_3$ intersects $C'_2$ at one node,
\item[.] $C''_1$ and $C''_2$ are disjoint, and each of $C''_1, C''_2$ intersects $C'_3$ at one node,
\item[.] the differential $\xi$ vanishes identically on $C'_3$ and has simple poles at the nodes between $C'_1$ and $C'_2$, and at the nodes of $C''_j, \, j=1,2$.
\end{itemize}

\item $\cS_{2,2}$ is the stratum where  $C$ has $4$ irreducible components, denoted by $C'_1, C'_2, C''_1, C''_2$, all of which are isomorphic to $\Pb^1$, such that
\begin{itemize}
\item[.] $C'_1$ and $C'_2$ are disjoint,
\item[.] $C'_1$ (resp. $C'_2$) contains two points in $\{p_1,\dots,p_4\}$, intersects $C''_1$ at two nodes, and is disjoint from $C''_2$.
\item[.] there are two nodes between $C''_1$ and $C''_2$, both of which are fixed by $\tau$,
\item[.] $\{p'_5, p''_5\}\subset C''_2$, and $\xi_{\left|C''_2\right.}\equiv 0$,
\item[.] $\xi_{\left|C''_1\right.}$ has a double zero at a node between $C''_1$ and $C''_2$, and simple poles at all the nodes between $C''_1$ and $C'_1\cup C'_2$.
\end{itemize}

\item $\cS_{1,3}$ is the stratum where  $C$ has $4$ irreducible components denoted by $C'$ and $C''_j, \, j=1,\dots,3$, such that
\begin{itemize}
\item[.] all the irreducible components are isomorphic to $\Pb^1$,
\item[.] $C'$ contains two points in $\{p_1,\dots,p_4\}$, each of $C''_1, C''_3$ contains one point in $\{p_1,\dots,p_4\}$, and $\{p'_5, p''_5\}\subset C''_2$,
\item[.] $C''_1$ intersects $C''_2$ at one node, and intersects each of $C'$ and $C''_3$ at two nodes,

\item[.] $C''_2$ intersects $C''_3$ at one node,

\item[.] $\xi_{\left|C''_2\right.}\equiv 0$, while $\xi_{\left|C''_1\right.}$ has a double zero at the node between $C''_1$ and $C''_2$, and has simple poles at all the nodes between $C''_1$ and $C'\cup C''_3$.
\end{itemize}
\end{enumerate}
\end{Theorem}
The proof of Theorem~\ref{th:bdry:eigen:form:H22} consists of a case by case verification following the topology of the quotient curve $E=C/\langle \tau \rangle$. It turns out that an immense majority of the cases will be ruled out by the charaterizing properties of limit Prym eigenforms proven in  Appendix \textsection\ref{sec:degenerate:eigen:form}. Since this proof is rather lengthy and has no significant impact on other parts of the paper, we provide it Appendix \textsection\ref{sec:prf:bdry:egein:form:H22}.   

\section{Geometry of the $\ol{\cX}_D$ near the boundary}\label{sec:geometry:bdry:XD}
In this section we study the geometry of $\ol{\cX}_D$ near its boundary. 
Let $\pp=(C,p_1,\dots,p_5,p'_5,\tau, [\xi])$ be a point in $\partial\ol{\cX}$.
To our purpose, we partition the boundary strata of $\partial\ol{\cX}_D$ into four groups as follows:
\begin{itemize}
\item[$\bullet$] Group I consists of the strata: $\cS_{1,0}, \cS_{2,0}^a, \cS_{0,2}$. The strata in this group contain $\pp$ such that $\xi$ does not have simple pole.

\item[$\bullet$] Group II consists of the strata: $\cS_{2,0}^b, \cS_{1,1}$. The strata in this group  contain $\pp$ such that the curve $C$ has two pairs of nodes that are exchanged by $\tau$, and $\xi$ has simple poles at all the nodes of $C$.

\item[$\bullet$] Group III consists of the strata: $\cS_{2,1}^a, \cS_{3,1}, \cS_{2,2}, \cS_{1,3}$. The strata in this group contain $\pp$ such that $\xi$ vanishes identically on one component of $C$, and is non-trivial on all other components. In particular, $\xi$ has simple poles at all non-separating nodes of $C$. 

\item[$\bullet$] Group IV consists of the strata: $\cS_{2,1}^b, \cS_{2,1}^c$. The strata in this group contain $\pp$ such that all the components of $C$ are isomorphic to $\Pb^1$, and  $\xi$ does not vanishes identically on any component.
\end{itemize}

\subsection{Triple of tori Prym eigenforms}\label{subsec:triple:tori:ef:def}
To investigate the boundary of $\ol{\cX}_D$ we need to generalize the notion of Prym eigenform to disconnected Riemann surfaces. 
A {\em triple of flat tori} is the data of $\{(X_j,x_j,\omega_j), \; j=0,1,2\}$, where for each $j\in \{0,1,2\}$
\begin{itemize}
	\item[$\bullet$] $X_j$ is a an elliptic curve,
	
	\item[$\bullet$] $x_j$ is a marked point on $X_j$,
	
	\item[$\bullet$] $\omega_j$ is a non-trivial holomorphic $1$-form on $X_j$.
\end{itemize}
Let us denote by $X$ the disjoint union of  $X_0,X_1,X_2$. The data of $\{(X_j,\omega_j), \; j=0,1,2\}$ can be viewed as a holomorphic $1$-form on $X$, which will be denoted by $\omega$. Thus the triple of tori $\{(X_j,x_j,\omega_j), \; j=0,1,2\}$ can be represented by the tuple $(X,x_0,x_1,x_2,\omega)$.

We call the triple  $\{(X_j,x_j,\omega_j), \; j=0,1,2\}$ a {\em Prym form} if there exists an isomorphism  $\phi: X_1 \to X_2$ such that $\phi^*\omega_2=-\omega_1$.
Combining with translations on $X_1$ and $X_2$, we can assume that $\phi(x_1)=x_2$. We extends $\phi$ to an involution $\tau$  of $X$ by setting $\tau_{|X_0}$ to be the unique non-trivial involution of $X_0$ fixing $x_0$ of, $\tau_{|X_1}=\phi$ and $\tau_{|X_2}= \phi^{-1}$. 
We will call $\tau$ the Prym involution of $X$.
Note that we have $\tau^*\omega=-\omega$.

Let $\Omega(X)^-$ denote the space of holomorphic $1$-form $\xi$ on $X$ such that $\tau^*\xi=-\xi$. 
We have $\dim_\C\Omega(X)^-=2$ and $\omega \in \Omega(X)^-$.
Define $ H_1(X,\Z)^-:=\{c\in H_1(X,\Z), \; \tau_*c=-c\}$.
We have $H_1(X,\Z)^-\simeq \Z^4$, and  the  intersection form on $H_1(X,\Z)^-$ has signature $(1,2)$. It follows  that $\Prym(X):=(\Omega(X)^-)^*/H_1(X,\Z)^-$ is an Abelian variety of dimension $2$.

Let $\Omega E_{D}(0^3)$ denote the space of triples of flat tori $(X,x_0,x_1,x_2,\omega)$ as above such that $\End(\Prym(X))$ contains a self-adjoint proper subring isomorphic to $\cO_D$ for which $\omega$ is an eigenform. 
We will call elements of $\Omega E_D(0^3)$ {\em triple of tori Prym eigenforms}.  
It is shown in \cite{LN:finite} that $\Omega E_D(0^3)$ is contained in the boundary of $\Omega E_D(2,2)^\odd$. 
We have a natural action of $\C^*$ on the space of triples of tori by simultaneously multiplying the same scalar to the Abelian differentials  on all three components. Let $W_{D}(0^3)$ denote the quotient $\Omega E_{D}(0^3)/\C^*$. 
We will  see that $W_D(0^3)$ consists of finitely many hyperbolic surfaces, each of which is a finite cover of the modular curve $\Hbb/{\SL(2,\Z)}$ (cf. \textsection\ref{sec:triple:tori:comput}). 

\subsection{Strata of group I}\label{subsec:geom:bdry:str:gp:I}
Our goal is to prove the following
\begin{Proposition}\label{prop:bdry:str:gp:I}
The strata $\cS_{1,0}, \cS_{2,0}^a, \cS_{0,2}$ have codimension $1$ in $\ol{\cX}_D$. All the points in $\cS_{1,0}\sqcup\cS_{2,0}^a\sqcup\cS_{0,2}$ are smooth points of $\ol{\cX}_D$ as an orbifold (that is each of those points admits a neigborhood isomorphic to a finite quotient of an open ball in $\C^2$). Moreover
\begin{itemize}
\item[(i)] Each component of $\cS_{1,0}$ is a finite cover of a Teichm\"uller curve in $W_D(4) \subset \Pb\Omega \cM_3^{\odd}(4)$.

\item[(ii)] Each component of $\cS_{2,0}^a$ is a finite cover of a curve in $W_{D}(0^3)$. 

\item[(iii)] Each component of $\cS_{0,2}$ is a finite cover of a curve in $W_{D'}(2)$, with $D'\in \{D, D/4\}$. 
\end{itemize}
\end{Proposition}
\begin{figure}[htbp]
\begin{minipage}[t]{0.2\linewidth}
\centering
\begin{tikzpicture}[scale=0.5]
	\draw (-2,4) arc (180:360:2cm);
	\draw (2,7)  arc (0:180:2cm);
	\draw (-2,4) -- (-2,7)  (2,4) -- (2,7);
	\draw (0,1) circle (1cm);
	
	\draw (0,7) ..controls (-0.5,7.2) and (-0.5,7.8) .. (0,8);
	\draw  (0,8) -- (0.25,8.1) (0,7) -- (0.25,6.9);  
	\draw (0,7) ..controls (0.5,7.2) and (0.5,7.8) .. (0,8); 
	
	\draw (0,5) ..controls (-0.5,5.2) and (-0.5,5.8) .. (0,6);
	\draw  (0,6) -- (0.25,6.1) (0,5) -- (0.25,4.9);  
	\draw (0,5) ..controls (0.5,5.2) and (0.5,5.8) .. (0,6);
	
	\draw (0,3) ..controls (-0.5,3.2) and (-0.5,3.8) .. (0,4);
	\draw  (0,4) -- (0.25,4.1) (0,3) -- (0.25,2.9);  
	\draw (0,3) ..controls (0.5,3.2) and (0.5,3.8) .. (0,4);
	
	\foreach \x in {(0,8.5), (0,6.5), (0,4.5), (0,0.5)} \filldraw \x circle (3pt);
	
	\foreach \x in {(-0.5,1), (0.5,1)} \filldraw[fill=white] \x circle (3pt);
	
	\filldraw[fill=blue!30] (0,2) circle (3pt);
	
	\draw (0,-1) node {$C \in \cS_{1,0}$};
\end{tikzpicture}
\end{minipage}
\begin{minipage}[t]{0.35\linewidth}
	\centering
	\begin{tikzpicture}[scale=0.4]
		\draw (0,4) ellipse (2 and 3);
		\draw (0,3) .. controls (-0.5,3.5) and (-0.5,4.5) .. (0,5);
		\draw (0,5) -- +(0.25,0.25); \draw (0,3) -- +(0.25,-0.25); 
		\draw (0,3) .. controls (0.5,3.5) and (0.5,4.5) .. (0,5);
		
		\draw (-4,0) ellipse ( 3 and 2);
		\draw (-5,0) .. controls (-4.5,-0.5) and (-3.5,-0.5) .. (-3,0);
		\draw (-5,0) -- +(-0.25,0.25); \draw (-3,0) -- +(0.25,0.25); 
		\draw (-5,0) .. controls (-4.5,0.5) and (-3.5,0.5) .. (-3,0);
		
		\draw (4,0) ellipse (3 and 2);
		\draw (3,0) .. controls (3.5,-0.5) and (4.5,-0.5) .. (5,0);
		\draw (3,0) -- +(-0.25,0.25); \draw (5,0) -- +(0.25,0.25); 
		\draw (3,0) .. controls (3.5,0.5) and (4.5,0.5) .. (5,0);
		
		\draw (0,0) circle (1);
		
		\foreach \x in {(0,6.5), (0,5.7), (0,2), (0,-0.5)} \filldraw \x circle (4pt);
		
		\foreach \x in {(-0.5,0), (0.5,0)} \filldraw[fill=white] \x circle (4pt);
		
		\foreach \x in {(0,1), (1,0), (-1,0)} \filldraw[fill=blue!30] \x circle (4pt);
		\draw (0,-4) node {$C\in \cS_{2,0}^a$};
	\end{tikzpicture}
	
\end{minipage}
\begin{minipage}[t]{0.35\linewidth}
	\centering
	\begin{tikzpicture}[scale=0.4]
		\draw (-2,-1) arc (0:180:1);
		\draw (4,-1) arc (0:180:1);
		\draw (-2,-1) arc (180:360:2 and 1);
		\draw (-4,-1) arc (180:360:4 and 3);
		
		\draw (-3,0) -- (3,0) (-3,5) -- (3,5);
		\draw (-3,5) arc (90:270:2.5);
		\draw (3,0) arc (-90:90:2.5);
		
		\draw ( -3,2.5) .. controls (-2.5,2) and (-1.5,2) .. (-1,2.5);
		\draw (-3,2.5) -- +(-0.3,0.3); \draw (-1,2.5) -- + (0.3,0.3);
		\draw ( -3,2.5) .. controls (-2.5,3) and (-1.5,3) .. (-1,2.5);
		
		\draw ( 3,2.5) .. controls (2.5,2) and (1.5,2) .. (1,2.5);
		\draw (3,2.5) -- +(0.3,0.3); \draw (1,2.5) -- + (-0.3,0.3);
		\draw ( 3,2.5) .. controls (2.5,3) and (1.5,3) .. (1,2.5);
		
		\foreach \x in {(-5,2.5), (-4,2.5), (4,2.5), (5,2.5)} \filldraw \x circle (4pt);
		
		\foreach \x in {(-1,-3), (1,-3)} \filldraw[fill=white] \x circle (4pt);
		
		\foreach \x in {(-3,0), (3,0)} \filldraw[fill=blue!30] \x circle (4pt);
		
		\draw  (0,-6) node {$C \in \cS_{0,2}$};
	\end{tikzpicture}
	
\end{minipage}
\label{fig:boundry:curves:gp:I}
\caption{Curves underlying differentials in strata of group I: $\bullet \in \{p_1,\dots,p_4\}, \circ \in \{p_5,p'_5\}$.}

\end{figure}
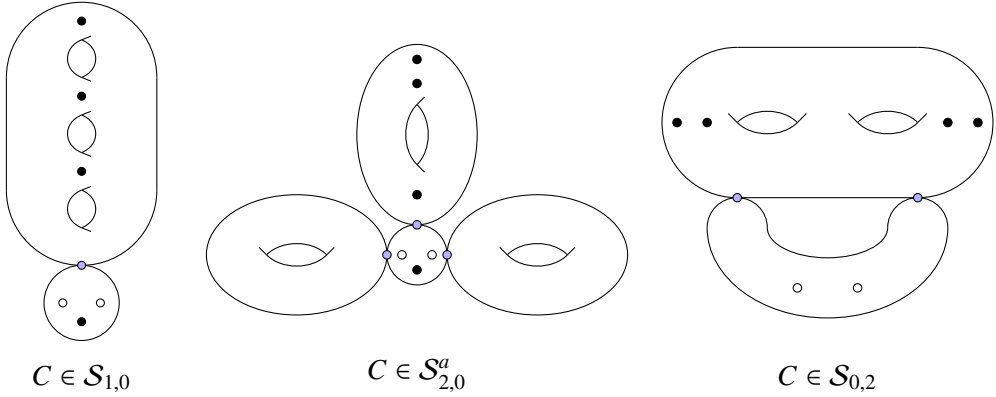
Suppose that $\pp$ is a point in $\cS_{1,0}\cup\cS_{2,0}^a\cup\cS_{0,2}$. By definition, $\xi$ vanishes identically on a unique irreducible component of $C$, which is isomorphic to $\Pb^1$.  
Let us denote this component by $C_0$.
Note that $C_0$ comes equipped with an involution with two fixed points, which is the restriction of $\tau$.
It follows from Theorem~\ref{th:twisted:diff} that $C_0$ carries a meromorphic Abelian differential $\eta$ satisfying $\tau^*\eta=-\eta$  with prescribed orders at its zeros and poles, and zero residues at its poles (which correspond to the  nodes of $C$). It turns out that these conditions  determine $\eta$ up to a constant.

\begin{Lemma}\label{lm:gp:I:diff:on:vanish:comp}
We have
\begin{itemize}
\item[$\bullet$] If $\pp \in \cS_{1,0}$, then we have $C_0=C'$ and up to a scalar $(C_0,\eta)\simeq (\Pb^1, (x^2-1)^2dx)$. 

\item[$\bullet$] If $\pp \in \cS_{2,0}^a$, we have $C_0=C'_2$ and up to s scalar $(C_0,\eta) \simeq (\Pb^1,\frac{(x^2-1)^2dx}{x^2(x^2+3)^2})$.

\item[$\bullet$] If $\pp \in \cS_{0,2}$, then $C_0=C''$ and up to a scalar $(C_0,\eta)\simeq (\Pb^1,\frac{(x^2-1)^2}{x^2}dx)$. 
\end{itemize}
In all cases the restriction of $\tau$ to $C_0$ is given by $x\mapsto -x$.
\end{Lemma}
\begin{proof}
We can always identify $C_0$ with $\Pb^1$ such that $\tau_{\left|C_0\right.}$ is given by $x\mapsto -x$ (here $x$ is the inhomogeneous coordinate on $\Pb^1$). In all the cases, $C_0$ contains the points $p_5, p'_5$. We can further assume that $p_5=1,  p'_5=-1$.

If $\pp\in \cS_{1,0}$, then there is  one node between $C_0$ and the other component of $C$.  Since this node is fixed by $\tau$, we can assume that it corresponds to the point $\infty$ under the identification $C_0\simeq \Pb^1$. In this case $\eta$ has double zeros at $\pm1$ and a pole of order $6$ at $\infty$. Thus up to a scalar, we have $\eta=(x-1)^2(x+1)^2dx$.

If $\pp\in\cS^a_{2,0}$, then $C$ has $4$ components denoted by $C'_1,C'_2, C''_1,C''_2$, where $C'_1, C''_1, C''_2$ are smooth elliptic curves, while  $C'_2\simeq \Pb^1$. The components $C'_1, C''_1, C''_2$ are pairwise disjoint, and intersect $C'_2$  at three nodes. In this case we have $C_0=C'_2$. 

Let $s_0$ is the node between $C'_2$ and $C'_1$, and $s_1$ (resp. $s_2$) the node between $C'_2$ and $C''_1$ (resp. between $C'_2$ and $C''_2$). Since $s_0$ is fixed by $\tau$, we can assume that $s_0=0$. Let $\pm b$,  $b\in \C\setminus\{0,\pm 1\}$, be the coordinates of $s_1, s_2$ respectively.
In this case, $\eta$ has double poles at $s_0, s_1, s_2$. Thus up to a scalar we have
$$
\eta=\frac{(x-1)^2(x+1)^2dx}{x^2(x-b)^2(x+b)^2}
$$
The global residue condition in Theorem~\ref{th:twisted:diff} implies that $\res_0(\eta)=\res_{b}(\eta)=\res_{-b}(\eta)=0$.
We always have $\res_0(\eta)=0$. The condition $\res_b(\eta)=\res_{-b}(\eta)=0$ implies that $b=\pm \sqrt{3}\imath$, and we get the desired conclusion.

Finally, if $\pp\in \cS_{0,2}$, then $C_0=C''$ intersects the other component of $C$ at two nodes both of which are fixed by $\tau$.
These two nodes correspond to $0$ and $\infty$ under the identification $C_0\simeq \Pb^1$.
In this case $\eta$ has a pole of order $4$ and a pole of order $2$ at the nodes. Using the involution $x\mapsto 1/x$, we can assume that $\infty$ is the pole of order $4$ and $0$ is the pole of order $2$ of $\eta$. Thus, up to a scalar, we have
$$
\eta=\frac{(x^2-1)^2dx}{x^2}.
$$
\end{proof}
The component $C_0$ together with the marked points in $C_0\cap\{p_1,\dots,p_5,p'_5\}$ and the nodes is a pointed genus zero curve. By a slight abuse of notation, we denote this pointed curve again by $C_0$. As a consequence of Lemma~\ref{lm:gp:I:diff:on:vanish:comp}, we have
\begin{Corollary}\label{cor:gp:I:vanish:comp:unique}
For each stratum in group I, the pointed curve $C_0$ is uniquely determined up to isomorphism.
\end{Corollary}

\begin{Lemma}\label{lm:gp:I:limit:diff}
Let $C_1$ be the union of all components of $C$ on which $\xi$ does not vanish identically, and $\xi_1:=\xi_{\left|C_1\right.}$.
We have
\begin{itemize}
\item[(i)] If $\pp \in\cS_{1,0}$, then $(C_1,\xi_1) \in \Omega E_D(4)$.

\item[(ii)] If $\pp \in \cS^a_{2,0}$, then $(C_1,\xi_1) \in \Omega E_{D}(0^3)$.

\item[(iii)] If $\pp \in \cS_{0,2}$, then $(C_1,\xi_1) \in \Omega E_{D'}(2)$, with $D'\in \{D, D/4\}$.
\end{itemize}
\end{Lemma}
\begin{proof}
Let $\tau_1$ be the restriction of $\tau$ to $C_1$.  If $\pp\in \cS_{1,0}$, then $C_1$ is a Riemann surface of genus $3$, and $\tau_1$ has $4$ fixed point on $C_1$ namely  three points in $\{p_1,\dots,p_4\}$ and the node between $C_0$ and $C_1$. If $\pp \in \cS^a_{2,0}$, then $C_1$ is the dis joint union of three tori $C'_1, C''_1, C''_2$. The involution $\tau_1$ preserves $C'_1$ and exchanges $C''_1$ and $C''_2$.
In the case $\pp \in \cS_{0,2}$, $C_1$ is a genus two Riemann surface, and $\tau_1$ has $6$ fixed points, with the two additional fixed points being the nodes between $C_0$ and $C_1$. This means that $\tau_1$ is the hyperelliptic involution of $C_1$.

Let $\Omega(C_1)$ denote the space of holomorphic Abelian differentials on $C_1$, and
$$
\Omega(C_1)^-=\{\omega \in \Omega(C_1), \; \tau^*\omega=-\omega\}.
$$
We first observe that $\dim_\C\Omega^-(C_1)=2$.
This claim is straightforward in the cases $\pp \in \cS_{0,1}$ and $\pp \in \cS_{0,2}$. In the case $\pp \in  \cS^a_{2,0}$, that is $C_1=C'_1\sqcup C''_1\sqcup C''_2$, the claim follows from the fact that elements of $\Omega^-(C_1)$ are triples of differentials $((C'_1,\omega'_1), (C''_1,\omega''_1), (C''_2, \omega''_2))$ such that $\tau_1^*\omega''_2=-\omega''_1$.
Let
$$
H_1(C_1,\Z)^-:=\{c\in H_1(C_1,\Z), \; \tau_{1*}c=-c\}.
$$
It is not difficult to check that $H_1(C_1)^*\simeq \Z^4$ and the restriction of the intersection form on $H_1(C_1,\Z)$ to $H_1(C_1,\Z)^-$ is non-degenerate. It follows in particular that $\Prym(C_1):=(\Omega(C_1)^-)^*/H_1(C_1,\Z)^-$ is an Abelian variety of dimension $2$.

Let $\xx=(X,x_1,\dots,x_5,x'_5,\tau_X,[\omega])$ be an element of $\cX_D$ close enough to $\pp$. Topologically, the surface $X$ is obtained from $C$ by smoothening the nodes. There is a surjective map $f: X\to C$ that sends a multicurve $\gamma$ (that is a family of pairwise disjoint simple closed curves) on $X$ onto the  nodes of $C$. The restriction of $f$ to  $X\setminus \gamma$ gives a homeomorphism from $X\setminus\gamma$ onto $C\setminus\{{\rm nodes}\}$. We have $f_*H_1(X,\Z)^- \subset H_1(C_1,\Z)^-$
in all cases. In the case $\pp\in \cS_{1,0}\sqcup \cS^a_{2,0}$, since all the components of the multicurve $\gamma \subset X$ are separating, we have $f_*H_1(X,\Z)^-=H_1(C_1)^-$. However, if $\pp \in \cS_{0,2}$, then $f_*H_1(X,\Z)^-$ is a sublattice of index $2$ in $H_1(C_1)^-=H_1(C_1)$.

By assumption, there exists $T\in \End(\Prym(X))$ such that $\Z[T]\simeq \cO_D$ and $\omega$ is an eigenvector of the action of $T^*$ on $\Omega(\Prym(X))=\Omega(X)^-$. In particular, we have $T^*\omega=\lambda\cdot \omega$ for some $\lambda\in\cO_D$.

By definition, $T$ is given by a $\C$-linear map on $(\Omega(X)^-)^*\simeq \C^2$ preserving the lattice $H_1(X,\Z)^-$.
In the case $\pp\in \cS_{1,0}\sqcup \cS^a_{2,0}$, since $H_1(X,\Z)^-$ can be identified with $H_1(C_1,\Z)^-$, we can view $T$ as an endomorphism $T: H_1(C_1,\Z)^- \to H_1(C_1,\Z)^-$. The condition $T^*\omega = \lambda \omega$ then implies that $T^*\xi_1=\lambda\xi_1$, since $\xi_1$ is the limit of  $\omega$ as $\xx$ converges to $\pp$. It follows from the argument of \cite[Th. 3.2]{McM:prym} that $T\in \End(\Prym(C_1))$ and therefore $(C_1,\xi_1)\in \Omega E_D(4)\sqcup \Omega E_{D}(0^3)$.

In the case $\pp\in \cS_{0,2}$, by using $f_*$ we can consider $H_1(X,\Z)^-$ as a sublattice of index $2$ in $H_1(C_1,\Z)^-$.  Thus we have $2\cdot H_1(C_1,\Z)^-\subset H_1(X,\Z)^-$. As a consequence $\tilde{T}:=2T$ can be extended to an endomorphism of $H_1(C_1,\Z)^-$. As we have $\tilde{T}^*\omega=2\lambda\cdot\omega$, it follows that  $\tilde{T}^*\xi_1=2\lambda\cdot\xi_1$. Therefore, $\xi_1$ is an eigenform for some quadratic order $\cO_{D'}$ acting by self-adjoint endomorphisms on $\Prym(C_1)$, that is $(C_1,\xi_1) \in \Omega E_{D'}(2)$. It turns out that $\cO_{D'}$ is generated either by  $T$, or by $T/2$. Thus $D'\in \{D,D/4\}$. For a proof of this fact we refer to \cite[Th. 8.6]{LN:finite}. 
This completes the proof of the lemma.
\end{proof}

\subsubsection*{Proof of Proposition~\ref{prop:bdry:str:gp:I}}
\begin{proof}
The proof of the proposition in the case $\pp \in \cS_{1,0}\sqcup \cS_{2,0}^a$ is rather standard since all the nodes of $C$ are separating.  We will only give the proof for the case $\pp\in \cS_{0,2}$. In this case $C_1$ is a genus two Riemann surface and $\xi_1$ has a double zero at one of the nodes between $C_1$ and $C_0$.
By Lemma~\ref{lm:gp:I:limit:diff}, $(C_1,[\xi_1])\in \Pb\Omega E_{D'}(2)$ for some $D'\in \{D,D/4\}$.
Let $U$ be a neighborhood of $(C_1,[\xi_1])$ in $\Pb\Omega E_{D'}(2)$. Since $\dim\Pb\Omega E_{D'}(2)=1$, we can suppose that $U$ is a neighborhood of $0$ in $\C$.
Taking a local lift in $\Omega E_{D'}(2)$  (and reducing $U$ if necessary), we have a holomorphic family of Abelian differentials $(C_{1,z},\xi_{1,z})_{z\in U}$, where  $(C_{1,0},\xi_{1,0})=(C_1,\xi_1)$ and $(C_{1,z},\xi_{1,z})\in \Omega E_{D'}(2)$.

Let $f : \cC_1 \to U$ be the underlying family of Riemann surfaces, that is $f^{-1}(z) \simeq C_{1,z}$ for all $z \in U$.
Let $w_0$ and $w_1$ be the points in $C_1$ which correspond to the nodes between $C_1$ and $C_0$, where $w_0$ is the unique zero of $\xi_1$.
Let $w_{0,z}$ and $w_{1,z}$ be the corresponding  Weierstrass points on $C_{1,z}$.
There is a neighborhood $W_0$ (resp. $W_1$) of the section  $z\mapsto w_{0,z}$ (resp. $z\mapsto w_{1,z}$) in $\cC_1$ together with a holomorphic map $\varphi_{0}: W_0 \to \C$ (resp. $\varphi_1: W_1 \to \C$) such that for all $z\in U$
\begin{itemize}
\item[$\bullet$] $\varphi_0(w_{0,z})=0$ (resp. $\varphi_1(w_{1,z})=0$).

\item[$\bullet$] Let $W_{0,z}:=W_0\cap C_{1,z}$ (resp. $W_{1,z}:=W_1\cap C_{1,z}$), then the restriction $\varphi_{0,z}:=\varphi_{0\left|W_{0,z}\right.}$ (resp. $\varphi_{1,z}:=\varphi_{1\left|W_{1,z}\right.}$) is a local coordinate on $W_{0,z}$ (resp. $W_{1,z}$).

\item[$\bullet$] $\xi_{0,z}=\varphi^2_{0,z}d\varphi_{0,z}$ on $W_{0,z}$ (resp. $\xi_{1,z}=d\varphi_{1,z}$ on $W_{1,z}$).
\end{itemize}
The last condition means that $\xi_{0,z}$ and $\xi_{1,z}$ are the pullbacks by $\varphi_{0,z}$ and $\varphi_{1,z}$ of the Abelian differentials $x^2dx$  and $dx$ on $\C$ respectively.

We identify $C_0$ with $\Pb^1$ such that the restriction of $\tau$ on $C_0$ corresponds to the involution $x\mapsto -x$. Since $\tau$ fixes $0$ and $\infty$, these two points are mapped to the nodes between $C_0$ and $C_1$. We can suppose that $0\in C_0$ is identified with $w_0\in C_1$, and $\infty \in C_0$ with $w_1\in C_1$.

Let $\eta=\frac{(x^2-1)^2dx}{x^4}$. Note that $\eta$ has a pole of order $2$ at $\infty$. Since $\res_\infty(\eta)=\res_0(\eta)=0$, there exist a neighborhood $V_0\subset \Pb^1$ of $0$ (resp. $V_1\subset\Pb^1$ of $\infty$) and a local coordinate $\phi_0$ on $V_0$ (resp. $\phi_1$ on $V_1$) such that
$\phi_0(0)=0$ and $\eta_{\left|V_0\right.}=\frac{d\phi_0}{\phi_0^4}$ (resp. $\phi_1(\infty)=0$ and $\eta_{\left|V_1\right.}=\frac{d\phi_1}{\phi_1^2}$).
We now choose $\delta \in \R_{>0}$ small enough such that
\begin{itemize}
\item[$\bullet$]  $\Delta_{\delta}\subset \varphi_{0,z}(W_{0,z})$ and $\Delta_{\delta^3} \subset \varphi_{1,z}(W_{1,z})$ for all $z\in U$,

\item[$\bullet$]  $\Delta_\delta\subset \phi_0(V_0)$ and $\Delta_{\delta^3} \subset \phi_1(V_1)$.
\end{itemize}
For all $0 < \delta' < \delta$, denote by $A_{\delta',\delta}$ the annulus $\{x\in \C, \; \delta' < |x| <  \delta\}$.
For all $t\in \Delta_{\delta^2}$ let $C_{z,t}$ denote the curve defined as follows

\begin{itemize}
\item[$\bullet$] For $t=0$, $C_{z,0}$ is the nodal curve obtained from $C_{1,z}$ and $\Pb^1$ by identifying $w_{0,z}\in C_1$ with $\infty\in \Pb^1$, and $w_{1,z}$ with $0\in \Pb^1$.

\item[$\bullet$] For $0  < |t| < \delta^2$,  we remove $\varphi_{0,z}^{-1}(\Delta_{|t|/\delta})$ from $W_{0,z}$ and $\phi_0^{-1}(\Delta_{|t|/\delta})$ from $V_0$. We then glue the annuli $\varphi^{-1}_{0,z}(A_{|t|/\delta, \delta})$ and $\phi^{-1}_0(A_{|t|/\delta,\delta})$ together by the relation $\varphi_{0,z}\phi_0=t$. Similarly, we remove $\varphi_{1,z}^{-1}(\Delta_{(|t|/\delta)^3})$ from $W_{1,z}$ and $\phi_1^{-1}(\Delta_{(|t|/\delta)^3})$ from $V_1$, and glue $\varphi^{-1}_{1,z}(A_{(|t|/\delta)^3, \delta^3})$ and $\phi^{-1}_1(A_{(|t|/\delta)^3, \delta^3})$ together by the relation $\varphi_{1,z}\phi_1=t^3$.
\end{itemize}
We thus obtain a holomorphic family of nodal curves $F: \cC \to U\times\Delta_{\delta^2}$ such that $F^{-1}(z,t)\simeq C_{z,t}$.
By construction, the family $(C_{1,z})_{z\in U}$ comes equipped with the differentials $(\xi_{1,z})_{z\in U}$.
If $t=0$, we define an Abelian differential $\xi_{z,0}$ on $C_{z,0}$ by setting $\xi_{z,0}=\xi_{1,z}$ on $C_{1,z}$ and $\xi_{z,0} \equiv 0$ on $C_0$.
For $t\neq 0$, by construction, $\xi_{1,z}$ and $-t^3\eta$ coincide on the overlap annuli $\varphi^{-1}_{0,z}(A_{|t|/\delta, \delta})\simeq \phi^{-1}_0(A_{|t|/\delta, \delta})$, and   $\varphi^{-1}_{1,z}(A_{(|t|/\delta)^3, \delta^3})\simeq \phi^{-1}_1(A_{(|t|/\delta)^3, \delta^3})$.
Thus we get a differential $\xi_{z,t}$ on $C_{z,t}$ which  coincides with $\xi_{1,z}$ on $C_{1,z}\setminus(W_{0,z}\cup W_{1,z})$, and coincides with $-t^3\eta$ on $C_0\setminus(V_0\cup V_1)$.
It is clear that $(C_{z,t},\xi_{z,t})\in \Omega \ol{\cB}_{4,1}$ for all $(z,t)\in U\times\Delta_{\delta^2}$.
Reversing the arguments of Lemma~\ref{lm:gp:I:limit:diff}, we  conclude that $(C_{z,t},\xi_{z,t})\in \Omega E_D(2,2)^\odd$ if $t\neq 0$.
Taking quotient by $\C^*$ we then get a holomorphic map $\Psi: U\times \Delta_{\delta^2} \to \Pb\Omega \ol{\cB}_{4,1}$ such that $\Psi(U\times\Delta^*_{\delta^2}) \subset \cX_D$.
Thus $\Psi(U\times \Delta_{\delta^2}) \subset \ol{\cX}_D$.
It is a well known fact that the map $(z,t) \mapsto C_{z,t}$ gives an embedding  of $U\times \Delta_{\delta^2}$ into an orbifold local chart of $(C,p_1,\dots,p_5, p'_5)$ in $\ol{\cB}_{4,1}$. 
As a consequence, $\Psi$ is a biholomorphism from $U\times\Delta_{\delta^2}$ onto its image. 

For every $\xx=(X,\ul{x},\tau_X,[\omega]) \in\cX_D$ close enough to $\pp$, let $f_\xx: X \to C$ be an associated degenerating map. The preimage of $C_0$ minus the nodes is an annulus $A$ in $X$ which contains the two zeros of $\omega$. There is a pair of saddle connections $s,s'$ connecting these two zeros whose union forms a core curve of $A$. 
Note that $s$ and $s'$ have the same period.
As $\xx$  converges to $\pp$, the flat metric defined by $\omega$ on $A$ collapses to $0$. Thus there cannot exist others saddle connections connecting the zeros of $\omega$ whose length is smaller than $|s|$. By the arguments of \cite[Th. 8.6]{LN:finite}, one can collapse $s$ and $s'$ to obtain a point $(X_1,[\omega_1]) \in U$. 
It follows that $\xx=\Psi((X_1,[\omega_1]),t)$ for some $t\in \Delta_{\delta^2}$.
We can then conclude that $\Psi(U\times\Delta_{\delta^2})$ is an orbifold local chart of $\pp$ in $\ol{\cX}_D$. 
It is also clear from the construction that $(C_{z,t},[\xi_{z,t}]) \in \cS_{0,2}$ if and only if $t=0$. Finally, the correspondence $(C_{z,0},[\xi_{z,0}]) \mapsto (C_{1,z}, [\xi_{1,z}])$ provides us with locally biholomorphic map from $\cS_{0,2}$ onto $\Pb\Omega E_{D'}(2)$. This completes the proof of the proposition.
\end{proof}

\subsection{Strata of group II}\label{subsec:geom:bdry:str:gp:II}
There are two strata in group II: $\cS^b_{2,0}$ and $\cS_{1,1}$.
We will show
\begin{Proposition}\label{prop:bdry:str:gp:II}
Let $\pp$ be a point in $\cS^b_{2,0}\cup \cS_{1,1}$. Then every irreducible component of the germ of $\ol{\cX}_D$ at $\pp$ is isomorphic to the germ at $0$ of the analytic set
$$
\cA=\{(z,t_1,t_2) \in \C^3, \; t_1^{m_1}=t_2^{m_2}\} \subset \C^3,
$$
where $m_1, m_2 \in \Z_{>0}$ are such that $\gcd(m_1,m_2)=1$. In this identification, the stratum of $\pp$ corresponds to the set $\cA\cap\{t_1=t_2=0\}$. In particular, we have $\dim \cS_{2,0}^b=\dim \cS_{1,1}=1$.
\end{Proposition}

If $\pp\in \cS^b_{2,0}$ then $C$ has three components $C'_1,C'_2, C''$, where $C'_1$ and $C'_2$ are two disjoint copies of $\Pb^1$, $C''$ is an elliptic curve which intersects each of $C'_1$ and $C'_2$ at two nodes. The differential $\xi$ has two double zeros in $C''$ and simple poles at all the nodes of $C$.
Let $\xi'_i:=\xi_{\left|C'_i\right.}, \; i=1,2$, and $\xi'':=\xi_{\left|C''\right.}$. We can identify  $C'_i$ with $\Pb^1$ and suppose that the restriction of $\tau$ to $C'_i$ is given by $x\mapsto 1/x$.  By assumption, we have $(C'_i,\xi'_i)\simeq (\Pb^1, \lambda_i\frac{dx}{x})$, for some $\lambda_i\in \C^*$. 
Let $r_i, r'_i$ denote the nodes between $C''$ and $C'_i$.
Note that $r_i$ and $r'_i$ are exchanged by $\tau$. The differential $\xi''$ has simple poles at $r_i, r'_i, \; i=1,2$, and we have
$$
\res_{r_i}(\xi'')=-\res_{r'_i}(\xi'').
$$

Consider now the case $\pp\in \cS_{1,1}$. In this case  $C$ has two irreducible components $C'$ and $C''$ , where $C'$ is  isomorphic to $\Pb^1$, and $C''$ is a curve of genus two with two self-nodes which intersects $C'$ at two other nodes. 
The differential $\xi$ has two double zeros on $C''$ and simple poles at all the nodes of $C$. 

We can identify the normalization $\tilde{C}''$ of $C''$ with $\Pb^1$ and suppose that the restriction of $\tau$ to $C''$ is given by $x \mapsto -x$ on $\tilde{C}''$. We can further suppose that $\{p_5, p'_5\}=\{\pm 1\}$.
Let $\pm r_1$ be the points in $\Pb^1$  that correspond to the nodes between $C''$ and $C'$.
The two self-nodes of $C''$ give rise to two pairs of points on $\Pb^1$ that are permuted by $\tau$.
Let $\pm r_2, \pm r_3$ denote those points, where $r_2$ and $r_3$ (resp. $-r_2$ and $-r_3$) map to the same node on $C''$.
The restriction $\xi''$ of $\xi$ to $C''$ has double zeros at $\pm 1$, and simple poles at the points $\pm r_i, \; i=1,2,3$.  Since $\tau^*\xi''=-\xi''$, we have
$$
\res_{r_1}\xi''=-\res_{-r_1}\xi'', \quad \text{ and } \quad \res_{r_2}\xi''=-\res_{-r_2}\xi''=-\res_{r_3}\xi''=\res_{-r_3}\xi''.
$$
%

\subsubsection{Coordinate system in a neighborhood of $\pp$}\label{subsec:str:gp:II:loc:coord}
In what follows, we will show that there is an analytic subset of $\Pb\Omega'\ol{\cB}_{4,1}(2,2)$ isomorphic to a ball in $\C^3$ that contains the germ of $\ol{\cX}_D$ at $\pp$.
We will only focus on the case $\pp \in \cS^b_{2,0}$, the proof for the case $\pp \in \cS_{1,1}$ follows the same lines.

Let $\tilde{\cQ}(4,-2,-2)$ be  the moduli space of triples $(Z,\rho,\zeta)$, where $Z$ is an elliptic curve, $\rho$ is an involution without fixed points on $Z$, and $\zeta$ is an Abelian differentials on $Z$ which has two double zeros and four simple poles such that $\rho^*\zeta=-\zeta$.
Denote by $\Pb\tilde{\cQ}(4,-2,-2)$ the projectivization of $\tilde{\cQ}(4,-2,-2)$, that is the quotient $\tilde{\cQ}(4,-2,-2)/\C^*$.
The image of $(Z,\rho,\zeta)$ in $\Pb\tilde{\cQ}(4,-2,-2)$ is denoted by $(Z,\rho,[\zeta])$.

Since $\rho$ has no fixed points, $Y:= Z/\langle\rho\rangle$ is an elliptic curve. The quadratic differential $\zeta^{2}$ descends to a meromorphic quadratic differential $\eta$ on $Y$.
By construction, $(Y,\eta)$ is an element of $\cQ(4,-2,-2)$, that is the moduli space of quadratic differentials on elliptic curves with one zero of order $4$ and two double poles, that are not the square of an Abelian differential.
The correspondence $(Z,\rho,\zeta) \mapsto (Y,\eta)$ allows us to identify $\tilde{\cQ}(4,-2,-2)$ with $\cQ(4,-2,-2)$.
It is shown in \cite{BCGGM2} that $\tilde{\cQ}(4,-2,-2)\simeq \cQ(4,-2,-2)$ is a complex orbifold of dimension $3$.

Recall that $C''$ is the elliptic component of $C$.   Let $\tau''$ be the restriction of $\tau$ to $C''$, and $\xi'':=\xi_{\left|C''\right.}$.
We then have $(C'',\tau'',\xi'') \in \tilde{\cQ}(4,-2,-2)$.
Let us fix a path $\gamma$ from $p_5$ to $p'_5$ in $C''$. For any $(Z,\rho, \zeta)$ in  a neighborhood of $(C'',\tau'', \xi'')$ in $\tilde{\cQ}(4,-2,-2)$, one can specify a path in $Z$ joining the zeros of $\zeta$, and a labeling of the poles of $\zeta$ by $z_1, z'_1, z_2, z'_2$ such that $z_i$ (resp. $z'_i$) correspond to $r_i$ (resp. $r'_i$).
A local chart of $\tilde{\cQ}(4,-2,-2)$ in a neighborhood of $(C'',\tau'',\xi'')$ is given by the map (cf. \cite{BCGGM2})
$$
(Z,\rho,\zeta) \mapsto (\zeta(\gamma), \res_{z_1}(\zeta), \res_{z_2}(\zeta)).
$$ 
This implies that the map $(Z,\rho, [\zeta]) \mapsto \left(\zeta(\gamma)/\res_{z_1}(\zeta), \res_{z_2}(\zeta)/\res_{z_1}(\zeta)\right)$  
gives a local chart of $\Pb\tilde{\cQ}(4,-2,-2)$ in a neighborhood of $(C'',\tau'',[\xi''])$.  
Define 
$$
\alpha:=\frac{\res_{r_2}(\xi'')}{\res_{r_1}(\xi'')}.
$$  
Let $\cW$ be a neighborhood of $(C'',\tau'',[\xi''])$ in $\Pb\tilde{\cQ}(4,-2,-2)$.   
The set 
$$
U:=\{(Z,\rho,[\zeta])\in \cW, \; \res_{z_2}\zeta/\res_{z_1}\zeta=\alpha\}
$$ 
can be identified with an open subset of $\C$ via the map
$(Z,\rho,[\zeta]) \mapsto \zeta(\gamma)/\res_{z_1}(\zeta)$.
Let $x_0\in U$ be the image of $(C'',\tau'',[\xi''])$ under this map.
By definition, there is a family of pointed elliptic curves $f: \cC'' \to U$ and a meromorphic section $\Xi''$ of the relative canonical line bundle $K_{\cC''/U}$ such that the for all $x\in U$, the restriction $\Xi''_x$ of $\Xi''$ to the fiber $C''_x:=f^{-1}(x)$ is an element of $\tilde{\cQ}(4,-2,-2)$, and $(C''_{x_0},\Xi''_{x_0})\simeq (C'',\xi'')$.
Note that $\cC''$ comes quipped with an involution $\rho$ whose restriction to each fiber $C_x$ gives an involution $\rho_x$ such that $\rho_x^*\Xi''_x=-\Xi''_x$.

Let $r_{i,x}$ (resp. $r'_{i,x}$) be the pole of $\Xi''_x$ corresponding to $r_i$ (resp. $r'_i$) for $i=1,2$.
Let $R_i$ (resp. $R'_i$) denote the  section of $f$ associated with the marked points $r_{i,x}$ (resp. $r'_{i,x}$).
There is a neighborhood $\cU_1$ (resp. $\cU'_1$)  of $R_1$ (resp. $R'_1$) that can be identified with $U\times V_1$, where $V_1$ is a neighborhood of $0\in \C$,  such that $R_1 \simeq U\times\{0\}$ (resp. $R'_1 \simeq U\times\{0\}$), and the restriction of $\Xi$ to $\cU_1$ (resp. to $\cU'_1$) is given by $\frac{1}{2\pi\imath}\cdot\frac{dz}{z}$ (resp. by $\frac{-1}{2\pi\imath}\cdot\frac{dz}{z}$), where $z$ is the coordinate on $V_1$ (resp. on $V'_1$).
Similarly, there is a neighborhood $\cU_2$ (resp. $\cU'_2$) of $R_2$ (resp. of $R'_2$) that can be identified with $U\times V_2$ (resp. $U\times V'_2$), where $V_2$ (resp. $V'_2$) is another neighborhood of $0\in \C$, such that $R_2\simeq U\times \{0\}$ (resp. $R'_2 \simeq U\times\{0\}$), and the restriction of $\Xi$ to $\cU_2$ (resp. to $\cU'_2$) is given by $\frac{\alpha}{2\pi\imath}\cdot\frac{dz}{z}$ (resp. by $\frac{-\alpha}{2\pi\imath}\cdot\frac{dz}{z}$).
We can furthermore suppose that $\cU_1, \cU'_1,\cU_2,\cU'_2$ are pairwise disjoint, and that  $\cU'_1:=\rho(\cU_1)$ and $\cU'_2=\rho(\cU_2)$.

\medskip

Let $C'_1$ and $C'_2$ be two copies of $\Pb^1$. We endow $C'_1$ with the Abelian differential $\xi'_1=\frac{1}{2\pi\imath}\cdot \frac{dw}{w}$ and $C'_2$ with the differential $\xi'_2=\frac{\alpha}{2\pi\imath}\cdot\frac{dw}{w}$. 
Let $s_1$ and $s'_1$ (resp. $s_2$ and $s'_2$) be the points in $C'_1$ (resp. in $C'_2$) which correspond to $0$ and $\infty$ in $\Pb^1$ respectively. There is a neighborhood $W_1$ of $s_1$ (resp. a neighborhood $W'_1$ of $s'_1$) with local coordinate $w$ such that $\xi'_{1\left|W_1\right.}= \frac{1}{2\pi\imath}\cdot dw/w$ (resp. $\xi'_{1\left|W'_1\right.}=\frac{-1}{2\pi\imath}\cdot dw/w$).
Similarly, there are neighborhoods $W_2$ of $s_2$ and $W'_2$ of $s'_2$ such that $\xi'_{2\left|W_2\right.}=\frac{\alpha}{2\pi\imath} \cdot dw/w$ and $\xi'_{2\left|W'_2\right.}=\frac{-\alpha}{2\pi\imath} \cdot dw/w$.
We can suppose that $W'_1$ (resp. $W'_2$) is the image of $W_1$ (resp. of $W_2$) under the involution $w\mapsto 1/w$.

\medskip 

Let $\delta \in \R_{>0}$ be small enough so that $\Delta_\delta$ is contained in all of $V_1, V_2, W_1, W_2$.
We can now define  a map $\Phi: U\times\Delta_{\delta^2}\times\Delta_{\delta^2} \to \Omega\ol{\cB}_{4,1}$ as follows: for all $(x,t_1,t_2)\in  U\times\Delta_\delta\times\Delta_\delta$,
\begin{itemize}
\item[$\bullet$] if $t_i=0$, we glue $C'_i$ to $C''_x$ by identifying $s_i$ with $r_{i,x}$ and $s'_i$ with $r'_{i,x}$.

\item[$\bullet$] if $t_i\in \Delta^*_{\delta^2}$, we remove the neighborhoods of $r_{i,x}$ and  $s_i$ that correspond to $\Delta_{t/\delta}\subset \Delta_\delta$. We then glue the annuli $A_{t_i/\delta,\delta} \subset V_i$ 	and $A_{t_i/\delta,\delta} \subset W_i$ together using the relation $zw=t_i$.
We carry the same plumbing construction in the neighborhoods of $r'_i$ and $s'_i$.
\end{itemize}
Let $C_{x,t_1,t_2}$ denote the resulting curve.
By construction the differentials $\Xi''_x,  \xi'_1, \xi'_2$  agree on the overlaps of different components of $C_{x,t_1,t_2}$. 
Therefore, we obtain an Abelian differential $\xi_{x,t_1,t_2}$ on the curve $C_{x,t_1,t_2}$.
Note that $\xi_{x,t_1,t_2}$ has two double zeros that are the zeros of $\Xi''_x$  located on $C''_x$.
The involution $\rho_x$ on $C''_x$ extends to an involution on $C_{x,t_1,t_2}$ which has four fixed points and satisfies $\rho^*_x\xi_{x,t_1,t_2}=-\xi_{x,t_1,t_2}$.
Therefore $(C_{x,t_1,t_2}, \xi_{x,t_1,t_2})\in \Omega' \ol{\cB}_{4,1}(2,2)$. The data of $C_{x,t_1,t_2}$, the zeros of $\Xi''_x$, and the fixed points of $\rho_x$ give a point in $\Pb\Omega'\ol{\cB}_{4,1}(2,2)$, which is defined to be $\Phi(x,t_1,t_2)$.

\begin{Lemma}\label{lm:neigh:p:in:Sb:20:germs}
All the components of the germ of $\ol{\cX}_D$ at $\pp$ are contained in $\Phi(U\times \Delta_{\delta^2}\times\Delta_{\delta^2})$.
\end{Lemma}
\begin{proof}
We have  $\dim\Pb\Omega'\cB_{4,1}(2,2)=\dim\Pb\cQ(4,-1^4)=4$.
Consider a neighborhood $\cV$ of $\pp$ in $\Pb\Omega'\ol{\cB}_{4,1}(2,2)$.
Denote by $\cV^*$ in the intersection $\cV\cap \Pb\Omega'\cB_{4,1}(2,2)$.
For every $\xx=(X,\ul{x},\tau_X,[\omega]) \in \cV^*$ one can specify two pairs of simple closed curves $\{c_1,c'_1\}, \{c_2,c'_2\}$, where $c_i$ and $c'_i$ are contracted to the nodes $r_i$ and $r'_i$ respectively. 
The map $\varphi: \xx \mapsto \omega(c_2)/\omega(c_1)$ is a well defined holomorphic function on $\cV$ (when $c_i$ degenerates to the node $r_i$, $\omega(c_i) = 2\pi\imath\cdot\res_{r_i}(\omega)$). 

We claim that if $\cV$ is small enough then $\ol{\cX}_D\cap \cV$ is contained in the set $\{\xx\in \cV, \; \varphi(\xx)=\alpha\}$. This is because  if $\xx$ is close enough to $\pp$ then $c_1$ and $c_2$ are core curves of  two parallel cylinders on  $(X,\omega)$. 
By Proposition~\ref{prop:stable:cyl:dec}, we can suppose that corresponding cylinder decomposition is stable. Thus $\omega(c_2)/\omega(c_1)$ belongs to a finite set by Proposition~\ref{prop:cyl:dec:length:ratios:22}.
It follows that $\varphi$ is constant on all irreducible components of $\ol{\cX}_D\cap\cV$. Since  $\varphi(\pp)=\alpha$, the claim follows.

It can be shown that $d\varphi(\pp)\neq 0$. Thus $\varphi^{-1}(\{\alpha\})$ is a complex manifold of dimension $3$. 
By construction the map $\Phi$ is holomorphic, injective, and satisfies $\Phi(U\times\Delta_{\delta^2}\times\Delta_{\delta^2}) \subset \varphi^{-1}(\{\alpha\})$. 
Since $\dim(U\times\Delta_{\delta^2}\times\Delta_{\delta^2})=\dim \varphi^{-1}(\{\alpha\})=3$, we conclude that $\Phi(U\times\Delta_{\delta^2}\times\Delta_{\delta^2})$ is a neighborhood of $\pp$ in $\varphi^{-1}(\{\alpha\})$.
As the germ of $\ol{\cX}_D$ at $\pp$ is contained in $\varphi^{-1}(\{\alpha\})$, the lemma follows.
\end{proof}

\subsubsection{Proof of Proposition~\ref{prop:bdry:str:gp:II}}
\begin{proof}
We now give the proof of Proposition~\ref{prop:bdry:str:gp:II} in the case $\pp  \in \cS^b_{2,0}$.
Let $\cA$ be an irreducible component of the germ of $\ol{\cX}_D$ at $\pp$. By Lemma~\ref{lm:neigh:p:in:Sb:20:germs}, we can identify $\cA$ with a germ of analytic subsets of $U\times \Delta_{\delta^2}\times\Delta_{\delta^2}$.
Let $\cA^*$ denote the intersection $\cA\cap U\times \Delta^*_{\delta^2}\times\Delta^*_{\delta^2}$.
For every $\xx=(X,\ul{x},\tau_X,[\omega]) \in \cA^*$  close enough to $\pp$, the nodes $r_i$ and $r'_i$ correspond to two homotopic simple closed curves on $X$ that are contained in a cylinder $E_i$ invariant by $\tau_X$.
We claim that $E_1$ and $E_2$ are parallel.
Indeed, assume that they are not.
Let $\ell(E_i)$ and $h(E_i)$ be the length and the height of $E_i$.
Since $\ell(E_i)=\omega(c_i)$, as $\xx$ converges to $\pp$, $\ell(E_i)$ is bounded above by some constant $K$, while $h(E_i)$ tends to $+\infty$.
Since $(X,\omega)$ is completely periodic (cf. \textsection~\ref{subsec:commplete:per:n:cyl:dec}), $X$ admits a cylinder decomposition in the direction of $E_2$. The cylinder $E_1$ must intersect some cylinder, say $E$, parallel to $E_2$.  Since $E$ must cross $E_1$ entirely, we have $\ell(E) \geq h(E_1)$. It follows that $\ell(E)/\ell(E_2) \to 0$ as $\xx$ converges to $\pp$. But by  Proposition~\ref{prop:cyl:dec:length:ratios:22}, the ratio $\ell(E_2)/\ell(E)$ belongs to a finite set. We thus get a contradiction which proves the claim.

The complement of $E_1\cup E_2$ in $X$ is a four-holed torus on which  $\tau$ acts by a translation of order $2$.
We can choose a basis $(a_1, b_1, a_2,  b_2)$ of $H_1(X,\Z)^-$ as shown in Figure~\ref{fig:symp:basis:near:str:Sb:20}.
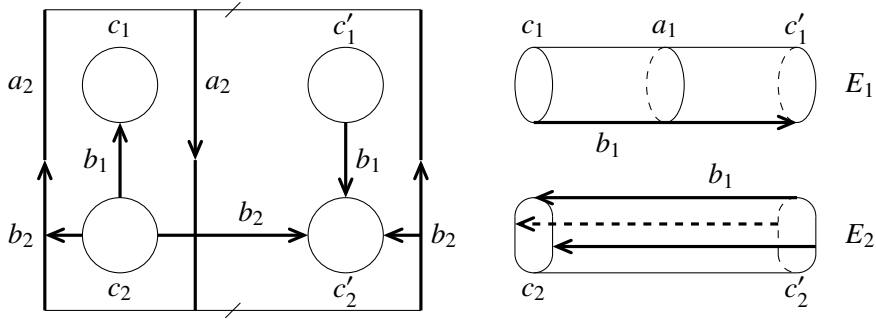
\begin{figure}[htb]
\centering
\begin{tikzpicture}[scale=0.5]
\draw (0,0) -- (10,0) (0,8) -- (10,8);
\foreach \x in {(5,0),(5,8)} \draw \x -- +(0.2,0.2) \x -- +(-0.2,-0.2);

\foreach \x in {(2,2), (2,6), (8,2), (8,6)} \draw \x 	circle (1);
\draw[very thick, ->, >=angle 45] (0,0) -- (0,4);
\draw[very thick, ->, >=angle 45] (10,0) -- (10,4);
\draw[very thick, ->, >=angle 45] (4,8) -- (4,4);
\draw[very thick] (0,4) -- (0,8);
\draw[very thick] (10,4) -- (10,8);
\draw[very thick] (4,4) -- (4,0);
\foreach \x in{(13,1), (13,3), (13,5), (13,7)} \draw \x -- +(7,0);
\foreach \x in{(13,6)} \draw \x ellipse (0.5 and 1);
\foreach \x in {(20,5), (16.5,5)} \draw \x arc (-90:90:0.5 and 1);
\foreach \x in {(20,7), (16.5,7)} \draw[dashed] \x arc (90:270:0.5 and 1);
\draw (13.5,2.5) arc (0:180:0.5);
\draw (12.5,1.5) arc (180:360:0.5);
\foreach \x in {(12.5,2.5), (13.5,2.5)} \draw \x -- +(0,-1);

\draw(20.5, 2.5) arc (0:90:0.5); \draw (20,1) arc (270:360:0.5); \draw (20.5,1.5) -- +(0,1);

\draw[dashed] (20,3) arc (90:180:0.5); \draw[dashed] (19.5,1.5) arc (180:270:0.5); \draw[dashed] (19.5,1.5) -- (19.5,2.5);

\draw[very thick, ->, >= angle 45] (1,2) -- (0,2);
\draw[very thick, ->, >= angle 45] (10,2) -- (9,2);
\draw[very thick, ->, >= angle 45] (3,2) -- (7,2);

\draw[very thick, ->, >=angle 45] (2,3) -- (2,5); \draw (2,4) node[left] {$\tiny b_1$}; 
\draw[very thick, ->, >=angle 45] (8,5) -- (8,3); \draw (8,4) node[right] {$\tiny b_1$};

\draw[very thick, ->, >=angle 45] (20,3) -- (13,3); 
\draw[very thick, ->, >=angle 45] (13,5) -- (20,5);
\draw[very thick, ->, >=angle 45] (20.5,1.7) -- (13.5,1.7);
\draw[very thick, dashed, ->, >=angle 45] (19.5,2.3) -- (12.5,2.3);

\draw (2,7) node[above] {$\tiny c_1$} (8,7.5) node {$\tiny c'_1 $};
\draw (2,1) node[below] {$\tiny c_2$}  (8,0.5) node {$\tiny c'_2$};
\draw (13,7) node[above] {$\tiny c_1$} (20,7.5) node {$\tiny c'_1 $};
\draw (13,1) node[below] {$\tiny c_2$}  (20,0.5) node {$\tiny c'_2$};
\draw (16.5,7) node[above] {$\tiny a_1$} (15,5) node[below] {$\tiny b_1$} (18,3) node[above] {$\tiny b_1$};
\draw (0,6) node[left] {$\tiny a_2$} (4,6) node[right] {$\tiny a_2$}
(0,2) node[left] {$\tiny b_2$} (5.5,2) node[above] {$\tiny b_2$} (10,2) node[right] {$\tiny b_2$};
\draw (21,6) node[right] {$\tiny E_1$} (21,2) node[right] {$\tiny E_2$};
\end{tikzpicture}
\caption{Symplectic basis of $H_1(X,\Z)^-$: $a_1$ and $b_1$ are simple closed curves, $a_2$ and $b_2$ have two components.}
\label{fig:symp:basis:near:str:Sb:20}
\end{figure}
Note that we $a_1=c_1, a_2=c_2-c_1$, and $\langle a_i,b_i\rangle=i, \; i=1,2$.
Since $b_1$ and $b_2$ cross the cylinders $E_1, E_2$, there is no consistent way to specify these elements of $H_1(X,\Z)$ when $\xx$ varies in $\cA^*$. Nevertheless, there is an open dense subset $\cA^*_0$ of $\cA^*$ such that the basis $\{a_1, b_1, a_2, b_2\}$ can be consistently chosen for all $\xx\in \cA^*_0$.
From now on, we will suppose that $\xx$ is a point in $\cA^*_0$.
By Proposition~\ref{prop:eigen:form:per:eq}, there is $T \in \End(\Prym(X,\tau))$ which  is given in the basis $(a_1, b_1,a_2, b_2)$ by an integral matrix of the form  $T=\left(\begin{smallmatrix}
e\cdot I_2 &  2B \\
B^* & 0 \\
\end{smallmatrix} \right)$, where $B=\left(\begin{smallmatrix}
a & b \\ c & d \end{smallmatrix} \right) \in \Mb_2(\Z)$, such that $T^*\omega=\lambda\cdot\omega$, where $\lambda=\frac{e+\sqrt{D}}{2} \in \R_{>0}$.
As a consequence, we have
\begin{equation}\label{eq:explicit:rel:per:symp:basis}
\omega(a_2)= \frac{2a}{\lambda}\omega(a_1)+\frac{2c}{\lambda}\omega(b_1) \quad \text{ and } \quad \omega(b_2)=\frac{2b}{\lambda}\omega(a_1)+\frac{2d}{\lambda}\omega(b_1).
\end{equation}
Since $\omega(a_1)=\omega(c_1), \omega(a_2)=\omega(c_2)-\omega(c_1)$, we get that
$$
\omega(c_2)=(1+\frac{2a}{\lambda})\omega(c_1)+\frac{2c}{\lambda}\omega(b_1).
$$
Since $\omega(c_1)$ and $\omega(c_2)$ (viewed as vectors in $\R^2$) are proportional, and $\omega(a_1)\wedge\omega(b_1) \neq  0$, we must have $c=0$ and $\omega(c_2)=(1+\frac{2d}{\lambda})\omega(c_1)$, which means that
\begin{equation}\label{eq:ratio:res:n:prototype}
\alpha = 1+\frac{2a}{\lambda}.
\end{equation}
Let us now prove
\begin{Claim}\label{clm:str:Sb:20:per:beta}
\begin{equation}\label{eq:str:Sb:20:per:beta}
\omega(b_1)=\frac{\ln(t_1)}{\pi\imath}+\frac{\alpha\ln(t_2)}{\pi\imath}+h_1(x) \quad \text{ and } \quad \omega(b_2)=\frac{2\alpha\ln(t_2)}{\pi\imath}+h_2(x)
\end{equation}
where $h_1$ and $h_2$ are holomorphic functions on $U$.
\end{Claim}
\begin{proof}
To see this, for all  $x \in  U$, let $u_{i,x}$ (resp. $u'_{i,x}$) be the point in $\cU_i$ (resp. $\cU'_i$) of coordinates $(x,\delta)$ in the identification $\cU_i\simeq U\times V_i$ (resp. $\cU'_i=U\times V_i$).
For all $\theta\in [0;2\pi]$, let $e^{\imath\theta}u_{i,x}$ (resp. $e^{\imath\theta}u'_{i,x}$) be the point of coordinates $(x,\delta e^{\imath\theta})$ in the same identification.
Note that we have $e^{\imath\theta}u_{i,x}\in C''_x$.

Let $v_i$ (resp. $v'_i$) denote the point of coordinate $\delta$ in $W_i$ (in $W'_i$).
For all $(t_1,t_2) \in \Delta_{\delta^2}\times \Delta_{\delta^2}$, we can choose a representative of $b_1$ which consists of
\begin{itemize}
\item[$\bullet$] a  path $\gamma_0 \subset C''_x$ from $u_{1,x}$ to $-u_{2,x}$, and $\gamma'_0 \subset C''_x$ from $-u'_{2,x}$ to $u'_{1,x}$,

\item[$\bullet$] a path $\sigma_1$ (resp. $\sigma'_1$) from $v_1$ to $u_{1,x}$ (resp. from $u'_{1,x}$ to $v'_{1}$) corresponding to a path from $\delta$ to $t_1$ in the annulus $A_{t_1/\delta, \delta}$,

\item[$\bullet$] a path $\sigma_2 \subset C'_2$ (resp. $\sigma'_2$) from $-u_{2,x}$ to $-v_{2,x}$ (resp. from $-v'_{2,x}$ to $-u'_{2,x}$) corresponding to a path from $-\delta$ to $-t_2$ in the annulus $A_{t_2/\delta,\delta}$,

\item[$\bullet$] a path $\gamma_1 \subset C'_1$ (resp. $\gamma_2\subset C'_2$) from $v'_{1,x}$ to $v_{1,x}$ (resp. from $v_{2,x}$ to $v'_{2,x}$).
\end{itemize}
The paths $\gamma_0, \gamma'_0, \gamma_1, \gamma_2$ can be chosen consistently for all $x \in U$. However, the paths $\sigma_i, \sigma'_i$ can be chosen consistently only on the domain $\{(x,t_1,t_2) \in U\times\Delta^*_{\delta^2}\times\Delta^*_{\delta^2},  \; -\pi < \arg(t_1) < \pi, \; -\pi < \arg(t_2) < \pi\}$. We can fix the homotopy class of $\sigma_i$ (resp. of $\sigma'_i$) by supposing that it does not cross the ray $\R_{\leq 0}\times\{0\}$.
We have
$$
\int_{b_1}\omega=\int_{\gamma_0}\omega+\int_{\gamma'_0}\omega+\int_{\gamma_1}\omega+\int_{\gamma_2}\omega+\sum_{i=1,2}\left(\int_{\sigma_i}\omega+ \int_{\sigma'_i}\omega\right).
$$
By construction, $\int_{\gamma_0}\omega+\int_{\gamma'_0}\omega+\int_{\gamma_1}\omega+\int_{\gamma_2}\omega$ is a holomorphic function on $U$. Since the restriction of $\omega$ to $V_1$ (resp. to $V'_1$) is given by $\frac{1}{2\pi\imath}\cdot dz/z$ (resp. $\frac{-1}{2\pi\imath}\cdot dz/z$), and the restriction of $\omega$ to $V_2$ (resp. $V'_2$) is given by $\frac{\alpha}{2\pi\imath}\cdot dz/z$  (resp. $\frac{-\alpha}{2\pi\imath}\cdot dz/z$), we get
$$
\sum_{i=1,2}\left(\int_{\sigma_i}\omega+ \int_{\sigma'_i}\omega\right)=\frac{1}{\pi\imath}(\ln(t_1)+\alpha \ln(t_2))+ {\rm const.}
$$
This proves the first equality. The second ones follows from similar arguments.
\end{proof}
It follows from \eqref{eq:explicit:rel:per:symp:basis} and \eqref{eq:str:Sb:20:per:beta}  that we have
$$
\frac{2\alpha\ln(t_2)}{\pi\imath}+h_2(x)=\frac{2b}{\lambda}+\frac{2d}{\lambda}\left(\frac{\ln(t_1)+\alpha\ln(t_2)}{\pi\imath}+h_1(x)\right).
$$
which is equivalent to
$$
d\ln(t_1)=\alpha(\lambda-d)\ln(t_2)+\phi(x)=(1+\frac{2a}{\lambda})(\lambda-d)\ln(t_2)+\phi(x)
$$
where $\phi$ is a holomorphic function on $U$. Since $\lambda$ is a root of the polynomial $P(x)=x^2-ex -2ad$, we have
$$
(1+\frac{2a}{\lambda})(\lambda-d)=2a-d+e.
$$
Thus $(x,t_1,t_2)$ satisfies
\begin{equation}\label{eq:XD:near:str:Sb:20}
t_1^d=t_2^{2a-d+e}\exp({\phi(x)}).
\end{equation}
Since every irreducible component of the germ of the analytic set defined by \eqref{eq:XD:near:str:Sb:20} in $\C^3$ is isomorphic to the set $\{(z,t_1,t_2), \; t_1^{m_1}=t_2^{m_2}\}$, with $\gcd(m_1,m_2)=1$, we get the desired conclusion.
\end{proof}

\subsection{Strata of group III}\label{subsec:bdry:str:gp:III}
Recall that strata in group III are $\cS^a_{2,1}, \cS_{3,1}, \cS_{2,2}, \cS_{1,3}$. If $\pp$ belongs to one of those strata then $C$ has a unique irreducible component, denoted by $C_0$, such that $\xi_{\left|C_0\right.}\equiv 0$.
All the nodes  incident to this component are fixed by $\tau$.
Outside of the nodes incident to $C_0$ there are four other nodes at which the differential $\xi$ has simple poles.
These nodes are partitioned into two pairs, the nodes in each pair are permuted by $\tau$.

\begin{Proposition}\label{prop:bdry:str:gp:III}
Let $\pp$ be a point in a stratum $\cS$  in group III. Then every irreducible component of the germ of $\ol{\cX}_D$ at $\pp$ is isomorphic to the germ at $0$ of the analytic set $\cA=\{(t_0,t_1,t_2,)\in \C^3, t_1^{m_1}=t_2^{m_2}\} \subset \C^3$, with $\gcd(m_1,m_2)=1$. 
In this identification, we have $\pp=0$ and $\cA\cap \cS=\{\pp\}$. In particular, the strata in group III consist of finitely many points in $\ol{\cX}_D$.
\end{Proposition}
\begin{proof}[Sketch of proof]
Let us denote by $r_i, r'_i, \, i=1,2$, the nodes at which $\xi$ has simple poles, where $r_i$ and $r'_i$ are permuted by $\tau$.
Set
$$
\alpha:=\frac{\res_{r_2}(\xi)}{\res_{r_1}(\xi)}.
$$
\begin{Claim}\label{clm:str:gp:III:ratio:res:const}
The number $\alpha$ is real and belongs to a finite subset of   $\R$. Moreover, for every $\xx=(X,\ul{x},\tau_X,[\omega])$ in the germ of $\ol{\cX}_D$ at $\pp$, we have
$$
\left(\int_{c_2}\omega\right)/\left(\int_{c_1}\omega\right)=\alpha.
$$
where $c_1$ (resp. $c_2$) is a simple closed curve on $X$ that is mapped to $r_1$ (resp. to $r_2$) by a degenerating map $f: X\to C$.
\end{Claim}
\begin{proof}
We first notice that $\varphi(\xx):=\left(\int_{c_2}\omega\right)/\left(\int_{c_1}\omega\right)$ is a well defined holomorphic function on a neighborhood of $\pp$ in $\Pb\Omega\ol{\cB}_{4,1}$.
For all $\xx\in \cX_D$, the nodes $\{r_i, r'_i\}$ correspond to either an invariant cylinder, or a pair of cylinders on $(X,\omega)$ permuted by $\tau_X$.
Since the moduli of those cylinders are large, they must be parallel, and therefore belong to the same cylinder decomposition of $(X,\omega)$.
By Proposition~\ref{prop:stable:cyl:dec}, we can suppose that the associated cylinder decomposition of $(X,\omega)$ is stable, thus given by one of the models in Proposition~\ref{prop:stable:cyl:dec:models:H22}.
Since $c_i$ is a core curve of the cylinder(s) associated to $\{r_i, r'_i\}$, $\varphi(\xx)$ is actually the ratio of the lengths of the corresponding cylinders. 
By Proposition~\ref{prop:cyl:dec:length:ratios:22}, the restriction of $\varphi$ to an open subset of $\cX_D$ containing $\xx$ takes values in a finite subset of $\R$.
Thus $\varphi$ is constant on each irreducible component of $\ol{\cX}_D$ in a neighborhood of $\pp$.
By definition we have $\varphi(\pp)=\alpha$. Thus $\varphi\equiv \alpha$ on all irreducible components of the germ of $\ol{\cX}_D$ at $\pp$.
\end{proof}

In all cases the component $C_0$ contains the marked points $\{p_5,p'_5\}$.
It follows from Theorem~\ref{th:twisted:diff} that $C_0$ carries a meromorphic Abelian differential $\eta_0$ that vanishes to the order $2$ at $p_5, p'_5$ and has poles with prescribed orders at the nodes incident to $C_0$.
The residues of $\eta_0$ at the nodes incident to $C_0$ are all zero (since all of these nodes are fixed by $\tau$).
Since $C_0$ is isomorphic to $\Pb^1$, these conditions determine $\eta_0$ up to a multiplicative scalar.

Let $C_j,\; j=1,\dots,m$, be the irreducible components of $C$ different from $C_0$. Then $\xi_j:=\xi_{\left|C_j\right.}$ is a non-trivial Abelian differential with at most simple poles on $C_j$.
The nodes between $C_j$ and $C_0$ are either regular points or zeros of $\xi_j$, while the self-nodes of $C_j$ (if any) and the nodes between $C_j$ and the other components of $C$ are simple poles of $\xi_j$. The condition that $\res_{r_2}(\xi)/\res_{r_1}(\xi)=\alpha$ then determines $\xi_j$ up to a multiplicative scalar. 

Let $r$ be a node of $C$.
\begin{itemize}
\item[$\bullet$] If $r$ is a node between $C_0$ and another component $C_j$, we specify a neighborhood $U$ of $r$ in $C_0$ and a neighborhood $V$ of $r$ in $C_j$ together with local coordinates $u$ on $U$, $v$ on  $V$ such that
$\zeta_{0\left|U\right.}=u^{-k(r)-1}du$,  $\xi_{j\left|V\right.}=v^{k(r)-1}dv$. Note that we always have $k(r)\geq 1$.

\item[$\bullet$] If $r$ is not incident to $C_0$, then let $C_{j}$ and $C_{j'}$, with $j,j'\in \{1,\dots,m\}$ (it may happen that $j=j'$), be the components that contain $r$. We choose a neighborhood $W$ of $r$ in $C_j$ and a neighborhood  $W'$ of $r$ in $C_{j'}$ together with local coordinates $w$ on $W$ and $w'$ on $W'$ such that
\begin{itemize}
\item[.] if $r\in \{r_1, r'_1\}$ then $\xi_{j\left|W\right.}=\frac{1}{2\pi\imath}\cdot \frac{dw}{w}$ and $\xi_{j'\left|W'\right.}=\frac{-1}{2\pi\imath}\cdot \frac{dw'}{w'}$, 

\item[.] if $r\in \{r_2, r'_2\}$ then $\xi_{j\left|W\right.}=\frac{\alpha}{2\pi\imath}\cdot \frac{dw}{w}$ and $\xi_{j'\left|W'\right.}=\frac{-\alpha}{2\pi\imath}\cdot\frac{dw'}{w'}$.
\end{itemize}

\end{itemize}

We can now use the data of $\{(C_0,\eta_0),(C_1,\xi_1),\dots,(C_m,\xi_m)\}$ to construct a holomorphic map $\Phi: \Bb \to \Pb\Omega'\ol{\cB}_{4,1}(2,2)$, where $\Bb$ is a small ball about $0$ in $\C^3$, as follows:
for all $t:=(t_0,t_1,t_2)\in \Bb$, the curve $C_t$ underlying $\Phi(t)$ is obtained from $C$ by smoothing its nodes in the following way
\begin{itemize}
\item[$\bullet$] Any node $r$ incident to $C_0$ corresponds to a collar on $C_t$ isomorphic to
\[
\{(u,v) \in \C^2, \; |u| <  \delta, |v| < \delta, uv=t_0^{n(r)}\},
\]
for some $\delta\in \R_{>0}$ and $n(r) \in \Z_{>0}$.  The numbers $n(r)$ 	are chosen so that $n(r)k(r)=n(r')k(r')$ if  $r$ and $r'$ are both incident to $C_0$, and
$$
\gcd\{n(r), \; \hbox{$r$ incident to $C_0$}\}=1. 
$$
\item[$\bullet$] Each of the nodes  $\{r_1, r'_i\}, \; i=1,2$, corresponds to a collar in $C_t$ isomorphic to
$$
\{(w,w')\in \C^2, \; |w| < \delta, |w'|
< \delta, \; ww'=t_i\}.
$$
\end{itemize}
Let $n$ be the common value of the products $n(r)k(r)$ with $r$ incident to $C_0$.  
The Abelian differentials $t^n_0\eta_0$, and $\{\xi_j, \; j=1,\dots,m\}$ induce a family of differentials each of which is defined on an open sub-surface of $C_t$. By construction, the differentials in this family coincide on the overlaps of the sub-surfaces. As a consequence, we obtained a well defined Abelian differential $\omega_t$ on $C_t$.
It also follows from the construction that $C_t$ inherits from $C$ an involution $\tau_t$ with four fixed points such that $\tau^*_t\omega_t=-\omega_t$.
The data of $(C_t,\tau_t,\omega_t)$ thus defines an element of $\Omega' \ol{\cB}_{4,1}$. 
Note that $\omega_t$ has two double zeros if $t_0\neq 0$. Therefore $(C_t,\tau_t,\omega_t)\in \Pb\Omega'\ol{\cB}_{4,1}(2,2)$.
By definition, $\Phi(t)$ is the projection of $(C_t,\rho_t,\omega_t)$ in $\Pb\Omega'\ol{\cB}_{4,1}$.
Clearly, we have $\Phi(0)=\pp$.
It is straightforward to check that $\Phi$ is injective, which means that $\Phi$ is a biholomorphic map onto its image.

\medskip

We now claim that $\Phi(\Bb)$ contains all the germs of $\ol{\cX}_D$ at $\pp$.  To see this, consider the function $\varphi$ defined in the proof of Claim~\ref{clm:str:gp:III:ratio:res:const}.
Recall that $\varphi$ is a well defined holomorphic function on a neighborhood $\cU$ of $\pp$ in $\Pb\Omega'\ol{\cB}_{4,1}(2,2)$.
It is a well known fact that $\pp$ is a regular point for $\varphi$. Thus $\varphi^{-1}(\{\alpha\}) \cap \cU$ is a $3$-dimension complex manifold.
By construction, $\Phi(\Bb)\subset \varphi^{-1}(\{\alpha\})$.
It follows that $\Phi(\Bb)$ is an open neighborhood of $\pp$ in $\varphi^{-1}(\{\alpha\})$ and the claim follows.

\medskip

Let $\cA$ be an irreducible component of the germ of $\ol{\cX}_D$ at $\pp$.
By the above claim, we can assume that $\cA\subset\Phi(\Bb)$.
Consider a point  $\xx=(X,\ul{x},\tau_X,[\omega])$ in $\cA\cap\cX_D$. Let $a_1=c_1-c'_1$ and $a_2=c_2-c'_2$, where $c'_i=\tau_X(c_i)$ is a simple closed curve on $X$ which is mapped to the node $r'_i$ on $C$.
Clearly we have  $a_1,a_2\in H_1(X,\Z)^-$.
We can find $b_1,b_2 \in H_1(X,\Q)^-$ such that $\{a_1,b_1,a_2,b_2\}$ is a symplectic basis of $H_1(X,\Q)^-$. 
By the arguments of  Proposition~\ref{prop:eigen:form:per:eq}, there exists $(a,b,d,e)\in \Q^4$ such that we have
\begin{equation}\label{eq:symp:basis:per:rel}
\omega(a_2)=\frac{a}{\lambda}\omega(a_1), \quad \text{ and } \quad  \omega(b_2)= \frac{b}{\lambda}\omega(a_1)+\frac{d}{\lambda}\omega(b_1) ,
\end{equation}
where $\lambda \in \R_{>0}$ satisfies $\lambda^2-e\lambda-ad=0$.
By assumption, we have
\begin{equation}\label{eq:ratio:res}
\frac{a}{\lambda}=\frac{\omega(a_2)}{\omega(a_1)}=2\alpha.
\end{equation}
By the same arguments as in the proof of Proposition~\ref{prop:bdry:str:gp:II}, we can write
\[
\omega(b_1)=\frac{\ln(t_1)}{2\pi\imath}+\phi_1(t), \quad \omega(b_2)= \frac{\alpha\ln(t_2)}{2\pi\imath}+\phi_2(t)
\]
where $\phi_1, \phi_2$ are holomorphic functions on $\Bb$.
Combine with \eqref{eq:symp:basis:per:rel}, we get that
\[
\alpha\ln(t_2)=\frac{d}{\lambda}\ln(t_1)+\phi(t)
\]
where $\phi$ is holomorphic on $\Bb$. Since $\alpha=\frac{a}{2\lambda}$, we get
$$
a\ln(t_2)=2d\ln(t_1)+2\lambda\phi(t)
$$
and therefore
\begin{equation}\label{eq:str:grp:III:a}
t_2^{m_2}=t_1^{m_1}\exp(\tilde{\phi}(t))
\end{equation}
for some $m_1, m_2 \in \Z_{>0}$ such that $\gcd(m_1,m_2)=1$ and $\tilde{\phi}$ a holomorphic function on $\Bb$.
Up to a change of coordinates, \eqref{eq:str:grp:III:a} is equivalent to $t_2^{m_2}=t_1^{m_1}$. In particular, the analytic subset $\tilde{\cA}$ of $\Bb$ defined by \eqref{eq:str:grp:III:a} has dimension 2. Since $\dim \cA=\dim \tilde{\cA}$, $\tilde{\cA}$ contains an open subset of $\cA$, and both $\tilde{\cA}$ and $\cA$ are irreducible, we conclude that $\tilde{\cA}=\cA$. The proposition is then proved.
\end{proof}

\subsection{Strata of group IV}\label{subsec:bdry:str:gp:IV}
There are two strata in group IV: $\cS_{2,1}^b$ and $\cS^c_{2,1}$. If $\pp$ is a point in one of those strata, then the curve $C$ has four irreducible components and six nodes. The differential $\xi$ has simple poles at all the nodes. In particular, $\xi$ is non-trivial on all components of $C$.

\begin{Proposition}\label{prop:bdry:str:gp:IV}
The strata of group IV consist of finitely many isolated points.
Every irreducible component of the germ of $\ol{\cX}_D$ at each of these points is isomorphic to the germ at $0\in \C^3$ of a surface $\{t_0^{m_0}=t_1^{m_1}t_2^{m_2}, \;  (t_0,t_1,t_2) \in\C^3\}$ with $(m_0,m_1,m_2) \in \Z_{>0}^3$ such that $\gcd(m_0,m_1,m_2)=1$.
\end{Proposition}
\begin{proof}[Sketch of proof]
Assume that $\pp$ is a point in $\cS_{2,1}^b\cup\cS^c_{2,1}$.
The nodes of $C$ are partitioned into $3$ pairs, the nodes in each pair are permuted by $\tau$. Let us denote the nodes of $C$ by $r_i, r'_i$, with $i\in \{0,1,2\}$, where $r'_i=\tau(r_i)$. 
%
For every point $\xx=(X,\ul{x},\tau_X,[\omega]) \in \Pb\Omega'\cB_{4,1}$ close enough to $\pp$, there is a degenerating map $f: X \to C$ such that the preimage of every node of $C$ is a simple close curve on $X$, and the restriction of $f$ to the complement of those curves is a homeomorphism onto the complement of the nodes in $C$.
Let $c_i$ and  $c'_i$  be respectively the preimages of $r_i$ and $r'_i$ in $X$. Note that since $\xi$ has simple poles at all the nodes, $c_i$ is non-separating for all $i=0,1,2$. If $c'_i$ is homologous to $-c_i$ then we set $a_i:=c_i$. Otherwise define $a_i:=c_i-c'_i$. By definition, $a_i \in H_1(X,\Z)^-$. We can always suppose that $a_1,a_2$ are part of a symplectic basis $(a_1,b_1,a_2,b_2)$ of $H_1(X,\Q)^-$, where $\langle a_i,b_i \rangle =1, \; i=1,2$. 
We also have 
\begin{equation}\label{eq:str:IV:rel:of:nodes}
a_0=s_1a_1+s_2a_2	
\end{equation} 
with $s_1,s_2 \in \Z$. Note that we have 
\begin{equation}\label{eq:str:IV:rel:intersections}
\langle a_0, b_i\rangle =s_i, \quad i=1,2.	
\end{equation}

The following claim follows from the same argument as Claim~\ref{clm:str:gp:III:ratio:res:const}

\begin{Claim}\label{clm:str:gp:IV:ratio:res:const}
There is a constant $\alpha\in \C$ such that for all $\xx\in \cX_D$ close enough to $\pp$ we have 
$$
\omega(a_2)/\omega(a_1)=\alpha. 
$$
\end{Claim}
Denote the components of $C$ by $\{C_j, \; j=1,\dots,4\}$.
Let $\xi_j$ be the restriction of $\xi$ to $C_j$.
Using the data $\{(C_1,\xi_1),\dots,(C_4,\xi_4)\}$, we define a holomorphic map $\Phi: \Bb\to \Pb\Omega'\ol{\cB}_{4,1}$, where $\Bb$ is small ball about $0$ in $\C^3$, by the standard plumbing constructions with parameters $t_i$ at the nodes $r_i$ and $r'_i$ for $i=0,1,2$.
It is not difficult to see that $\Phi$ is a biholomorphism onto its image.
By construction, we have $\Phi(0)=\pp$.
If $t_i\neq 0$ for all  $i=1,2,3$, then by construction $\Phi(t_0,t_1,t_2)$ is an element of $\Pb\Omega'\cB_{4,1}(2,2)$. It follows that $\Phi(\Bb) \subset \Pb\Omega'\ol{\cB}_{4,1}(2,2)$.
As a consequence of Claim~\ref{clm:str:gp:IV:ratio:res:const} we get
\begin{Claim}\label{clm:str:gp:IV:Psi:contains:germ}
The germ of $\ol{\cX}_D$ at $\pp$ is contained in $\Phi(\Bb)$.
\end{Claim}

Consider now a point $\xx \in \cX_D$ close to $\pp$.  By Claim~\ref{clm:str:gp:IV:Psi:contains:germ}, we can assume that  $\xx=\Phi(t)$, where $t=(t_0,t_1,t_2) \in \Bb$.
The arguments of Proposition~\ref{prop:eigen:form:per:eq}  imply that there exists $(a,b,d,e)\in \Q^4$ such that
\begin{equation}\label{eq:str:IV:per:rel}
\omega(a_2)=\frac{a}{\lambda}\omega(a_1), \quad \text{ and } \quad \omega(b_2) = \frac{b}{\lambda}\omega(a_1)+ \frac{d}{\lambda}\omega(b_1).
\end{equation}
where $\lambda\in \R_{>0}$ satisfies $\lambda^2-e\lambda-ad=0$. It follows from Claim~\ref{clm:str:gp:IV:ratio:res:const} that
\begin{equation}\label{eq:str:IV:alpha:ratio}
\alpha=\frac{a}{\lambda}.
\end{equation}
We can normalize $\omega$ by setting $\omega(a_1)=1$.
Since  $\langle a_0,b_i\rangle= s_i, \; i=1,2$, we have
\begin{align}
\label{eq:str:IV:per:b:1} \omega(b_1) & =\frac{\ln(t_1)}{2\pi\imath}+\frac{s_1(s_1+\alpha s_2)\ln(t_0)}{2\pi\imath}+\phi_1(t), \\
\label{eq:str:IV:per:b:2} \omega(b_2) & =\frac{\alpha\ln(t_2)}{2\pi\imath}+\frac{s_2(s_1+\alpha s_2)\ln(t_0)}{2\pi\imath}+\phi_2(t)
\end{align}
where $\phi_1, \phi_2$ are holomorphic functions on $\Bb$. Combining \eqref{eq:str:IV:per:b:1} and \eqref{eq:str:IV:per:b:2} with  \eqref{eq:str:IV:per:rel} and \eqref{eq:str:IV:alpha:ratio} we get
$$
\omega(b_2)=\frac{a}{\lambda}\frac{\ln(t_2)}{2\pi\imath} + (s_1s_2+s_2^2\frac{a}{\lambda})\frac{\ln(t_0)}{2\pi\imath} +\phi_2(t)
=\frac{d}{\lambda}\left(\frac{\ln(t_1)}{2\pi\imath} +(s_1^2+s_1s_2\frac{a}{\lambda})\frac{\ln(t_0)}{2\pi\imath}\right) + \phi_3(t) 
$$
which implies
\begin{align}
\label{eq:str:IV:log:rel} d\ln(t_1) & =a\ln(t_2)+(as_2^2-ds_1^2)\ln(t_0) +s_1s_2(\lambda-\frac{ad}{\lambda})\ln(t_0) +\phi(t) \\
\nonumber & = a\ln(t_2)+(as_2^2-ds_1^2+es_1s_2)\ln(t_0) + \phi(t) \quad \hbox{(here we used $\lambda^2-e\lambda-ad=0$)}
\end{align}
where $\phi$ is a holomorphic function on $\Bb$. 
Let $\Bb^*:=\{(t_0,t_1,t_2) \in \Bb, \; t_0t_1t_2 \neq 0\}$.
Then $\cX_D$ is contained in the set of $t \in\Bb^*$ which satisfies \eqref{eq:str:IV:log:rel}.
Up to a change of coordinates of $\Bb$,  every irreducible component of the set of $t\in \Bb$ satisfying  \eqref{eq:str:IV:log:rel} is defined by 
\begin{equation}\label{eq:str:IV:poly:rel}
t_0^{m_1}=t_1^{m_1}t_2^{m_2}
\end{equation}
with $(m_0,m_1,m_2)\in \N^3$ such that $\gcd(m_0,m_1,m_2)=1$.
Let $\cA$ be the irreducible component of the analytic set defined  by \eqref{eq:str:IV:poly:rel} that contains $\xx$.  Since $\xx$ is a regular point of $\cX_D$ by assumption, and $\dim\ol{\cX}_D=\dim \cA=2$, $\cA$  must equal an irreducible component of $\ol{\cX}_D$ in a neighborhood of $\pp$. This completes the proof of the proposition.
\end{proof}

\section{The normalization of $\ol{\cX}_D$ and the universal curve}\label{sec:normal:univ:curve}
Let $\hat{\cX}_D$ be the normalization of the space $\ol{\cX}_D$.
As a consequence of the results of \textsection~\ref{sec:geometry:bdry:XD}, we get
\begin{Proposition}\label{prop:norm:orbifold}
The space $\hat{\cX}_D$ is a complex orbifold.
\end{Proposition}
\begin{proof}
Since the local branches of $\ol{\cX}_D$ are separated in $\hat{\cX}_D$, it is enough to show that the normalization of every irreducible component of $\ol{\cX}_D$ at a point $\pp \in \ol{\cX}_D$ has at worst finite quotient singularities.
This is obvious if $\pp$ is a point in $\cX$.
Thus we only need to consider the case $\pp\in \partial\ol{\cX}_D=\ol{\cX}_D\setminus\cX_D$.
If $\pp$ belongs to a stratum of Group I then by Proposition~\ref{prop:bdry:str:gp:I} $\ol{\cX}_D$ is smooth at $\pp$, and we have nothing to prove.
If $\pp$ belongs to a stratum of group II or a stratum of group III, then by Proposition~\ref{prop:bdry:str:gp:II} and Proposition~\ref{prop:bdry:str:gp:III}, any irreducible local component of $\ol{\cX}_D$ at $\pp$ is isomorphic to the germ at $0$ of the set $\cA=\{(t_0,t_1,t_2)\in \C^3, \; t_1^{m_1}=t_2^{m_2}\}$, where $(m_1,m_2)\in \Z_{>0}^2$ satisfies $\gcd(m_1,m_2)=1$. Since the normalization of $\cA$ is $\C^2$ with the normalizing map $(t_0,t) \mapsto (t_0, t^{m_2}, t^{m_1})$, all the preimages of $\pp$ are smooth points in $\hat{\cX}_D$.
Finally, if $\pp$ is a point in a stratum of group IV, then by Proposition~\ref{prop:bdry:str:gp:IV}, any irreducible local component of $\ol{\cX}_D$ at $\pp$ is isomorphic to the germ at $0$ of the set $\cA=\{(t_0,t_1,t_2)\in \C^3, \; t_0^{m_0}=t_1^{m_1}t_2^{m_2}\}$, where $(m_0,m_1,m_2)\in \Z_{>0}^3$ satisfies $\gcd(m_0,m_1,m_2)=1$. It is a well known fact that the normalization $\hat{\cA}$ of $\cA$ is a quotient of $\C^2$ by an action of the cyclic group $\Z/m$, where $m=\frac{m_0}{\gcd(m_0,m_1)\gcd(m_0,m_2)}$, and the normalizing map $\hat{\cA}\to \cA$ is induced by the map
$$
(s,t) \in \C^2 \mapsto (s^{\frac{m_1}{\gcd(m_0,m_1)}}t^{\frac{m_2}{\gcd(m_0,m_2)}}, s^{\frac{m_0}{\gcd(m_0,m_1)}}, t^{\frac{m_0}{\gcd(m_0,m_2)}}) \in \cA.
$$
Note that the action of $\Z/m$ on $\C^2$ is generated by $(s,t) \mapsto (\zeta_m s, \zeta^k_m t)$, where $\zeta_m=\exp(2\pi\imath/m)$, and $k\in \Z$ is such that $\frac{km_2}{\gcd(m_0,m_2)}=-\frac{m_1}{\gcd(m_0,m_1)} \mod m$ (see \cite[\textsection 8]{Bai:GT} or \cite[\textsection III.6]{BHPVdV} for more details). In particular, all the points in the preimage of $\pp$ in $\hat{\cX}_D$ are finite quotient singularities.
Thus, we can  conclude that $\hat{\cX}_D$ is an orbifold.
\end{proof}

Let $\nu: \hat{\cX}_D \to \ol{\cX}_D$ be the normalizing map. 
Since the restriction of $\nu$ to $\nu^{-1}(\cX_D)$ is an isomorphism, we can consider $\cX_D$ as an open dense subset in  $\hat{\cX}_D$. The set $\partial \hat{\cX}_D:=\hat{\cX}_D\setminus \cX_D$ is called the {\em boundaries}  of $\hat{\cX}_D$. 
In what follows, we will label the  strata of $\partial \hat{\cX}_D$ by the same notation as their direct image in $\partial \ol{\cX}_D$.

Let $\hat{\cC}_D$ be the pullback of the universal curve on $\ol{\cX}_D$ to $\hat{\cX}_D$.
For $i=1,\dots,4$, there is a section of the projection $\hat{\pi}: \hat{\cC}_D \to \hXD$ which map associates to each $\pp=(C,p_1,\dots,p_5,p'_5,\tau,[\xi])$  the marked point $p_i$ on the fiber $\hat{\pi}^{-1}(\{\pp\}) \simeq C$. Denote by $\Sigma_i$ the image of this section.
Note that $\Sigma_i$ is a divisor in $\hat{\cC}_D$.
We have another divisor in $\hat{\cC}_D$ which intersects the fiber $\pi^{-1}(\{\pp\})$ at the points $p_5$ and $p'_5$. We denote this divisor by $\Sigma_5$. By a slight abuse of language, we will also call $\Sigma_5$ a section of $\hat{\pi}$.

We will translate the volume of $\cX_D$ into intersections of  cohomology classes on $\hat{\cC}_D$.
To this purpose, it is essential that the complex space underlying $\hat{\cC}_D$ has an orbifold structure.
Unfortunately, this is not the case in general.
For this reason, we need to consider a modification $\tilde{\cC}_D$ of $\hat{\cC}_D$ which is an orbifold with the following properties: let $\tilde{\pi}: \tilde{\cC}_D \to \hat{\cX}_D$ be the composition of the map $\varphi: \tilde{\cC}_D \to \hat{\cC}_D$ and the  projection $\hat{\pi}:\hat{\cC}_D \to \hat{\cX}_D$. Then
\begin{itemize}
\item[$\bullet$] all the fibers of $\tilde{\pi}$ are semistable curves,

\item[$\bullet$] $\tilde{f}$ restricts to an isomorphism from $\tilde{\pi}^{-1}(\cX_D)$ onto $\hat{\pi}^{-1}(\cX_D)$,

\item[$\bullet$] $\Sigma_i, \; i=1,\dots,5$, extends  to $\tilde{\cC}_D$ as section of $\tilde{\pi}$.
\end{itemize}
It is a well known fact that such a modification of $\hat{\cC}_D$ always exists. In what follows, we will give an explicit construction of $\tilde{\cC}_D$ adapted to our situation. The detailed description of $\tilde{\cC}_D$  is useful for the computations in \textsection~\ref{sec:curv:curent:vol}.

We will construct the space $\tilde{\cC}_D$ by gluing together analytic sets arising from neighborhoods of points in $\hat{\cC}_D$ possibly with some modification.
We call a point in $\hat{\cC}_D$ a {\em regular point} if it is either a smooth point or a finite quotient singularity.
Our construction of $\tilde{\cC}_D$ does not modify the analytic structure in a neighborhood of regular points.
In what follows $\qq$ will be a point in $\hat{\cC}_D$ and $B$ a neighborhood of $\qq$.
Let $\pp:=\hat{\pi}(\qq) \in \hat{\cX}_D$, and denote by $C_\pp$ the fiber $\hat{\pi}^{-1}(\{\pp\})$.

If $\qq$ is a smooth point on the curve $C_\pp$ then $\qq$ is a regular point in $\hat{\cC}_D$. In this case, we do not make any change to $B$.
From now on we will only focus on the case where $\qq$ is a node on the fiber $C_\pp$.

\begin{itemize}
\item[(a)] If $\pp$ is a point in a stratum of group I, then $\pp$ is a smooth point of $\hat{\cX}_D$ by  Proposition~\ref{prop:bdry:str:gp:I}.  As a consequence $\qq$ is a regular point of $\hat{\cC}_D$. In this case, we leave the neighborhood $B$ of $\qq$ unchanged.\\

\item[(b)] If $\pp$ is a point in a stratum of group II or of group III, then by Proposition~\ref{prop:bdry:str:gp:II} and Proposition~\ref{prop:bdry:str:gp:III}  $\pp$ is a smooth point in $\hat{\cX}_D$, and the normalizing map from a neighborhood of $\pp$ in $\hat{\cX}_D$ to  $\ol{\cX}_D$ is given by
$$
\begin{array}{cccc}
f:&  U\subset \C^2 & \to & \cA=\{(t_0,t_1,t_2) \in \C^3, \; t_1^{m_1}=t_2^{m_2}\}\\
  & (z,t) & \mapsto & (z,t^{m_2},t^{m_1})
\end{array}
$$
where $(m_1,m_2)\in \Z_{>0}^2$ satisfies $\gcd(m_1,m_2)=1$, and $U$ is a neighborhood of $0\in \C^2$.

By assumption $\qq$ is a node in $C_{\pp}$. Without loss of generality, we can suppose that $t_1$ is the smoothing parameter of this node.
This means that a neighborhood of $\qq$ in $\hat{\cC}_D$ is isomorphic to a neighborhood of $0$ in the analytic set $B=\{(x,y,z,t)\in \C^4, \; xy=t^{m_2}\}$. In this case,  $B$ is isomorphic to  a  quotient $\hat{B}/(\Z/m_2)$,
where $\hat{B}$ is an open subset in $\C^3$ containing $0$, and the action of $\Z/m_2$ on $\C^3$ is given by $\theta\cdot(u,v,z)=(\theta u,\theta^{-1}v,z)$ for all $\theta\in \Ub_{m_2}\simeq \Z/m_2$. The isomorphism between $\hat{B}/(\Z/m_2)$ and $B$ is induced by the map $(u,v,z) \mapsto (u^{m_2},v^{m_2},z,uv)$. In particular, $\qq$ is a regular point of $\hat{\cC}_D$, and we leave $B$ unchanged.\\

\item[(c)] In the case $\pp$ is a point in a stratum of group IV, by Proposition~\ref{prop:bdry:str:gp:IV}, any irreducible component of the germ of $\ol{\cX}_D$  at $\nu(\pp)$ is isomorphic to the germ of the analytic set $\cA=\{(t_0,t_1,t_2) \in \C^3, \; t_0^{m_0}=t_1^{m_1}t_2^{m_2}\}$ at $0 \in \C^3$, where $(m_0,m_1,m_2) \in \Z_{>0}^3$ satisfies $\gcd(m_0,m_1,m_2)=1$.
As a consequence, a neighborhood of $\pp$ in $\hat{\cX}_D$ is isomorphic to $\hat{\cA}=\C^2/(\Z/m)$,
where $m=\frac{m_0}{\gcd(m_0,m_1)\gcd(m_0,m_2)}$ and the action of $\Z/m$ on $\C^2$ is generated by
$(s,t) \mapsto (\zeta_m s, \zeta_m^kt)$, with $\zeta_m=\exp(2\pi\imath/m)$, and $k\in \Z$ such that $\frac{km_2}{\gcd(m_0,m_2)}=-\frac{m_1}{\gcd(m_0,m_1)} \mod m$.  The normalizing map $\hat{\cA} \to \cA$ is induced by the map
$$
\begin{array}{cccc}
\varphi: &  \C^2  & \to & \cA \\
				 & (s,t) & \mapsto & (s^{\frac{m_1}{\gcd(m_0,m_1)}}t^{\frac{m_2}{\gcd(m_0,m_2)}}, s^{\frac{m_0}{\gcd(m_0,m_1)}}, t^{\frac{m_0}{\gcd(m_0,m_2)}})
\end{array}
$$
Let $\Omega$ be a neighborhood of $0\in \C^2$ which is invariant by the action of $\Z/m$. Note that  the map $\varphi$ has degree $m$ and  $\Omega/(\Z/m)$ is isomorphic to a neighborhood of $\pp$ in $\hat{\cX}_D$.
Consider the pullback $\ol{\cC}_{\Omega}$ of the universal curve over $\cA$ to $\Omega$ by $\varphi$.
The preimage of $\qq$ in $\ol{\cC}_{\Omega}$,  which will be denoted by $\qq'$, is a node on the fiber over $0$.
Let us write $\varphi=(\varphi_0,\varphi_1,\varphi_2)$.
A neighborhood of $\qq'$ in $\ol{\cC}_{\Omega}$ is isomorphic to $B'=\{(x,y,s,t) \in \C^4, \; xy=\varphi_i(s,t)\}$, with some $i\in\{0,1,2\}$.
Note that $\Z/m$ acts on $B'$ by $\theta\cdot(x,y,s,t)=(x,y,\theta s, \theta^k t)$, and a neighborhood of $\qq$ in $\hat{\cC}_D$ is isomorphic to $B:=B'/(\Z/m)$.

\medskip

If $i\in \{1,2\}$ then $B'$ is isomorphic to the analytic set defined by the equation $xy=t^a$ for some $a\in \Z_{>0}$.
This implies that $B'$ is  isomorphic to the quotient of a neighborhood of $0\in \C^3$ by a linear action of $\Z/a$.
As a consequence, $B$ is also a finite quotient of an open subset of $\C^3$, which means that $\qq$ is regular. In this case, we leave $B$ unchanged.
It remains to consider the case where $i=0$, that is
$$
B' \simeq \{(x,y,s,t) \in \C^4, \; xy=s^at^b\},
$$
with
$$
a=\frac{m_1}{\gcd(m_0,m_1)}, \; b=\frac{m_2}{\gcd(m_0,m_2)}.
$$
Note that $\gcd(a,m)=\gcd(b,m)=1$.
We will replace $B'$ by a complex orbifold $\hat{B}$ together with a compatible action of $\Z/m$.
Define
$$
U:=\{(x,y,s,t,u)\in \C^5, \; x=us^a,  t^b=uy\} \; \text{ and } \;  V:=\{(x,y,s,t,u)\in \C^5, \; s^a=vx,  y=vt^b \}.
$$
Let $U^*:=\{(x,y,s,t,u) \in U, \; u\neq 0\} \subset U$ and $V^*:=\{(x,y,s,t,v)\in V, \; v\neq 0\} \subset V$.
We identify $U^*$ with $V^*$ by the mapping $(x,y,s,t,u) \leftrightarrow (x,y,s,t,1/v)$. Let $B''$ denote the  complex space obtained from $U\sqcup V$ by identifying $U^*$ with $V^*$ as above.
We define an action of $\Z/m$ on $U$ by
$$
\theta\cdot(x,y,s,t,u)=(x,y,\theta s, \theta^k t, \theta^{-a}u)
$$
and an action of $\Z/m$ on $V$ by
$$
\theta\cdot(x,y,s,t,v)=(x,y,\theta s, \theta^k t, \theta^{a}v),
$$
(recall that $k\in \Z$ satisfies $kb \equiv -a\mod m$).
These actions of $\Z/m$ are compatible with the identification $U^*\simeq V^*$. Thus, we have a well defined $\Z/m$ action on $B''$.

Note that $B''$ is an orbifold since it only has finite quotient singularities.
We have a natural projection $\phi: B'' \to B', \; (x,y,s,t,u) \mapsto (x,y,s,t)$.
Note that $\phi^{-1}(0)$ is isomorphic to $\Pb^1$, and $\phi$ restricts to an isomorphism from $B''\setminus\phi^{-1}(\{0\})$ onto $B'\setminus\{0\}$.
The $\Z/m$-actions on $B''$ and $B'$ are equivariant with respect to $\phi$. Therefore we have a well defined map
$$
\bar{\phi}: B''/(\Z/m) \to B'/(\Z/m)\simeq B
$$
which is an isomorphism outside of the set $\bar{\phi}^{-1}(\{\bar{0}\})$ (here $\bar{0}$ denotes the image of $0\in B'$ in $B'/(\Z/m)$).
We then replace $B$ by $\hat{B}:=B''/(\Z/m)$.
Remark that  $\bar{\phi}^{-1}(\{\bar{0}\})$ is isomorphic to $\Pb^1$.
\end{itemize}
In all cases, by construction $\hat{B}$ contains an open dense subset $\hat{B}^*$ that can be embedded into $\hat{\cC}_D$. In the case $\qq$ is regular, $\hat{B}^*=\hat{B}$. Therefore the analytic sets $\hat{B}$'s defined above patch together to give a complex space $\tilde{\cC}_D$.

\begin{Proposition}\label{prop:univ:curves:orb}
Let $\tilde{\cC}_D$ be the complex space constructed above. Then $\tilde{\cC}_D$ is an orbifold which comes equipped with a surjective map $\varphi: \tilde{\cC}_D \to \hat{\cC}_D$ such that the following diagram is commutative
\begin{center}
 \begin{tikzpicture}[scale=0.3]
    \node (A) at (0,6) {$\hat{\cC}_{D}$};
    \node (B) at (6,6) {$\ol{\cC}_{D}$};
    \node (C) at (0,0) {$\hat{\cX}_D$};
    \node (D) at (6,0) {$\ol{\cX}_D$};
    \node (E) at (-6,6) {$\tilde{\cC}_D$};

    \path[->, font=\scriptsize, >= angle 60] (A) edge node[above]{$\hat{\nu}$} (B);
    \path[->, font=\scriptsize, >= angle 60]
    (A) edge node[left]{$\hat{\pi}$} (C)
    (C) edge node[above]{$\nu$} (D)
    (B) edge node[right]{$\pi$} (D)
    (E) edge node[above] {$\varphi$} (A)
    (E) edge node[left] {$\tilde{\pi}$} (C);
 \end{tikzpicture}
\end{center}
The boundary $\partial \tilde{\cC}_D:=\tilde{\pi}^{-1}(\partial\hat{\cX}_D)$ is a normal crossing divisor in $\tilde{\cC}_D$.
Moreover, there is an $\Z/2$-action preserving the fibers of $\tilde{\pi}$, and $\varphi$ is  equivariant with respect to the $\Z/2$-actions on $\tilde{\cC}_D$ and $\hat{\cC}_D$,
All the fibers of $\tilde{\pi}$ are semistable curves, and their quotient by the $\Z/2$ action is a nodal genus one curve.

For every $\pp\in \hat{\cX}_D$, denote by $\tilde{C}_{\pp}$  and  by $C_\pp$ the fibers  of $\tilde{\pi}$ and $\hat{\pi}$ over $\pp$ respectively.
Then $\tilde{C}_{\pp} \simeq C_\pp$ if $\pp$ is not contained in the strata $\cS_{2,1}^b\cup \cS_{2,1}^c$.
In the case $\pp\in\cS_{2,1}^b\cup \cS_{2,1}^c$, $\tilde{C}_{\pp}$ has two extra $\Pb^1$ components that are mapped to two nodes of $C_\pp$ permuted by the Prym involution, the other components of $\tilde{\cC}_{\pp}$ are mapped isomorphically onto components of $C_\pp$.
\end{Proposition}

By a slight abuse of notation, we denote by $\Sigma_k, \; k=1,\dots,4$, the divisor in $\tilde{\cC}_D$ which intersects each fiber of $\tilde{\pi}$ at the $k$-th fixed point of the Prym involution, and by $\Sigma_5$  the divisor which intersects each fiber of $\tilde{\pi}$ at the two marked points that are permuted by the Prym involution.  
\section{Relations of divisors in $\tilde{\cC}_D$}\label{sec:div:relations:in:CD}
In preparation to the proof of Theorem~\ref{th:vol:Omg:E:D:2:2}, in this section we will prove some important relations between the tautological divisors in $\tCD$.
For all strata $\cS_{x,y}^{\bullet} \subset \partial\hXD$, we will denote its closure in $\hXD$ by $\ol{\cS}_{x,y}^{\bullet}$.  The inverse image of $\cS_{x,y}^{\bullet}$ in $\tCD$ will be denoted by $\cT_{x,y}^\bullet$. 

Let $\cT_{1,0}^0$ denote the subset of $\cT_{1,0}$ defined as follows: for all $\xx =(C_\xx,\ul{x},\tau_\xx,[\omega_\xx]) \in \cS_{1,0}$, $\cT_{1,0}^0$ intersects $C_\xx$ (considered as the fiber  $\tilde{\pi}^{-1}(\{\xx\}))$ in the $\Pb^1$ component of $C_\xx$. Note that $\omega_\xx$ vanishes identically on this component. 
Similarly, we define $\cT_{2,0}^{a,0}$ (resp. $\cT_{0,2}^0$) to be the subset of $\cT_{2,0}^a$ (resp. of $\cT_{0,2}$) such that for all $\xx \in \cS_{2,0}^a$ (resp. for all $\xx \in \cS_{0,2}$) $\cT_{2,0}^{a,0}$ (resp. $\cT_{0,2}^0$) intersects the fiber $C_\xx$ in the unique $\Pb^1$ component on which $\omega_\xx$ vanishes identically. 
Denote  by $\ol{\cT}_{1,0}^0, \ol{\cT}_{2,0}^{a,0}, \ol{\cT}_{0,2}^0$ the closures of $\cT_{1,0}^0, \cT_{2,0}^{a,0}, \cT_{0,2}^0$ in $\tCD$. Note that these subsets are divisors in $\tCD$. 

Recall that for all $\xx \in \cS_{2,0}^a$, $C_\xx$ has three irreducible components that are elliptic curves. One of those components is invariant while the other two are permuted by the Prym involution. Let $\cT_{2,0}^{a,1}$ denote the subset of $\cT_{2,0}^a$ such that for all $\xx\in \cS_{2,0}^a$, $\cT_{2,0}^{a,1}$ intersects the fiber $C_\xx$ in the invariant elliptic component of $C_\xx$.
Denote by $\ol{\cT}_{2,0}^{a,1}$ the closure of $\cT_{2,0}^{a,1}$ in $\tCD$.

Let  us denote by $\partial_1\hXD$  the union of all strata in group I, and by $\partial_\infty\hXD$ the union of all strata in the groups II, III, and IV in $\partial \hXD$. The points in $\partial_\infty\hXD$ correspond to Abelian differentials with simple poles at some nodes in $\partial\hXD$. Note that $\partial_\infty\hXD$ is in fact the closure of the strata in group II, and therefore a divisor in $\hXD$.
The inverse image of $\partial_\infty \hXD$ (resp. $\partial_1\hXD$) in $\tCD$ is denoted by $\partial_\infty\tCD$ (resp. $\partial_1\tCD$).
The main result of this section is the following

\begin{Proposition}\label{prop:rel:cotangent:class}
	We have the following relation in $\mathrm{Pic}(\tCD)\otimes\Q$,
	\begin{equation}\label{eq:rel:cotangent:class}
		[\omega_{\tilde{\cC}_D/\hat{\cX}_D}]=\frac{1}{6}\cdot[\ol{\cT}_{0,2}]+\sum_{i=1}^4[\Sigma_i]+2[\ol{\cT}^0_{1,0}]+[\ol{\cT}^{a,0}_{2,0}]+3[\ol{\cT}^{a,1}_{2,0}]+[\cR_1],
	\end{equation}
	where $\cR_1$ is a divisor with support contained in $\partial_\infty\tilde{\cC}_D$.
\end{Proposition}

\subsection{Fundamental relation in $\tilde{\cC}_D$}\label{subsec:fund:rel:tCD}
By definition $\ol{\cX}_D$ is a subvariety of $\Pb\Omega\ol{\cB}_{4,1}$. Let $\OO(-1)_{\ol{\cX}_D}$ denote the restriction of the tautological line bundle over $\Pb\Omega\ol{\cB}_{4,1}$ to $\ol{\cX}_D$. The pullback of $\OO(-1)_{\ol{\cX}_D}$ to $\hat{\cX}_D$ will be denoted by $\OO(-1)_{\hat{\cX}_D}$. For simplicity, when the context is clear we will write $\OO(-1)$ for the restriction of this bundle to various subsets of $\hXD$.

\begin{Proposition}\label{prop:fund:rel:tCD}
We have the following relation in $\mathrm{Pic}(\tCD)$
\begin{equation}\label{eq:fund:rel:tCD}
\tilde{\pi}^*\OO(-1)\sim \omega_{\tCD/\hXD}-2 \cdot[\Sigma_5] -5\cdot[\ol{\cT}_{1,0}^0]-[\ol{\cT}_{2,0}^{a,0}] - 3\cdot[\ol{\cT}_{0,2}^0].
\end{equation}
\end{Proposition}
\begin{proof}
Let $U$ be an open neighborhood of a point $\xx\in\hXD$ and suppose that there exists a  trivializing section of $\OO(-1)$ over $U$ given by $\xx \mapsto \omega_\xx$.
If $\xx \in \XD$, then $C_\xx$ is a smooth genus three curve, and $\omega_\xx$ has two double zeros at $x_5$ and $x'_5$. Since the open $U$ can be chosen to be disjoint from $\partial\hXD$,  \eqref{eq:fund:rel:tCD} holds true in $\CD= \tilde{\pi}^{-1}\XD$.

We now consider the case $\xx\in \partial\hXD$. 
Since $\tilde{\cC}_D$ is an orbifold, it is enough to show that \eqref{eq:fund:rel:tCD} holds true for all $\xx$ contained in strata of codimension $1$ in $\partial\hXD$. This means that we only need to consider the case $\xx$ belongs to a stratum in group I or group II. 

Assume first that $\xx$ is contained in a stratum of group II, that is $\xx\in \cS_{2,0}^b\cup \cS_{1,1}$. Then $\omega_\xx$ has double zeros at $\{x_5, x'_5\}$ and simple poles at all the nodes of $C_\xx$. This means that $\omega_\xx$ is a trivializing section of $(\omega_{\tCD/\hXD}-2\Sigma_5)_{\left|C_\xx\right.}$. In particular, \eqref{eq:fund:rel:tCD} holds true since we can choose $U$ to be disjoint from $\cS_{1,0}\cup\cS_{2,0}^a\cup\cS_{0,2}$.

Assume now that $\xx\in \cS_{1,0}$. Then $C_\xx$ has two irreducible components denoted by $C^0_\xx$ and $C^1_{\xx}$ meeting at one node where  $\omega_\xx$ vanishes identically on $C^0_\xx$, and $(C^1_\xx,\omega_\xx)$ is an element of $\Omega\cM_{3}(4)$.
Let $q$ be the unique node of $C_\xx$. There is a neighborhood $V$ of $q$ in $\tCD$ together with a coordinate system $(x,y,z)$, where $q \simeq (0,0,0) \in \C^3$, such that $C_\xx^0 =\{x=z=0\}, \; C_\xx^1=\{y=z=0\}$, and the projection $\tilde{\pi}$ is given by $\tilde{\pi}(x,y,z)=(xy,z)$ (here $\xx \simeq (0,0)$).
In this case, $dx/x$ is a trivializing section of $\omega_{\tCD/\hXD}$.
Up to a non-vanishing holomorphic function on $U$, we have $\omega_\xx= x^4dx=x^5\cdot  dx/x$.
Since $x^5$ can be seen as a trivializing section of the line bundle $-5\cdot[\ol{\cT}_{1,0}^0]$ in $V$, we get the desired conclusion.
The cases $\xx\in \cS_{2,0}^a\cup \cS_{0,2}$ follow from similar arguments.
\end{proof}

\subsection{Quotient and forgetful mappings}\label{subsec:forget:map}
Recall that the Prym involution stabilizes each fiber of $\tilde{\pi}: \tilde{\cC}_D \to \hat{\cX}_D$. Let $\tilde{\cE}_D$ denote the quotient of $\tilde{\cC}_D$ by the Prym involution, and $Q: \tilde{\cC}_D \to \tilde{\cE}_D$ the associated projection.
By definition, $\tilde{\cE}_D$ comes equipped with a projection $\tilde{\varpi}: \tilde{\cE}_D \to \hat{\cX}_D$,  whose fiber over a point $\xx=(C_\xx,x_1,\dots,x_5,x'_5,\tau_\xx, [\omega_\xx])$ is the tuple $(E_\xx, p_1,\dots,p_5)$, where $E_\xx:=C_\xx/\langle\tau_\xx\rangle$, and  $p_i$ is the image of $x_i$.
Note that $(E_\xx, p_1,\dots,p_5)$ is a semi-stable genus one curve with $5$ marked points that is actually stable unless $\xx$ belongs to the strata of group IV (which is a finite set of points) in $\partial\hat{\cX}_D$.

Removing the $4$ first  marked points $p_1,\dots,p_4$ on $E_\xx$, and passing to the  stable model, we obtain a family $\varpi: \cE \to \hat{\cX}_D$ of $1$-pointed genus one curves  over $\hat{\cX}_D$. The  fiber of  $\varpi$ over $\xx$ is the pair $(E'_\xx, p_5)$, which  is the stable model of $(E_\xx, p_5)$.
Recall that $E'_\xx$ is obtained from $E_\xx$ by successively collapsing the $\Pb^1$ components that either have only one node, or have two nodes and do not contain $p_5$. In particular $E'_\xx=E_\xx$ if $\xx\in \cX_D$. For $\xx$ contained in the strata of codimension $1$ in $\partial\hXD$, we have
\begin{itemize}
\item[$\bullet$] If $\xx\in \cS_{1,0}\cup\cS_{2,0}^a\cup\cS_{2,0}^b\cup\cS_{1,1}$ then $E_\xx$ has either one or two  $\Pb^1$ components.  In those cases, $E'_\xx$ is obtained by collapsing all the $\Pb^1$ components of $E_\xx$.


\item[$\bullet$] If $\xx\in \cS_{0,2}$ then $E_\xx$ has two $\Pb^1$ components which intersect at two nodes, and $E'_\xx$ is obtained by collapsing the $\Pb^1$ component that does not contains $p_5$ to a node.
\end{itemize}
We have naturally a map $F: \tilde{\cE}_D \to \cE_D$, and  the following commutative diagram
\begin{figure}[htb]
\begin{tikzpicture}[scale=0.5]
	\node (A) at (-5,5) {$\tilde{\cC}_{D}$};
	\node (B) at (0,5) {$\tilde{\cE}_{D}$};
	\node (C) at (5,5) {$\cE_D$};
	\node (D) at (0,0) {$\hat{\cX}_D$};
		
	\path[->, font=\scriptsize, >= angle 60] (A) edge node[above]{$Q$} node[below]{$[2:1]$} (B);
	\path[->, font=\scriptsize, >= angle 60] (B) edge node[above]{$F$} (C);
	\path[->, font=\scriptsize, >= angle 60]
	(A) edge node[left]{$\tilde{\pi}$} (D)
	(B) edge node[left]{$\tilde{\varpi}$} (D)
	(C) edge node[right]{$\varpi$} (D);
\end{tikzpicture}
\end{figure}

We will be interested in the pullback of the relative dualizing sheaf $\omega_{\cE/\hat{\cX}_D}$ to $\tCD$.
\begin{Proposition}\label{prop:rel:dual:pullback}
We have the following relation in $\mathrm{Pic}(\tCD)$
\begin{equation}\label{eq:rel:dual:pullback}
\omega_{\tilde{\cC}_D/\hat{\cX}_D} \sim Q^*\circ F^*\omega_{\cE_D/\hXD} +\sum_{i=1}^4[\Sigma_i] + 2[\ol{\cT}_{1,0}^0] +  [\ol{\cT}_{2,0}^{a,0}] + 3[\ol{\cT}_{2,0}^{a,1}]+ [\cR],
\end{equation}
where $\cR$ is a divisor with support contained in $\partial_\infty\tilde{\cC}_D$.
\end{Proposition}
\begin{proof}
We first compute the class of  $F^*\omega_{\cE_D/\hat{\cX}_D}$ in $\mathrm{Pic}(\tilde{\cE}_D)$.
Let $\ol{\cE}_{1,0}^0, \ol{\cE}_{2,0}^{a,0}, \ol{\cE}_{2,0}^{a,1}$ be respectively the images of $\ol{\cT}_{1,0}^0, \ol{\cT}_{2,0}^{a,0}, \ol{\cT}_{2,0}^{a,1}$ in $\tilde{\cE}_D$.
Consider a point $\xx=(C_\xx,x_1,\dots,x_5,x'_5,\tau_\xx, [\omega_\xx]) \in \hat{\cX}_D$. Recall that the fiber $\tilde{\varpi}^{-1}(\{\xx\})$ is the pointed curve $(E_\xx,p_1,\dots,p_5)$ where $E_\xx=C_\xx/\langle\tau_\xx\rangle$.
The map $F$ is defined by successively removing the marked points $p_1, p_2, p_3, p_4$ from the curve $E_\xx$ and passing to the stable model.
Thus, we have a sequence of maps 
\begin{equation}\label{eq:seq:collapse:maps}
\tilde{\cE}_D=\cE^1_D \overset{f_1}{\to} \cE^2_D \overset{f_2}{\to} \dots \overset{f_4}{\to} \cE^5_D=\cE_D,
\end{equation}
where each $f_i$ consists of passing to the stable model after removing the $i$-th marked point, and $F=f_4\circ\dots\circ f_1$.
Let $\varpi_i:\cE^i_D\to \hXD$ be the natural projection.
For $k=1,\dots,5$, let $\Gamma_k\subset \tilde{\cE}_D$ be the section of $\tilde{\varpi}$ that meets the fiber $E_\xx$ at $p_k$. 
By an abuse of notation, the images of $\Gamma_k$ in $\cE^i_D$ (which is a section of $\varpi_i$) will be denoted again by $\Gamma_k$.
It is a well known fact that we have
$$
f_i^*(\omega_{\cE^{i+1}_D/\hXD}(\Gamma_{i+1}+\dots+\Gamma_5)) \sim \omega_{\cE^{i}_D/\hXD}(\Gamma_{i+1}+\dots+\Gamma_5)
$$
(see for instance \cite[Ch. X, Prop.6.7]{ACG11}). Thus
$$
\omega_{\cE^{i}_D/\hat{\cX}_D}-f_i^*\omega_{\cE^{i+1}_D/\hat{\cX}_D}=(f^*_i\Gamma_{i+1}-\Gamma_{i+1})+\dots+(f^*_i\Gamma_5-\Gamma_5).
$$
By construction,  $f^*_i\Gamma_k-\Gamma_k$ (with $k>i$) is a divisor  in $\cE^{i}_D$ whose support meets the fibers of $\varpi_i$ in a $\Pb^1$ component that  contains only the $i$-th and $k$-th marked points together with a node. 

Let $\xx$ be a generic point in the image of the support of  $f^*_i\Gamma_k-\Gamma_k$. Denote by $E^{(i)}_\xx$ the fiber $\varpi^{-1}(\{\xx\})$ and by $P^{(i)}_\xx$ the component of $E^{(i)}_\xx$ that is contained in an irreducible  component $\cD_i$ of $\supp(f^*_i\Gamma_k - \Gamma_k)$. Let $q_\xx$ denote the node of $E^{(i)}_\xx$ contained in $P^{(i)}_\xx$.  The preimage $\tilde{q}_\xx$  of $q_\xx$ in $C_\xx$ consists of either one or two nodes.
\begin{itemize}
	\item[(i)] $\tilde{q}_\xx$ consists of one node. A neighborhood of $\tilde{q}$ in $\tCD$ can be identifies with a neighborhood $\cU$ of $0 \in \C^4$ in the set $\{(x,y,z,t) \in \C^4, \; xy=t\}$, and the projection $\tilde{\pi}$ is given by $\tilde{\pi}(x,y,z,t)=(z,t)$. The action of the Prym involution in $\cU$ corresponds to $(x,y) \mapsto (-x,-y)$. Thus a neighborhood of $q_\xx$ in  $\cE^i_D$ is identified with a neighborhood  $\cV$ of $0\in \C^4$ in the set $\{(u,v,z,t)\in \C^4, \; uv=t^2\}$, and the restriction of the map $f_{i-1}\circ\dots\circ f_1 \circ Q: \tCD \to \cE^i_D$ to $\cU$ is given by $(x,y,z,t) \mapsto (x^2,y^2,z,t)$.
	
	We can suppose that $\cD_i\cap \cV$ is defined by $u=t=0$. The collapsing map $f_i$ is then given by $f(u,v,z,t)=(u,z,t)$, and $\Gamma_k\cap f_i(\cV)$ is defined by the equation $u=0$. It follows that  $f^*_i\Gamma_k$ is the sum of the proper transform of $\Gamma_k$ (which is denoted by $\Gamma_k$ by a slight abuse of notation) and the divisor $\ord_{\cD_i}(u)\cdot\cD_i$, where $\ord_{\cD_i}(u)$ is the order of $u$ along $\cD_i$. 
	Since $\cV$ is defined by $uv=t^2$, in a neighborhood of a smooth point of $\cD_i$, we have $u\sim t^2$, while $\cD_i$ is defined by $t=0$. Thus we have $\ord_{\cD_i}(u)=2$, which implies that
	$$
	f^*_i\Gamma_k -\Gamma_k \sim 2\cD_i.
	$$
	
	\item[(ii)] $\tilde{q}_\xx$ contains two points. A neighborhood of $q_\xx$ in $\cE^i_D$ is isomorphic to a neighborhood of either point in $\tilde{q}_\xx$. One can easily check that in this case 
	$$
	f^*_i\Gamma_k -\Gamma_k \sim \cD_i.
	$$     
\end{itemize}
Analyzing the irreducible components of $\partial\tCD$ that are contracted in $\cE_D$,   we get that
\begin{equation*}\label{eq:rel:can:tED}
F^*\omega_{\cE_D/\hat{\cX}_D}\sim \omega_{\tilde{\cE}_D/\hat{\cX}_D}-2[\ol{\cE}_{1,0}^0]-[\ol{\cE}_{2,0}^{a,0}] -3[\ol{\cE}_{2,0}^{a,1}] + [\cR'],
\end{equation*}
where $\cR'$ is a divisor with support in $\partial_\infty\tilde{\cE}_D:= Q(\partial_\infty\tCD)$.
Finally as $Q^*\omega_{\tilde{\cE}_D/\hat{\cX}_D}\sim \omega_{\tilde{\cC}_D/\hat{\cX}_D} -\sum_{i=1}^4[\Sigma_i]$, we obtain
$$
Q^*\circ F^* \omega_{\cE_D/\hat{\cX}_D} \sim \omega_{\tilde{\cC}_D/\hat{\cX}_D}-\sum_{i=1}^4 [\Sigma_i] -2[\ol{\cT}_{1,0}^0]-[\ol{\cT}_{2,0}^{a,0}] -3[\ol{\cT}_{2,0}^{a,1}] + [\cR],
$$
where $\cR$ is a divisor with support in $\partial_\infty\tilde{\cC}_D$.
\end{proof}

\subsection{Proof of Proposition~\ref{prop:rel:cotangent:class}}
\begin{proof}
Since the restriction of $\omega_{\cE_D/\hat{\cX}_D}$ to the fiber of $\varpi$ is trivial, we have $\omega_{\cE/\hXD} \sim \varpi^*\cL$, where $\cL$ is a line bundle over $\hXD$. By construction, we have a morphism $\varphi: \hXD \to \ol{\cM}_{1,1}$ such that $\cL=\varphi^*\ol{\cH}$, where $\ol{\cH} \to \ol{\cM}_{1,1}$ is the Hodge bundle. It is well known that $\ol{\cH}\sim \frac{1}{12}\cdot[\delta_{\rm irr}]$, where $\delta_{\rm irr}$ is the point in $\ol{\cM}_{1,1}$ which represents the genus one curve with a non-separating node (see for instance \cite{Zvo} or \cite{ACG11}).
Thus we have
$$
\cL \sim \frac{1}{12}\cdot\varphi^*[\delta_{\rm irr}].  
$$
\begin{Claim}\label{clm:coeff:S:0:2:in:delta}
We have	
\begin{equation}\label{eq:pullback:del:irr}
\varphi^*[\delta_{\rm irr}] \sim 2[\ol{\cS}_{0,2}]+[\ol{\cS}_{1,1}].	
\end{equation}	
\end{Claim}
\begin{proof}
We first observe that $\varphi^{-1}(\delta_{\rm irr})=\ol{\cS}_{0,2} \cup \ol{\cS}_{1,1}$. Thus $\varphi^*[\delta_{\rm irr}]$ is a combination of $[\ol{\cS}_{0,2}]$ and $[\ol{\cS}_{1,1}]$.

Consider a point $\xx \in \cS_{0,2}$. The curve $C_\xx$ has two irreducible components: $C^0_\xx$ is isomorphic to $\Pb^1$, and $C^1_\xx$ is a smooth curve of genus two. These two components meet each other at two nodes both are fixed by the Prym involution.
Denote by  $q_1$ and $q_2$ the two nodes of $C_\xx$.

Let $(z,t)$ be a local system of coordinates in of $\hXD$ in a neighborhood $U$ of $\xx$ such that $\xx \simeq (0,0)$ and $\cS_{2,0}$ is defined by $t=0$. Using this coordinate system, we identify $U$ with a neighborhood of $0$ in $\C^2$. For all $\uu=(z,t)\in U$, the fiber of $\tilde{\pi}$ over  $\uu$ will be denoted by $C_{z,t}$. 

Recall from \textsection~\ref{subsec:geom:bdry:str:gp:I} that a neighborhood of one of the nodes of $C_\xx\simeq C_{0,0}$, say  $q_1$, in $\tCD$ is isomorphic to  the set 
$$
\cU_1:= \{(x,y,z,t) \in \Omega, \; xy=t\},
$$ 
while a neighborhood of $q_2$ is isomorphic to 
$$
\cU_2:=\{(x,y,z,t) \in \Omega, \;xy=t^3\},
$$ 
where $\Omega$ is a neighborhood of $0$ in $\C^4$.
We can suppose that in both cases, $y$ is the coordinate on the component $C^0_\xx\simeq\Pb^1$ of $C_\xx$. We now remark that there is  
an automorphism $\iota_\xx: C_\xx \to C_\xx$ that fixes $C^1_\xx$ pointwise and restricts to the involution of $C^0_\xx$ fixing $q_1$ and $q_2$. The automorphism $\iota_\xx$ gives rise to an involution $\iota$ on $\tilde{\pi}^{-1}(U)$ whose restriction to $\cU_i$ is given by  $(x,y,z,t)\mapsto (x,-y,z,-t)$. In particular, we have that $\iota(C_{z,t})=C_{z,-t}$, that is $C_{z,t}$ and $C_{z,-t}$ are isomorphic. Therefore,  $\varphi(z,t)=\varphi(z,-t)\in \ol{\cM}_{1,1}$. 
In a suitable local coordinate of $\ol{\cM}_{1,1}$ such that $\delta_{\rm irr}\simeq 0$, the restriction of $\varphi$ to $U$ is  given by $\varphi(z,t)=t^2$. This implies that the coefficient of $[\ol{\cS}_{0,2}]$ in $\varphi^*[\delta_{\rm irr}]$ is $2$.

In the case $\xx \in \cS_{1,1}$, none of the node of  $C_\xx$ is fixed by $\tau_\xx$. Therefore, the coefficient of $[\ol{\cS}_{1,1}]$ in $\varphi^*[\delta_{\rm irr}]$ is $1$. This completes the proof of the claim. 
\end{proof}

It follows from Claim~\ref{clm:coeff:S:0:2:in:delta} that we have
\begin{equation}\label{eq:pullback:rel:cotangent:M11}
Q^*\circ F^*\omega_{\cE_D/\hXD} \sim \tilde{\pi}^*\cL \sim \frac{1}{12}\cdot\left(2[\ol{\cT}_{0,2}] +[\ol{\cT}_{1,1}]\right).
\end{equation}
Note that  $\ol{\cT}_{1,1}$ is contained in $\partial_\infty\tCD$.
Combining \eqref{eq:pullback:rel:cotangent:M11} with \eqref{eq:rel:dual:pullback} we obtain \eqref{eq:rel:cotangent:class}.
\end{proof}

 \section{Curvature, current, and volume of $\cX_D$}\label{sec:curv:curent:vol}
\subsection{Definition of the $(2,2)$-form $\Theta$}\label{subsec:Theta:def}
We consider $\cX_D$ as an open dense subset of $\hat{\cX}_D$. 
Over $\cX_D$, we have a Hermitian metric on $\OO(-1)$ given by the Hodge norm. 
Let $\xx:=(X,\ul{x},\tau_\xx,[\omega_\xx])$ be an element of $\cX_D$. Then the fiber $\OO(-1)_\xx$ of $\OO(-1)$ over $\xx$ is the precisely the line $\C\cdot\omega_\xx \subset H^{1,0}(X)$.  The Hodge norm of $\omega_\xx$ is given by
$$
||\omega_\xx||^2:=\frac{\imath}{2}\int_X\omega_\xx\wedge\ol{\omega}_\xx.
$$
Let $\vartheta$ denote the curvature form of the Hogde norm. Recall that by definition, $\vartheta$ is given by
$$
\vartheta=-\partial\ol{\partial}\ln(||\omega_\xx||^2).
$$
where $\sigma$ is any local holomorphic section of $\OO(-1)$.
\begin{Lemma}\label{lm:curv:form:loc:express}
Let $\alpha$ be a combination of simple closed curves on $X$ which represents a non-trivial element of $H_1^(X,\Z)^-$. 	For all $\yy=(Y,\ul{y},\tau_\yy,[\omega_\yy])$ in a neighborhood $U$ of $\xx$, we can consider $\alpha$ as an element of $H_ 1(Y,\Z)^-$.
Suppose that there is an assignment $\yy \mapsto \omega_\yy$ such that $\omega_\yy(\alpha)=1$ for all $\yy\in U$. Define 
$$
\Ab(\yy):=||\omega_\yy||^2.
$$
Then we have 
\begin{equation}\label{eq:curv:form:loc:express}
	\vartheta = \frac{\partial \Ab\wedge \bar{\partial}\Ab}{\Ab^2}.
\end{equation}
\end{Lemma}
\begin{proof}
	By definition, the correspondence $\sigma: \yy \to \omega_\yy$ is a local section of $\OO(-1)$ on $U$. Thus 
	$$
	\vartheta= -\partial\bar{\partial}\ln(\Ab)=-\frac{\partial\bar{\partial} \Ab}{\Ab} + \frac{\partial \Ab\wedge \bar{\partial}\Ab}{\Ab^2}.
	$$
	We will show that $\partial\bar{\partial} \Ab=0$.
	There is a symplectic basis $(a_1,b_1,a_2,b_2)$ of $H_1(X,\Q)^-$ with $\alpha=a_1$.  By Proposition~\ref{prop:eigen:form:per:eq},  there is a matrix $M\in \Mb_{2}(\Q(\sqrt{D}))$ such that the following holds
	$$
	(\omega_\yy(a_2) \quad \omega_\yy(b_2))=(\omega_\yy(a_1) \quad \omega_\yy(b_1))\cdot M.
	$$
	for all $\yy \in U$. This means that $\omega_\yy(a_2)$ and $\omega_\yy(b_2)$ are linear functions of $(\omega_\yy(a_1),\omega_\yy(b_1))$.
	Since $\omega_\xx(a_1)\equiv 1$, $\omega_\yy(a_2)$ and $\omega_\yy(b_2)$ are real affine functions of $\omega_\yy(b_1)$. Let $\beta(\yy):=\omega_\yy(b_1)$. We then get
	\begin{align*}
		\Ab(\yy):=||\omega_\yy||^2 & = \frac{\imath}{2}\left(\bar{\beta}(\yy)-\beta(\yy)+ \omega_\yy(a_2)\ol{\omega}_\yy(b_2)-\ol{\omega}_\yy(a_2)\omega_\yy(b_2)\right)\\
		& =\frac{\imath}{2}\cdot R \cdot (\bar{\beta}(\yy)-\beta(\yy))
	\end{align*}
	where $R$ is a real constant. Since $\beta$ is a holomorphic function, we must have $\partial\bar{\partial}\Ab=0$. The lemma is then proved.
\end{proof}

Let $\pi: \cC_D \to \cX_D$ denote the universal curve over $\cX_D$.
By a slight abuse of notation, the pullback of the curvature form of the Hodge norm to $\cC_D$ will be also denoted by $\vartheta$.
Recall that a point $\hxx$ in the fiber $\pi^{-1}(\{\xx\})$,  is a pair $(\xx,x)$, where $x$ is a point in $X$. Consider a path $c(\hxx)$ from $x$ to $\tau_\xx(x)$ on $X$.
For every $\hyy=(\yy,y)\in \cC_D$ close enough to $\hxx$, there is a distinguished homeomorphism $h_\yy: (Y,y) \to (X,x)$, where $Y$ is the Riemann surface underlying $\yy$, determined up to homotopy. We can suppose that $h_\xx\circ \tau_\yy = \tau_\xx\circ h_\xx$. Let $c(\hyy)$ be the image of $c(\hxx)$ by such a map. Then $c(\hyy)$ is a path from $y$ to $\tau_\yy(y)$. 
Define
\begin{equation}\label{eq:normalized:per:funct}
\varphi_c(\hyy):=\frac{\left|\int_{c(\hyy)}\omega_\yy\right|^2}{||\omega_\yy||^2}.
\end{equation}
Observe that $\varphi_c(\hxx)$ does not depend on the choice of the representative $\omega_\xx$ of the line $[\omega_\xx] \subset \Omega(X)^-$. 
\begin{Proposition}\label{prop:Theta:def:22:form}
	The closed $(2,2)$-form
	\begin{equation}\label{eq:Theta:CD:def}
		\Theta:=(\imath\vartheta)\wedge\left(\frac{\imath}{2}\partial\bar{\partial}\varphi_c\right)
	\end{equation}
	does not depend on the choice of the path $c$, and therefore is  well defined on $\cC_D$.
\end{Proposition}
\begin{proof}
Let $\cU$ be an open neighborhood of $\hxx$ in $\cC_D$ and $U$  the projection of $\cU$ in $\cX_D$.
We can suppose that for all $\yy=(Y,\ul{y},\tau_\yy,[\omega_\yy]) \in U$ there is a distinguished symplectic basis $(a_1,b_1,a_2,b_2)$ of $H_1(Y,\Z)^-$. We can also assume that $\omega_\yy$ satisfies $\omega_\yy(a_1)=1$ for all $\yy\in U$. This means that the correspondence  $\sigma: \yy \to \omega_\yy$ is a section of $\OO(-1)$ defined on $U$. Let $\beta(\yy):=\omega_\yy(b_2)$ and $\Ab(\yy)=||\omega_\yy||^2$. It follows from Lemma~\ref{lm:curv:form:loc:express} that we have $\Ab=\frac{\imath}{2}\cdot R(\bar{\beta}-\beta)$, where $R$ is a real constant, and
$$
\vartheta = \frac{\partial \Ab\wedge \bar{\partial}\Ab}{\Ab^2}=\frac{d\beta\wedge d\bar{\beta}}{4\Im(\beta)^2}.
$$

Let $P(\hyy):=\int_{c(\hyy)}\omega_\yy$.
By definition, $\varphi_c(\hyy)=|P(\hyy)|^2/\Ab(\yy)$. Thus
\begin{equation*}
	\partial\bar{\partial}\varphi_c =\frac{dP\wedge d\bar{P}}{\Ab}+\frac{\imath R}{2}\cdot \frac{P}{\Ab^2}\cdot d\beta \wedge d\bar{P} - \frac{\imath R}{2} \cdot \frac{\bar{P}}{\Ab^2}\cdot dP\wedge d\bar{\beta}+\frac{R^2}{2}\cdot \frac{|P|^2}{\Ab^3}\cdot d\beta\wedge d\bar{\beta},
\end{equation*}
and therefore
\begin{equation}\label{eq:Theta:loc:express}
\left(\imath\vartheta\right)\wedge\left(\frac{\imath}{2}\partial\bar{\partial}\varphi_c\right)=-\frac{1}{2}\cdot\left(\frac{d\beta\wedge d\bar{\beta}}{4\Im(\beta)^2}\right)\wedge \left(\frac{dP\wedge d\bar{P}}{\Ab}\right).
\end{equation}

Let $c'(\hxx)$ be another path on $X$ from $x$ to $\tau_\xx(x)$. Then $\hat{c}:=c'*(-c)$ is an element of $H_1(X,\Z)$. Note that we can identify $H^1(X,\Z)^-$ with $H^1(Y,\Z)^-$ for all $\yy=(Y,\ul{y},\tau_\yy,[\omega_\yy]) \in U$. 
We can write $\hat{c}=\hat{c}^+ + \hat{c}^-$, where $\tau_{\yy*}\hat{c}^+=\hat{c}^+$ and $\tau_{\yy*}\hat{c}^-=-\hat{c}^-$. 
Since $\omega_\yy \in \Omega(Y)^-$, we have $\omega_\yy(\hat{c})=\omega_\yy(\hat{c}^-)$.
By Proposition~\ref{prop:eigen:form:per:eq}, $\omega_\yy(\hat{c}^-)$ is a linear function with real coefficients  in the variables $(\omega_\yy(a_1),\omega_\yy(b_1))$. Since $\omega_\yy(a_1)\equiv 1$, $\omega_\yy(\hat{c}^-)$ is actually a real affine function of $\beta$. 
	
Let $P'(\hyy)$ be the integral of $\omega_\yy$ along the path $c'(\hyy)$.  We then have $P'(\hyy)=P(\hyy)+\omega_\yy(\hat{c}^-)$. Therefore, $dP'=dP + rd\beta$  and $d\bar{P}'=d\bar{P}+rd\bar{\beta}$, 
where $r\in\R$. It follows immediately from \eqref{eq:Theta:loc:express} that
	$$
	\vartheta\wedge\partial\bar{\partial}\varphi_{c'}=\vartheta\wedge\partial\bar{\partial}\varphi_{c}
	$$
and the proposition follows.
\end{proof}

Our goal now to prove the following
\begin{Theorem}\label{th:Theta:current}
	The $(2,2)$-form $\Theta$ defined in Proposition~\ref{prop:Theta:def:22:form} is a closed current on $\tilde{\cC}_D$.
\end{Theorem}

Recall that $\partial\tilde{\cC}_D$ is a divisor with normal crossings (in the orbifold sense) in $\tilde{\cC}_D$. Since $\Theta$ is a smooth closed $(2,2)$-form in $\tilde{\cC}_D\setminus\partial\tilde{\cC}_D$, to show that $\Theta$ defines a closed current on $\tilde{\cC}_D$ is amount to prove the following: for all $\hpp=(\pp,p)\in \partial\tilde{\cC}_D$, that is $\pp\in \partial\hXD$ and $p$ is a point in the fiber $\tilde{\pi}^{-1}(\{\pp\})$, let $(x_1,x_2,x_3)$ be a local coordinate system in a neighborhood of $\hpp$ such that $\partial\tilde{\cC}_D$ is defined by $x_1\dots x_r=0, \; r\in\{1,2,3\}$. Then we have

\begin{itemize}
\item[(A)] For all $I=\{i_1,i_2\}\subset \{1,2,3\}$, and $J=\{j_1,j_2\}\subset \{1,2,3\}$, the function
$$
a_{I,J}:=\Theta(\partial x_{i_1},\partial x_{i_2}, \partial \bar{x}_{j_1}, \partial \bar{x}_{j_2})
$$
is $L^1_{\rm loc}$, and

\item[(B)] For all $\eps>0$, denote by $\cU_\eps$  the $\eps$-neighborhood of $\partial \tilde{\cC}_D$, then we have
$$
\lim_{\eps\to 0}\int_{\partial \cU_\eps}\Theta\wedge dx_i= \lim_{\eps\to 0}\int_{\partial \cU_\eps}\Theta\wedge d\bar{x}_i=0
$$
for all $i\in \{1,2,3\}$.
\end{itemize}
To prove those properties of $\Theta$ it is essential to have a convenient expression of the $1$-forms $dP$ and $d\bar{P}$ in \eqref{eq:Theta:loc:express}. 

Let us consider a family of nodal curves $\varrho: \cY \to U$, where $U$ is an open neighborhood of $0\in \C^N$. 
For all $\xx \in U$, denote the fiber $\varrho^{-1}(\{\xx\})$ by $Y_\xx$. We assume that 
\begin{itemize} 
\item[(i)] There is an involution $\tau_\cY$ on $\cY$ which restricts to an admissible involution  on each fiber $Y_\xx$. This implies in particular that if $q$ is a node of $Y_\xx$  fixed by $\tau_\cY$, then the two local branches of $Y_\xx$ at $q$ are invariant by $\tau_\cY$.  

\item[(ii)] There is a system of coordinates  $(z_1,\dots,z_{N-n},t_1,\dots,t_{n})$ on $U$ such that $Y_\xx$ is smooth if and only if $\xx\in U^*:=\{(z_1,\dots,z_{N-n},t_1,\dots,t_n)\in U, \; t_1\cdots t_n\neq 0\}$.

\item[(iii)] Let $\{q_j, \; j\in J\}$ be the set of nodes of $Y_0$. 
For every  $j\in J$, there exist  $i=i(j) \in \{1,\dots,n\}$ and a positive integer $r=r(j)$ such that a neighborhood  of $q_j$ in $\cY$ is isomorphic to the  analytic set 
$$
\cA_j:=\{(u,v,z_1,\dots,z_m,t_1,\dots,t_n) \in \C^2\times U, \; |u| < \delta, |v| < \delta, \; uv=t_{i}^{r}\},
$$
with $\delta\in \R_{>1}$. We suppose moreover that the sets $\cA_j$'s are pairwise disjoint, and for each $j\in J$, either  $\cA_j$ is invariant by $\tau_{\cY}$ in which case the restriction of $\tau_\cY$ to $\cA_j$ is given by $(u,v,z_1,\dots,z_{N-n},t_1,\dots,t_n) \mapsto (-u,-v,z_1,\dots,z_{N-n},t_1,\dots,t_n)$, or  there exists $j'\in J\setminus\{j\}$ such that $\cA_j$ and $\cA_{j'}$ are permuted by $\tau_\cY$.
\end{itemize}
For simplicity, in what follows we will write $z=(z_1,\dots,z_{N-m})$,  $t=(t_1,\dots,t_n)$, and for any subset $V\subset U$, $\cY_{|V}=\varrho^{-1}(V)$.
For all $\xx\in U^*$ and $j\in J$, let $a_j(\xx)$ denote a core curve of the annulus $\cA_j\cap Y_\xx$. The monodromy of the family $\cY_{|U^*}$ is generated by products of simultaneous Dehn twists about the curves $a_j(\xx)$.
The set $U^*$ can be covered by a finite family of open subsets $\{U^*_k, \; k=1,\dots,m\}$ such that for each $k$ the fiberation $\varrho: \cY_{|U^*_k} \to U^*_k$ is trivial. This means that we have an isomorphism of fiberations $\cY_{|U^*_k} \simeq U^* _k\times Y_{\xx_k}$, where $\xx_k$ is an arbitrary point in $U^*_k$. 

\medskip 
 
Let $y_0$ be a point in $Y_0$ and consider a neighborhood $\cU$ of $y$ in $\cY$. We wish to specify for each $y \in \cU\cap \varrho^{-1}(U^*)$ a path from $y$ to $\tau_{\cY}(y)$ in the smooth curve $Y_{\varrho(y)}$ in a coherent manner. We distinguish two cases:
\begin{itemize}
	\item[(i)] $y_0$ is fixed by $\tau_\cY$. We have two subcases:
	 \begin{itemize}
	 	\item[(i.a)] $y_0$ is a smooth point in $Y_0$.  We choose $\cU$ to be a neighborhood of $y$ such that  $(\cU, y_0) \simeq (\Delta(\rho)\times V, 0)$, where $\rho$ is a small positive real number, and $V$ is an open neighborhood of $0$ in $U$, and the restriction of $\tau_\cY$ to $\cU$  is given by $(w,z,t)\mapsto (-w,z,t)$. In this case, for all $y \simeq (w,z,t)\in \cU$ we denote by $c(y)$ the segment $[w,-w]\times\{(z,t)\} \subset \cU\cap Y_{(z,t)}$. 
	 	
	 	\item[(i.b)] $y_0=q_j$ is a node of $Y_0$. In this case we take $\cU=\cA_j$. For all $y\simeq (u,v,z,t)$, denote by $c(y)$ the path $\theta \mapsto (e^{\imath\theta}u, e^{-\imath\theta}v, z,t)$, with $\theta\in [0;\pi]$. One readily checks that $c(y)$ joins $y$ to $\tau_{\cY}(y)$ and is contained in $\cU$. 
	 \end{itemize}

	\item[(ii)] $y_0$ is not invariant by $\tau$. Again, we have two subcases:
	\begin{itemize}
		\item[(ii.a)] $y_0$ is a smooth point of $Y_0$. Let $y'_0:=\tau_\cY(y_0)$. We choose a neighborhood $\cU$ of $y_0$ such that $(\cU,y_0) \simeq (\Delta(\rho)\times V,0)$, with $\rho$ being a small positive real number, and $V$  an open neighborhood of $0$ in $U$. Let $\cU':=\tau_{\cY}(\cU)$.
		We identify $(\cU',y'_0)$ with $\Delta(\rho)\times V$ so that the restriction of $\tau_\cY$ to $\cU$ is given by $(w,z,t) \mapsto (-w,z,t)$.
		We can suppose that $\cU\cup \cU'$ is disjoint from $\cA_j$ for all $j\in J$.
		
		For each $k\in \{1,\dots,m\}$ pick a point $\xx_k$ in $V^*_k:= V\cap U^*_k$. The trivializing $\cY_{|V^*_k}  \simeq Y_{\xx_k}\times V^*_k$  provides us with homeomorphisms $h_\xx: Y_\xx\to Y_{\xx_k}$,  for all   $\xx \in V^*_k$. We can assume that the restrictions of $h_\xx$ to $Y_\xx\cap \cU$ and to $Y_\xx\cap\cU'$  are given by $(w,z(\xx),t(\xx)) \mapsto (w,z(\xx_k),t(\xx_k))$. Let $f_k: Y_{\xx_k} \to Y_0$ be a degenerating map, that is $f_{k}(a_j(\xx_k))=q_j$ for all $j\in J$, and the restriction of $f_k$ to the complement of $\bigcup_{j\in J}a_j(\xx_k)$, denoted by $Y^0_{\xx_k}$, is a homeomorphism from $Y^0_{\xx_k}$ onto $Y_0\setminus\{q_j, \; j\in J\}$. 
		We can assume that the restrictions of $f_k$ to $\cU\cap Y_{\xx_k}$ and to $\cU'\cap Y_{\xx_k}$ satisfy $f_k(w,z(\xx_k),t(\xx_k))=(w,0,0)$.
		We can also suppose that the $\Z/2$-action generated by $\tau_\cY$ is equivariant with respect to $h_\xx$ and $f_k$. 
		
		Let us  pick a simple path $c(y_0)$ from $y_0$ to $\tau_\cY(y_0)$ in $Y_0$.    Let $y_k \in \cU\cap Y_{\xx_k}$ be the point  of coordinate $(0,z(\xx_k), t(\xx_k))$, and $y'_k:=\tau_\cY(y_k)$. Note that we have $f_k(y_k)=y_0$ and $f_k(y'_k)=y'_0$. 
		Let  $c(y_k)$ is a path in $Y_{\xx_k}$ joining $y_k$ to $y'_k$ such that $f_k(c(y_k))$ is homotopic to $c(y_0)$ by a homotopy with fixed endpoints in $Y_0$. For all $y =(w,z,t)\in \cU\cap \cY_{|V^*_k}$, let $c(y)$ be the path from $y$ to $y':=\tau_\cY(y)$ in $Y_{\xx}$, where $\xx=(z,t)$, which is the concatenation of 
		\begin{itemize}
			\item[$\bullet$] a path in $\cU\cap Y_\xx \simeq \Delta(\rho)$ from $y=(w,z,t)$ to $(0,z,t)=h_\xx^{-1}(y_k)$,
			
			\item[$\bullet$] the path $h^{-1}_\xx(c(y_k))$ from $h^{-1}_\xx(y_k)$ to $h^{-1}_\xx(y'_k)$, 
			
			\item[$\bullet$] a path in $\cU'\cap Y_\xx \simeq \Delta(\rho)$ from $h^{-1}_\xx(y'_k)$ to $y'$.  
		\end{itemize}

    \item[(ii.b)] $y_0=q_j$ is a node of $Y_0$. We have $\tau_\cY(q_j)=q_{j'}$ for some $j'\in J, j' \neq j$.  
    In this  case, we choose $\cU$ to be $\cA_j$. 
    Let $i=i(j)=i(j')$ and $r=r(j)=r(j')$. We can assume that the restriction $\tau_{\cY|\cA_j}: \cA_j \to \cA_{j'}$ is given by  $(u,v,z,t) \mapsto (-u,-v,z,t)$, where $(u,v,z,t)$ is the coordinate system in the definition of $\cA_j$ and $\cA_{j'}$.
    For all $\xx\in U$, let $y_1(\xx)$ denote the point in $\cA_j$ of coordinate $(1,t_i^r(\xx),z(\xx),t(\xx))$, and $y'_1(\xx):=\tau_\cY(y_1(\xx)) \simeq (-1,-t^r_i(\xx), z(\xx), t(\xx)) \in \cA_{j'}$. 
    We can suppose that the maps $h_\xx: Y_\xx \to Y_{\xx_k}$ and $f_k: Y_{\xx_k} \to Y_0$ satisfy $h_{\xx}(y_1(\xx))=y_1(\xx_k), h_\xx(y'_1(\xx))=y'_1(\xx_k)$, and $f_k(y_1(\xx_k))=y_1(0), f_k(y'_1(\xx_k))=y'_1(0)$. 
    
    Consider a simple path $c(q_j)$ in $Y_0$ joining $q_j$ and $q_{j'}$. Without loss of generality we can assume that $c(q_j)\cap \cA_j$ (resp. $c(q_j)\cap\cA_{j'}$) is contained in the local branch $\{v=0\}$ of $Y_0$, and that $c(q_j)$ contains the segments $c_0(q_j):=[q_j,y_1(0)] \simeq [0,1]\times \{0\} \subset \cA_j$ and $c'_0(q_j):= [y'_1(0), q_{j'}]\subset \cA_{j'}$.
    Let $c_1(q_j)$ denote the path from $y_1(0)$ to $y'_1(0)$ that is contained in $c(q_j)$.

    Consider now a point $y=(u,v,z,t) \in \cA_j\cap \cY_{|U^*_k}$. We wish to specify a path $c(y)$ from $y$ to $y':=\tau_\cY(y)$ on $Y_\xx$, where $\xx=(z,t)$ in a coherent manner. To this purpose, let us pick a simple path $c_1(\xx_k)$ in $Y_{\xx_k}$ from $y_1(\xx_k)$ to $y'_1(\xx_k)$  such that $f_k(c_1(\xx_k))$ is homotopic to $c_1(q_j)$ in $Y_0$ (note that $f_k(c_1(\xx_k))$ and $c_1(q_j)$ have the same endpoints). For all $\xx \in U^*_k$, let $c_1(\xx):=h^{-1}_\xx(c_1(\xx_k))$. A convenient way to construct a path from $y$ to $\tau_\cY(y)$ is to concatenate $c_1(\xx)$, where $\xx=\varrho(y)\in U^*_k$, with a path from $y$ to $y_1(\xx)$ and a path from $y'_1(\xx)$ to $y'$. Unfortunately, since $Y_\xx\cap \cA_j$ is an annulus, there does not exist any distinguished path from $y$ to $y_1(\xx)$ up to homotopy.  
    To remedy this issue we consider  
     $\cA^{0*}_{j,k}:=\{(u,v,z,t) \in \cA_j\cap\cY_{|U^*_k}, \, \arg(u)\neq \pi/2\}$, and  $\cA^{1*}_{j,k}:=\{(u,v,z,t) \in \cA_j \cap \cY_{|U^*_k}, \, \arg(u)\neq -\pi/2\}$. 
    If $y \in \cA^{0*}_{j,k}$,  there is a unique path from $y$ to $y_1(\xx)=(1,t_i^r,z,t)$ which is contained in $\cA^{0*}_{j,k}\cap Y_{\xx}$ up to homotopy. We denote this path by $c_0(y)$ and its image by $\tau_\cY$ by $c'_0(y)$. The concatenation $c_0(y)*c_1(\xx)*c'_0(y)$ is denoted by $c(y)$. We have a similar construction for all $y \in \cA^{1*}_{j,k}$.     
    \end{itemize}
\end{itemize} 
We summarize the construction above in the following
\begin{Lemma}\label{lm:construct:path}
Let $y_0$ be a point in the central fiber $Y_0$. 
\begin{itemize}
	\item[$\bullet$] If $y_0$ is fixed by $\tau_\cY$, then there exists a neighborhood $\cU$ of  $y_0$ such that one can specify for all $y\in \cU$ a distinguished path $c(y)$ in $\cU\cap Y_{\varrho(y)}$ joining  $y$ to $\tau_\cY(y)$, where $c(y)$ is constant if $y$ is fixed by $\tau_\cY$. 
	
	\item[$\bullet$] If $y_0$ is not fixed by $\tau_\cY$, then there exists a neighborhood $\cU$ of $y_0$ such that $\cU^*:=\cU\cap\cY_{|U^*}$  can be covered by a finite family $\{\cU^*_k, \; k=1,\dots,\ell\}$   of open subsets such that for each $k\in \{1,\dots,\ell\}$, for all $y\in \cU^*_k$, one can specify a distinguished path $c(y)\subset Y_{\varrho(y)}$ from $y$ to $\tau_\cY(y)$. Note that the choice of the path $c(y)$ depends on $\cU_k$.
\end{itemize}	
\end{Lemma}

We now prove 
\begin{Proposition}\label{prop:int:s:arc:express} 
Suppose that there exists a holomorphic section $\Omega$ of the relative dualizing sheave  $\omega_\varrho$ on $\cY$ such that $\tau_\cY^*\Omega=-\Omega$. For all $\xx\in U^*$, denote by $\Omega_\xx$ the restriction of $\Omega$ to the smooth curve $Y_ \xx$. We assume that  for every $j\in J$, the restriction of $\Omega$ to $\cA_j$ is either $\lambda_j u^{m_j}du$, or $\lambda_j v^{m_j}dv$, where $\lambda_j \in \C$, and $m_j\in \Z_{\geq -1}$. 
Let $y_0$ be a point in the central fiber $Y_0$, and $\cU$ a neighborhood of $y_0$ as described in Lemma~\ref{lm:construct:path}.

\begin{itemize}
	\item[(a)] Assume that $y_0$ is fixed by $\tau_\cY$. Define 
	$$
	P(y):=\int_{c(y)}\Omega_{\varrho(y)}
	$$
	for all $y\in \cU\cap \cY_{|U^*}$. Then $P$ is the restriction to $\cU^*$ of a holomorphic function on $\cU$.
	
	\item[(b)] Assume that $y_0$ is not fixed by $\tau$. Let $\cU$, and $\cU^*_k, \, k=1\dots,\ell$, be as in Lemma~\ref{lm:construct:path}. Fix a $k \in \{1,\dots,\ell\}$. For all $y\in \cU^*_k$  define
	$$
	P_k(y):=\int_{c(y)} \Omega_{\xx}
	$$
	where $\xx:=\varrho(y)$.
	\begin{itemize} 
		\item[(b.1)] If $y_0$ is a smooth point of $Y_0$ then  we have
		\begin{equation}\label{eq:P:expr:p:smooth}
			P_k(y) =  \phi + \sum_{i=1}^n \mu_i\cdot \ln(t_i(\xx))
		\end{equation}
		where $\phi$ is the restriction to $\cU^*_k$ of a  holomorphic function on $\cU$, and the $\mu_i$'s are complex constants satisfying $\mu_i\neq 0$ only if there exists $j\in J$ such that $\Omega_0$ has a simple pole at $q_j$, $i=i(j)$, and $q_j$ is contained in the interior of $c(y_0)$.  		 
		\item[(b.2)] If $y_0$ is a node $q_{j}$ of $Y_0$ then up to a permutation  of the coordinates $(u,v)$ on  $\cA_{j}$, for all $y=(u,v,z,t) \in \cU^*_k$,  we have
		\begin{equation}\label{eq:P:expr:p:node:s:pole}
			P_k(y)=\phi+\mu_0\cdot \ln(u)+ \sum_{i=1}^n \mu_i\cdot \ln(t_i)
		\end{equation}
		where $\phi$ is the restriction to $\cU^*_k$ of a  holomorphic  function on $\cU$, $\mu_0 \in \C$ is non-zero only if $\Omega_0$ has simple pole at $q_{j_0}$,  and the numbers  $\{\mu_i, \; 1 \leq i \leq n\}$ satisfy the same properties as in \eqref{eq:P:expr:p:smooth}.
	\end{itemize}
\end{itemize}
\end{Proposition}
\begin{proof}
Suppose first that $y_0$ is fixed by $\tau_\cY$. If $y_0$ is a smooth point of $Y_0$ then we can choose the neighborhood $\cU$ of $y_0$ such that $(\cU,y_0) \simeq (\Delta\times V, 0)$, where $V$ is a an open neighborhood of $0$ in $U$. In this case
$\Omega_{|\cU}= \varphi(w,z,t) dw$, where $w$ is the coordinate on $\Delta$, and $\varphi$ is a holomorphic function. By construction, all the paths $c(y)$ are contained in $\cU$. Thus $P(.)$ is the restriction to $\cU^*$ of the function 
$$
(w,z,t) \mapsto \int_{w}^{-w}\varphi(s,z,t)ds
$$
which is a holomorphic function on $\cU$, and the conclusion follows.

If $y_0$ is a node $q_j$ of $Y_0$ which is fixed by $\tau_\cY$, then we have $\cU=\cA_j$ and $c(y)\subset\cA_j$ for all $y\in\cA_j$.
Without loss of generality, we can assume that $\Omega=\lambda_j u^{m_j}du$ in $\cA_j$.  Recall that the restriction of $\tau_\cY$ to $\cA_j$ is given by $(u,v,z,t)=(-u,-v,z,t)$. It follows from the assumption  $\tau^*_\cY\Omega=-\Omega$ that we have $m_j$ is an even number, which implies that $m_j\geq 0$ (since we must have $m_j \geq -1$). Since for all $y \in \cA_j\cap \cY_{|U^*}$ the path $c(y)$ is entirely contained in $\cA_j$, and the conclusion follows.

\medskip 

We now turn to the case $y_0$ is not fixed by $\tau_\cY$. Consider a point $y \in \cU^*_k$. Let $\xx:= \varrho(y)\in U^*$.
As $y$ varies in $\cU^*_k$, for all $j\in J$, one can specify  
 a simple arc $\delta_j(\xx)$ in $\cA_j(\xx):=\cA_j\cap Y_\xx$ joining $(1,t_i^{r},z,t)$ to $(t_{i}^{r},1,z,t)$, where $i=i(j), r=r(j)$.
 Without loss of generality, we can assume that the restriction of $\Omega$ to $\cA_j$ is given by $\lambda_j u^{m_j}du$. 
 We then have
\begin{equation}\label{eq:int:along:cross:arc}
 \int_{\delta_j(\xx)}\Omega_\xx =\left\{
 \begin{array}{cl}
 	\lambda_j r\cdot\ln(t_{i}) & \hbox{ if $m_j=-1$}\\
 	\frac{\lambda_j}{m_j+1}\cdot (t_i^{r(m_j+1)}-1) & \hbox{ if $m_j\geq 0$.}
 \end{array}
 \right.
 \end{equation}
Note that $m_j=-1$ if and only if $\Omega_0$ has simple poles at $q_j$. 

\medskip

Assume that $y_0$ is a smooth point in $Y_0$. We can suppose that $y_0$ and $y'_0:=\tau_\cY(y_0)$ are not contained in any $\cA_j, \; j \in J$.  Let $J_c:=\{j \in J, \; q_j\in  c(y_0)\} \subset J$. 
For all $y\in \cU^*_k$, up to  homotopy (with fixed endpoints), we can assume that for all $j\in J_c$, the path $c(y)$ contains the arc 
$\delta_j(\xx)$. 
Let $\hat{c}_0(y)$ denote the complement of $\bigcup_{j\in J_c}\delta_j(\xx)$ in $c(y)$. Then $\hat{c}_0(y)$ is a finite union of simple arcs in $Y_\xx$ whose image by the degenerating map $f_\xx:=f_k\circ h_\xx: Y_\xx \to Y_0$ is  contained in the smooth part of $Y_0$.   Therefore
$$
\phi_0(y):=\int_{\hat{c}_0(y)}\Omega_\xx 
$$
is the restriction of a holomorphic function on $\cU$ to $\cU^*_k$.
Let $J^*_c$ denote the set of $j\in J_c$ such that $\Omega_0$ has simple poles at the node $q_j$. 
As a consequence of \eqref{eq:int:along:cross:arc} we get
$$
P_k(y)=\phi_0(y)+\sum_{j\in J_c}\int_{\delta_j(\xx)}\Omega_\xx = \sum_{j\in J^*_c}\lambda_j\cdot r(j)\cdot \ln(t_{i(j)})+\phi(y)
$$
where $\phi$ is a holomorphic function on $\cU$. We get the desired conclusion by setting  
$$
\mu_i:= \sum_{j\in J^*_c, \, i(j)=i} \lambda_{j}\cdot r(j).
$$	

\medskip 

Finally, let us assume that $p$ is a node $q_{j_0}$ of $Y_0$ not fixed by $\tau_\cY$. In this case we can take $\cU=\cA_{j_0}$. Without loss of generality, we can assume that the arc $c(y_0)\cap \cA_{j_0}$ is contained in the local branch $\{v=0\}$ of $Y_0$. Recall that for all $y=(u,v,z,t)\in \cU^*_k$, $c(y)$ is the concatenation $c_0(y)*c_1(y)*c'_0(y)$, where
\begin{itemize}
	\item[$\bullet$] $c_0(y)$ is a path in $Y_{(z,t)}\cap \cA_{j_0}$ from $y$ to $y_1(z,t):=(1,uv,z,t)$,
	\item[$\bullet$] $c_1(y)$ is a path in $Y_{(z,t)}$ from $y_1(z,t)$ to $y'_1(z,t):=\tau_\cY(y_1(z,t))$,
	\item[$\bullet$] $c'_0(y)= -\tau_\cY(c_0(y))$.	
\end{itemize}  
Using the fact that $\tau_\cY^*\Omega=-\Omega$, we get 
$$
\int_{c_0(y)}\Omega_{(z,t)} +\int_{c'_0(y)}\Omega_{(z,t)}=2\int_{c_0(y)}\Omega_{(z,t)}. 
$$ 
If $m_{j_0}=-1$ then we have
$$
\int_{c_0(y)}\Omega_{(z,t)}=-\lambda_{j_0}\ln(u).
$$
If $m_{j_0} \geq 0$ then
$$
\int_{c_0(y)}\Omega_{(z,t)} =\left\{
\begin{array}{cl}
		\frac{\lambda_{j_0}}{m_{j_0}+1}\cdot(1-u^{m_{j_0}+1}) & \hbox{ if $\Omega=\lambda_{j_0}u^{m_{j_0}}du$}\\
		\frac{\lambda_{j_0}}{m_{j_0}+1}\cdot v^{m_{j_0}+1}\cdot(u^{m_{j_0}+1}-1) & \hbox{ if $\Omega=\lambda_{j_0}v^{m_{j_0}}dv$}\\
\end{array} 
\right. 
$$
Since $m_{j_0} = -1$ if and only if $\Omega_0$ has simple poles at $q_{j_0}$, the same argument of the previous case allows us to conclude. 
\end{proof}

\begin{Proposition}\label{prop:Area:loc:express}
Let $\cY, U, U^*, \Omega$ as in Proposition~\ref{prop:int:s:arc:express}. Let $J^*$ denote the set of $j\in J$ such that $\Omega_0$ has simple poles at the node $q_j$ of $Y_0$. Then for all $\xx\in U^*$, we have
\begin{equation}\label{eq:Area:loc:express}
||\Omega_\xx||^2=\Ab(Y_\xx,\Omega_\xx)= -\sum_{i=1}^n a_i\ln|t_i| +\psi,
\end{equation}
where the $a_i$'s are real constants in $\R_{\geq 0}$ satisfying $a_i >0$ in and only if there exists $j\in J^*$ such that $i(j)=i$, and $\psi$ is a smooth positive function on $U$. 	
\end{Proposition}
\begin{proof}
We first observe that $\cY^{0}:=\cY\setminus\left(\bigcup_{j\in J}\cA_j\right)$ is a fibration over $U$ with fiber being a surface with boundary diffeomorphic to the complement in $Y_0$ of a neighborhood of its nodes. For all $\xx in U$, let $Y^{0}_\xx$ denote the fiber of $\cY^{(0)}$ over $\xx$. Define
$$
\psi(\xx):=\Aa(Y^{0}_\xx,\Omega_\xx)=\frac{\imath}{2}\int_{Y^{0}_\xx}\Omega_\xx\wedge\ol{\Omega}_\xx.
$$
Then $\psi$ is a smooth positive function on $U$.
For each $j\in J$, let $A_j(\xx)$ denote the annulus $\cA_j\cap Y_\xx$. We have  	
$$
\frac{\imath}{2}\int_{A_j(\xx)}\Omega_\xx\wedge\ol{\Omega}_\xx = \frac{\imath}{2}|\lambda_j|^2\int_{|t_i|^{r}< |u| < 1}|u|^{2m_j}du d\bar{u} =\left\{
\begin{array}{cl}
	\frac{2\pi|\lambda_j|^2}{2(m_j+1)}(1-|t_i|^{2r(m_j+1)}) & \hbox{ if $m_j\geq 0$} \\
	-2\pi|\lambda_j|^2r\ln|t_i| & \hbox{ if $m_j=-1$}
\end{array}
\right.
$$
where $i=i(j)$ and $r=r(j)$.	Since 
$$
\Aa(Y_\xx,\Omega_\xx)=\Aa(Y^{0}_\xx,\Omega_\xx)+\sum_{j\in J}\Aa(A_j(\xx),\Omega_\xx)
$$
we get the desired conclusion.
\end{proof}

As a consequence we obtain

\begin{Corollary}\label{cor:Theta:extend:str:gp:I}
The $(2,2)$-form $\Theta$ extends smoothly across strata of group I in $\partial\tCD$. 	
\end{Corollary}
\begin{proof}
Consider a point $\hpp$ in a stratum of group I in $\partial \tCD$. Let $\pp$ be the projection of $\hpp$ in $\hXD$. By definition, $\pp$ is contained in one of the strata $\cS_{1,0}, \cS_{2,0}^a$, $\cS_{0,2}$. Let $\omega_\pp$ be an Abelian differential on $\tilde{C}_\pp:=\tilde{\pi}^{-1}(\{\xx\})$ which generates the line $\OO(-1)_\pp$. Note that $\omega_\pp$ is holomorphic at all the nodes of $\tilde{C}_\pp$.

We know that $\pp$ is a smooth point of $\ol{\cX}_D$, hence a smooth point of $\hXD$ (see Proposition~\ref{prop:bdry:str:gp:I} and Proposition~\ref{prop:norm:orbifold}).
In \textsection~\ref{subsec:geom:bdry:str:gp:I}, we showed that a neighborhood of $\pp$ (in $\hXD$) is isomorphic to an open subset $U \subset \C^2$ with coordinates $(z,t)$, where $t$ is the smoothing parameter of the nodes of $\tilde{C}_\pp$. From our construction, we obtain actually the universal curve $\tilde{\cC}_{D|U}$ over $U$ and for each $\xx$ in $U$ an Abelian differential $\omega_\xx$ generating the line $\OO(-1)_\xx$. The differential $\omega_\xx$ is in fact the restriction to $\tilde{C}_\xx$ of a section  $\Omega$ of the  relative dualzing sheaf $\omega_{\tilde{\pi}}$ over $\tilde{\cC}_{D|U}$.

One readily checks that the family $\tilde{\pi}: \tilde{\cC}_{D|U} \to U$ and the section $\Omega$ satisfy all the conditions of Proposition~\ref{prop:int:s:arc:express}. It follows from Proposition~\ref{prop:Area:loc:express} that the function $\Ab(\xx):= \Aa(\tilde{C}_\xx, \Omega_\xx)$ defined on $U^*:=\{(z,t)\in U, \, t \neq 0\}$ extends smoothly to $U$. 

By Proposition~\ref{prop:int:s:arc:express}, there is a holomorphic function $P$ defined on neighborhood $\cU$ of $\hpp$ such that  the function $\varphi_c(.)$ in \eqref{eq:normalized:per:funct} satisfies
$$
\varphi_c(\hxx)=\frac{|P(\hxx)|^2}{\Ab(\xx)}
$$
for all $\hxx\in \cU^*:=\cU\cap \tilde{\cC}_{D|U^*}$ and $\xx:=\tilde{\pi}(\hxx)$.  Since $\Theta= (\imath\vartheta)\wedge\left(\frac{\imath}{2}\partial\bar{\partial}\varphi_c\right) = (-\imath\partial\bar{\partial}\ln\Ab)\wedge \left(\frac{\imath}{2}\partial\bar{\partial}\varphi_c\right)$ the corollary follows.
\end{proof}

\subsection{Proof that $\Theta$ is a closed current on $\tilde{\cC}_D$}\label{subsec:prf:Theta:closed:current}
We now proceed to the proof of Theorem~\ref{th:Theta:current}. 

In what follows $\hpp$ will be a point in $\partial\tCD$ whose projection in $\partial \hXD$ is denoted by $\pp$. 
The fiber $\tilde{\pi}^{-1}(\{\pp\})$ is denoted by $\tilde{C}_\pp$, and the Prym involution on $\tilde{C}_\pp$ is denoted by $\tau_\pp$. Let $\omega_\pp$ be an Abelian differential on $\tilde{C}_\pp$ generating the line $\OO(-1)_\pp$.
We will denote by $C_\pp$ the fiber $\hat{\pi}^{-1}(\{\pp\})$ which is isomorphic to the curve underlying $\nu(\pp) \in \ol{\cX}_D$. Note that $\tilde{C}_\pp$ and $C_\pp$ are isomorphic unless $\pp$ is contained in a stratum of group IV.  

By Corollary~\ref{cor:Theta:extend:str:gp:I} we already know that $\Theta$ extends smoothly across the strata of group I in $\partial \tCD$. Therefore, we will only focus on the case $\pp$ is contained in a stratum of group II, III, or IV.
For all of those cases, in \textsection~\ref{subsec:geom:bdry:str:gp:II}, \ref{subsec:bdry:str:gp:III}, \ref{subsec:bdry:str:gp:IV}, we constructed a holomorphic embedding $\Phi: \Bb \to \Pb\Omega'\ol{\cB}_{4,1}(2,2)$, where $\Bb$ is an open neighborhood of $0$ in $\C^3$ with the following properties
\begin{itemize}
	\item[$\bullet$] $\Phi(0) = \nu(\pp)$.
	
		\item[$\bullet$] $\Phi(\Bb)$ contains a neighborhood of $\nu(\pp)$ in $\ol{\cX}_D$, that is the germ of $\ol{\cX}_D$ at $\nu(\pp)$ is isomorphic to the germ of an analytic subset of $\Bb$ at $0$. 
	
	\item[$\bullet$] A neighborhood of $\pp$ in $\hat{\cX}_D$ is the normalization of an irreducible analytic subset of $\Bb$. 
\end{itemize}
Let  $\pi: \ol{\cC}_{|\Bb} \to \Bb$ be the family of curves  which is the pullback of the universal curve over $\Pb\Omega'\ol{\cB}_{4,1}$ by $\Phi$. There is by construction a section of the tautological line bundle $\Phi^*\OO(-1)$ on $\Bb$. This section corresponds to a section $\Omega$ of the relative dualizing sheaf $\omega_{\pi}$ on $\ol{\cC}_{|\Bb}$. One readily checks that $\ol{\cC}_{|\Bb}$ and $\Omega$ satisfy all the conditions of Lemma~\ref{lm:construct:path} and Proposition~\ref{prop:int:s:arc:express}.

\subsubsection{Case $\pp$ contained in a boundary stratum of group II}\label{subsec:prf:Theta:current:str:gp:II}
\begin{proof}
In this case $\Bb$ is endowed with a system of coordinates $(x,t_1,t_2)$ where $t_1$ and $t_2$ are the smoothing parameters of the  nodes of $C_\pp \simeq \tilde{C}_\pp$. Note also that $\omega_\pp$ has simple poles at all the nodes of $C_\pp$.
By Proposition~\ref{prop:bdry:str:gp:II}, any irreducible component of the germ of $\ol{\cX}_D$ at $\nu(\pp)$ is isomorphic to the germ of $\cA:=\{(x,t_1,t_2)\in \C^3, \, t_1^{m_1}=t_2^{m_2}\}$ at $0$, where $m_1,m_2\in \Z_{>0}$ and $\gcd(m_1,m_2)=1$. Therefore, a neighborhood of $\pp$ in $\hXD$ can be identified with a neighborhood $U$ of $0\in \C^2$, and the restriction of the normalizing map $\nu: \hXD \to \ol{\cX}_D$ to $U$ is given by $\nu: (z,t) \to \Phi(z,t^{m_2},t^{m_1})$, where $(z,t)$ are the coordinates on $U$.
Define $U^*:=\{(z,t) \in U, \, t\neq 0\}$.

Let $\tilde{\pi}: \tilde{\cC}_{D|U} \to U$ denote the family of curves which is the pullback of the universal curve on $\Pb\Omega'\ol{\cB}_{4,1}$ by $\Phi\circ\nu$. The pullback the section of $\Phi^*\OO(-1)$ on $\Bb$ corresponds to a section of the relative dualizing sheaf $\omega_{\tilde{\pi}}$ that we will denote again by $\Omega$.  
One readily checks that  $\tilde{\pi}, U, U^*,\Omega$ satisfy all the conditions of Proposition~\ref{prop:int:s:arc:express}. Thus it follows from Proposition~\ref{prop:Area:loc:express} that up to a multiplicative constant we have
$$
\Ab(\xx):=||\Omega_\xx||^2= -2\ln(|t|)+\phi,
$$
for all $\xx=(z,t)\in U^*$, where $\phi$ is a smooth positive function  on $U$. It follows from Lemma~\ref{lm:curv:form:loc:express} that
$$
\vartheta=-\partial\bar{\partial}\ln(\Ab)=\frac{\partial \Ab\wedge \bar{\partial}\Ab}{\Ab^2} =\frac{(dt/t-\partial\phi)\wedge(d\bar{t}/\bar{t}-\bar{\partial}\phi)}{(-2\ln|t|+\phi)^2}.
$$

We now  have two cases:

\begin{itemize}
\item[(i)] $\hpp$ is a smooth point in $\tilde{C}_\pp$. 
Since $\pp$ is either contained in $\cS_{2,0}^b$ or in $\cS_{1,1}$,  each component of $\tilde{C}_\pp$ is invariant by the Prym involution. Therefore, there exists a path $c$ in $\tilde{C}_\pp$ joining $\hpp$ to $\tau(\hpp)$ which does not cross any node of $\tilde{C}_\pp$. For all $\hxx$ in a neighborhood of $\hpp$, $c$ gives rise to a distinguished homotopy class $c(\hxx)$ of path from $\hxx$ to $\tau(\hxx)$ on $\tilde{C}_\xx$, where $\xx=\tilde{\pi}(\hxx)$. This implies that $P: \hxx \mapsto \int_{c(\hxx)}\Omega_\xx$ is a holomorphic function on a neighborhood of $\hpp$. From \eqref{eq:Theta:loc:express}, we get that
\begin{align*}
\Theta & =-\frac{1}{2}\cdot\left(\frac{\partial \Ab\wedge \bar{\partial}\Ab}{\Ab^2}\right)\wedge \left( \frac{dP\wedge d\bar{P}}{\Ab}\right) \\ &=R_2\cdot\frac{(dt/t-\partial\phi)\wedge(d\bar{t}/\bar{t}-\bar{\partial}\phi)\wedge  dP\wedge d\bar{P}}{(-2\ln|t|+\phi)^3}
\end{align*}
where $R_2$ is a constant.
Since the functions $\frac{1}{|t|^2(-2\ln|t|+\phi)^3}$, and  $\frac{1}{|t|(-2\ln|t|+\phi)^3}$ are integrable over a neighborhood of $0$ in $\C^3$, (A) follows.

A neighborhood  $\cU$ of $\hpp$ can be identified with  $\Delta^3\subset\C^3$.  Let $(x,z,t)$ be a coordinate system on $\Delta^3$ such that the projection $\tilde{\pi}: \cU \to \hXD$ is given by $\tilde{\pi}(x,z,t) =(z,t)$. 
In these coordinates, $\partial\tilde{\cC}_D$ is defined by $\{t=0\}$. The boundary of the $\eps$-neighborhood of $\partial\tilde{\cC}_D\cap\cU$ corresponds to the set $\Delta^2\times\{|t|=\eps\}$. For all $1$-form $\eta$ with compact support in $\cU$, we have
$$
\left| \int_{\Delta\times\{|t|=\eps\}\times\Delta}\Theta\wedge\eta\right| \leq \frac{K}{-(\ln|\eps|)^3}\cdot\int_{\{|t|=\eps\}}\frac{|dt|}{|t|}=\frac{2\pi K}{-(\ln|\eps|)^3}
$$
where $K$ is a constant, from which (B) follows.\\

\item[(ii)] Case $\hpp$ is a node of $\tilde{C}_\pp$. A neighborhood $\cU$ of $\hpp$ is isomorphic to a quotient $\Delta^3/(\Z/m)$, were the action of $\Z/m$ on $\C^3$ is given by $k\cdot(z,u,v) \mapsto (z,e^{2\pi\imath k/m}u, e^{-2\pi\imath k/m}v)$. 
In this local chart, the projection $\tilde{\pi}$  reads $\tilde{\pi}(z,u,v)=(z,uv)$.
Thus the pullback of $\vartheta$ to $\cU$ is given by 
$$
\vartheta=  \frac{(du/u+dv/v-\partial\phi)\wedge(d\bar{u}/\bar{u}+d\bar{v}/\bar{v}-\bar{\partial}\phi)}{(-2\ln|u|-2\ln|v|+\phi)^2}.
$$
Since $\omega_\pp$ has simple pole at all the nodes of $\tilde{C}_\pp$, $\hpp$ is exchanged by $\tau$ with another node. Note that $\hpp$ and $\tau(\hpp)$ are contained in the same component of $\tilde{C}_\pp$. In particular, there is a path $c$ in $\tilde{C}_\pp$ joining $\hpp$ and $\tau(\hpp)$ which does not contain any node in the interior. 
By Proposition~\ref{prop:int:s:arc:express}, $\cU^*$ can be covered by a finite family of open subsets $\{\cU^*_k, \, k=1,\dots,\ell\}$ such that for each $k\in \{1,\dots,\ell\}$, and for all $\hxx\in \cU^*_k$, one can construct a distinguished path $c(\hxx)$ from $\hxx$ to $\tau(\hxx)$ in $\tilde{C}_\xx$. The integral of $\Omega_\xx$ along $c(\hxx)$ provides us with a function $P_k(.)$ on $\cU^*_k$ which satisfies 
$$
P_k(z,u,v)=\mu\ln(u)+Q 
$$
where $\mu$ is a constant and $Q$ is the restriction to $\cU^*_k$ of a holomorphic function on $\cU$. 
Note that the constant $\mu$ is determined by the residue of $\Omega_0$ at the node $\hpp$. Therefore, the $1$-forms $dP_k$'s  (resp. $d\bar{P}_k$'s) give rise to a well defined $1$-form on $\cU^*$ that we will denote by $dP$ (resp. $d\bar{P}$), and we have 
$$
dP=\mu\cdot \frac{du}{u} + dQ, \quad d\bar{P}=\bar{\mu}\cdot \frac{d\bar{u}}{\bar{u}} + d\bar{Q}.
$$ 
It follows
$$
\Theta=R_3\cdot\frac{(du/u+dv/v-\partial\phi)\wedge(d\bar{u}/\bar{u}+d\bar{v}/\bar{v}-\bar{\partial}\phi)\wedge(\mu du/u+ dQ)\wedge(\bar{\mu}d\bar{u}/\bar{u}+d\bar{Q})}{(-2\ln|u|-2\ln|v|+\phi)^3}
$$
where $R_3$ is a real constant. We now remark that
\begin{align*}
\left|\int_{\Delta^3}\frac{du d\bar{u}dv d\bar{v}dz d\bar{z}}{|u|^2|v|^2(-\ln|u|-\ln|v|+\phi)^3}\right| & \leq K\cdot\int_0^1\int_0^1\frac{dr ds}{rs(-\ln(r)-\ln(s)+K')^3}\\
& \leq  \frac{K}{2}\cdot \int_0^1\frac{dr}{r(-\ln(r) +K')^2} = \frac{K}{K'}
\end{align*}
where $K$ and $K'$ are some positive real constants, from which 
(A) follows.

We have $\cU\cap\partial\tilde{\cC}_D=\{uv=0\}$. Hence the boundary of the $\eps$-neighborhood of $\partial\tilde{\cC}_D\cap \cU$ is the union of  $\Delta\times\{|u|=\eps\}\times\{\eps \leq |v| < 1\}$ and $\Delta\times\{\eps \leq |u|<1\}\times\{|v|=\eps\}$. For any $C^\infty$  $1$-form $\eta$ with compact support in $\cU$, we have
\begin{align*}
\left|\int_{|z|<1}\int_{|u|=\eps}\int_{\eps \leq |v| < 1} \Theta\wedge\eta\right| &\leq  K\cdot\int_{|u|=\eps} \left(\int_{\eps \leq |v| < 1} \frac{dv d\bar{v}}{|v|^2(-\ln(\eps)-\ln(|v|))^3}\right) \frac{|du|}{|u|}  \leq \frac{K'}{\ln(\eps)^2}\\
\end{align*}
which implies that
$$
\lim_{\eps \to 0}\int_{\Delta}\int_{|u|=\eps}\int_{\eps \leq |v| < 1}\Theta\wedge\eta=0.
$$
A similar computation shows
$$
\lim_{\eps \to 0}\int_{\Delta}\int_{\eps \leq |u| < 1} \int_{|v|=\eps}\Theta\wedge\eta=0,
$$
and (B) follows. This completes the proof of Theorem~\ref{th:Theta:current} in the case $\pp$ is contained in a stratum of group II in $\partial\hXD$.
\end{itemize}
\end{proof}
	
\subsubsection{Proof of Theorem~\ref{th:Theta:current}, case $\pp$ is contained in a  stratum of group III}
\begin{proof}
Recall that  group III consists of the following strata $\cS^a_{2,1}, \cS_{3,1}, \cS_{2,2}$, and $\cS_{1,3}$, which have dimension $0$ by Proposition~\ref{prop:bdry:str:gp:III}.  A neighborhood $U$ of $\pp$ in $\hat{\cX}_D$ is the normalization of the germ at $0$ of the analytic set $\cA=\{(t_0,t_1,t_2)\in \C^3, \; t_1^{m_1}=t_2^{m_2}\}$, with $m_1,m_2 \in \Z_{>0}$ satisfying $\gcd(m_1,m_2)=1$. Note that $t_0$ is the smoothing parameter of the nodes on $\tilde{C}_\pp$ at which $\omega_\pp$ is holomorphic, and $t_1,t_2$ are the smoothing parameters of the nodes at which $\omega_\pp$ has simple poles.
It is well known that $U$ is isomorphic to an open neighborhood of $0\in \C^2$, and the normalization map $\nu: U \to \cA$ is given by $\nu: (t_0,t) \mapsto (t_0,t^{m_2},t^{m_1})$. 
Let $U^*:=\{(t_0,t) \in U, \; t_0t \neq 0\}$.
By Proposition~\ref{prop:Area:loc:express}, we get that
$$
\vartheta= \frac{(dt/t-\partial\phi)\wedge(d\bar{t}/\bar{t}-\bar{\partial}\phi)}{(-2\ln|t|+\phi)^2}
$$
up to a constant, where $\phi$ is a real positive $C^\infty$ function on $U$. 

\begin{itemize}
	\item[(a)] Case $\hpp$ is fixed by $\tau$. By Proposition~\ref{prop:int:s:arc:express}, there is a neighborhood $\cU$ of $\hpp$ such that for all $\hxx\in \cU^*:=\cU\cap \tilde{\cC}_{D|U^*}$ one can specify a path $c(\hxx)$ from $\hxx$ to $\tau(\hxx)$ which is contained in $\cU$. It follows that the function $P: \hxx \mapsto \int_{c(\hxx)}\Omega_\xx$ is the restriction to $\cU^*$ of a holomorphic function on $\cU$.
	By Proposition~\ref{prop:Theta:def:22:form} (cf. \eqref{eq:Theta:loc:express}), we have
	$$
	\Theta=R\cdot \frac{(dt/t-\partial\phi)\wedge(d\bar{t}/\bar{t}-\bar{\partial}\phi)\wedge dP\wedge d\bar{P}}{(-2\ln|t|+\phi)^3}
	$$
	where $R$ is some real constant. We have two subcases:
	\begin{itemize}
		\item[(a.1)] $\hpp$ is a smooth point of $\tilde{C}_\pp$. In this case, we can suppose $\cU \simeq \Delta^3$ with coordinates  $(x,t_0,t)$.  Since the function $\frac{1}{|t|^2(-2\ln|t|+\phi)^3}$ is integrable in $U$, (A) follows.
		
		We have $\cU\cap\partial\tilde{\cC}_D=\{t_0t=0\}$. Therefore the boundary of the $\eps$-neighborhood of $\partial\tilde{\cC}_D\cap\cU$ consists of $\cV_1(\eps):=\{(x,t_0,t)\in \Delta^3, \; |t_0|=\eps, \eps \leq |t| < 1\}$, and $\cV_2(\eps):=\{(x,t_0,t) \in \Delta^3, \; \eps \leq |t_0| < 1, |t|=\eps\}$. For all $C^\infty$ $1$-form $\eta$ with support in $\cU$, we have
		$$
		\left|\int_{\cV_1(\eps)}\Theta\wedge\eta\right| \leq K\cdot\eps\int_\eps^1\frac{dr}{r(-\ln(r)+K')^3} =O(\eps),
		$$
		while
		$$
		\left|\int_{\cV_2(\eps)}\Theta\wedge\eta\right| \leq \frac{K}{(-\ln(\eps)+K')^3}\cdot\int_{|t|=\eps}\frac{|dt|}{|t|} = O\left(\frac{-1}{\ln(\eps)^3}\right)
		$$
		(here $K$ and $K'$ are some real positive constants).
		Thus we have
		$$
		\lim_{\eps \to 0}\int_{\cV_1(\eps)}\Theta\wedge\eta= \lim_{\eps \to 0}\int_{\cV_2(\eps)}\Theta\wedge\eta=0,
		$$
		and (B) follows.
		
		\item[(a.2)] $\hpp$ is a node $q_j$ of $\tilde{C}_\pp$ fixed by $\tau$. An orbifold neighborhood of $\hpp$ is isomorphic to $\Delta^3$ with coordinates $(u,v,t)$ and the projection  $\tilde{\pi}$ given by $\tilde{\pi}(u,v,t) =(uv,t)$. It follows immediately that (A) is satisfied.  The boundary $\partial\tilde{\cC}_D$ is defined by $uvt=0$ in this case. Thus the boundary of the $\eps$-neighborhood of $\partial\tilde{\cC}_D\cap \cU$ consists of $\cV_1=\partial\Delta(\eps)\times A(\eps,1)\times A(\eps,1), 
		 \cV_2= A(\eps,1)\times \partial\Delta(\eps)\times A(\eps,1),
		 \cV_3= A(\eps,1)\times A(\eps,1) \times \partial\Delta(\eps)$.
		 One readily checks that (B) is also satisfied in this case.
	\end{itemize}
	
\item[(b)] $\hpp$ is not fixed by $\tau$. By Proposition~\ref{prop:int:s:arc:express}, there is a  neighborhood $\cU$ of $\hpp$ such that $\cU^*$ can be covered by a finite family $\{\cU^*_k, \; k=1,\dots,\ell\}$ of open subset such that for all $k\in \{1,\dots,\ell\}$, for all $\hxx\in \cU^*_k$, one can specify a distinguished path $c(\hxx)$ from $\hxx$ to $\tau(\hxx)$ in $\tilde{C}_\xx$.  Let $P_k(\hxx):=\int_{c(\hxx)}\Omega_\xx$. Then $dP_k$'s coincide on the overlaps  of different $\cU^*_k$'s. Thus we have  well defined $1$-forms $dP$ and $d\bar{P}$ on $\cU^*$.
We have two subcases
\begin{itemize}
	\item[(b.1)] $\hpp$ is a smooth point of $\tilde{C}_\pp$ or a node at which $\omega_\pp$ is holomorphic. It follows from Proposition~\ref{prop:int:s:arc:express} (b) that either $dP$ and $d\bar{P}$ are restrictions to $\cU^*$ of smooth $1$-forms on $\cU$, or $dP=\alpha\left(\frac{dt}{t} + dQ\right), d\bar{P}=\bar{\alpha}\left( \frac{d\bar{t}}{\bar{t}}+ d\bar{Q}\right)$, where $\alpha\in \C$ and  $Q$ is a holomorphic function on $\cU$. In both cases, the same calculations as in the previous case allow us to conclude.  
	
	\item[(b.2)] $\hpp$ is a node of $\tilde{C}_\pp$ at which $\omega_\pp$ has simple poles. 
	In an orbifold local chart of $\tilde{\cC}_D$, a neighborhood    of $\hpp$  can be identified with $\Delta^3$ with coordinates $(t_0,u,v)$ and the  projection $\tilde{\pi}$  is  given by $\tilde{\pi}: (t_0,u,v) \mapsto (t_0,uv)$. In these coordinates
	$$
	\vartheta=\frac{(du/u+dv/v-\partial \phi) \wedge(d\bar{u}/\bar{u}+d\bar{v}/\bar{v}-\bar{\partial}\phi)}{(-2\ln|u|-2\ln|v|+\phi)^2}
	$$
	It follows from Proposition~\ref{prop:int:s:arc:express} (b.2) that  $dP=\alpha\frac{du}{u}+ \beta\frac{dv}{v}+ dQ$, and $d \bar{P}=\bar{\alpha}\frac{d\bar{u}}{\bar{u}}+\bar{\beta}\frac{d\bar{v}}{\bar{v}}+ d\bar{Q}$, where $\alpha$ and $\beta$ are complex constants and $Q$ is a holomorphic function on $\cU$. Thus  we have
	$$
	\Theta = \frac{\left(\frac{du}{u}+\frac{dv}{v}-\partial\phi\right)\wedge\left(\frac{d\bar{u}}{\bar{u}}+ \frac{d\bar{v}}{\bar{v}} -\bar{\partial}\phi\right)\wedge \left(\alpha \frac{du}{u}+ \beta \frac{dv}{v}+ dQ\right)\wedge \left(\bar{\alpha}\frac{d\bar{u}}{\bar{u}} + \bar{\beta} \frac{d\bar{v}}{\bar{v}} +d\bar{Q}\right)}{(-2\ln|u|-2\ln|v|+\phi)^3}.
	$$
	Since for all $K\in \R_{>0}$, we have
	$$
	\left|\int_{\Delta^3}\frac{du d\bar{u} dv d\bar{v} dt_0 d\bar{t}_0}{|u|^2|v|^2(-\ln|u|-\ln|v|+K)^3} \right| \leq K'\cdot \int_{0}^1\int_0^1\frac{dr ds}{rs(-\ln(r)-\ln(s)+K)^3} = \frac{K'}{2K}
	$$
	for some  $K' \in \R_{>0}$, condition (A) is verified.
	For condition (B), notice that the boundary of the $\eps$-neighborhood of $\partial\tilde{\cC}_D\cap \cU$ consists of $\cV_1(\eps)=\{|t_0|=\eps, \; \eps \leq |u|, \; \eps \leq |v|\}, \cV_2(\eps) = \{\eps \leq |t_0|, \; |u|=\eps, \; \eps \leq |v|\}$, and $\cV_3(\eps)=\{\eps \leq |t_0|, \; \eps \leq |u|, \; |v|=\eps\}$. For all $C^\infty$ $1$-form $\eta$ with compact support in $\cU$, we have
	$$
	\left|\int_{\cV_1(\eps)}\Theta\wedge\eta\right| \leq K_1\cdot\eps\cdot\int_\eps^1\int_{\eps}^1\frac{dr ds}{rs(-\ln(r)-\ln(s)+K)^3} =  O(\eps),
	$$
	$$
	\left|\int_{\cV_2(\eps)}\Theta\wedge\eta\right| \leq K_2\cdot \int_{\eps}^1\frac{ds}{s(-\ln\eps-\ln(s)+K)^3} =O\left(\frac{1}{\ln^2(\eps)}\right),
	$$
	$$
	\left|\int_{\cV_3(\eps)}\Theta\wedge\eta\right| \leq K_3\cdot \int_{\eps}^1\frac{dr}{r(-\ln(r)-\ln(\eps)+K)^3} = O\left(\frac{1}{\ln^2(\eps)}\right).
	$$
	Therefore, condition (B) is also verified. This completes the proof of Theorem~\ref{th:Theta:current} in the case $\pp$ is contained in a stratum of $\partial\hXD$ in group III.	
\end{itemize} 	
\end{itemize}

\end{proof}

\subsubsection*{Proof of Theorem~\ref{th:Theta:current}, case $\pp$ is contained in a stratum of group IV}
\begin{proof}
In this case the holomorphic embedding $\Phi: \Bb \to \Pb\Omega'\ol{\cB}_{4,1}(2,2)$ constructed in \textsection \ref{subsec:bdry:str:gp:IV} satisfies the following 
\begin{itemize} 
	\item[$\bullet$] There is a system of coordinates  $(t_0,t_1,t_2)$ on $\Bb$ such that each $t_i$ is the smoothing parameter of a pair of nodes in $C_\pp$.
	
	\item[$\bullet$] Via $\Phi$  any irreducible component of the germ  $(\ol{\cX}_D,\nu(\pp))$ is isomorphic to the germ at $0\in \C^3$ of an analytic set $\cA= \{(t_0,t_1,t_2) \in \C^3, \; t_0^{m_0}=t_1^{m_1}t_2^{m_2}\}$, where $m_0,m_1,m_2 \in \Z_{>0}$ satisfy $\gcd(m_0,m_1,m_2)=1$.
\end{itemize}

A neighborhood of $\pp$ in $\hXD$ is the normalization $\hat{\cA}$ of $\cA$.  It is a well known fact that $\hat{\cA}$ is isomorphic to a quotient  $U/(\Z/m)$, where $U$ is an open neighborhood of $0\in \C^2$, $m=\frac{m_0}{\gcd(m_0,m_1)\gcd(m_0,m_2)}$.
The normalizing map $\nu: \hat{\cA} \to \cA$ is given by
$$
\nu: (s,t) \mapsto (s^{\frac{m_1}{\gcd(m_0,m_1)}}t^{\frac{m_2}{\gcd(m_0,m_2)}}, s^{\frac{m_0}{\gcd(m_0,m_1)}}, t^{\frac{m_0}{\gcd(m_0,m_2)}}).
$$
Let $\hat{\cC}_{D|U}$ denote the pullback of the universal curve on $\Bb$ to $U$ by $\Phi\circ\nu$, and $\tilde{\cC}_{D|U}$ the family of curves constructed in \textsection~\ref{sec:normal:univ:curve}.  Remark that $\tilde{\cC}_{D|U}$ satisfies all the conditions preceding Lemma~\ref{lm:construct:path} with $U^*=\{(s,t)\in U, \, st \neq 0\}$. By the construction of $\Phi$, we get a section $\sigma$ of $\OO(-1)$ on $\Phi(\Bb)$. The pullback of this section to $U$ corresponds to a section $\Omega$ of the relative dualizing sheaf $\omega_{\tilde{\pi}}$ on $\tilde{\cC}_{D|U}$. 
One readily checks that $\Omega$ satisfies the hypotheses of Proposition~\ref{prop:int:s:arc:express}, and that the restriction of $\Omega$ to the fiber $\tilde{C}_\pp$ has simple poles at all the nodes of $\tilde{C}_\pp$.  
%
%
%
It follows from  \eqref{eq:Area:loc:express} that  we have
$$
\vartheta=\frac{(\lambda dt/t+\mu ds/s-\partial\phi)\wedge(\lambda d\bar{t}/t+\mu d\bar{s}/\bar{s} -\bar{\partial}\phi)}{(-\lambda\ln|t|^2-\mu\ln|s|^2+\phi)^2}
$$
for all $(s,t)\in U^*$, where $\lambda,\mu \in \R$, and  $\phi$ is a $C^\infty$ real positive function on $U$. 

If $\hpp$ is a smooth point in $\tilde{C}_\pp$, then it follows from Proposition~\ref{prop:int:s:arc:express} that there is a neighborhood $\cU$ of  $\hpp$ such that on $\cU^*:=\cU \cap \tilde{\cC}_{D|U^*}$ we can write
\begin{align*}
	\Theta 	& = \frac{\left(\lambda\frac{dt}{t}+\mu\frac{ds}{s}- \partial\phi\right)\wedge\left(\lambda \frac{d\bar{t}}{\bar{t}}+\mu\frac{d\bar{s}}{\bar{s}} -\bar{\partial}\phi\right) \wedge\left(\alpha\frac{dt}{t}+\beta\frac{ds}{s}+d\varphi\right) \wedge \left(\bar{\alpha}\frac{d\bar{t}}{\bar{t}} +\bar{\beta}\frac{d\bar{s}}{\bar{s}}+d\bar{\varphi}\right)}{(-2\lambda\ln|t|-2\mu\ln|s|+\phi)^3}
\end{align*}
where $\alpha$ and $\beta$ are some complex constants which are both zero if $\hpp$ is fixed by $\tau$, and $\varphi$ is a holomorphic function on $\cU$.

If $\hpp$ is a node of $\tilde{C}_\pp$ then a neighborhood $\cU$ of $p$ in $\tilde{\cC}_D $ is isomorphic to a neighborhood of $0$ in the set $\{(u,v,s,t)\in \Delta^2\times U, \; uv=t^a\}$. It follows from Proposition~\ref{prop:int:s:arc:express} that on $\cU^*:= \cU\cap \tilde{\cC}_{D|U^*}$ we have
$$
\Theta=-\frac{1}{2}\cdot\vartheta\wedge\frac{dP\wedge d\bar{P}}{\Ab}
$$  
where $dP=\alpha\frac{du}{u}+\beta\frac{dt}{t}+\gamma\frac{ds}{s}+\varphi$ with $\alpha,\beta,\gamma\in\C$ and $\varphi$ a holomorphic function on $\cU$.   We now remark that
$$
\frac{dt}{t}=\frac{dt^a}{at^a}=\frac{1}{a}\cdot\left(\frac{du}{u}+ \frac{dv}{v} \right)
$$
Therefore, up to  a multiplicative constant we have
\begin{equation*}
	\Theta  =  \frac{\left(\frac{du}{u}+\frac{dv}{v}+\mu_1\frac{ds}{s}-\partial\phi\right)\wedge\left(\frac{d\bar{u}}{\bar{u}}+\frac{d\bar{v}}{\bar{v}}+ \mu_1\frac{d\bar{s}}{\bar{s}}-\bar{\partial}\phi\right)\wedge \left(\alpha_1\frac{du}{u}+\alpha_1\frac{dv}{v}+ \beta\frac{ds}{s}+d\varphi\right)\wedge\left(\bar{\alpha}_1\frac{d\bar{u}}{\bar{u}}+\bar{\alpha}_1\frac{d\bar{v}}{\bar{v}}+ \bar{\beta}\frac{d\bar{s}}{\bar{s}}+d\bar{\varphi}\right)}{(-2\ln|u|-2\ln|v|-2\mu_1\ln|s|+\phi)^3}
\end{equation*}
with $m_1=a\mu/\lambda$ and $\alpha_1=\alpha/a$. One can now readily check that in both cases $\Theta$ satisfies the conditions (A) and (B). The details are left to the reader.
\end{proof}

\section{Properties of $\Theta$}\label{sec:prop:Theta}
Our goal now is to prove some characteristics of $\Theta$. 
By Theorem~\ref{th:Theta:current}, we know that the trivial extension of $\Theta$ to $\tilde{\cC}_D$ defines a closed current. 
We denote by $[\Theta]$ its cohomology class in $H^{2,2}(\tilde{\cC}_D)$. 
One of the fundamental properties of $[\Theta]$ is the following
\begin{Theorem}\label{th:Theta:pushforward}
	We have
	\begin{equation}\label{eq:class:Theta:pushforward}
	\tilde{\pi}_*[\Theta]=4\cdot(\imath\vartheta)=8\pi\cdot c_1(\OO(-1)).	
	\end{equation}
\end{Theorem}
Theorem~\ref{th:Theta:pushforward} will follows from
\begin{Lemma}\label{lm:int:Theta:pullback}
Let $\varphi$ be a smooth $(1,1)$-form on $\hat{\cX}_D$. Then
\begin{equation}\label{eq:int:Theta:pullback}
\int_{\tilde{\cC}_D}\Theta\wedge\tilde{\pi}^*\varphi=4\cdot\int_{\hat{\cX}_D}(\imath\vartheta)\wedge\varphi.
\end{equation}
\end{Lemma}
\begin{proof}
We have
$$
\int_{\tilde{\cC}_D}\Theta\wedge\tilde{\pi}^*\varphi=\int_{\cC_D}\Theta\wedge\tilde{\pi}^*\varphi.
$$	
Locally, open subsets of  $\cC_D$ are diffeomorphic to $U\times S$, where $U$ is an open subset of $\hat{\cX}_D$ and $S$ is a reference Riemann surface, with the map $\tilde{\pi}$ being the projection onto the first factor.
Shrinking $U$ if necessary, we can assume that there is a trivializing holomorphic section $\sigma$ of $\OO(-1)$ over $U$. This section  assigns a holomorphic $1$-form $\omega_\xx$ on the fiber $C_\xx$ for all $\xx\in U$.
By construction, we have
$$
\int_{U\times S}\Theta\wedge\tilde{\pi}^*\varphi=\int_U\left(\frac{\imath}{2}\int_{C_\xx} \frac{d P_c\wedge d\bar{P}_c}{\Ab(\xx)}\right)\cdot(\imath\vartheta(\xx))\wedge \varphi(\xx).
$$
On  the fiber $C_\xx$,   $P_c$ is locally defined by
$$
P_c(x)= \int_x^{\tau(x)}\omega_\xx,
$$
where the integral is taken along a chosen path  $c$. Since $\tau^*\omega_\xx=-\omega_\xx$, it follows that we have $dP_c(x)_{\left| C_\xx\right.}=-2\omega_\xx(x)$, for all $x\in C_\xx$ (independently of the choice of the path $c$). Thus
$$
\frac{\imath}{2}\int_{C_\xx}\frac{dP_c\wedge d \bar{P}_c}{\Ab(\xx)}=4\cdot \frac{\frac{\imath}{2}\int_{C_\xx}\omega_\xx\wedge\ol{\omega}_\xx}{\Ab(\xx)}=4,
$$
and \eqref{eq:int:Theta:pullback} follows.
\end{proof}

\begin{proof}[Proof of Theorem~\ref{th:Theta:pushforward}]
It is a well known fact that $\imath\vartheta$ defines a closed $(1,1)$-current on $\hat{\cX}_D$ whose cohomology class in $H^{1,1}(\hat{\cX}_D)$ equals $2\pi\cdot c_1(\OO(-1))$ (see for instance~\cite{Bai:GT, Ng25}). Thus \eqref{eq:class:Theta:pushforward}  follows from Lemma~\ref{lm:int:Theta:pullback}.
\end{proof}

\begin{Corollary}\label{cor:int:Theta:divisor}
	Let $\cD$ be a divisor in $\hat{\cX}_D$, such that the support $|\cD|$ of $D$ is not contained in the closure of the union of strata of group II in $\partial\hat{\cX}_D$.
	Denote by $\cD_\reg$ the set of regular points of $\cD$, and $\cD_0$ the set $\cD_{\reg}\setminus\partial_\infty\hat{\cX}_D$.   Then we have
	\begin{equation}\label{eq:inters:Theta:divisor}
		\langle [\Theta],[\tilde{\pi}^*\cD]\rangle= 8\pi c_1(\OO(-1))\cdot[\cD] = 4\int_{\cD_0}\imath\vartheta.
	\end{equation}
\end{Corollary}
\begin{proof}
	By Theorem~\ref{th:Theta:pushforward}, we have
	$$
	\langle [\Theta],\tilde{\pi}^*\cD\rangle=4\langle[\imath\vartheta],[\cD]\rangle.
	$$
	By the main result of \cite{Ng25}, we have that
	$$
	\langle [\imath\vartheta], [\cD]\rangle =2\pi c_1(\OO(-1))\cdot[\cD]=\int_{\cD_0}\imath\vartheta,
	$$
	and \eqref{eq:inters:Theta:divisor} follows.
\end{proof}

\begin{Proposition}\label{prop:Theta:inters:str:infty}
Let $\cS$ be an irreducible component of  $\partial_\infty\hat{\cX}_D$.  Then we have
\begin{equation}\label{eq:inters:Theta:str:infty}
\langle[\Theta],[\tilde{\pi}^*(\cS)]\rangle=0.
\end{equation}
\end{Proposition}
\begin{proof}
  By Theorem~\ref{th:Theta:pushforward}, we have
  $$
  \langle [\Theta],[\tilde{\pi}^*\cS]\rangle=\langle\tilde{\pi}_*[\Theta],[\cS]\rangle=8\pi\cdot \left( c_1(\OO(-1))\cdot[\cS]\right).
  $$
  By definition, a generic point of $\cS$  parametrizes an Abelian differentials on nodal curves having simple poles at all the nodes. One can pick out one of the nodes, and define a trivializing section of $\OO(-1)_{\left|\cS\right.}$ by setting the residue of the Abelian differentials at this node to be $1$. This means that the tautological line bundle $\OO(-1)$ is trivial on $\cS$. We thus have $c_1(\OO(-1))\cdot[\cS]=0$ and the proposition follows.
\end{proof}

Another important property of $\Theta$ is the following
\begin{Proposition}\label{prop:int:Theta:section}
Let $\iota: \hat{\cX}_D \to \tilde{\cC}_D$ be a section of $\tilde{\pi}$ whose image is denoted by $\Sigma$. Suppose that for all $\xx \in \hat{\cX}_D$, $\iota(\xx)$ is a smooth point in $C_\xx$. Then we have
\begin{equation}\label{eq:int:Theta:section}
\langle [\Theta],[\Sigma]\rangle = \int_{\Sigma\setminus\partial_\infty\tCD}\Theta.
\end{equation}
where $\partial_\infty\tCD$ is the preimage of $\partial_\infty\hXD$ in $\tCD$.
\end{Proposition}
\begin{proof}
Since $\Sigma$ is the image of a section, it is a suborbifold of $\tilde{\cC}_D$. By definition,	
$$
\langle [\Theta],[\Sigma]\rangle = \int_{\tilde{\cC}_D}\Theta\wedge\Phi_\Sigma,
$$
where $\Phi_\Sigma \in H^{1,1}(\tilde{\cC}_D)$ is the Poincaré dual of $\Sigma$.
The $(1,1)$-form $\Phi_\Sigma$ is in fact a  representative  of the Thom class of the normal bundle $\cN_\Sigma$ of $\Sigma$. By assumption, $\tilde{\pi}$ is a submersion in a neighborhood of $\Sigma$. Therefore, one can identify $\cN_\Sigma$ with the vertical tangent bundle of $\Sigma$ whose fiber at a point $(\xx,x)\in \Sigma$ is identified with $T_xC_\xx$. In particular, we can view $\cN_\Sigma$ as a holomorphic complex line bundle over $\Sigma$.
We now briefly recall the construction of $\Phi_\Sigma$, details of this construction can be found in \cite[Ch. 1,\textsection 6]{BT82}.
Denote by $p: \cN_\Sigma \to \Sigma\simeq \hat{\cX}_D$ the natural projection.
Let $\cN^*_\Sigma$ denote the complement in $\cN_\Sigma$ of the zero section.
There exists a smooth $1$-form $\psi$ on $\cN^*_\Sigma$ known as the {\em global angular form} which is defined as follows: 
let $\{U_\alpha, \alpha\in A\}$ be an open cover of $\Sigma$ such that $\cN_\Sigma$ is trivial on each $U_\alpha$. 
Let $d\theta$ denote the angular form on $\C^*$. On each $U_\alpha$  the restriction of $\psi$ to  $\cN^*_{\Sigma\left|U_\alpha\right.} \simeq U_\alpha\times \C^*$ is given by
\begin{equation}\label{eq:glob:ang:form:loc}
\psi=\frac{d\theta}{2\pi}- p^*\xi_\alpha,
\end{equation}
where $\xi_\alpha$ is a smooth $1$-form on $U_\alpha$. 
Note that $\psi$ is not necessarily closed. In fact, we have
$d\psi= - p^*\eta$, where $\eta$ is a smooth closed $2$-form on $\Sigma$ representing the Euler class of $\cN_\Sig$.

Chose some small $\eps_0\in \R_{>0}$. Let $\rho: \R^+ \to \R$ be a smooth function  such that $-1 \leq \rho(t)\leq 0$ for all $t\in \R^+$, $\rho \equiv -1$ on $[0;\eps_0/2]$, and $\rho\equiv 0$ on $[\eps_0;+\infty)$.
Fix a $C^\infty$ Hermitian metric $|.|$ on $\cN_\Sig$ and define $h: \cN_\Sigma \to \R$ by   $h(\hxx,v)=\rho(|v|)$, for all $\hxx\in \Sigma$ and $v\in  p^{-1}(\{\hxx\})$.
For all $0< \eps' < \eps $, let
$$
\cN_\Sigma(\eps):=\{(\hxx,v) \in \cN_\Sigma, \; |v| < \eps\} \quad \text{ and } \quad \cN_{\Sig}(\eps,\eps')=\{(\hxx,v)\in \cN_\Sig, \, \eps' < |v| < \eps \}.
$$
Define
$$
\Phi:=d(h\cdot\psi) = dh\wedge \psi -h\cdot p^*\eta.
$$
Then $\Phi$ is a closed $2$-form on $\cN_\Sigma$, with support contained in  $\cN_\Sigma(\eps_0)$. Note that the support of $\dh\cdot\psi$ is contained in $\cN_\Sigma(\eps_0,\eps_0/2)$. 
If $\eps_0$ is small enough, $\cN_\Sigma(\eps_0)$ can be embedded into $\tilde{\cC}_D$ by a smooth embedding.
Thus we can consider $\Phi$ as a closed $2$-form on $\tilde{\cC}_D$.
By construction, $\Phi$ is a representative of the Poincaré dual of $[\Sigma]$.
As a consequence,
$$
\langle[\Theta],[\Sigma]\rangle =\int_{\tilde{\cC}_D}\Phi\wedge\Theta.
$$
Let us now fix a $C^\infty$ Riemannian metric on $\tilde{\cC}_D$ whose restriction to $\cN_\Sigma$  coincides with the metric $|.|$.
Given $0<\eps < \eps_0/2$ and $\delta>0$, let  $\cU_\eps$ and $\cV_\delta$ be respectively the $\eps$-neighborhood of $\Sigma$ and the $\delta$-neighborhood of $\partial_\infty\tilde{\cC}_D$ with respect to this metric.
Since $\Theta$ extends smoothly across the strata of group I in $\partial \tCD$, we have
$$
\langle[\Theta],[\Sigma]\rangle =\lim_{\delta\to 0}\lim_{\eps \to 0}\int_{\tilde{\cC}_D\setminus(\cU_\eps\cup\cV_\delta)}\Phi\wedge\Theta.
$$
Since $\Phi\wedge\Theta=d(h\cdot\psi)\wedge\Theta=d(h\cdot\psi\wedge\Theta)$ on $\tilde{\cC}_D\setminus(\cU_\eps\cup\cV_\delta)$, Stokes' formula gives
$$
\int_{\tilde{\cC}_D\setminus(\cU_\eps\cup\cV_\delta)}\Phi\wedge\Theta=-\int_{\partial(\cU_\eps\cup\cV_\delta)}h\cdot\psi\wedge\Theta =-\int_{\partial\cU_\eps \setminus\cV_\delta}h\cdot\psi\wedge \Theta -\int_{\partial\cV_\delta\setminus \cU_\eps}h\cdot \psi\wedge\Theta.
$$

By compactness, modulo a negligible subset, we can decompose  $\partial\cU_\eps\setminus \cV_\delta$ into a finite union of subsets $\{\tilde{U}'_i, \; i\in I\}$ where for each $i\in I$, $\tilde{U}'_i\simeq U'_i\times \partial \Delta_\eps$ with $U_i \subset \Sigma$ being a relatively compact subset contained in one of the open subsets $\{U_\alpha, \, \alpha\in A\}$. Since $h\equiv -1$ on $U'_i\times\partial\Delta_\eps$, we have
\begin{align*}
-\int_{\tilde{U}'_i}h\cdot\psi\wedge\Theta & = \int_{U'_i\times\partial\Delta_\eps}\psi\wedge\Theta = \int_{U'_i\times\partial\Delta_\eps} \frac{d\theta}{2\pi}\wedge\Theta - \int_{U'_i\times\partial\Delta_\eps} p^*\xi_{\alpha}\wedge\Theta.
\end{align*}
Since 
$$
\lim_{\eps\to 0}\int_{U'_i\times\partial\Delta_\eps}\frac{d\theta}{2\pi}\wedge\Theta=\int_{U'_i}\Theta, \quad \text{ and } \quad
\int_{U'_i\times\partial\Delta_\eps}p^*\xi_{\alpha}\wedge\Theta=O(\eps),
$$
it follows that
$$
\lim_{\eps\to 0}\int_{\partial\cU_\eps\setminus\cV_\delta}\psi\wedge\Theta=\int_{\Sigma\setminus \cV_\delta}\Theta,
$$
and therefore,
$$
\int_{\tilde{\cC}_D\setminus\cV_\delta}\Phi\wedge\Theta=\lim_{\eps\to 0}\int_{\tilde{\cC}_D\setminus(\cU_\eps\cup\cV_\delta)}\Phi\wedge\Theta =\int_{\partial\cV_\delta}h\cdot\psi\wedge\Theta +\int_{\Sigma\setminus\cV_\delta}\Theta.
$$
Recall that by construction, ${\rm supp}(h) \subset \cU_{\eps_0}$.
By compactness, for $\eps_0>0$ small enough, we can cover $\cV_{\delta}\cap\cU_{\eps_0}$ by a finite family $\{W_j, \; j\in J \}$ of open subsets of $\tilde{\cC}_D$, where for each $j\in J$, $W_j$ is biholomorphic to $\Delta_r^3$ for some $r>\eps_0$ with a coordinate system $(s,t,x)$ such that 
\begin{itemize}
	\item[$\bullet$] $W_j\cap \Sigma=\{x=0\}$,
	
	\item[$\bullet$] $W_j\cap\cU_{\eps_0}=\{|x| < \eps_0\}$, and
	
	\item[$\bullet$] either (a) $W_j\cap\partial_\infty\tilde{\cC}_D=\{t=0\}$, or (b)  $W_j\cap\partial_\infty\tilde{\cC}_D=\{st=0\}$.
\end{itemize}
Case (a) occurs when $W_j$ is a neighborhood of a point $(\xx,x) \in \Sigma$, where $\xx$ is contained in stratum of group II or group III in $\partial\hat{\cX}_D$, and case (b) occurs when $\xx$ is contained in a stratum of group IV. In both cases $x$ is a smooth point in $\tilde{C}_\xx$.
We have in  case (a)
$$
V_j:=W_j\cap(\partial\cV_\delta\cap\cU_{\eps_0})\simeq \Delta_r\times\partial\Delta_\delta\times \Delta_{\eps_0},
$$
and in  case (b)
$$
V_j:=W_j\cap(\partial\cV_\delta\cap\cU_{\eps_0})\simeq A(r,\delta)\times\partial\Delta_\delta\times \Delta_{\eps_0}\cup \partial\Delta_\delta\times A(r,\delta)\times \Delta_{\eps_0},
$$
where $A(r,\delta)=\Delta_r\setminus\Delta_\delta$.
In these local coordinates, we have 
$$
d\theta=\frac{\imath}{2}(d\bar{x}/\bar{x}-dx/x)\quad \text{ and } \quad  h(s,t,x) = \rho(|x|).
$$
It follows from the proof of Theorem~\ref{th:Theta:current} that up to a multiplicative constant in case (a)
$$
\Theta=\frac{(dt/t-\partial\phi)\wedge(d\bar{t}/\bar{t}-\bar{\partial}\phi)\wedge (\mu_1 dt/t +d\varphi)\wedge (\bar{\mu}_1d\bar{t}/\bar{t} +d\bar{\varphi})}{(-2\ln|t|+\phi)^3}
$$
while in case (b)
$$
\Theta = \frac{\left(\lambda\frac{dt}{t}+\mu\frac{ds}{s}-\partial\phi\right)\wedge\left(\lambda \frac{d\bar{t}}{\bar{t}}+\mu\frac{d\bar{s}}{\bar{s}} -\bar{\partial}\phi\right) \wedge\left(\lambda_1\frac{dt}{t}+\mu_1\frac{ds}{s}+d\varphi\right)\wedge \left(\bar{\lambda}_1\frac{d\bar{t}}{\bar{t}}+\bar{\mu}_1\frac{d\bar{s}}{\bar{s}}+d\bar{\varphi}\right)}{(-2\lambda\ln|t|-2\mu\ln|s|+\phi)^3}
$$
where $\lambda,\mu \in \R_{>0}, \lambda_1, \mu_1 \in \C$, $\phi$ is a smooth function, and $\varphi$ a holomorphic function  on $W_j$.
It follows that in case (a)
$$
\int_{V_j}h\cdot\psi\wedge\Theta=\int_{V_j}h\cdot(\frac{d\theta}{2\pi} - p^*\xi_\alpha)\wedge\Theta = O(\frac{1}{-(\ln|\delta|)^3})
$$
while in case (b)
$$
\int_{V_j}h\cdot\psi\wedge\Theta= \int_{V_j}h\cdot(\frac{d\theta}{2\pi}-p^*\xi_\alpha)\wedge\Theta = O(\frac{1}{(\ln|\delta|)^2}).
$$
As a consequence, we get
$$
\lim_{\delta\to 0}\int_{\partial\cV_\delta}h\psi\wedge\Theta =0,
$$
and therefore
\begin{align*}
\langle[\Theta], [\Sigma]\rangle &= \int_{\tilde{\cC}_D}\Phi\wedge\Theta = \lim_{\delta\to 0}\left(\lim_{\eps\to 0}\int_{\tilde{\cC}_D\setminus(\cU_\eps\cup\cV_\delta)}\Phi\wedge\Theta\right)\\
&= \lim_{\delta\to 0}\left( \int_{\partial\cV_\delta} h\cdot \psi\wedge\Theta + \int_{\Sigma\setminus\cV_\delta}\Theta\right)\\
&=\int_\Sigma\Theta.
\end{align*}
\end{proof}

To our purpose, we will need the following result which strengthens  Proposition~\ref{prop:Theta:inters:str:infty}.
\begin{Proposition}\label{prop:Theta:inters:infty:str:enhanced}
  Let $\cE$ be an irreducible component of $\partial_\infty\tilde{\cC}_D:=\tilde{\pi}^{-1}(\partial_\infty\hXD)$. Then we have
  \begin{equation}\label{eq:Theta:inters:infty:str:enhanced}
  \langle[\Theta],[\cE]\rangle=0.
  \end{equation}
\end{Proposition}
\begin{proof}
Let $\cS:=\tilde{\pi}(\cE)$. Then $\cS$ is an irreducible component of $\partial_\infty\hat{\cX}_D$, that is $\cS$ is the closure of a component $\cS^*$ of a stratum in group II.   For  every $\pp \in \cS^*$, $\cE$ intersects the fiber  $\tilde{C}_\pp=\tilde{\pi}^{-1}(\{\pp\})$ in an irreducible component $E_\pp$ of $\tilde{C}_\pp$. 
We fist consider the case where $E_\pp$ is  smooth. This case occurs when $\cS^*$ is a component of $\cS^b_{2,0}$, or $\cS^*$ is a component of $\cS_{1,1}$ and $E_\pp$ is  the $\Pb^1$ component of $\tilde{C}_\pp$. Note that in all of these cases, $E_\pp$ is invariant by the Prym involution.

By assumption, $\cE$ is a suborbifold of $\tilde{\cC}_D$, and the Poincaré dual of $[\cE]$ is represented by a $2$-form $\Phi$ supported in a tubular neighborhood of $\cE$ ($\Phi$ also represents the Thom class of the normal bundle $\cN_\cE$ of $\cE$).
Recall that $\Phi=d(h\cdot\psi)$, where
\begin{itemize}

\item[$\bullet$]  $\psi$ is the global angular form  defined on the complement of the zero section in the normal bundle $\cN_\cE$,

\item[$\bullet$] With a choice of smooth Hermitian metric on $\cN_\cE$, $h$ is a function with support  contained in the $\eps_0$-neighborhood of $\cE$ which satisfies $h\equiv -1$ in the  $\eps_0/2$-neighborhood  of $\cE$ (here $\cE$ is identified with the zero section of $\cN_\cE$).

\end{itemize}
Let us fix a  Riemannian metric on $\tilde{\cC}_D$ whose restriction to $\cN_ \cE$ coincides with the Hermitian metric used to define $h$. For all $\eps>0$ denote by $\cU_\eps$  the $\eps$-neighborhood of $\cE$, and by  $\cV_\eps$ the $\eps$-neighborhood of $\partial_\infty\tilde{\cC}_D$  with respect to this metric.
By assumption, $\cU_{\eps_0}$ is isometric to $\cE\times\Delta_{\eps_0}$.
Since $\Theta$ is a well  defined smooth $(2,2)$-form outside of $\partial_\infty\tilde{\cC}_D$, for all $0 < \eps < \eps_0/2$, we have
$$
(\Phi\wedge\Theta)_{\left|\tilde{\cC}_D\setminus\cV_\eps\right.}=d(h\cdot\psi\wedge\Theta)_{\left|\tilde{\cC}_D\setminus\cV_\eps\right.}.
$$
It follows from Stokes' formula that
$$
\int_{\tilde{\cC}_D\setminus\cV_\eps}\Phi\wedge\Theta=-\int_{\partial\cV_\eps\cap\cU_{\eps_0}}h\cdot\psi\wedge\Theta.
$$
Let $\partial'_\infty\tilde{\cC}_D$ be the union of all the irreducible components of $\partial_\infty\tilde{\cC}_D$ except $\cE$.
Note that $\cE$ intersects $\partial'_\infty\tilde{\cC}_D$  transversely.

For all $\hpp \in \cE$, $\hpp$ has a neighborhood $U$ in $\cU_{\eps_0}$ which is isometric to $\Delta_{\eps_0}\times\Delta_\delta\times\Delta_{\delta'}$, for some $\delta,\delta'\in\R >0$, with coordinates $(x,y,z)$ such that $\cE\cap U \simeq \{0\}\times\Delta_{\delta}\times\Delta_{\delta'}$.
We will give an estimate for the integral of $h\cdot\psi\wedge\Theta$ on $\partial\cV_\eps\cap U$. This estimate depends on the  geometry of $\partial\cV_\eps$ as well as the expression of $\Theta$ in the neigborhood of $\hpp$. Recall that
$$
\psi=\frac{d\theta}{2\pi} - p^*\xi = \frac{\imath}{2}\cdot\left(\frac{d\bar{x}}{\bar{x}} - \frac{dx}{x} \right) - p^*\xi,
$$
where $p: U \to \Delta_{\delta}\times\Delta_{\delta'}$ is the natural projection, and $\xi$ is a smooth $1$-form on $\Delta_{\delta}\times\Delta_{\delta'} \subset \cE$.
By convention, in what follows $\phi$ (resp. $\varphi$) is be a real positive smooth function (resp. holomorphic function) on $U$,  $\lambda, \mu$ are positive real numbers, and $\alpha,\beta,\gamma$ are some complex numbers.

Let $\pp$ is the image of $\hpp$ in $\partial_\infty\hXD$. We have the following cases:
\begin{itemize}
	\item[(a)] Case $\pp \in \cS^*$.  
	We have two subcases
	\begin{itemize}
		\item[(a.1)] Case $\hpp$ is a smooth point in $\tilde{C}_\pp$. We have $\partial_\infty\tCD\cap U=\cE\cap U=\{x=0\}$. 
		From  the proof of Theorem~\ref{th:Theta:current} we get that
		$$
		\Theta=\frac{(\frac{dx}{x}-\partial\phi)\wedge(\frac{d\bar{x}}{\bar{x}}-\bar{\partial}\phi)\wedge (\alpha\cdot\frac{dx}{x}+d\varphi)\wedge (\bar{\alpha}\cdot\frac{d\bar{x}}{\bar{x}} + d\bar{\varphi})}{(-2\ln|x|+\phi)^3}.
		$$
		We thus have
		$$
		-\int_{U\cap \partial \cV_\eps} h\cdot\psi\wedge \Theta = \int_{\Delta_{\delta}}\int_{\Delta_{\delta'}}\int_{|x|=\eps}\left(\frac{\imath}{2}\cdot\left( \frac{d\bar{x}}{\bar{x}}-\frac{dx}{x}\right)-p^*\xi\right)\wedge\Theta =O\left(\frac{1}{-(\ln\eps)^3}\right).
		$$
		
		\item[(a.2)] Case $\hpp$ is a node of $\tilde{C}_\pp$. In this case $\hpp$ in an intersection point of $\cE$ and $\partial'_\infty\tilde{\cC}_D$ (recall that by assumption the fiber $E_\pp$ does not have self-node). 
		We can choose the labeling of the coordinates on $U$ such that $\partial_\infty\tilde{\cC}_D\cap U \simeq \{xy=0\}$. 
		From the proof of Theorem~\ref{th:Theta:current}, the restriction of $\Theta$ to $U$ can be written as
		$$
		\Theta=\frac{\left(\frac{dx}{x} + \frac{dy}{y} - \partial\phi\right)\wedge \left(\frac{d\bar{x}}{\bar{x}} + \frac{d\bar{y}}{\bar{y}} - \bar{\partial}\phi \right)\wedge \left(\beta\frac{dy}{y} + d\varphi \right) \wedge \left(\bar{\beta}\frac{d\bar{y}}{\bar{y}} + d\bar{\varphi} \right)}{(-2\ln|x| -2\ln|y| +\phi)^3}.
		$$
		Note that $\cV_\eps\cap U$ is the union $\Delta_\eps\times\Delta_{\delta}\times\Delta_{\delta'}\cup \Delta_{\eps_0}\times\Delta_{\eps}\times\Delta_{\delta'}$. Thus
		$$
		\partial \cV_\eps\cap U= \partial\Delta_\eps\times A(\delta,\eps)\times\Delta_{\delta'}\cup A(\eps_0,\eps)\times\partial\Delta_\eps\times \Delta_{\delta'}
		$$
		One readily checks that
		\begin{equation*}
		-\int_{\partial\Delta_\eps\times A(\delta,\eps)\times\Delta_{\delta'}} h\cdot\psi\wedge\Theta = \int_{|z|<\delta'}\int_{\eps < |y| <\delta} \int_{|x|=\eps} \left( \frac{\imath}{2}\left(\frac{d\bar{x}}{\bar{x}}-\frac{dx}{x}\right)-p^*\xi\right)\wedge\Theta = O\left( \frac{1}{(\ln\eps)^2}\right)
		\end{equation*}
		and
		\begin{equation*}
		-\int_{A(\eps_0,\eps)\times\partial\Delta_\eps\times \Delta_{\delta'}} h\cdot\psi\wedge\Theta = -\int_{|z|<\delta'}\int_{\eps < |x| <\eps_0} \int_{|y|=\eps} h\cdot\left( \frac{\imath}{2}\left(\frac{d\bar{x}}{\bar{x}}-\frac{dx}{x}\right)-p^*\xi\right)\wedge\Theta = O\left( \frac{1}{(\ln\eps)^2}\right).
		\end{equation*}
		Hence
		\begin{equation*}
		-\int_{U\cap \partial\cV_\eps}h\cdot\psi\wedge\Theta  =  O\left( \frac{1}{(\ln\eps)^2}\right).
		\end{equation*}
	\end{itemize}
	
	\item[(b)] Case $\pp$ is contained in a stratum of group III. Again, we have two subcases: either $\hpp$ is a smooth point of $\tilde{C}_\pp$ or $\hpp$ is a node of $\tilde{C}_\pp$. In the former case,  $\partial_\infty\tilde{\cC}_D\cap U = \cE\cap U \simeq \{x=0\}$, and the restriction of $\Theta$ to $U$ is given by
	$$
	\Theta=\frac{(\frac{dx}{x}-\partial\phi)\wedge(\frac{d\bar{x}}{\bar{x}}-\bar{\partial}\phi)\wedge (\alpha\cdot\frac{dx}{x}+d\varphi)\wedge (\bar{\alpha}\cdot\frac{d\bar{x}}{\bar{x}} + d\bar{\varphi})}{(-2\ln|x|+\phi)^3}
	$$
	In the latter case, $\cS\cap U\simeq \{x=0\}$, while $\partial_\infty\tilde{\cC}_D\cap U \simeq \{xy=0\}$, and the restriction of $\Theta$ is given by
	$$
	\Theta=\frac{\left(\frac{dx}{x} + \frac{dy}{y} - \partial\phi\right)\wedge \left(\frac{d\bar{x}}{\bar{x}} + \frac{d\bar{y}}{\bar{y}} - \bar{\partial}\phi \right)\wedge \left(\beta\frac{dy}{y} + d\varphi \right) \wedge \left(\bar{\beta}\frac{d\bar{y}}{\bar{y}} + d\bar{\varphi} \right)}{(-2\ln|x| -2\ln|y| +\phi)^3}.
	$$
	We  can then conclude by the same arguments as Case (a).
	
	\item[(c)] Case $\pp$ is contained in a stratum of  group IV. We have two subcases:
	\begin{itemize}
		\item[(c1)]  $\hpp$ is a smooth point of $\tilde{C}_\pp$.
		In this case $\partial_\infty\tilde{\cC}_D\cap U\simeq \{xy=0\}$.
		From Theorem~\ref{th:Theta:current}, the restriction of $\Theta$ to $U$ is given by 
		\begin{equation*}
			\Theta 	 = \frac{\left(\lambda\frac{dx}{x}+\mu\frac{dy}{y}- \partial\phi\right)\wedge\left(\lambda \frac{d\bar{x}}{\bar{x}}+\mu\frac{d\bar{y}}{\bar{y}} -\bar{\partial}\phi\right) \wedge\left(\alpha\frac{dx}{x}+\beta\frac{dy}{y}+d\varphi\right)\wedge \left(\bar{\alpha}\frac{d\bar{x}}{\bar{x}}+\bar{\beta}\frac{d\bar{y}}{\bar{y}}+d\bar{\varphi}\right)}{(-2\lambda\ln|x|-2\mu\ln|y|+\phi)^3}
		\end{equation*}
		It follows that
		$$
		-\int_{\cV_\eps\cap U}h\cdot\psi\wedge \Theta =O\left(\frac{1}{(\ln\eps)^2}\right).
		$$
		
		\item[(c2)] $\hpp$ is a node of $\tilde{C}_\pp$. In this case $\partial_\infty\tilde{\cC}_D\cap U \simeq \{xyz=0\}$.
		From Theorem~\ref{th:Theta:current}, up to a multiplicative constant, the restriction of $\Theta$ to $U$ is given by 
		\begin{equation*}
			\Theta  =  \frac{\left(\frac{dx}{x}+\frac{dy}{y}+\mu\frac{dz}{z}-\partial\phi\right)\wedge\left(\frac{d\bar{x}}{\bar{x}}+\frac{d\bar{y}}{\bar{y}}+ \mu\frac{d\bar{z}}{\bar{z}}-\bar{\partial}\phi\right)\wedge \left(\alpha\frac{dx}{x}+\beta\frac{dy}{y}+ \gamma\frac{dz}{z}+d\varphi\right)\wedge\left(\bar{\alpha}\frac{d\bar{x}}{\bar{x}}+\bar{\beta}\frac{d\bar{y}}{\bar{y}}+ \bar{\gamma}\frac{d\bar{z}}{\bar{z}}+d\bar{\varphi}\right)}{(-2\ln|x|-2\ln|y|-2\mu\ln|z|+\phi)^3},
		\end{equation*}
		It follows that
		$$
		-\int_{\cV_\eps\cap U}h\cdot\psi\wedge \Theta =O\left(\frac{1}{-\ln\eps}\right).
		$$
	\end{itemize}		
\end{itemize}
In all cases, we have 
$$
\lim_{\eps \to 0} \int_{\cV_\eps\cap U}h\cdot\psi\wedge\Theta =0.
$$		
Since we can cover the $U_{\eps_0}$ by a finite family of open subsets of $\tilde{\cC}_D$ of the form $\Delta_{\eps_0}\times\Delta_\delta\times\Delta_{\delta'}$, we obtain 
$$
	\langle[\Theta],[\cE]\rangle  = \lim_{\eps\to 0}\int_{\tilde{\cC}_D\setminus\cV_\eps}\Phi\wedge\Theta = -\lim_{\eps\to 0}\int_{\partial \cV_\eps\cap \cU_{\eps_0}}h\cdot\psi\wedge \Theta =0.
$$
We now turn to the case the fiber $E_\pp$ is not smooth for $\pp\in \cS^*$. This case only occurs when $\cS^*$ is a component of  $\cS_{1,1}$, and $E_\pp$ is the component of $\tilde{C}_\pp$ which is a nodal curve of genus two. The other component of $\tilde{C}_\pp$ is isomorphic to $\Pb^1$. We denote this component by $E'_\pp$ and the corresponding component of $\partial_\infty\tilde{\cC}_D$ by $\cE'$. By the first part of the proof, we have
$$
\langle[\Theta],[\cE']\rangle=0.
$$
By construction we have $[\tilde{\pi}^*\cS]=[\cE]+[\cE']$. By Proposition~\ref{prop:Theta:inters:str:infty} we know that $\langle[\Theta], [\tilde{\pi}^{-1}\cS]\rangle=0$. As a consequence, we get $\langle[\Theta],[\cE]\rangle=0$ as well.
\end{proof}

\section{Volume of $\cX_D$ and intersections in $\tCD$} \label{sec:vol:XD:n:Theta}
In this section, we will prove 
\begin{Theorem}\label{th:vol:XD}
	We have
	\begin{equation}\label{eq:vol:XD:form}
		\mu(\cX_D)=-\frac{\pi}{144}\langle [\Theta],[\ol{\cT}_{0,2}]\rangle - \frac{\pi}{8} \langle[\Theta],[\ol{\cT}^{a,1}_{2,0}]\rangle.
	\end{equation}
	where $[\Theta]$ is the cohomology class of $\Theta$ in $H^{2,2}(\tCD)$.
\end{Theorem}
Theorem~\ref{th:vol:XD} will follows from the results of \textsection\ref{sec:div:relations:in:CD} and Theorem~\ref{th:vol:XD:n:rel:dual:bdl} here below. 

\begin{Theorem}\label{th:vol:XD:n:rel:dual:bdl}
	We have
	\begin{equation}\label{eq:vol:XD:n:rel:dual:dbl}
		\mu(\cX_D)=\frac{-\pi}{24}\langle[\Theta],[\omega_{\tCD/\hXD}]\rangle.
	\end{equation}
\end{Theorem}
\begin{proof}
Let $\xx=(C_\xx,\ul{x},\tau_\xx,[\omega_\xx])$, where $\ul{x}=(x_1,\dots,x_5,x'_5)$, be a point in $\XD$.
Fix a homotopy class $c$ of continuous paths from $x_5$ to  $x'_5$ in $C_\xx$.
Let $\omega: \xx \mapsto \omega_\xx$ be a local holomorphic section of the tautological line bundle in a neighborhood of $\xx$.
Then by Proposition~\ref{prop:push:meas:proj:sp}, we have
$$
d\mu(\xx)=-\frac{\pi}{6}\cdot\imath\vartheta(\xx)\wedge\left(\frac{\imath}{2}\partial\bar{\partial}\left(\frac{\left|\int_c\omega\right|^2}{||\omega||^2}\right)(\xx)\right)
$$
Recall that $\Sigma_5$ is the divisor in $\tCD$ which intersects $C_\xx$ at the points $\{x_5,x'_5\}$. 
In particular, $\Sigma_5$ corresponds to two  local sections  of  $\tilde{\pi}$.  The local expression of the volume form $d\mu$ on $\cX_D$ is clearly the pullback of $-\frac{\pi}{6}\cdot\Theta$ by those local sections. It follows that we have
$$
\mu(\XD)=\int_{\XD}d\mu=-\frac{\pi}{12}\cdot\int_{\Sigma_5\cap\CD}\Theta =\frac{-\pi}{12}\cdot\int_{\Sigma_5\setminus\partial_\infty\tCD}\Theta.
$$
By Proposition~\ref{prop:fund:rel:tCD} and Proposition~\ref{prop:int:Theta:section}, we get that
\begin{equation}\label{eq:vol:XD:n:inters:Theta:a}
\mu(\XD) =\frac{-\pi}{12}\cdot\langle[\Theta],[\Sigma_5]\rangle  = \frac{-\pi}{24}\cdot\langle[\Theta],[\omega_{\tCD/\hXD}] - [\tilde{\pi}^*\OO(-1)]-5[\ol{\cT}^0_{1,0}]-[\ol{\cT}^{a,0}_{2,0}] -3[\ol{\cT}^0_{0,2}]\rangle
\end{equation}
We claim that
$$
\langle[\Theta],[\tilde{\pi}^*\OO(-1)]\rangle = \langle[\Theta],[\ol{\cT}^0_{1,0}]\rangle = \langle[\Theta],[\ol{\cT}^{a,0}_{2,0}]\rangle = \langle[\Theta],[\ol{\cT}^0_{0,2}]\rangle =0.
$$
Indeed, by Theorem~\ref{th:Theta:pushforward}, we have
$$
  \langle[\Theta],[\tilde{\pi}^*\OO(-1)]\rangle = 8\pi\cdot  c_1^2(\OO(-1))\cdot[\hXD].
$$
It follows from the main result of \cite{Ng25} that $\left(\frac{\imath}{2\pi}\cdot\vartheta\right)^2$ is a representative in the sense of current of $c_1^2(\OO(-1))$ on $\hXD$. Since $\vartheta^2$ vanishes identically, we conclude that $\langle[\Theta],[\tilde{\pi}^*\OO(-1)]\rangle=0$.

\medskip 

For $\langle[\Theta],[\ol{\cT}^0_{1,0}]\rangle$, we observe that $\ol{\cT}^0_{1,0}$ is a smooth divisor in $\tCD$  (the intersection $\ol{\cT}^0_{1,0}\cap\partial_\infty\tCD$ consists of some $\Pb^1$ components in the fiber of $\tilde{\pi}$ over points in the strata of group III).
By similar arguments as in Proposition~\ref{prop:int:Theta:section}, we get that
$$
\langle[\Theta],[\ol{\cT}^0_{1,0}]\rangle = \int_{\ol{\cT}^0_{1,0}\setminus\partial_\infty\tCD}\Theta
$$
Note that $\tilde{\pi}(\ol{\cT}^0_{1,0}\setminus\partial_\infty\tCD)=\cS_{1,0}$.
For any $\xx\in \cS_{1,0}$, let $C_\xx^0$ be the component of $C_\xx$ that is contained in $\ol{\cT}^0_{1,0}$.
Remark that $C_\xx^0$ is invariant by the Prym involution.
By definition,   $\omega_\xx$ vanishes identically on $C_\xx^0$. Therefore, the function $\varphi_c$ defined in \eqref{eq:normalized:per:funct} is identically zero on $C_\xx^0$. Consequently, $\Theta$ vanishes identically on $\ol{\cT}^0_{1,0}\setminus\partial_\infty\tCD$, and we have
$$
\langle[\Theta],[\ol{\cT}^0_{1,0}]\rangle = \int_{\ol{\cT}^0_{1,0}\setminus\partial_\infty\tCD}\Theta=0.
$$
The proofs of $\langle[\Theta],[\ol{\cT}^{a,0}_{2,0}]\rangle = \langle[\Theta],[\ol{\cT}^0_{0,2}]\rangle =0$ follow the same lines.
As a direct consequence, we obtain \eqref{eq:vol:XD:n:rel:dual:dbl} from \eqref{eq:vol:XD:n:inters:Theta:a}.
\end{proof}

\subsection*{Proof of Theorem~\ref{th:vol:XD}}\label{subsec:prf:th:vol:XD}
\begin{proof}
It follows from Theorem~\ref{th:vol:XD:n:rel:dual:bdl} and Proposition~\ref{prop:rel:cotangent:class}  that we have
\begin{align*}
\mu(\XD) & = \frac{-\pi}{24}\cdot\langle [\Theta],[\omega_{\tCD/\hXD}] \rangle \\
          & = \frac{-\pi}{24}\cdot\langle [\Theta], \frac{1}{6} [\ol{\cT}_{0,2}]+2[\ol{\cT}^0_{1,0}] +[\ol{\cT}^{a,0}_{2,0}]+3[\ol{\cT}^{a,1}_{2,0}]+\sum_{i=1}^4[\Sigma_i]+[\cR_1]\rangle
\end{align*}
where $\cR_1$ is a divisor with support contained in $\partial_\infty\tCD$.
By Proposition~\ref{prop:Theta:inters:infty:str:enhanced} and Proposition~\ref{th:vol:XD:n:rel:dual:bdl}
$$
\langle [\Theta],[\ol{\cT}^0_{1,0}]\rangle = \langle [\Theta],[\ol{\cT}^{a,0}_{2,0}]\rangle = \langle [\Theta],[\cR_1]\rangle=0.
$$
Since the function $\varphi_c$ in \eqref{eq:normalized:per:funct} vanishes identically on $\Sigma_i, \; i=1,\dots,4$, Proposition~\ref{prop:int:Theta:section} implies that $\langle[\Theta],[\Sigma_i]\rangle=0$ for all $i=1,\dots,4$. As a consequence, we obtain \eqref{eq:vol:XD:form}.
\end{proof}

\section{Triples of tori and modular curves in $\hat{\cX}_D$} \label{sec:triple:tori:comput}
Our goal in this section  is to calculate $\langle[\Theta],[\ol{\cT}_{2,0}^{a,1}]\rangle$.
Recall that $\Omega E_D(0^3)$ is the space of triples of tori Prym eigenforms (cf. \textsection\ref{subsec:triple:tori:ef:def}). 
Since the space $\Omega E_{D}(0^3)$ consists of finitely many $\GL^+(2,\R)$-orbits, $W_D(0^3):=\Pb\Omega E_{D}(0^3)$ is a finite union of hyperbolic surfaces (orbifolds) with finite area. Each component of $\Pb\Omega E_{D}(0^3)$ is actually a finite cover of the modular curve $\Hbb/\SL(2,\Z)$.
We will prove
\begin{Theorem}\label{th:int:Theta:T:2:0:a:1}
For all discriminant $D > 4$, $D$ is not a square,  we have
\begin{equation}\label{eq:int:Theta:tS:20:a1}
\langle[\Theta],[\ol{\cT}_{2,0}^{a,1}]\rangle= -48\pi\cdot\chi(W_{D}(0^3)).
\end{equation}
\end{Theorem}

In the case $D\equiv 1 \, [8]$, $\Omega E_D(2,2)^{\odd}$ has two connected components denoted by $\Omega E_{D+}(2,2)^{\odd}$ and $\Omega E_{D-}(2,2)^{\odd}$ (see \textsection~\ref{subsec:triple:tori:proj:1} for more details).
Recall that $\hat{\cX}_{D\pm}$ are the closures of the preimages  of $\Pb\Omega E_{D\pm}(2,2)^\odd$ in $\hXD$.
Denote by $\cS_{2,0}^{a\pm}$ the intersection of $\cS_{2,0}^a$ with $\hat{\cX}_{D\pm}$ respectively.
Finally, let $\cT_{2,0}^{a\pm,1}$ be the preimages of $\cS_{2,0}^{a\pm}$ in $\cT_{2,0}^{a,1}$.
We will prove a more precise version of Theorem~\ref{th:int:Theta:T:2:0:a:1} for this case

\begin{Theorem}\label{th:int:Theta:T:2:0:a:1:D:odd}
For all discriminant $D>9, \; D \equiv 1 \, [8]$, $D$ is not a square, we have
\begin{equation}\label{eq:int:Theta:on:tS:a1:20:D:odd}
\langle [\Theta],[\ol{\cT}_{2,0}^{a+,1}]\rangle = \langle [\Theta],[\ol{\cT}_{2,0}^{a-,1}]\rangle  = -24\pi\cdot \chi(W_D(0^3)).
\end{equation}
\end{Theorem}
The Euler characteristic of $W_D(0^3)$ can be computed explicitly.
For all  $m\in \N,  \; m \geq 2$, define
\[
c(m):=m\prod_{\substack{p \, | \, m \\ p \, {\rm prime}}}\left(1+\frac{1}{p} \right).
\]	
For all integer $e$ such that $e^2< D$ and $D \equiv e^2 \, [8]$, we can write $\frac{D-e^2}{8}= f^2q$, where $f,q \in \N$, and $q$ is  square-free.  Define
\[
m_D(e):=\sum_{\substack{r \, | \, f \\ \gcd(r,e)=1}}c(\frac{D-e^2}{8r^2}).
\]
We will prove
\begin{Proposition}\label{prop:Euler:char:W:D:0}
For all discriminant $D\equiv 0,1,4 \, [8], D>9$, which is not a square, we have
\begin{equation}\label{eq:Euler:char:W:D:4}
\chi(W_D(0^3))= \frac{-1}{6}\cdot\sum_{\substack{-\sqrt{D} < e < \sqrt{D}\\ e^2 \equiv D\; [8]}} m_D(e).	
\end{equation} 
\end{Proposition} 
The proof of Proposition~\ref{prop:Euler:char:W:D:0} is given in \textsection\ref{subsec:prf:euler:char:W:D:0}.

\subsection{Integration of the curvature form on Teichmüller curves}\label{subsec:int:Hodge:curv:on:Teich:curves}
We start by the following important observation.
\begin{Proposition}\label{prop:int:Hodge:curv:on:Teich:curves}
	Let $\cS$ be a connected component of $\cS_{1,0}\sqcup \cS_{2,0}^a\sqcup\cS_{0,2}$. Then we have
	\begin{equation}\label{eq:int:Hodge:curv:on:Teich:curves}
		\int_{\cS}\imath\vartheta = 2\pi\cdot c_1(\OO(-1))\cdot[\ol{\cS}]=-\pi\chi(\cS).
	\end{equation}
\end{Proposition}
\begin{proof}
	That $\int_{\cS}\imath\vartheta = 2\pi \cdot c_1(\OO(-1))\cdot[\ol{\cS}]$ is a consequence of the main result of \cite{Ng22} (see also \cite{Bai:GT}). Thus we will only give the proof of the equality
	\begin{equation}\label{eq:int:Hodge:curve:n:Euler:char}
		\int_{\cS}\imath\vartheta = -\pi\chi(\cS)
	\end{equation}
	To see this, we first remark that since $\cS$ is the projectivization of a closed $\GL^+(2,\R)$-orbit (that is $\cS$ is a Teichm\"uller curve), $\cS$ is isomorphic to a quotient $\Hbb/\Gamma$, where $\Gamma$ is Fuchsian group. Locally, a neighborhood of any point $\xx \in \cS$ can be identified with an open subset of $\Hbb=\{z\in \C, \, \Im(z) >0\}$ as follows: let $(C_\xx,[\omega_\xx])$ be the projectivized Abelian differential corresponding to $\xx$. Let $\gamma$ be a simple closed geodesic on a component of $C$ where $\omega_\xx$ does not vanish identically and $E$ the cylinder that contains $\gamma$.  Let $\sigma$ be a saddle connection contained in the closure of $E$ that crosses $\gamma$ once. We will call $\sigma$ a {\em crossing saddle connection} of  $E$.  For all $\xx' \simeq (C_{\xx'},[\omega_{\xx'}])$ in $\cS$ close to $\xx$, we can identify $\gamma$ with a closed geodesic and $\sigma$ with a saddle connection on $C_{\xx'}$.
	We can also normalize such that $\omega_{\xx'}(\gamma)=1$ for all $\xx'$ in a neighborhood of $\xx$.
	This means that the assignment $\xx' \mapsto \omega_{\xx'}$ is a holomorphic section of the tautological line bundle $\OO(-1)$.
	The mapping $\xx \mapsto z(\xx):=\omega_\xx(\sigma)$ then gives a local coordinate for $\cS$ in a neighborhood of $\xx$.
	With an appropriate orientation of $\sigma$, we have that $\Im(z) >0$, that is $z(\xx)\in \Hbb$.
	Note that if $(\gamma',\delta')$ is a is a different pair of (closed geodesic, crossing saddle connection) then the periods of  $\gamma'$ and $\delta'$ are related to those of $\gamma$ and $\delta$ by some matrix $A$ in $\GL^+(2,\R)$. Thus if $z'$ is the local coordinate associated to  $(\gamma',\delta')$, then $z'=A\cdot z$, where $A$ acts on $\Hbb$ by homography.
	
	Let us write $z(\xx)=x+\imath y$.
	Since the ratios of the widths and the ratios of the heights of parallel cylinders on Veech surfaces are constant, we get that
	$$
	\Aa(C_\xx,\omega_\xx)=R\cdot y,
	$$
	where $R$ is a positive real constant. Now, a direct calculation shows that
	\begin{align*}
		\imath\vartheta(\xx) = -\imath\partial\bar{\partial}\ln(||\omega_\xx||^2) = -\imath\partial\bar{\partial}\ln(R\cdot y) & = -\imath\partial\bar{\partial}\ln(\frac{\imath}{2}\cdot(\bar{z}-z)) =-\imath\cdot\frac{dz\wedge d\bar{z}}{(\bar{z}-z)^2} =\frac{dx\wedge dy}{2y^2}.
	\end{align*}
	Since the volume form $\nu$ of the hyperbolic metric on $\Hbb$ is given by $dx\wedge dy/y^2$, we get that
	\begin{align*}
		\int_{\cS}\imath\vartheta &= \frac{1}{2}\int_{\cS}\nu =\frac{-2\pi}{2}\cdot\chi(\cS)=-\pi\cdot\chi(\cS).
	\end{align*}
\end{proof}

\subsection{Forgetting the marked points}\label{subsec:triple:tori:forget:map}
Consider a point $\pp \sim (C,p_1,\dots,p_5,p'_5,\tau, [\xi]) \in \cS_{2,0}^a$. Recall from Theorem~\ref{th:bdry:eigen:form:H22} that $C$ has four irreducible components denoted by $C'_1,C'_2, C''_1,C''_2$, where  $C'_1, C''_1,C''_2$ are  (smooth) elliptic curves, $C'_2$ is  isomorphic to $\Pb^1$ and adjacent to all the other components. The differential $\xi$ vanishes identically on $C'_2$ and is nowhere vanishing on $C'_1, C''_1,C''_2$. Let $C_1$ denote the union of $C'_1,C''_1,C''_2$, and $\xi_1:=\xi_{\left|C_1\right.}$. Then $(C_1,\xi_1)$ is a triple of tori in $\Omega E_{D}(0^3)$ (see Lemma~\ref{lm:gp:I:limit:diff}).
The correspondence $\pp \mapsto (C_1,[\xi_1])$ defines a  map $\Psi_D: \cS_{2,0}^a \to \Pb\Omega E_{D}(0^3)=W_D(0^3)$.
\begin{Lemma}\label{lm:S20:a:cover:triple:tori}
The map $\Psi_D$ is a covering of degree $4!$.
\end{Lemma}
\begin{proof}
We first show that the projectivized Abelian differential $(C,[\xi])$ is uniquely determined by $(C_1,[\xi_1])$.
To see this recall that by assumption, $C'_2$ contains $p_5,p'_5$ and one of the points $\{p_1,\dots,p_4\}$. Let us assume that $p_4\in C'_2$.
Let $r_0$ be the node between $C'_2$ and $C'_1$, and $r_i, \; i=1,2$, the node between $C'_2$ and $C''_i$.
By definition, the Prym involution $\tau$ fixes $r_0, p_4$, and permutes $p_5$ and  $p'_5$ (resp. $r_1$ and $r_2$).
We can identify $C'_2$ with $\Pb^1$ such that the restriction of $\tau$ is given by $z\mapsto -z$.
We then have $(C'_2,r_0,p_4,p_5,p'_5,r_1,r_2) \simeq (\Pb^1,0,\infty,1,-1,b,-b)$, where $b\in \C\setminus\{0,\pm1\}$.
By Theorem~\ref{th:twisted:diff}, there exists a meromorphic Abelian differential $\eta$ on $C'_2$ such that 
$$
\div(\eta)=2p_5+ 2p'_5 -2r_0 -2 r_1 -2r_2
$$
and  residues  of $\eta$ at  the poles $r_0,r_1,r_2$ are all zero. Up to a scalar, there is a unique Abelian differential on $\Pb^1$ with the prescribed orders at the marked points, namely $\eta=\frac{(z^2-1)^2dz}{z^2(z^2-b^2)^2}$. The condition on the residues of $\eta$ at the poles implies that $b^2=-3$. Thus, we have
$$
\eta=\frac{(z^2-1)^2dz}{z^2(z^2+3)^2}.
$$
In particular, the pointed curve $(C'_2,r_0,p_4,p_5,p'_5,r_1,r_2)$ is uniquely determined and independent  of $C_1$. This proves our claim.

Since $(C,[\xi])$ is uniquely determined by $(C_1,[\xi_1])$, $\Psi_D$ is a covering onto its image.
Let $(X,\omega):=\{(X_j,x_j,\omega_j), \; j=0,1,2\}$ be  a triple of tori in $\Omega E_{D}(0^3)$.
Denote by $(X,[\omega])$ the corresponding point in $\Pb\Omega E_{D}(0^3)$.
We will show that $\#\Psi^{-1}_D((X,[\omega]))=4!$.

Let $C$ be the stable curve obtained as the union of $X_0,X_1,X_2$ and a copy of $\Pb^1$, denoted by $C_0$, where for all $j=0,1,2$, $x_j$ is identified with a point in $C_0$. We can assume that $x_0$ is identified with $0$, $x_1$ with $\sqrt{3}\imath$ and $x_2$ with $-\sqrt{3}\imath$.
Let $\xi\in H^0(C,\omega_C)$ be the differential on $C$ which vanishes identically on $C_0$ and equals $\omega_j$ on $X_j$.
Since $(X_1,\omega_1)$ and $(X_2,\omega_2)$ are isomorphic,  there is an involution $\tau$ of $C$ that exchanges $X_1$ and $X_j$ and leaves $X_0$ and $C_0$ invariant. 
By construction, $\tau$ has four regular fixed point in $C$, three of them are contained in $X_0$ and the forth one is contained in $C_0$.
Let $p_1,\dots,p_4$ denote the regular fixed points of $\tau$, and  $p_5$ and $p'_5$ the points in $C_0$ that correspond to $1$ and $-1$ respectively. Then $(C,p_1,\dots,p_5,p'_5,\tau,\xi)$ is an element of $\Omega'\ol{\cB}_{4,1}$.
We claim that $(C,p_1,\dots,p_5,p'_5,\tau, \xi)\in \Omega\ol{\cX}_D$.
To see this, let $\eta$ be the meromorphic differential on $C_0\simeq \Pb^1$ which is equal to $\frac{(z^2-1)^2dz}{(z^2+3)^2z^2}$.
Given $t\in \C^*$, $|t|$ small enough, the smoothing construction by plumbing simultaneously the three nodes of $C$ with parameter $t$ yields a smooth genus three curve $C_t$ together with a holomorphic Abelian differentials  $\xi_t$ such that 
\begin{itemize}
	\item[$\bullet$] the restriction of $\xi_t$ to the complement of a neighborhood of $x_j$ in $X_j$ is equal to $\omega_j,$ for $j=0,1,2$,
	
	\item[$\bullet$] the restriction of $\xi_t$ to the complement of a neighborhood of $\{0,\pm\imath \sqrt{3}\}$ in $C_0$ is equal to $t\eta$.
\end{itemize}
In particular, we have $(C_t,\xi_t) \in \Omega\cM_3(2,2)$. The involution $\tau$ of $C$ induces an involution on $C_t$ with four fixed points, we denote this involution again  by $\tau$. By construction, we have $\tau^*\xi_t=-\xi_t$.
Since $(X,\omega)\in \Omega E_{D}(0^3)$, it is straightforward to check that $(C_t,\xi_t) \in \Omega E_D(2,2)^\odd$.
The numbering of the fixed points of $\tau$ on $C$ induces naturally a numbering of the fixed points of $\tau$ on $C_t$. Thus we obtain a map $\varphi: \Delta_\eps \to \Omega\ol{\cX}_D$, for some $\eps>0$ small, such that $\varphi(0)=(C,p_1,\dots,p_5,p'_5,\xi)$ and $\varphi(\Delta^*_\eps)\subset \Omega\XD$.
It follows that $\pp:=(C,p_1,\dots,p_5,p'_5,\tau,[\xi]) \in \ol{\cX}_D$. Clearly we have $\pp \in \cS_{0,2}^a$, and $\Psi_D(\pp)=(X,[\omega])$. We can then conclude that $\Psi_D(\cS_{2,0}^a)=\Pb\Omega E_D(0^3)$.

We have seen that if we forget the numbering of fixed points of $\tau$, then the differential $(C,[\xi])$ is uniquely determined by $(X,[\omega])$. Thus $\Psi_D^{-1}(\{(X,[\omega])\})$ consists of the same projectivized differential $(C,[\xi])$ with different numberings of the fixed points of $\tau$. Since we are free to choose the numbering, we have $\#\Psi_D^{-1}(\{(X,[\omega])\})=4!$.
\end{proof}

\begin{Remark}\label{rk:triple:tori:plumb}
The map $\varphi$ in the proof of Lemma~\ref{lm:S20:a:cover:triple:tori} can also be defined using flat metric argument as follows:
given $t\in \C^*$ with $|t|$ small enough, on  each flat torus $(X_j,\omega_j)$ there is a unique geodesic segment $s_j$ centered at the marked point $x_j$ with period $t$. Slit open $X_j$ along $s_j$, we obtain three flat surfaces whose boundary consists of a pair of geodesic segments with period $t$. Gluing those surfaces together cyclically by identifying a segment in the boundary of $X_j$ with a segment in the boundary of $X_{j+1}$ (with the convention $X_3=X_0$), we obtain a surface $(Z_t,\eta_t)$ in $\Omega E_D(2,2)^{\rm odd}$. Remark that $(Z_t,\eta_t)$ has three homologous saddle connections with period $t$. This yields a holomorphic map $\phi: \Delta_\eps \to \Omega \ol{\cX}_D$ with the same properties as $\varphi$ (see \cite[\textsection 5]{LN:finite} for more details).
\end{Remark}

\subsection{Components of $\Omega E_D(0^3)$}\label{subsec:triples:tori:eigenform:comp}
Let $\{(X_j,x_j,\omega_j), \; j=0,1,2\}$ be a triple of flat tori. Then for each $j\in \{0,1,2\}$, there is a lattice $\Lambda_j$ in $\C$ such that $(X_j,\omega_j,x_j)\simeq (\C/\Lambda_j, dz, \bar{0})$, where $\bar{0}$ is the projection of $0\in \C$ in $\C/\Lambda_j$.

Let $\Pcal_D(0^3)$ denote the set of quadruples of integers $(a,b,d,e)$ satisfying the following conditions
$$
(\Pcal_{D}(0^3)) \quad \left\{
\begin{array}{l}
	a>0,  \; d > 0, \;  0 \leq b < a,\\
	D=e^2+8ad,\\
	\gcd(a,b,d,e)=1.
\end{array}
\right.
$$
Elements of $\Pcal_D(0^3)$ will  be called {\em prototypes for triple of tori}.
For every prototype $\frakp=(a,b,d,e)\in \Pcal_{D}(0^3)$, define $\lambda(\frakp):=\frac{e+\sqrt{D}}{2}$. We will call the {\em prototypical triple tori} associated to $\frakp$ the Abelian differential $(X,\omega)=\{(X_j,x_j,\omega_j), j=0,1,2\}$ defined as follows
\begin{itemize}
	\item[$\bullet$] $(X_0,\omega_0) \simeq (\C/(\lambda\cdot\Z +\imath\lambda\cdot\Z), dz)$,
	
	\item[$\bullet$] $(X_1,\omega_1) \simeq (X_2,\omega_2)\simeq (\C/(a\cdot\Z+(b+\imath d)\cdot\Z), dz)$.
\end{itemize}
The following result follows from the arguments of Proposition~\ref{prop:eigen:form:per:eq} (see also \cite[Prop. 8.2]{LN:finite} and \cite[App.]{LN:components}).
\begin{Proposition}\label{prop:triple:tori:eigen:form}
	All prototypical triples of tori are contained in $\Omega E_D(0^3)$. 
	A triple of flat  tori $\{(X_j,x_j,\omega_j), \, j=0,1,2\}$ belongs to $\Omega E_{D}(0^3)$ if and only if there is a matrix $A\in \GL^+(2,\R)$ such that $A\cdot(X,\omega)$ is a prototypical triple of tori.
\end{Proposition}
\begin{Remark}\label{rk:triple:tori:prot:no:unique}\hfill
The matrix $A$ and the prototypical triple of tori in the conclusion of Proposition~\ref{prop:triple:tori:eigen:form} are by no means unique. 
\end{Remark}

Given a lattice $\Lambda \subset \C$, for any sublattice $\Lambda' \subset \Lambda$ we define $\rho(\Lambda,\Lambda')$ to be the largest positive integer $r$ such that $\frac{1}{r}\cdot\Lambda' \subset \Lambda$.
The following lemma provides us with a characterization of the prototypical triples of tori contained in the same $\GL^+(2,\R)$-orbit. 

\begin{Lemma}\label{lm:triples:tori:eigenform:char}
Let $(X,\omega)=\{(X_j,x_j,\omega_j), \; j=0,1,2\}$ be a triple of tori in $\Omega E_D(0^3)$. Let $\Lambda_j, \; j=0,1,2$, be the lattices in $\C$ such that $(X_j,\omega_j) \simeq (\C/\Lambda_j, dz)$. Then there exists a unique integer $e =:\ee(X,\omega)$ such that for $\lambda=\frac{e+\sqrt{D}}{2}$ we have
\begin{itemize}
  \item[(i)] $\Lambda'_1:=\lambda\cdot\Lambda_1 \subset \Lambda_0$.

  \item[(ii)] Let $K:=[\Lambda_0:\Lambda'_1]$ and $r:=\rho(\Lambda_0,\Lambda'_1)$. Then $D=e^2+8K$ and $\gcd(r,e)=1$.
\end{itemize}
\end{Lemma}
\begin{proof}
From Proposition~\ref{prop:triple:tori:eigen:form} we know that the $\GL^+(2,\R)$-orbit of $(X,\omega)$ contains the prototypical triple of tori $(Y,\eta)=\{(Y_j,y_j,\eta_j), \, j=0,1,2\}$ associated with a prototype $\frakp=(a,b,d,e) \in \Pcal_D(0^3)$.
We claim that $e$ is uniquely determined by $(X,\omega)$. Indeed, with $\lambda:=\lambda(\frakp)$ we  have
\begin{align*}
\frac{\Aa(X_0,\omega_0)}{\Aa(X,\omega)}= \frac{\Aa(Y_0,\eta_0)}{\Aa(Y,\eta)} =\frac{\lambda^2}{\lambda^2+2ad} &= \frac{\lambda^2}{2\lambda^2-e\lambda} = \frac{e+\sqrt{D}}{2\sqrt{D}}
\end{align*}
which implies that $e$ is uniquely determined.

Since the properties (i) and (ii) are invariant under the simultaneous action of $\GL^+(2,\R)$ the pair $(\Lambda_0,\Lambda_1)$, we can suppose from now on that $(X,\omega)$ is the triple of tori associated to $\frakp$. In this case, we have $r=\gcd(a,b,d)$ and $K=\det\left(\begin{smallmatrix} a & b \\ 0 & d \end{smallmatrix} \right) =ad$. Since $(a,b,d,e)\in \Pcal_D(0^3)$, we have $D=e^2+8ad$ and $\gcd(r,e)=\gcd(a,b,d,e)=1$. The lemma is then proved.
\end{proof}

The following lemma was known to McMullen (cf. \cite[\textsection 2]{McM:spin}). We will provide here an alternative proof of this fact using Lemma~\ref{lm:triples:tori:eigenform:char}.
\begin{Lemma}\label{lm:triple:tori:prot:same:orbit}
	Let $(X,\omega)=\{(X_j,x_j,\omega_j), \; j=0,1,2\}$ and $(X',\omega')=\{(X'_j,x'_j,\omega'_j), \; j=0,1,2\}$ be the prototypical triples of tori associated respectively to the elements $\frakp=(a,b,d,e)$ and $\frakp'=(a',b',d',e')$ of $\Pcal_{D}(0^3)$.  Then $(X,\omega)$ and $(X',\omega')$ belong to the same $\GL^+(2,\R)$-orbit if and only if $e=e'$ and $\gcd(a,b,d)=\gcd(a',b',d')$.
\end{Lemma}
\begin{proof}
	Assume first that $(X,\omega)$ and $(X',\omega')$ belong to the same $\GL^+(2,\R)$-orbit. Then it follows from Lemma~\ref{lm:triples:tori:eigenform:char} that we must have $e=e'$ which implies that $\lambda=\lambda'$.
	Since $(X_0,\omega_0)$ and $(X'_0,\omega'_0)$ are both isomorphic to $(\C/\lambda\cdot(\Z+\imath \Z), dz)$, we must have $(X',\omega')=A\cdot (X,\omega)$ for some $A\in \SL(2,\Z)$. In particular, we have
	$$
	\left(\begin{array}{cc}
		a' & b' \\
		0 &  d'
	\end{array}
	\right)= A\cdot \left(\begin{array}{cc}
		a & b \\
		0 & d
	\end{array}
	\right)
	$$
	which implies that $\gcd(a,b,d)=\gcd(a',b',d')$.
	
	Conversely, assume that we have $e=e'$ and $\gcd(a,b,d)= \gcd(a',b',d')=\ell$. Let
	$$
	(a_1,b_1,d_1):=\frac{1}{\ell}(a,b,d) \quad \text{ and} \quad (a'_1,b'_1,d'_1):=\frac{1}{\ell}(a',b',d').
	$$
	Note that we have
	$$
	D=e^2+8ad=e^2+8\ell^2a_1d_1=e'{}^2+8\ell^2a'_1d'_1,
	$$
	which implies that  $a_1d_1=a'_1d'_1$ (since $e=e'$). Therefore, the lattices $\Lambda:=a_1\cdot\Z+(b_1+\imath d_1)\cdot \Z$ and $\Lambda':=a'_1\cdot\Z+(b'_1+\imath d'_1)\cdot \Z$ are both primitive and have same index in $\Z+\imath\cdot\Z$.
	It is a well known fact that there is a matrix $A\in \SL(2,\Z)$ such that $A(\Lambda) =\Lambda'$.
	As a consequence, we get that $(X',\omega')=A\cdot(X,\omega)$.
\end{proof}

Let $\Pcal^*_{D}(0^3)$ denote the set of triples of integers $(e,\ell,m)$ satisfying
$$
(\Pcal^*_D(0^3)): \quad \ell > 0, \, m >0, \, D=e^2+8\ell^2m, \; \gcd(e,\ell)=1.
$$
If $(e,\ell,m)\in \Pcal^*_{D}(0^3)$ then $(\ell,0,\ell m,e) \in \Pcal_{D}(0^3)$.
We denote by $\Omega E_{D,(e,\ell,m)}(0^3)$ the $\GL^+(2,\R)$-orbit of the prototypical triple of tori associated to the prototype $(\ell,0,\ell m,e)$.
As an immediate consequence of Lemma~\ref{lm:triple:tori:prot:same:orbit}, we get the following
\begin{Corollary}\label{cor:triple:tori:components}
	We have
	$$
	\Omega E_{D}(0^3) =\bigsqcup_{(e,\ell,m)\in \Pcal^*_D(0)} \Omega E_{D,(e,\ell,m)}(0^3).
	$$
\end{Corollary}
\begin{proof}
	For all $(a,b,d,e) \in \Pcal_{D}(0^3)$. Let $\ell:=\gcd(a,b,d)$ and $m:=ad/\ell^2$. We then have $\gcd(e,\ell)=\gcd(a,b,d,e)=1$, and
	$D=e^2+8ad=e^2+8\ell^2m$, which means that $(e,\ell,m) \in \Pcal^*_D(0^3)$.
	It follows from Lemma~\ref{lm:triple:tori:prot:same:orbit} that every prototypical triple of tori is contained in some $\Omega E_{D,(e,\ell,m)}(0^3)$, and if $(e',\ell',m')$ and $(e,\ell,m)$ are different then  $\Omega E_{D,(e',\ell',m')}(0^3)$ and $\Omega E_{D,(e,\ell,m)}(0^3)$ are disjoint.
	This proves the corollary.
\end{proof}

\subsection{Projection onto $\cM_{1,1}$} \label{subsec:triple:tori:proj:1}
Let $\pi_0: \Omega E_{D}(0^3) \to \Omega\cM_{1,1}$ denote the map that associates to a triple $\{(X_j,x_j,\omega_j), \; j=0,1,2\} \in \Omega E_D(0^3)$ the element $(X_0,x_0,\omega_0) \in \Omega\cM_{1,1}$. 
Let $e$ be an integer such that $e^2 < D$ and $e^2 \equiv D \, [8]$. Denote by $\Omega E_{D,e}(0^3)$ the set of all $(X,\omega) \in \Omega E_D(0^3)$ such that $\ee(X,\omega)=e$.
Let $\pi_0^{(e)}$ be the restriction of $\pi_0$ to $\Omega E_{D,e}(0^3)$.
The maps $\pi_0, \pi_0^{(e)}$ descend to  maps from $\Pb\Omega E_{D}(0^3)$ and $\Pb\Omega E_{D,e}(0^3)$ onto $\cM_{1,1} \simeq \Hbb/\SL(2,\Z)$ that we abusively denote again by $\pi_0, \pi_0^{(e)}$ respectively.
Let us define
$$
\Pcal_{D,e}(0^3):=\{(a,b,d)\in \Z^3, \; (a,b,d,e) \in \Pcal_D(0^3)\}. 
$$

\begin{Lemma}\label{lm:m:D:e:equals:deg:proj:1:e}
We have
\begin{equation}\label{eq:m:D:e:equals:deg:proj:1:e}
\deg \pi_0^{(e)}= \#\Pcal_{D,e}(0^3).
\end{equation}
\end{Lemma}
\begin{proof}
Since $\lambda=\frac{e+\sqrt{D}}{2}$ is fixed, by Lemma~\ref{lm:triples:tori:eigenform:char} we can identify $\Omega E_{D,e}(0^3)$ with the space of pairs $(\Lambda_0,\Lambda_1)$ where $\Lambda_0$ is a lattice in $\C$, and $\Lambda_1$ is a sublattice of $\Lambda_0$ which satisfies
\begin{itemize}
\item[(i)] $[\Lambda_0:\Lambda_1]=\frac{D-e^2}{8}$,

\item[(ii)] $\gcd(\rho(\Lambda_0,\Lambda_1),e)=1$.
\end{itemize}
Using this identification, the map $\pi_0$ is simply given by $\pi_0: (\Lambda_0,\Lambda_1) \mapsto \Lambda_0$.
The preimage of $\Lambda_0$ by  $\pi_0^{(e)}$ is the set of sublattices $\Lambda_1 \subset \Lambda_0$ satisfying (i)  and (ii).
We can suppose that $\Lambda_0=\Z^2$. For any $\Lambda_1$, there exists a unique positive integer $a$ such that $a\cdot\Z \times\{0\}=\Lambda_1\cap\Z\times\{0\}$. There also exists a unique vector $(b,d) \in \Lambda_1$ such that $d>0$, $0\leq b < a-1$, and for all $(x,y) \in \Lambda_1\setminus\Z\times\{0\}$, we have
$$
ad=\det\left(
\begin{array}{cc}
  a & b \\
  0 & d
\end{array}
\right) \leq \left|\det \left(
\begin{array}{cc}
  a & x \\
  0 & y
\end{array}
\right)\right|.
$$
It is elementary to show that $(a,0)$ and $(b,d)$ form a basis of $\Lambda_1$. Condition (i) then implies that $ad=(D-e^2)/8$. Since $\rho(\Lambda_0,\Lambda_1)=\gcd(a,b,d)$, condition (ii) implies that $\gcd(a,b,d,e)=1$. We can then conclude that $(a,b,d,e) \in \Pcal_D(0^3)$. We thus have shown that there is a bijection between the preimage of $\Lambda_0$ by $\pi_0^{(e)}$ and the set $\Pcal_{D,e}(0^3)$ from which the lemma follows.
\end{proof}

We will say that a discriminant $D$ is $(1,2)$-primitive if $D \equiv 0,1,4 \; [8]$, and there does not exist $f \in \Z_{>1}$ such that $D = f^2 D'$ with $D' \equiv 0,1,4 \; [8]$. Recall that for all $n\in \N$,
$$
\sigma_1(n)= \sum_{d \, | \, n, d\geq 1} d.
$$
\begin{Corollary}\label{cor:triple:tor:deg:proj:e:D:prim}
If $D$ is $(1,2)$-primitive then for all $e\in \Z$ such that $e^2 < D, \; e^2 \equiv D \; [8]$ we have
\begin{equation}\label{eq:triple:tor:deg:proj:e:D:prim}
\deg \pi_0^{(e)} =\sigma_1(\frac{D-e^2}{8}).
\end{equation}
\end{Corollary}
\begin{proof}
Given $e$ and $D$, if $(a,b,d,e) \in \Pcal_D(0)$ then we have $ad=(D-e^2)/8$. Thus $a \, | \, (D-e^2)/8$ and $d$ is uniquely determined by $a$. We claim that $\gcd(a,d,e)=1$. Indeed, let $k=\gcd(a,d,e)$. Assume that $k>1$. Let $(a_1,d_1,e_1):=(a/k,d/k,e/k)$. We then have
$$
D=e^2+8ad=k^2(e_1^2+8a_1d_1)
$$
which contradicts the hypothesis that $D$ is $(1,2)$-primitive. Therefore we must have $\gcd(a,d,e)=1$. As a consequence, for  all $b \in \{0,1,\dots,a-1\}$, we have $(a,b,d,e) \in \Pcal_D(0^3)$. It follows from Lemma~\ref{lm:m:D:e:equals:deg:proj:1:e} that we have
$$
\deg \pi_0^{(e)} = \#\Pcal_{D,e}(0^3)=\sum_{a \, |  \, (D-e^2)/8} a =\sigma_1(\frac{D-e^2}{8}),
$$
which proves the corollary.
\end{proof}
Our goal now is to provide a closed formula to compute $\deg \pi_0^{(e)}$ in the general case. 
For all $(e,\ell,m) \in \Pcal^*_{D}(0^3)$ denote by $\pi_0^{(e,\ell,m)}: \Omega E_{D,(e,\ell,m)}(0^3) \to \Omega\cM_{1,1}$   the restriction of $\pi_0$ to $\Omega E_{D,(e,\ell,m)}(0^3)$. We will also denote by $\pi_0^{(e,\ell,m)}$ the induced projection from $\Pb\Omega_{D,(e,\ell,m)}(0^3)$ onto $\cM_{1,1}$.
It follows from the argument of \cite[Th. 2.1]{McM:spin} that $\Pb\Omega E_{D,(e,\ell,m)}(0^3)$ is isomorphic to $\Hbb/\Gamma_0(m)$, where
$$
\Gamma_0(m)=\{\left(\begin{smallmatrix} a & b \\ c & d \end{smallmatrix} \right)\in \SL(2,\Z), \; c\equiv 0 \; [m] \}.
$$
It is a well known fact that
$$
[\SL(2,\Z):\Gamma_0(m)]=m\prod_{\substack{ p \, | \, m \\ p \, {\rm prime}}}\left(1+\frac{1}{p}\right)
$$
(see for instance \cite[\textsection 4]{Miyake}). We thus have the following
\begin{Lemma}\label{lm:deg:triple:tor:proj:1:e:l:m}
	We have
	$$
	\deg \pi_0^{(e,\ell,m)}= m\prod_{\substack{ p \, | \, m \\ p \, {\rm prime}}}\left(1+\frac{1}{p}\right) =: c(m).
    $$
\end{Lemma}

Corollary~\ref{cor:triple:tori:components} then implies
\begin{Corollary}\label{cor:triple:tori:proj:1:deg:e}
For all $e \in \Z$ such that $e^2 < D$, $e^2\equiv D \; [8]$, let us write $(D-e^2)/8= f^2m$, where $f,m\in \N$, $m$  square-free.
We then have
\begin{equation}\label{eq:trip:tori:proj:1:deg:e}
\deg \pi_0^{(e)}=\sum_{\substack{r \, | \, f\\ \gcd(r,e)=1}}c(\frac{D-e^2}{8r^2}) = m_D(e).
\end{equation}
\end{Corollary}

\subsection{Proof of Proposition~\ref{prop:Euler:char:W:D:0}}\label{subsec:prf:euler:char:W:D:0}
\begin{proof}
From Lemma~\ref{lm:triple:tori:prot:same:orbit} and Corollary~\ref{cor:triple:tori:proj:1:deg:e}, we have	
\begin{align*}
\chi(W_D(0^3)) & =\chi(\Pb\Omega E_D(0^3)) = \sum_{\substack{-\sqrt{D} < e < \sqrt{D} \\ e^2 \equiv D \; [8]} }\chi(\Pb \Omega E_{D,e}(0^3)) \\
& = -\frac{1}{6}\cdot \sum_{\substack{-\sqrt{D} < e < \sqrt{D} \\ e^2 \equiv D \; [8]} }\deg \pi_0^{(e)} = \frac{-1}{6}\cdot \sum_{\substack{-\sqrt{D} < e < \sqrt{D} \\ e^2 \equiv D \; [8]} } m_D(e) \quad \hbox{(since $\chi(\cM_{1,1})=-1/6$)}.
\end{align*}
\end{proof}

\subsection{Integration of $\Theta$ over the spaces of triples of tori eigenforms}\label{subsec:int:Theta:on:T:2:0:a:1}
Let $\cS_{2,0}^{a}(e)$ denote the preimages of $\Pb\Omega E_{D,e}(0^3)$ in $\cS_{2,0}^a  \subset \partial\hat{\cX}_D$.
Let $\cT_{2,0}^{a,1}(e)$ be the  preimage of  $\cS_{2,0}^a(e)$ in $\cT_{2,0}^{a,1}$.
Note that $\ol{\cT}_{2,0}^{a,1}(e)$ is a divisor in $\tilde{\cC}_D$.

\begin{Proposition}\label{prop:int:Theta:on:T:a:1:2:0:e}
We have
\begin{equation}\label{eq:int:Theta:T:a:1:2:0:e}
\langle [\Theta],[\ol{\cT}_{2,0}^{a,1}(e)]\rangle = -2\pi\cdot4!\cdot\frac{e+\sqrt{D}}{\sqrt{D}}\cdot \chi(\Pb\Omega E_{D,e}(0^3)) = 8\pi\cdot\frac{e+\sqrt{D}}{\sqrt{D}}\cdot \deg\pi_0^{(e)}.
\end{equation}
\end{Proposition}
\begin{proof}
Consider a point $\pp \simeq (C_\pp,p_1,\dots,p_5,p'_5,\tau, [\xi_\pp])$ in  $\cS_{2,0}^a$.  Let $(X,[\omega]) =\{(X_j,x_j,[\omega_j]), \; j=0,1,2\}\in \Pb\Omega E_{D}(0)$ be the image of $\pp$ by $\Psi_D$. By definition, we have
\begin{itemize}
\item[$\bullet$] $C_\pp$ is the stable curve formed by $X_0,X_1,X_2$ and an additional component  $C_0\simeq \Pb^1$ where each $x_j$ is a node between $X_j$ and $C_0$,
\item[$\bullet$] $\xi_\pp$ is the Abelian differential on $C_\pp$ that vanishes identically on $C_0$ and equals $\omega_j$ on $X_j$.
%
\end{itemize}
The fiber $\tilde{\pi}^{-1}(\{\pp\}) \subset \tilde{\cC}_D$ can be identified with the curve $C_\pp$, and its intersection with the divisor $\cT_{2,0}^{a,1}$ is precisely  the elliptic curve $X_0$ considered as an irreducible component of $C_\pp$.
Since $\Theta$ is smooth on $\cT_{2,0}^{a,1}$, we have
\begin{align*}
\langle [\Theta],[\ol{\cT}_{2,0}^{a,1}(e)]\rangle &= \int_{\cT_{2,0}^{a,1}(e)}\Theta =\int_{\cS_{2,0}^a(e)}\left(\frac{\imath}{2}\int_{C_\pp\cap\cT_{2,0}^{a,1}}\frac{dP\wedge d\bar{P}}{||\xi_\pp||^2} \right)(\imath\vartheta(\pp)),
\end{align*}
where $P$ is a function whose restriction to $X_0=C_\pp\cap\cT_{2,0}^{a,1}$ is given by $x \mapsto \int_x^{\tau(x)}\omega_0$. One readily checks that $\left(dP\wedge d\bar{P}\right)_{\left|X_0\right.}=4\omega_0\wedge\ol{\omega}_0$.
Hence
\begin{align*}
\frac{\imath}{2}\int_{X_0}\frac{dP\wedge d\bar{P}}{||\xi_\pp||^2} & = \frac{4}{\Aa(X,\omega)}\cdot\frac{\imath}{2}\cdot\int_{X_0}\omega_0\wedge\ol{\omega}_0 = 4\cdot\frac{\Aa(X_0,\omega_0)}{\Aa(X,\omega)}\\
& =4\cdot \frac{e+\sqrt{D}}{2\sqrt{D}}=\frac{2(e+\sqrt{D})}{\sqrt{D}}.
\end{align*}
It follows
\begin{align*}
\langle [\Theta],[\ol{\cT}_{2,0}^{a,1}(e)]\rangle & = \int_{\cS_{2,0}^a(e)}\left(\frac{\imath}{2}\int_{C_\pp\cap \cT_{2,0}^{a,1}} \frac{dP\wedge d\bar{P}}{||\xi_\pp||^2} \right)(\imath\vartheta(\pp))=\frac{2(e+\sqrt{D})}{\sqrt{D}}\int_{\cS_{2,0}^a(e)}\imath\vartheta.
\end{align*}
By Proposition~\ref{prop:int:Hodge:curv:on:Teich:curves}, we have that
$$
\int_{\cS_{2,0}^a(e)}\imath\vartheta = -\pi\cdot\chi(\cS_{2,0}^a(e)).
$$
Therefore,
\begin{align*}
\langle [\Theta],[\ol{\cT}_{2,0}^{a,1}(e)]\rangle &=-2\pi\cdot \frac{e+\sqrt{D}}{\sqrt{D}}\cdot \chi(\cS_{2,0}^a(e))\\
&= -2\pi\cdot\frac{e+\sqrt{D}}{\sqrt{D}}\cdot 4!\cdot \chi(\Pb\Omega E_{D,e}(0^3)) \quad \hbox{(by Lemma~\ref{lm:S20:a:cover:triple:tori} )} \\
&= -2\pi\cdot 4!\cdot\frac{e+\sqrt{D}}{\sqrt{D}}\cdot\deg \pi_0^{(e)}\cdot \chi(\cM_{1,1})\\
&= 8\pi\cdot\frac{e+\sqrt{D}}{\sqrt{D}}\cdot \deg\pi_0^{(e)} \quad  \hbox{(since $\chi(\cM_{1,1})=\chi(\Hbb/\SL(2,\Z))=-1/6$)}.
\end{align*}
and the proposition is proved.
\end{proof}

\subsection*{Proof of Theorem~\ref{th:int:Theta:T:2:0:a:1}}
\begin{proof}
It follows from \eqref{eq:trip:tori:proj:1:deg:e} that $\deg\pi_0^{(e)}=\deg \pi_0^{(-e)}$, for all integers $e$ such that $-\sqrt{D} < e < \sqrt{D}$ and $e^2 \equiv D \, [8]$. Therefore, $\chi(\Pb\Omega E_{D,e}(0^3)) = \chi(\Pb\Omega E_{D,-e}(0^3))$. Proposition~\ref{prop:int:Theta:on:T:a:1:2:0:e} then  implies that

\begin{align*}
\langle [\Theta],[\ol{\cT}_{2,0}^{a,1}]\rangle & = -2\pi\cdot 4!\cdot \sum_{\substack{-\sqrt{D} < e < \sqrt{D}\\ e^2 \equiv D \, [8] }} \frac{e+\sqrt{D}}{\sqrt{D}}\cdot\chi(\Pb\Omega E_{D,e}(0^3)) \\
& = -2\pi\cdot 4!\cdot \sum_{\substack{-\sqrt{D} < e < \sqrt{D}\\ e^2 \equiv D \, [8] }}\chi(\Pb\Omega E_{D,e}(0^3))\\
& = -48\pi\cdot\chi(W_{D}(0^3)).
\end{align*}
The theorem is then proved.
\end{proof}

\subsection{Case $D\equiv 1 \; [8]$}\label{subsec:triple:tori:D:1:mod:8}
In the case $D\equiv 1 \, [8]$, it was shown  in \cite{LN:components} that $\Omega E_D(2,2)^\odd$ has two components that we will denote by $\Pb \Omega E_{D+}(2,2)^\odd$ and $\Pb \Omega E_{D-}(2,2)^\odd$.  By convention the closure of $\Pb\Omega E_{D+}(2,2)^\odd$ (resp. of $\Pb\Omega E_{D-}(2,2)^\odd$) contains the triple of tori associated with the  prototype $(1,0,(D-1)/8,1)$ (resp. with the prototype $(1,0,(D-1)/8,-1)$) in $\Pcal_{D}(0^3)$ (cf. \textsection~\ref{subsec:triples:tori:eigenform:comp}).
Let $\hat{\cX}_{D\pm}$ be the closures of the preimages of $\Pb\Omega E_{D\pm}(2,2)^\odd$ in $\hat{\cX}_D$ respectively. Denote by  $\cS_{2,0}^{a+}$ (resp. $\cS_{2,0}^{a-}$) the intersection of $\cS_{2,0}^a$ with $\hat{\cX}_{D+}$ (resp. with $\hat{\cX}_{D-}$).
We start by
\begin{Lemma}\label{lm:triple:tori:same:comp:D:1:8}
Let $D \equiv 1 \; [8], \; D>9$, be a non-square discriminant. Let $(X,\omega)$ and $(X',\omega')$ be two triples of tori  with prototypes $\frakp:=(a,b,d,e)$ and $\frakp':=(a',b',d',e')$ in $\Pcal_{D}(0^3)$ respectively.
If $(X,\omega)$ and $(X',\omega')$ are contained  in the closure of same component of $\Omega E_D(2,2)^\odd$, then   $e'\equiv e \, [4]$.
\end{Lemma}
\begin{proof}
Recall that by definition,
$(X_0,\omega_0)\simeq (\C/\Lambda_0, dz), \quad (X_1,\omega_1) \simeq (X_2,\omega_2) \simeq (\C/\Lambda, dz)$,
where $\Lambda_0=\lambda\cdot\Z+\imath\lambda\cdot\Z$, and $\Lambda=a\cdot\Z+ (b+\imath d)\cdot\Z$.
Let $\alpha_0$ and $\beta_0$ denote the elements of $H_1(X_0,\Z)$ that  correspond to  $\lambda$ and $\imath\lambda$ (as elements of  $\Lambda_0$) respectively.
For $j=1,2$, let $\alpha_j$ (resp. $\beta_j$) denote the element of $H_1(X_j,\Z)$ corresponding to $a\in \Lambda$ (resp. to $b+\imath d \in \Lambda$). 
Let $\alpha:=\alpha_1+\alpha_2, \; \beta:=\beta_1+\beta_2$.
Since the Prym involution $\tau$ satisfies $\tau_*\alpha_0=-\alpha_0, \; \tau_*\beta_0=-\beta_0$, and $\tau_*\alpha_1=-\alpha_2, \tau_*\beta_1=-\beta_2$, it follows that $\cB:=(\alpha_0,\beta_0,\alpha,\beta)$ is a symplectic basis of $H_1(X,\Z)^-$.
Let $T$ be the element of $\End(\Prym(X))$ which is given in the basis $\cB$ by the matrix
$$
T=\left(
\begin{array}{rrcc}
e & 0 & 2a & 2b\\
0 & e & 0 & 2d\\
d & -b & 0 & 0 \\
0  & a & 0 & 0
\end{array}
\right).
$$
Then $T$ is self-adjoint with respect to the intersection form on $H_1(X,\Z)^-$ and satisfies $\Z[T] \simeq \Ocal_D$ and $T^*=\lambda(\frakp)\cdot\omega$.
We construct the symplectic basis $\cB'=\{\alpha'_0,\beta'_0,\alpha',\beta'\}$ of $H_1(X',\Z)^-$ and $T'\in \End(\Prym(X'))$ in the same manner.	

Let $(Y,\eta)$ (resp. $(Y',\eta')$) be an element of $\Omega E_D(2,2)^{\odd}$ which is obtained from $(X,\omega)$ (resp. from $(X',\omega')$) by the construction described in Remark~\ref{rk:triple:tori:plumb} (see also \cite[\textsection 8A]{LN:finite}).
We can identify $\cB$ (resp. $\cB'$)  with a symplectic basis of $H_1(Y,\Z)^-$ (resp. of $H_1(Y',\Z)^-$), and $T$ (resp. $T'$) with a self-adjoint endomorphism of $\Prym(Y)$ (resp. of $\Prym(Y')$) satisfying $T^*\eta=\lambda(\frakp)\cdot\eta$ (resp. $T'{}^*\eta'=\lambda(\frakp')\cdot\eta'$).

By assumption, $(Y,\eta)$ and $(Y',\eta')$ belong to the same component of $\Omega E_{D}(2,2)^\odd$.
Since $\Omega E_D(2,2)^\odd$ is a rank one invariant subvarieties, there is a continuous path $\gamma$ from $(Y,\eta)$ to $(Y',\eta')$  in $\Omega E_D(2,2)^\odd$ which is a concatenation of finitely many  paths $\gamma=\gamma_1*\dots*\gamma_k$, where each of the $\gamma_i$'s is either contained in a $\GL^+(2,\R)$-orbit, or in an isoperiodic leaf (equivalently, a leaf of the kernel foliation). As a consequence, there is an isomorphism $\phi: H_1(Y,\Z)^- \to H_1(Y',\Z)^-$ such that $\phi^*$ maps $\Span(\Re(\eta'), \Im(\eta'))$ on to $\Span(\Re(\eta),\Im(\eta))$ (see \cite[Th. 4.1]{LN:components} for more details).
It follows that $S:=\phi^{-1}\circ T' \circ\phi$ satisfies $S^*\eta=\lambda(\frakp')\cdot\eta$, and  we have $S \in \Z[T]$.

Recall that the map that associates to $R\in\Z[T]$ the eigenvalue $\lambda(R) \in \R$ of $R$ on the line $\C\cdot\eta$, that is $R^*\eta=\lambda(R)\cdot\eta$, is an isomorphism from $\Z[T]$ onto $\Ocal_D$.
Since
$$
(S-T)^*\eta=(\lambda(\frakp')-\lambda(\frakp))\cdot \eta =\frac{e'-e}{2}\cdot\eta
$$
we must have $S-T=\frac{e'-e}{2}\cdot \Id_4$ (note that both $e$ and $e'$ are odd numbers).

We now claim that $\frac{e'-e}{2}$ is even. To see this we notice that the endomorphisms $T$ and $T'$ satisfy the following property
$$
\langle Tu,v\rangle \equiv \langle u,v\rangle \mod 2, \quad \forall u,v \in H_1(Y,\Z)^-
$$
and
$$
\langle T'u',v'\rangle \equiv \langle u',v'\rangle \mod 2, \quad \forall u',v' \in H_1(Y',\Z)^-.
$$
As a consequence
$$
\langle(S-T)u,v\rangle = \frac{e'-e}{2}\cdot\langle u,v\rangle \equiv 0 \mod 2, \quad \forall u,v \in H_1(Y,\Z)^-.
$$
Thus $\frac{e'-e}{2}$ must be an even number. This completes the proof of the lemma.
\end{proof}

\begin{Corollary}\label{cor:triple:tor:in:Omg:E:D:2:2:D:odd}
If $e\equiv 1\, [4]$ then  $\cS_{2,0}^a(e)\subset \cS_{2,0}^{a+}$, and if $e \equiv -1 \, [4]$ then  $\cS_{2,0}^a(e) \subset \cS_{2,0}^{a-}$.
\end{Corollary}
\begin{proof}
Assume first that $e \equiv 1 \, [4]$. By  Lemma~\ref{lm:triple:tori:same:comp:D:1:8}, the triples of tori in $\cS_{2,0}^a(e)$ cannot be contained in the closure of $\Pb\Omega E_{D-}(2,2)^\odd$. Thus those triples of tori must be contained in the closure of $\Pb\Omega E_{D+}(2,2)^\odd$. This means that $\cS_{2,0}^a(e)\subset \hat{\cX}_{D+}$. The proof for the case $e\equiv -1 \, [4]$ follows the same lines.
\end{proof}

Lemma~\ref{lm:triple:tori:same:comp:D:1:8}
implies that $\Pb\Omega E_{D,e}(0^3)$ and $\Pb \Omega E_{D,-e}(0^3)$ are not contained in the same component of $\Pb\Omega\ol{E}_D(2,2)^\odd$ for all $e$ odd such that $e^2 <D$. 
Let us write $W_{D,e}(0^3)=\Pb \Omega E_{D,e}(0^3)$ and 
\[
W_{D+}(0^3):=\bigcup_{\substack{e^2 < D, \\ e \equiv 1\, [4]}}W_{D,e}(0^3), \quad W_{D-}(0^3):=\bigcup_{\substack{e^2 < D, \\ e \equiv -1\, [4]}}W_{D,e}(0^3)
\]
Note that $W_{D+}(0^3)$ (resp. $W_{D-}(0^3)$) is the union of the components of $W_D(0^3)$ which are contained in the boundary of $\Pb\Omega\ol{E}_{D+}(2,2)^\odd$ (resp. $\Pb\Omega\ol{E}_{D-}(2,2)^\odd$).
Since $m_D(e)=m_D(-e)$, we get
\begin{Corollary}\label{cor:euler:char:W:Dpm:0:equal}
We have	
\begin{equation}\label{eq:euler:char:W:Dpm:0:equal}
\chi(W_{D+}(0^3))=\chi(W_{D-}(0^3))=\frac{\chi(W_D(0^3))}{2}.
\end{equation}
\end{Corollary}
For the proof of Theorem~\ref{th:int:Theta:T:2:0:a:1:D:odd}, we will need the following result, whose proof is given in Appendix \textsection~\ref{subsec:prf:iden:triple:tor:D:odd}.
\begin{Theorem}\label{th:iden:triple:tor:e:times:deg:proj:1:e}
	For any $D >9, \, D\equiv 1 \; [8]$ not a square, we have
	\begin{equation}\label{eq:iden:triple:tor:e:times:deg:proj:1:e}
		\sum_{\substack{0 < e < \sqrt{D} \\ e \; {\rm odd}}}(-1)^{\frac{e-1}{2}}\cdot e \cdot m_D(e)=0.
	\end{equation}
\end{Theorem}

\subsection*{Proof of Theorem~\ref{th:int:Theta:T:2:0:a:1:D:odd}}
\begin{proof}
As a consequence of Corollary~\ref{cor:triple:tor:in:Omg:E:D:2:2:D:odd}, we get
\begin{align*}
\langle[\Theta],[\ol{\cT}_{2,0}^{a\pm,1}] \rangle &=\sum_{\substack{-\sqrt{D} < e < \sqrt{D} \\ e \equiv \pm1 \, [4]}}\langle[\Theta],[\ol{\cT}_{2,0}^{a,1}(e)]\rangle \\
&= -48\pi\sum_{\substack{-\sqrt{D} < e < \sqrt{D} \\ e \equiv \pm1 \, [4]}}\frac{e+\sqrt{D}}{\sqrt{D}}\cdot \chi(W_{D,e}(0^3)) \quad \hbox{(by Proposition~\ref{prop:int:Theta:on:T:a:1:2:0:e})} \\
&= -48\pi \sum_{\substack{-\sqrt{D} < e < \sqrt{D} \\ e \equiv \pm1 \, [4]}} \chi(W_{D,e}(0^3)) \pm \frac{8\pi}{\sqrt{D}}\sum_{\substack{ 0 < e < \sqrt{D}\\ e \, \odd }}(-1)^{\frac{e-1}{2}}\cdot e \cdot m_D(e) \\
&= -48\pi \sum_{\substack{-\sqrt{D} < e < \sqrt{D} \\ e \equiv \pm1 \, [4]}} \chi(W_{D,e}(0^3)) \quad \hbox{(by Theorem~\ref{th:iden:triple:tor:e:times:deg:proj:1:e})}.
\end{align*}
Since $\chi(W_{D,-e}(0^3)) = \chi(W_{D,e}(0^3))$ for all $e$ odd, $-\sqrt{D} < e < \sqrt{D}$, we get
$$
\langle[\Theta],[\ol{\cT}_{2,0}^{a\pm,1}] \rangle = -48\pi \sum_{\substack{-\sqrt{D} < e < \sqrt{D} \\ e \equiv \pm1 \, [4]}} \chi(W_{D,e}(0^3)) =-24\pi \sum_{\substack{-\sqrt{D} < e < \sqrt{D} \\ e \, \odd}} \chi(W_{D,e}(0^3)) =-24\pi\chi(W_D(0^3)).
$$
The theorem is then proved.
\end{proof}


\section{Weierstrass Teichmüller curves in the boundary of $\hat{\cX}_D$}\label{sec:W:Teich:curves}
In this section we compute the intersection number $\langle [\Theta],[\ol{\cT}_{0,2}]\rangle$.
Since $[\ol{\cT}_{0,2}] \sim \tilde{\pi}^*[\ol{\cS}_{0,2}]$, it follows from Theorem~\ref{th:Theta:pushforward} that we have
\begin{equation}\label{eq:inters:Theta:T:0:2}
\langle [\Theta], [\ol{\cT}_{0,2}] \rangle =8\pi c_1(\OO(-1))\cdot[\ol{\cS}_{0,2}].
\end{equation} 
Thus, it is enough to compute the degree of the tautological line bundle over the curve $\ol{\cS}_{0,2}$.
Recall that for all  $D'\in \N, \, D' >4, \, D' \equiv 0,1 \, [4]$, $W_{D'}(2):=\Pb\Omega E_{D'}(2)$ is a Teichmüller curve (not necessarily connected) which is the projectivization of closed $\GL^+(2)$-orbit(s) in $\Omega E_{D'}(2)$. By the result of \cite{McM:spin}, if $D' \equiv 0 \, [4]$ or $D'\equiv 5 \, [8]$ then $W_{D'}(2)$ is connected, and if $D' \equiv 1 \, [8]$, then $W_{D'}(2)$ has two components.
We will prove
\begin{Theorem}\label{th:eval:c1:taut:on:S02:D:even}
	Let $D> 4, \, D \equiv 0 \, [4]$ be an even discriminant which is not a square. Then we have
	\begin{equation}\label{eq:eval:c1:taut:on:S02:D:even}
		c_1(\OO(-1))\cdot[\ol{\cS}_{0,2}] = -12\cdot(\chi(W_D(2)) + b_D\cdot\chi(W_{D/4}(2))),
	\end{equation}
	where
	$$
	b_D= \left\{
	\begin{array}{ll}
		0 & \text{ if } D/4 \equiv 2,3 \, [4] \\
		4 & \text{ if } D/4 \equiv 0 \, [4] \\
		3 & \text{ if } D/4 \equiv 1 \, [8] \\
		5 & \text{ if } D/4 \equiv 5 \, [8] \\
	\end{array}
	\right.
	$$
	(here $\chi(.)$ designates the Euler characteristic).
\end{Theorem}

In the case $D\equiv 1 \, [8]$, let $\cS_{0,2}^{\pm}$ be respectively the intersection of $\cS_{0,2}$ with $\hat{\cX}_{D\pm}$.
We will show

\begin{Theorem}\label{th:eval:c1:taut:on:S02:D:odd}
	For all $D\in \N, \, D>9$ not a square, and $D\equiv 1 \; [8]$, we have
	\begin{equation}\label{eq:eval:c1:taut:on:S02:D:odd}
		c_1(\OO(-1))\cdot[\ol{\cS}_{0,2}^+] = c_1(\OO(-1))\cdot[\ol{\cS}_{0,2}^-] = -12\cdot\chi(W_{D}(2)).
	\end{equation}
\end{Theorem}

\subsection{ Weierstrass eigenforms in genus two with a marked  point}\label{subsec:eigen:fom:H2:w:reg:Wei:pt}\hfill\\
Let $\pp=(C,p_1,\dots,p_5,p'_5,\tau,[\xi])$ be a point in $\cS_{0,2}$.
By Lemma~\ref{lm:gp:I:diff:on:vanish:comp} and Lemma~\ref{lm:gp:I:limit:diff}, we know that $C$ has two irreducible components denoted by $C_0$ and $C_1$ where
\begin{itemize}
	\item[$\bullet$]  $C_0$ is isomorphic to $\Pb^1$,
	
	\item[$\bullet$] $C_1$ is a compact Riemann surface of genus $2$,
	
	\item[$\bullet$] $C_0$ and $C_1$ meet  at two nodes, both are fixed by the Prym involution,
	
	\item[$\bullet$] $\xi_{\left|C_0\right.}\equiv 0$ and $(C_1,\xi_{\left|C_1\right.})\in \Omega E_{D'}(2)$ for some $D'\in \{D,D'/4\}$.
\end{itemize}
Let $\xi_1:=\xi_{\left|C_1\right.}$. Then the nodes between $C_0$ and $C_1$ are the unique zero of $\xi_1$ and a Weierstrass point of $C_1$.
Denote by $q$ and $q'$ the nodes of $C$, where $q$ is the double zero of $\xi_1$.

\medskip 

Let $\Omega E^*_{D'}(2)$ denote the space of eigenforms in $\Omega E_{D'}(2)$ together with a marked Weierstrass point which is not the zero of the Abelian differential. Denote by $W^*_D(2)$ the projectivization of $\Omega E^*_{D'}(2)$, that is $W^*_{D'}(2)=\Pb\Omega E^*_{D'}(2)$. 
There is a natural finite covering $\cR_{D'}: W^*_{D'}(2) \to W_{D'}(2)$ consisting of forgetting the marked regular Weierstrass point. The problem of determining the number of connected components of $W^*_{D'}(2)$ and the degree of the map $\cR_{D'}$ on each components of $W^*_{D'}(2)$ has been resolved in \cite{GP24}.

Let $W^*$ denote the component of $W^*_{D'}(2)$ that contains $(C_1,q',[\xi_1])$.
Since the Prym involution fixes $q$ and $q'$, the pointed curve $(C_0,q,q',p_5,p'_5)$ is isomorphic to $(\Pb^1, 0,\infty,1,-1)$ with the Prym involution given by $z\mapsto -z$. In particular, $(C_0,q,q',p_5,p'_5)$ is independent of $\pp$.
As a consequence, we get
\begin{Lemma}\label{lm:deg:proj:Teich:curve:unnumber}
	Let $\cS$ be the component of $\cS_{0,2}$ which contains $\pp$. Then the map $F:\cS \to W^*$ which associates to $\pp$ the projectivized differential  with a marked regular Weierstrass point $(C_1,q',[\xi_1])$ is a covering of degree $4!$.
\end{Lemma}
\begin{proof}
	Since the differential without marked points $(C,[\xi])$  is uniquely determined by $(C_1,[\xi_1])$, $F$ is a covering. By construction, all the marked points $p_1,\dots,p_4$  of $C$ are  contained in $C_1$ and correspond actually to the regular Weierstrass points of $C_1$. Since the map $F$ consists of forgetting the numbering of those points, we get that $\deg F=4!$.
\end{proof}

Let $\cS'_{0,2}$ (resp. $\cS''_{0,2}$) denote the  set of $\pp \in \cS_{0,2}$ such that $(C_1,\xi_1) \in \Omega E_{D}(2)$ (resp. ($C_1,\xi_1)\in \Omega E_{D/4}(4)$ in the case $4 \, | \, D$).
Let $F': \cS'_{0,2} \to W_D(2)$ and $F'': \cS''_{0,2} \to W_{D/4}(2)$ denote the projections which associate to $\pp$ the projectivized Abelian differential (without marked points) $(C_1,[\xi_1])$.
Our goal now is  to compute the degrees of $F'$ and $F''$.

\medskip

Fix  $D'\in \{D, D/4\}$ and consider a surface $(X,\omega) \in \Omega E_{D'}(2)$. Let $w_0$ be the zero of $\omega$, which is a Weierstrass point of $X$.
Denote by $w_1,\dots,w_5$ the other Weierstrass points of $X$. For each $i=1,\dots,5$, the triple $(X,w_i,\omega)$ (resp. $(X,w_i,[\omega])$) is an element of $\Omega E^*_{D'}(2)$ (resp. of $\Pb\Omega E^*_{D'}(2)$).
If $(X,w_i,\omega)$ is contained in the closure of $\Omega E_D(2,2)^\odd$, then by the plumbing construction described in \textsection\ref{subsec:geom:bdry:str:gp:I} (c.f. the proof of Proposition~\ref{prop:bdry:str:gp:I}), one obtains a holomorphic map
$\varphi_i: \Delta_{\delta^2} \to \Omega\ol{E}_D(2,2)^\odd$ such that $\varphi_i(0)=(X,w_i,\omega)$, and $\varphi_i(\Delta^*_{\delta^2})\subset \Omega E_D(2,2)^\odd$.


\medskip

There is an alternative way to construct the family $\varphi_i(\Delta_{\delta^2})$ using techniques from flat metrics  that we now describe. Given $t\in \Delta^*_{\delta^2}$, by a standard construction known as ``breaking up a zero" (see for instance \cite{KZ03,Bai:GAFA,LN:finite}), we can modify the flat metric  in a small disc about the double zero $w_0$ to create two simple zeros connected by a saddle connection $\sig_0$ of period $t^3$.  Let $\sig_1$ be the unique geodesic segment centered at $w_i$ with period $t^3$. Slitting open the segments $\sig_0$ and $\sig_1$, we obtain a flat surface with two boundary components each of which is composed by two geodesic segments.
We can  glue together two pairs of segments in the boundary of this surface to obtain a translation surface $M^i_t$ of genus three with two singularities. One readily checks that this flat surface belongs to the stratum $\Omega\cM_3(2,2)$. Moreover, the hyperelliptic involution on $X$ induces an involution on the new surface with four fixed points, and the segments $\sig_0,\sig_1$ on $X$ give rise to a pair of saddle connections $\sig, \sig'$ on $M^i_t$ that are exchanged by this involution.
It is shown in \cite[\textsection 8C]{LN:finite} that the surfaces obtained from this construction belongs to $\Omega E_D(2,2)^\odd$ with $D \in \{D', 4D'\}$. 

Recall from \cite{McM:spin} that a splitting prototype for eigenform in $\Omega E_{D'}(2)$  is a quadruple of integers $(a,b,d,e)$ which satisfies
$$
(\Pcal_{D'}(2)) \quad \left\{
\begin{array}{lll}
	D'=e^2+4ad, & a,d > 0, & \gcd(a,b,d,e)=1 \\
	0\leq b < \gcd(a,d), &  a > d+e. &
\end{array}
\right.
$$
Note that the condition $a > d+e$ is equivalent to $\lambda':=\frac{e+\sqrt{D'}}{2} < a$. The prototype $(a,b,d,e)$ is called {\em reduced} if we have $d=1$ and hence $b=0$.

Denote by $\Pcal_{D'}(2)$ the set of prototypes for $\Omega E_{D'}(2)$. Associated to each prototype $(a,b,d,e) \in \Pcal_{D'}(2)$, we have a {\em prototypical surface} constructed from a square of size $\lambda'$ and a parallelogram whose sides correspond to the vectors $(a,0)$ and $(b,d)$ (see Figure~\ref{fig:prototype:surface:H2}).
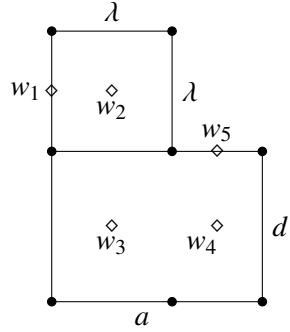
\begin{figure}[htbp]
\centering
\begin{tikzpicture}[scale=0.4]
	\draw (4,5) -- (4,9) -- (0,9) -- (0,0) -- (7,0) -- (7,5) -- (0,5);
	\foreach \x in {(0,9), (0,5), (0,0), (4,9), (4,5), (4,0), (7,0), (7,5)} \filldraw \x circle (4pt);
	
	\foreach \x in {(0,7), (2,7), (2,2.5), (5.5,5), (5.5,2.5)} \draw \x node {$\diamond$};
	
	\draw (2,9) node[above] {$\tiny \lambda$} (4,7) node[right] {$\tiny \lambda$} (3,0) node[below] {$\tiny a$} (7,2.5) node[right] {$\tiny d$};
	
	\draw (0,7) node[left] {$\tiny w_1$};
	\draw (2,7) node[below] {$\tiny w_2$};
	\draw (2,2.5) node[below] {$\tiny w_3$};
	\draw (5,2.5) node[below] {$\tiny w_4$};
	\draw (5.5,5) node[above] {$\tiny w_5$};
\end{tikzpicture}	
	\caption{Prototypical surface where $b=0$, the $w_i$'s are regular Weierstrass points.}
	\label{fig:prototype:surface:H2}
\end{figure}

\begin{Proposition}\label{prop:split:prot:surf:H2}
	Let  $M:=(X,\omega)$ be the prototypical surface associated to a prototype $(a,b,d,e)\in \Pcal_{D'}(2)$, where $b=0$. Denote by $w_0$ be the unique zero of $\omega$ and label the remaining Weierstrass points of $X$ by $w_1,\dots,w_5$ as in Figure~\ref{fig:prototype:surface:H2}.
	We then have
	\begin{itemize}
		\item[(i)] $(M,w_1) \in \Omega\ol{E}_{D'}(2,2)^\odd$ if $a$ is even, and   $(M,w_1) \in \Omega\ol{E}_{4D'}(2,2)^\odd$ if $a$ is odd.
		
		\item[(ii)] $(M,w_2)\in \Omega\ol{E}_{D'}(2,2)^\odd$ if both $a$ and $d$ are even, $(M,w_2)\in \Omega\ol{E}_{4D'}(2,2)^\odd$ otherwise.
		
		\item[(iii)] $(M,w_3) \in \Omega\ol{E}_{D'}(2,2)^\odd$ if both $d$ and $e$ are even, and $(M,w_3) \in \Omega\ol{E}_{4D'}(2,2)^\odd$ otherwise.
		
		\item[(iv)] $(M,w_4) \in \Omega\ol{E}_{D'}(2,2)^\odd$ if both $a-e$ and $d$ are even, and $(M,w_4) \in \Omega\ol{E}_{4D'}(2,2)^\odd$ otherwise.
		
		\item[(v)] $(M,w_5)\in \Omega \ol{E}_{D'}(2,2)^\odd$ if $a-d-e$ is even, $(M, w_5) \in \Omega \ol{E}_{4D'}(2,2)^\odd$ $a-d-e$ is odd.
	\end{itemize}
%
%
%
%
\end{Proposition}
\begin{proof}
	For $i=1,\dots,5$, let $M_i$ be a surface constructed from
	$(M,w_i)$ by the surgery described above.
	For (i), we can suppose that $M_1$ is constructed from horizontal slits on $M$ (that is with a parameter $t\in \R$). Then $M_1$ is decomposed into three cylinders in the horizontal direction, one of which is fixed while the other two are permuted by the Prym involution $\tau$ (see Figure~\ref{fig:prot:surf:splitting:1}).
	One can pick out a symplectic basis $(\alpha_i,\beta_i,\; i=1,2)$ of $H_1(M_1,\Z)^-$ as follows
	\begin{itemize}
		\item[$\bullet$] $\alpha_1=\alpha'_1+\alpha''_1$, where $\alpha'_1$ and $\alpha''_1$ are the core curves of the horizontal cylinders permuted by $\tau$,
		
		\item[$\bullet$] $\beta_1=\beta'_1+\beta''_1$, where $\beta'_1$ (resp. $\beta''_1$) is contained in the closure of the cylinder with core curve $\alpha'_1$ (resp. $\alpha''_1$) such that $(\alpha'_1,\beta'_1)=1$ (resp. $(\alpha''_1,\beta''_1)=1$).
		
		\item[$\bullet$] $\alpha_2$ is the core curve of the horizontal cylinder fixed by $\tau$,
		
		\item[$\bullet$] $\beta_2$ is a simple closed curve contained in the closure of the cylinder with core curve  $\alpha_2$ such that $(\alpha_2,\beta_2)=1$.
	\end{itemize}
	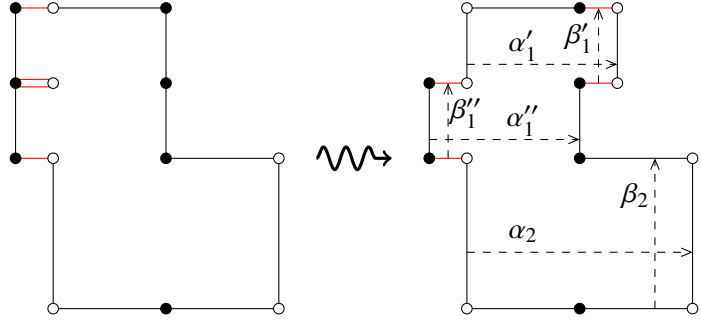
\begin{figure}[htbp]
		\centering 
		\begin{tikzpicture}[scale=0.5]
			\draw (0,8) -- (0,4) (1,4) -- (1,0) -- (7,0) -- (7,4) -- (4,4) -- (4,8) -- (1,8);
			
			\foreach \x in {(0,8), (0,6.1), (0,5.9), (0,4)} \draw[red] \x -- +(1,0);
			
			\foreach \x in {(0,8), (0,6), (0,4), (4,8), (4,6), (4,4), (4,0)} \filldraw \x circle (4pt);
			
			\foreach \x in {(1,8),(1,6), (1,4), (1,0), (7,4), (7,0)} \filldraw[fill=white] \x circle (4pt);
			
			\draw (11,6) -- (11,4) (12,4) -- (12,0) -- (18,0) -- (18,4) -- (15,4) -- (15,6) (16,6) -- (16,8) (15,8)-- (12,8) -- (12,6);
			
			\foreach \x in {(11,6),(11,4), (15,8), (15,6)} \draw[red] \x -- +(1,0);
			
			\foreach \x in {(11,6), (11,4), (15,8), (15,6), (15,4), (15,0)} \filldraw \x circle (4pt);
			
			\foreach \x in {(12,8), (12,6), (12,4), (12,0), (16,8), (16,6), (18,4), (18,0)} \filldraw[fill=white] \x circle (4pt);
			
			\foreach \x in {(11,4.5), (12,6.5)} \draw[->, >=angle 45, dashed] \x -- +(4,0);
			\draw[->, >=angle 45, dashed] (12,1.5) -- (18,1.5);
			
			\foreach \x in {(11.5,4), (15.5,6)} \draw[->,>=angle 45, dashed] \x -- +(0,2);
			
			\draw[->, >= angle 45, dashed] (17,0) -- (17,4);
			
			\draw (13.5,7) node {$\tiny \alpha'_1$} (13.5,5) node {$\tiny \alpha''_1$} (15,7.2) node {$\tiny \beta'_1$} (12,5.2) node {$\tiny \beta''_1$} (13.5,2) node {$\tiny \alpha_2$} (16.5,3) node {$\tiny \beta_2$};
			
			\draw[->, very thick, snake=snake, segment amplitude=1.5mm, segment length=3mm, line after snake=1mm] (8,4) -- (10,4);
		\end{tikzpicture}
		
		\caption{Construction of $M_1$}
		\label{fig:prot:surf:splitting:1}
	\end{figure}
	Let $v=(2\lambda',2\imath\lambda', a, \imath d) \in \C^4$, with $\lambda'=\frac{e+\sqrt{D'}}{2}$, be the vector recording the periods of $(\alpha_1,\beta_1,\alpha_2,\beta_2)$.
	Let $T$ be the endomorphism of $H_1(M_1, \Z)^-$ given in the basis $(\alpha_1,\beta_1,\alpha_2,\beta_2)$ by the matrix
	$T=\left(
	\begin{smallmatrix}
		2e & 0 & a & 0\\
		0 & 2e & 0 & 2d\\
		4d & 0 & 0 & 0 \\
		0 & 2a & 0 & 0
	\end{smallmatrix}
	\right).$
	One readily checks that $T$ is self-adjoint with respect to the intersection form and satisfies $T^2=2eT+4ad$. Moreover, we have
	$$
	{}^tv\cdot T =2\lambda'\cdot {}^tv.
	$$
	Let $(X_1,\omega_1)$ be the Abelian differential corresponding to $M_1$.
	By the arguments of \cite[Th. 3.5]{McM:prym} (see also \cite[\textsection 4]{LN:H4}), $T$ generates a subring isomorphic to $\Ocal_{4D'}$ in $\End(\Prym(X_1))$, for which $\omega_1$ is an eigenform.
	If $\langle T \rangle$ is the maximal self-adjoint subring of $\End(\Prym(X_1))$ that preserves the line $\C\cdot\omega_1$, then by definition we have $M_1\in \Omega E_{4D'}(2,2)^\odd$.
	This is the case if and only if $\gcd(a,2d,2e)=1$. Since $\gcd(a,d,e)=1$, this occurs when $a$ is odd. 
	If $a$ is even then $T/2\in \End(\Prym(X_1))$, and $\langle T/2 \rangle \simeq \Ocal_{D'}$ which means that $M_1\in \Omega E_{D'}(2,2)^\odd$. This completes the proof of (i).
	
	\medskip
	
	For (ii), we also consider a surface $M_2:=(X_2,\omega_2)$ obtained from $M$ by some horizontal slitting. In particular, $M_2$ is horizontally periodic with the same cylinder diagram as $M_1$. We choose a symplectic basis $(\alpha_1,\beta_1,\alpha_2,\beta_2)$ of $H_1(M_2,\Z)^-$ in the same way as for $M_1$.
	We consider the endomorphism of $H_1(M_2, \Z)^-$ given in the basis $(\alpha_1,\beta_1,\alpha_2,\beta_2)$ by the matrix
	$T=\left(
	\begin{smallmatrix}
		2e & 0 & a & -d\\
		0 & 2e & 0 & 2d\\
		4d & 2d & 0 & 0 \\
		0 & 2a & 0 & 0
	\end{smallmatrix}
	\right).
	$ One readily checks that $T\in \End(\Prym(X_2))$ is self-adjoint and generates a subring isomorphic to $\Ocal_{4D'}$ in $\End(\Prym(X_2))$ for which $\omega_2$ is an eigenform. We conclude by similar arguments as case (i).
	
	\medskip
	
	For (iii), we consider a surface $M_3=(X_3,\omega_3)$ obtained from $M$ by a small vertical slitting (see Figure~\ref{fig:split:vert:eigen:form:H2}). In this case $M_3$ is decomposed into $4$ horizontal cylinders with the diagram I.A (see \textsection\ref{sec:eigen:form:g3}). We can pick out a basis $(\alpha_1,\beta_1,\alpha_2,\beta_2)$ such that $\langle \alpha_i,\beta_i\rangle =i, \; i=1,2$, and  whose periods are given by the vector $v=(\lambda',a+\imath\lambda',2a,\lambda'+\imath d)$. By considering the endomorphism of $H_1(M_3, \Z)^-$ given by the matrix
	$T=\left(
	\begin{smallmatrix}
		2e & 0 & 4a & 2e\\
		0 & 2e & 0 & 2d\\
		d & -e & 0 & 0 \\
		0 & 2a & 0 & 0
	\end{smallmatrix}
	\right)
	$ we get the desired conclusion.
	
	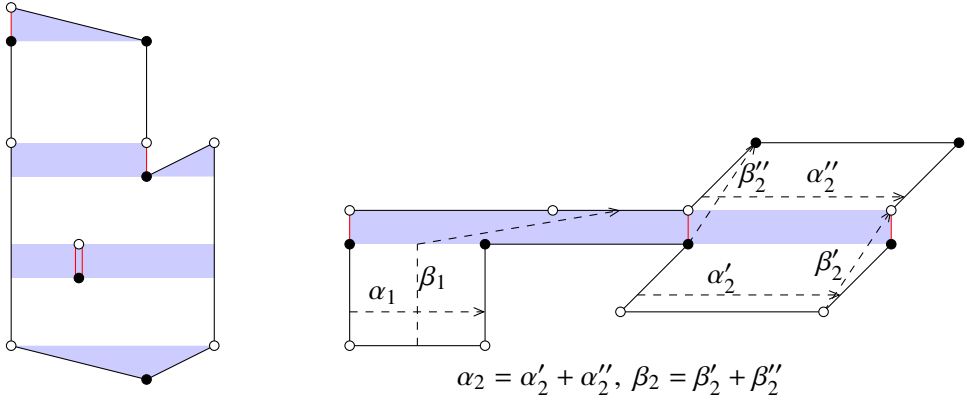
\begin{figure}[htbp]
		\centering
		\begin{tikzpicture}[scale=0.45]
			\fill[blue!20] (0,10) --( 0,9) -- (4,9) -- cycle;
			\fill[blue!20] (0,6) -- (0,5) -- (4,5) -- (4,6) --cycle;
			\fill[blue!20] (4,5) -- (6,5) -- (6,6) -- cycle;
			\fill[blue!20] (0,3) -- (0,2) -- (6,2) -- (6,3) -- cycle;
			\fill[blue!20] (0,0) -- (4,-1) -- (6,0) -- cycle;
			\fill[blue!20] (10,4) -- (10,3) -- (26,3) -- (26,4) -- cycle;
			
			\draw  (0,9) -- (0,0) -- (4,-1) -- (6,0) -- (6,6)  -- (4,5) (4,6) -- (4,9) -- (0,10);
			\foreach \x in {(0,9), (4,5), (1.9,2), (2.1,2)} \draw[red] \x -- +(0,1);
          \draw (10,3) -- (10,0) -- (14,0) -- (14,3) -- (20,3) -- (18,1) -- (24,1) -- (26,3) (26,4) -- (28,6) -- (22,6) -- (20,4) --(10,4);
          
          \foreach \x in {(10,3), (20,3), (26,3)} \draw[red] \x -- +(0,1);
          
          \draw[->, >=angle 45, dashed] (20,3) -- (22,6);
          \draw[->, >=angle 45, dashed] (24,1) -- (26,4);
          
          \foreach \x in {(0,10), (0,6), (0,0), (2,3), (4,6), (6,6), (6,0), (10,4), (10,0), (14,0), (16,4), (18,1), (20,4), (24,1), (26,4)} \filldraw[fill=white] \x circle (4pt);
          
          \foreach \x in {(0,9), (2,2), (4,9), (4,5), (4,-1), (10,3), (14,3), (20,3), (22,6), (26,3), (28,6)} \filldraw \x circle (4pt);
          
          \draw[dashed] (12,0) -- (12,3); 
          \draw[->, >=angle 45, dashed] (12,3) -- (18,4); 
          \draw[->, >=angle 45, dashed] (10,1)  -- (14,1);
          \draw (11,1.5) node {$\tiny \alpha_1$} (12.5,2) node {$\tiny \beta_1$};  
          \draw (22,5) node {$\tiny \beta''_2$};
          \draw (24.2,2.4) node {$\tiny \beta'_2$};
          \draw[->, >=angle 45, dashed] (20.4, 4.4) -- +(6,0);
          \draw (24,5) node {$\tiny \alpha''_2$};
          \draw[->, >=angle 45, dashed] (18.5, 1.5) -- +(6,0);
          \draw (21,2) node {$\tiny \alpha'_2$};
          \draw (18,-1) node {$\tiny \alpha_2=\alpha'_2+\alpha''_2, \; \beta_2=\beta'_2+\beta''_2$};
		\end{tikzpicture}
		
		\caption{Construction of $M_3$ by a vertical splitting}
		\label{fig:split:vert:eigen:form:H2}
	\end{figure}

	Finally, for (iv) and (v), by rotating $M$ by the angle $\pi/2$, then rescaling by a diagonal  matrix, one can transform $M$ into the prototypical surface associated with the prototype $(a^*,b^*,d^*,e^*)$, where
	$a^*=a-d-e, b^*=0, d^*=d$, and $e^*=-e-2d$. We can then conclude by the arguments of cases (i) and (ii).
\end{proof}

Let us now prove 
\begin{Proposition}\label{prop:degree:forget:map:S02}
	Let $D\equiv 0 \; [4], \; D\geq 8$ be an even discriminant which is not a square. Recall that  $F':\cS'_{0,2} \to  W_{D}(2)$ and $F'': \cS''_{0,2} \to  W_{D/4}(2)$ are the maps consisting of forgetting the marked regular Weierstrass points on the genus two components of the underlying stable curves.
	We have
	\begin{itemize}
		\item[(i)] If $D/4\equiv 2,3 \; [4]$, then $\deg F'=4!$ and $\deg F''=0$.
		
		\item[(ii)] If $D/4\equiv 0 \; [4]$, $D/4 \geq 8$, then $\deg F'=4!$ and $\deg F''=4\cdot 4!$.
		
		\item[(iii)] If $D/4\equiv 1 \; [8]$, $D/4 \geq 17$, then $\deg F'=4!$ and $\deg F''=3\cdot 4!$.
		
		\item[(iv)] If $D/4\equiv 5 \; [8]$, $D/4 \geq 13$,  then $\deg F'=4!$ and $\deg F''=5!$.
	\end{itemize}
\end{Proposition}
\begin{proof}\hfill 
	\begin{itemize}
		\item[(i)] Since $D/4\equiv 2,3 \; [4]$, $D/4$ is not a discriminant. Therefore $\cS''_{0,2}=\vide$, and $\deg F''=0$. 
		We have either $D=8k$, or $D=8k+4$, where $k$ is an odd number. 
		In the former case, let $M$ be the surface constructed from the prototype $(2k,0,1,0) \in \Pcal_D(2)$.
		Let $w_1,\dots,w_5$ be the regular Weierstrass points of $M$ as in Proposition~\ref{prop:split:prot:surf:H2}.  Then only $(M,w_1)$ belongs to $\Omega \ol{E}_{D}(2,2)^\odd$. This means that the preimage of $M$ in $\cS_{0,2}$ consists of one point up to a numbering of the fixed points of the Prym involution. Therefore we have $\deg F= \deg F'=4!$ in this case.
		
		In the latter case, that is $D=8k+4$, $k$ odd, let $M$ be the surface  associated to the prototype $(2k+1,0,1,0) \in \Pcal_D(2)$.
		and $w_1,\dots,w_5$  the regular Weierstrass points of $M$. From Proposition~\ref{prop:split:prot:surf:H2}, only $(M,w_5)$ is contained in $\Omega  \ol{E}_D(2,2)^\odd$. Thus we also have $\deg F'=4!$.
		
		\medskip 
		
		\item[(ii)] In this case, we can write $D=16k, \; k\in \N$, $k \geq 2$. Let $M$ be the surface constructed from the prototype $(4k,0,1,0) \in \Pcal_D(2)$ and $w_1,\dots,w_5$ be the regular Weierstrass points of $M$. 
		By Proposition~\ref{prop:split:prot:surf:H2}, only $(M,w_1)\in \Omega \ol{E}_D(2,2)^\odd$. Since $W_D(2)$ is connected, we conclude that $\deg F'=4!$.
		
		Consider now the surface $M$ constructed from the prototype $(k,0,1,0) \in \Pcal_{D/4}(2)$. Note that $k$ can be odd or even. In both cases, it follows from Proposition~\ref{prop:split:prot:surf:H2} that four pairs among $\{(M,w_i),\; i=1,\dots,5\}$ belong to $\Omega \ol{E}_D(2,2)^\odd$. Thus $\deg F''=4\cdot 4!$.
		
		\medskip

		\item[(iii)] Let us write $D/4=8k+1$. Then $D=32k+4$. Note that $W_D(2)$ is connected. By considering the surface associated with the prototype $(8k+1,0,1,0) \in \Pcal_D(2)$, we get that $\deg F'=4!$.
		
		By \cite{McM:spin}, we know that $W_{D/4}(2)$ has two components. We will denote those components by $W_{D/4\pm}(2)$ where  $W_{D/4+}(2)$ contains the surface $M^+$ constructed from the prototype $(2k,0,1,-1)$, and  $W_{D/4-}(2)$ contains the surface $M^-$ constructed from the prototype $(2k,0,1,1)$.
		Let $w^+_1,\dots,w^+_5$ (resp. $w^-_1,\dots, w^-_5$) be the regular Weierstrass points of $M^+$ (resp. of $M^-$). From Proposition~\ref{prop:split:prot:surf:H2}, $(M^\pm,w_i^{\pm})$ belongs to $\Omega \ol{E}_D(2,2)^\odd$ if and only if $i\in\{2,3,4\}$. Thus we have $\deg F''=3\cdot 4!$.
		
		\medskip
		
		\item[(iv)] Let us write $D/4=8k+5$ or equivalently $D=32k+20$.
		By considering the surface constructed from the prototype $(8k+5,0,1,0) \in \Pcal_D(2)$ we get that $\deg F'=4!$.
		Consider now the surface $M$ constructed from the prototype $(2k+1,0,1,1) \in \Pcal_{D/4}(2)$. Let $w_1,\dots,w_5$ be the regular Weierstrass points of $M$. It follows from Proposition~\ref{prop:split:prot:surf:H2} that $(M,w_i)\in \Omega \ol{E}_D(2,2)^\odd$ for all $i=1,\dots,5$.
		Thus, we have $\deg F''=5\cdot 4!=5!$.
	\end{itemize}
\end{proof}

\subsection{Proof of Theorem~\ref{th:eval:c1:taut:on:S02:D:even}}
\begin{proof}
	By Proposition~\ref{prop:int:Hodge:curv:on:Teich:curves}, we have
	\begin{align*}
		c_1(\OO(-1))\cdot[\ol{\cS}_{0,2}] & = -\frac{1}{2}\cdot\chi(\cS_{0,2}) = -\frac{1}{2}\cdot\left(\chi(\cS'_{0,2}) + \cdot\chi(\cS''_{0,2})\right)\\
		& = -\frac{1}{2}\cdot\big(\deg F'\cdot\chi(W_D(2))+\deg F''\cdot\chi(W_{D/4}(2))\big)
	\end{align*}
	and we conclude by Proposition~\ref{prop:degree:forget:map:S02}.
\end{proof}

\subsection{Case $D\equiv 1 \, [8]$}
In this case $W_D(2)$ has two connected components (cf. \cite{McM:spin}).
Let $W_{D+}(2)$ (resp. $W_{D-}(2)$) be the component of $W_D(2)$ that contains the surface constructed from the prototype $((D-1)/4), 0, 1, -1)$ (resp. $((D-1)/4,0,1,1)$) in $\Pcal_D(2)$.

Since $4 \nmid D$,  we have $\cS_{0,2}=\cS'_{0,2}$.
Let $\cS_{0,2}^\pm$  denote respectively the intersections of $\hat{\cX}_{D\pm}$  with $\cS_{0,2}$. As a consequence of Proposition~\ref{prop:split:prot:surf:H2}, we get

\begin{Proposition}\label{prop:degree:FD:odd}
	For $D\equiv 1 \, [8], D>9$, $F(\cS_{0,2}^+)=W_{D+}(2), \;  F(\cS_{0,2}^-)=W_{D-}(2)$, and we have
	$$
	\deg(F_{\left|\cS_{0,2}^+\right.})=\deg(F_{\left|\cS_{0,2}^-\right.})=2\cdot4!
	$$
\end{Proposition}
\begin{proof}
	Let $M^+$ be the surface associated with the prototype $((D-1)/4,0,1,-1) \in \Pcal_D(2)$. Let $w_0, w_1,\dots,w_5$ be as in Proposition~\ref{prop:split:prot:surf:H2}. Note that in this case $a=(D-1)/4$ is even. It follows from Proposition~\ref{prop:split:prot:surf:H2} that $(M^+,w_i) \in \Omega \ol{E}_{D}(2,2)^\odd$ if and only if $i=1$ or $i=5$.
	We claim that $(M^+,w_1)\in \Omega \ol{E}_{D+}(2,2)^\odd$.
	To see this we  consider a surface $M_1^+$ obtained from $(M^+,w_1)$ by some small horizontal slits.
	By construction, there are a triple of homologous horizontal saddle connections that decompose $M^+_1$ into a connected sum of three tori.
	We can collapse this triple of saddle connections to  obtain a triple of tori $\hat{M}^+_1$. Rescaling $\hat{M}^{+}_1$ by the matrix $\left(\begin{smallmatrix} 1/\lambda & 0 \\ 0 & (D-1)/(4\lambda) \end{smallmatrix}\right)$, where $\lambda=\frac{-1+\sqrt{D}}{2}$, we obtain the triple of tori associated with the prototype $(1,0,(D-1)/8,1)\in \Pcal_D(0^3)$ (cf. \textsection \ref{subsec:triples:tori:eigenform:comp}). This means that $M^+_1\in \Omega E_{D+}(2,2)^\odd$. Therefore $(M^+,w_1) \in \Omega \ol{E}_{D+}(2,2)^\odd$.
	By the results of \cite{GP24}, $(M^+,w_1)$ and $(M^+,w_5)$ belong to the same $\GL^+(2,\R)$-orbit. Therefore, we also have $(M^+,w_5)\in \Omega \ol{E}_{D+}(2,2)^\odd$.
	
	Let $M^-$ be the surface in  $\Omega E_{D-}(2)$ associated with the prototype $((D-1)/4,0,1,1) \in \Pcal_D(2)$, and $w_1,\dots,w_5$ be the regular Weierstrass points on $M^-$.
	By similar arguments as above $(M^-,w_i)\subset \Omega \ol{E}_{D-}(2,2)^\odd$ if and only if $i=1$ and $i=5$.
	
	Since $4 \, \nmid \, D$, we must have $F(\cS_{0,2})\subset W_D(2)=W_{D+}(2)\sqcup W_{D-}(2)$. The arguments above show that $F(\cS_{0,2}^+)=W_{D+}(2)$,  $F(\cS_{0,2}^-)=W_{D-}(2)$, and we have
	$$
	\#F^{-1}(\{M^+\})= \#F^{-1}(\{M^-\})=2.4!
	$$
	This completes the proof of the proposition.
\end{proof}

\subsection*{Proof of Theorem~\ref{th:eval:c1:taut:on:S02:D:odd}}
\begin{proof}
	It follows from Proposition~\ref{prop:int:Hodge:curv:on:Teich:curves} that
	$$
	c_1(\OO(-1))\cdot[\ol{\cS}_{0,2}^{\pm}]=-\frac{1}{2}\cdot\chi(\cS_{0,2}^{\pm}) =\frac{-1}{2}\cdot\deg F_{\left|\cS_{0,2}^{\pm}\right.}\cdot\chi(W_{D\pm}(2)).
	$$
	In \cite{Bai:GT}, it was shown that $\chi(W_{D+}(2))=\chi(W_{D-}(2))=1/2\cdot\chi(W_D(2))$. We can then conclude by Proposition~\ref{prop:degree:FD:odd}.
\end{proof}

\section{Volume of $\Pb\Omega E_D(2,2)^\odd$} \label{sec:vol:Omg:E:D:2:2}
\begin{proof}[Proof of Theorem~\ref{th:vol:Omg:E:D:2:2}]
	If $ 4 \, | \, D$, combining Theorem~\ref{th:vol:XD},
	Theorem~\ref{th:int:Theta:T:2:0:a:1}, and Theorem~\ref{th:eval:c1:taut:on:S02:D:even}, we get
	
	\begin{equation}\label{eq:eval:vol:XD:D:even}
		\mu(\XD)= \frac{2\pi^2}{3}\Big(\chi(W_D(2))+ b_D\cdot\chi(W_{D/4}(2)) \Big) +6\pi^2\cdot\chi(W_{D}(0^3)).
	\end{equation}
	Since the map $\cX_D \to \Pb\Omega E_D(2,2)^\odd$ has degree $4!=24$, \eqref{eq:vol:Omg:E:D:2:2:formula} follows.
	
	\medskip
	
	In the case $D\equiv 1 \; [8]$,
	Theorem~\ref{th:int:Theta:T:2:0:a:1:D:odd} and Theorem~\ref{th:eval:c1:taut:on:S02:D:odd}  imply
	\begin{equation}\label{eq:eval:vol:XD:pm:D:odd}
		\mu(\cX_{D+})= \mu(\cX_{D-})=\frac{2\pi^2}{3}\cdot\chi(W_D(2)) + 3\pi^2\cdot \chi(W_D(0^3)).
	\end{equation}
	Since $\mu(\Pb\Omega E_{D\pm}(2,2)^\odd)=\frac{1}{4!}\cdot\mu(\cX_{D\pm})$, \eqref{eq:vol:Omg:E:D:2:2:D:odd} follows.
\end{proof}

\section{Siegel-Veech constants}\label{sec:Siegel:Veech}
\subsection{Degenerating by collapsing saddle connections}\label{subsec:collapse:sc}
Let $(X,\omega)$ be an eigenform in  $\Omega E_D(2,2)^\odd$. Denote the zeros of $\omega$ by $x_1, x_2$. By convention, any saddle connection $\s$ on $X$ connecting $x_1$ and $x_2$ is endowed with the orientation from $x_1$ to $x_2$.
We say that $\s$ has multiplicity $k$, $k=1,2,\dots$, if there are exactly $k$ saddle connections on $X$ with the same endpoints and the same period as $\s$.
Since the zeros of $\omega$ have order $2$, the multiplicity of any saddle connection cannot be greater than $3$.
The following proposition generalizes \cite[Prop. 5.5]{LN:components}, its proof is left to the reader.

\begin{Proposition}\label{prop:collapse:fam:s:c}
	Let $\tilde{\sigma}:=\{\sigma_1,\dots,\sigma_k\}, \; k \in \{1,2,3\}$, be a maximal family of saddle connections with the same period joining the two zeros of $\omega$. Assume that any saddle connection $\sigma'$ parallel to $\sigma_1$ not in $\tilde{\sigma}$ (if exists) satisfies $|\sigma'| >|\sigma_1|$.
	Then the family $\tilde{\sigma}$ can be collapsed simultaneously along the isoperiodic leaf of $(X,\omega)$ and the resulting surface belongs to $\Omega E_D(4)$ if $k=1$, to $\Omega E^*_{D'}(2)$ for some $D'\in \{D, D/4\}$ if $k=2$, and to $\Omega E_D(0^3)$ if $k=3$.
\end{Proposition}

As a byproduct of Proposition~\ref{prop:collapse:fam:s:c}, we get

\begin{Corollary}\label{cor:collapse:s:c:non:par:abs:per}
	Let $(X,\omega)\in\Omega E_D(2,2)^\odd$ with $D$ not a square, and $\tilde{\sigma}:=\{\sigma_1,\dots,\sigma_k\}$ be a maximal family of saddle connections with the same period joining the two zeros of $\omega$.
	Assume that $\sigma_1$ is  not parallel to any vector in the set
	$$
	\mathrm{Per}(\omega):=\{\omega(c), \; c \in H_1(X,\Z)\}-\{0\} \subset \R^2.
	$$
	Then $\tilde{\sigma}$ can be collapsed simultaneously along the isoperiodic leaf of $(X,\omega)$.
\end{Corollary}
\begin{proof}
	It is enough to show that there is no saddle connection parallel to $\sigma_1$ but not in $\tilde{\sigma}$. Let $\sigma'$ be such a saddle connection.  If $\sigma'$ joins a zero of $\omega$ to itself then it represents an element of $H_1(X,\Z)$, and we have a contradiction to the hypothesis. Therefore, $\sigma'$ must join the two zeros of $\omega$.
	As a consequence $c:=(-\sigma')*\sigma_1$ is an element of $H_1(X,\Z)$ satisfying $\omega(c)=\lambda\cdot\omega(\sigma_1)$ for some $\lambda\in\R$. Again, by the hypothesis we must have $\lambda=0$. It follows that
	$\omega(\sigma')=\omega(\sigma_1)$, which means that $\sigma'\in \tilde{\sigma}$ and we have again a contradiction.  We can now conclude by Proposition~\ref{prop:collapse:fam:s:c}.
\end{proof}

It follows from Proposition~\ref{prop:collapse:fam:s:c},  that $\Omega E_D(4)$ and $\Omega E_D(0^3)$ are contained in the boundary of $\Omega E_D(2,2)^\odd$. 
Denote by $\Omega E^*_{[D]}(2)$ denote the union of the components of $\Omega E^*_{D}(2)$ and $\Omega E^*_{D/4}(2)$ that are contained in the boundary of $\Omega E_D(2,2)^\odd$.

\medskip

In the case $D\equiv \pm 1 \; [8]$, $\Omega E_D(2,2)^\odd$ is a disjoint union of two connected components $\Omega E_{D+}(2,2)^\odd$ and $\Omega E_{D-}(2,2)^\odd$, where $\Omega E_{D+}(2,2)^\odd$ (resp. $\Omega E_{D-}(2,2)^\odd$) contains the closed orbit $\Omega E_{D-}(4)$
(resp. $\Omega E_{D+}(4)$) in its closure. Let $\Omega E^*_{D\pm}(2)$ denote respectively the union of the components of $\Omega E^*_D(2)$ that are contained in the boundary of $\Omega E_{D\pm}(2,2)^\odd$. Finally, let $\Omega E_{D\pm}(0^3)$ denote the union of the components of $\Omega E_D(0^3)$ that are contained in the boundary of  $\Omega E_{D\pm}(2,2)^\odd$ respectively.

To simplify the notation, we will denote the projectivization spaces $\Pb\Omega E_D(4), \Pb\Omega E^*_{[D]}(2), \Pb\Omega E_D(0^3)$ by $W_D(4), W^*_{[D]}(2)$ and  $W_D(0^3)$ respectively.  Similarly, if $D\equiv 1 \, [8]$, we will write $W_{D\pm}(\kappa)=\Pb\Omega E_{D\pm}(\kappa)$ for $\kappa\in \{4, 2, 0^3\}$.


\subsection{Prym eigenforms with a marked saddle connection}\label{subsec:marked:s:connect}
To prove Theorem~\ref{th:Siegel:Veech}, we will consider the Siegel-Veech transforms of the indicator function of a small disc in $\C$. The supports of the Siegel-Veech transforms are tubular neighborhoods of some components of the boundary  of $\Omega_1 E_D(2,2)^\odd$. The corresponding Siegel-Veech constants are obtained from the ratio of the volumes of those neighborhoods and the volume of $\Omega_1 E_D(2,2)^\odd$.   Even though this method is already well known since the pioneer works \cite{EMZ03, MZ08}, the calculation of the Siegel-Veech constants in our situation is however not straightforward because of different normalizations  of the volume forms on different spaces of eigenforms. 
In the sequel, we will focus on the case of saddle connection of multiplicity one. The proofs for the other cases follows the same lines.

For $k=1,2,3$, let $\Omega\tilde{E}^{(k)}_D(2,2)^\odd$ denote the space of triples $(X,\omega,\tilde{\sigma})$, where $(X,\omega) \in \Omega E_D(2,2)^\odd$ and $\tilde{\sigma}=\{\sigma_1,\dots,\sigma_k\}$ is a maximal family of saddle connections connecting the two zeros of $\omega$ having the same period.  Let $\Upsilon_k: \Omega\tilde{E}^{(k)}_D(2,2)^\odd \to \Omega E_D(2,2)^\odd$ be the forgetting map.  Note that $\Upsilon_k$ is a local diffeomorphism. 
The pullback of the volume form on $\Omega E_{D}(2,2)^\odd$ to $\Omega\tilde{E}^{(k)}_{D}(2,2)^\odd$ will be denoted again by $d\Vol$.

Let $\Omega_1 \tilde{E}^{(k)}_D(2,2)^\odd$ denote  the set of surfaces in $\Omega \tilde{E}^{(k)}_D(2,2)^\odd$ which have area one.
As in the case of $\Omega E_D(2,2)^\odd$, we have a volume form $d\vol_1$ on $\Omega_1 \tilde{E}_D^{(k)}(2,2)^\odd$ defined  as follows:  for any $U$ open subset of $\Omega_1 \tilde{E}^{(k)}_D(2,2)^\odd$, $\vol_1(U):=\Vol(C_1(U))$, where $C_1(U):=\bigcup_{t\in(0,1]}t\cdot U$ is the cone over $U$.

\medskip
Consider a surface $(X_0,\omega_0)\in \Omega_1 E_D(4)$. Let $v$ be a vector in $\R^2\setminus\{0\}\simeq \C^*$ such that all the saddle connections of $(X_0,\omega_0)$ in the direction of $\pm v$ (if exist) have length at least $2|v|$. Then  one can ``break up" the unique zero of order $4$ of $\omega_0$ into two double zeros that are connected by a saddle connection $\sigma_v$ with period $v$ (see \cite{KZ03,EMZ03}). Let $(X_v,\omega_v)$ denote the resulting translation surface. Then $(X_v,\omega_v,\sigma_v)$ is an element of $\Omega_1 \tilde{E}^{(1)}_D(2,2)^\odd$.
We will call this construction the ``zero splitting'' with parameter $v$.

Since the zero of $\omega_0$ has order $4$,  there are $5$ pairs of symmetric rays in directions $\pm v$ issued from this zero. As a consequence we obtain $5$ distinct elements in $\Omega_1 \tilde{E}^{(1)}_D(2,2)^\odd$ from $(X_0,\omega_0)$ and $v$  (see for instance \cite[\textsection 5.3]{LN:components} for more details).  Note also that since the zeros of $\omega_v$ are not numbered, the surfaces obtained from $v$ and $-v$ are actually the same.

Let us now fix a small positive real number $\eps_0>0$. The set of $v\in \Delta^*_{\eps_0}$ such that one can break up the zero of $\omega_0$ into two zeros connected by a saddle connection with period $v$ is an open dense subset of $\Delta^*_{\eps_0}$. Therefore, there is  an open dense subset  $\cU_{\eps_0}$ of $\Omega_1 E_D(4)\times \Delta_{\sqrt[5]{\eps_0}}$ and a map $F_1: \cU_{\eps_0} \to \Omega_1 \tilde{E}^{(1)}_D(2,2)^\odd$, which associates to $((X_0,\omega_0),t)$ an element $(X_t,\omega_t,\sigma_t) \in \Omega_1 \tilde{E}^{(1)}_D(2,2)^\odd$ such that all the absolute periods of $\omega_t$ equal the corresponding absolute periods of $\omega_0$, and  $\omega_t(\sigma_t)=t^5$.
The condition $\omega(\sigma_t)=t^5$ reflects the fact that for all $v\in \Delta^*_{\eps_0}$,  the zero splitting with parameter $v$ produces five elements of $\Omega_1 \tilde{E}^{(1)}_D(2,2)^\odd$.

\begin{Lemma}\label{lm:breakup:zero:covering}
The map $F_1$ is a two to one covering onto its image.
\end{Lemma}
\begin{proof}
Given $(X,\omega,\sigma) \in F_1(\cU_{\eps_0}) \subset \Omega_1 \tilde{E}^{(1)}_D(2,2)^\odd$ collapsing the marked saddle connection allows us to recover the surface  $(X_0,\omega_0) \in \Omega_1 E_D(4)$. It follows that $F_1$ is a local diffeomorphism. Moreover, since the surface $(X_0,\omega_0)$ is uniquely determined by $(X,\omega,\sigma)$, and the period of $\sigma$ depends on the labelling of the zeros of $\omega$ (recall that $\sigma$ is endowed with the orientation from $x_1$ to $x_2$ by convention),  the preimage of $(X,\omega,\sigma)$ consists of two elements $((X_0,\omega_0),\pm t)$ with $\omega(\sigma)=\pm t^5$. Therefore, we have $\deg F_1= 2$.
\end{proof}

Theorem~\ref{th:Siegel:Veech} will follow from
\begin{Proposition}\label{prop:vol:neigh:Omg:E:D:4}
	We have
	\begin{equation}\label{eq:vol:neigh:Omg:E:D:4}
	\int_{\cU_{\eps_0}}F^*_1 d\vol_1 = \frac{5\pi^3\eps^2_0}{6}\chi(W_D(4)).
    \end{equation}	
\end{Proposition}

\subsection{Volume form on $\Omega E_D(4)$}\label{subsec:vol:form:Omg:E:D:4}
Recall that $\Omega E_D(4)$ is endowed with a natural volume form $d\Vol'$ locally defined as follows: a neighborhood of $(X_0,\omega_0)$ in $\Omega E_D(4)$ can be identified with an open subset of the subspace $\Vb:=\Span(\Re(\omega_0), \Im(\omega_0)) \subset H^1(X_0,\C)$.
The restriction $(.,.)_{|\Vb}$ of the intersection form on $H^1(X_0,\C)$ to $V$ has signature $(1,1)$. In particular $(.,.)_{|\Vb}$ is non-degenerate.
Therefore the imaginary part $\Omega$ of $(.,.)_{|\Vb}$ is a symplectic form on $\Vb$. We define $d\Vol' = \frac{\Omega^2}{2!}$.
As usual, the volume form $d\Vol'$ induces a volume form $d\vol'_1$ on $\Omega_1 E_D(4)$ by the formula $\vol'_1(B)=\Vol'(C_1(B))$, for all  $B \subset \Omega_1 E_D(4)$. 
We endow $\Omega_1E_D(4)\times\Delta_{\sqrt[5]{\eps_0}}$ with the product measure $d\vol'_1\times \lambda_{\rm Leb}$, where $\lambda_{\rm Leb}$ is the Lebesgue measure on $\Delta_{\sqrt[5]{\eps_0}}$. Our goal now is to compare this measure and $F_1^*d\vol_1$.

Let $U \subset \Omega E_D(4)$ be an open subset which can be equipped with a system of coordinates by period mappings.
Consider a surface $(X_0,\omega_0)\in \Omega_1 E_D(4)^\odd\cap U$. Let $(\alpha_1,\beta_1,\alpha_2,\beta_2)$ be a symplectic basis of $H_1(X_0,\Z)^-$ such that $\langle \alpha_i,\beta_i\rangle=i$, and the cycles $\alpha_i$ are represented by the core curves of some parallel cylinders in $X_0$. By Proposition~\ref{prop:eigen:form:per:eq}, there is a matrix $A\in \Mb_2(\R)$ such that $(\omega_0(\alpha_2),\omega(\beta_2))= (\omega_0(\alpha_1),\omega(\beta_1))\cdot A$.
Since $\omega(\alpha_1)$ and $\omega(\alpha_2)$ are parallel, we must have
$A=\left(\begin{smallmatrix} a & b \\ 0  & d \end{smallmatrix}\right)$.
As a consequence, we get that
\begin{align*} 
\Aa(X,\omega_0)  & = \frac{\imath}{2}\int_{X_0}\omega_0\wedge\ol{\omega}_0 = \frac{\imath}{2}\cdot \sum_{k=1}^2\frac{1}{k}(\omega_0(\alpha_k)\ol{\omega}_0(\beta_k)-\omega_0(\beta_k)\ol{\omega}_0(\alpha_k))\\
& = K\cdot \Im(\ol{\omega}_0(\alpha_1)\omega_0(\beta_1))
\end{align*}
where $K$ is a  positive real number.

We can parametrize the neighborhood $B:=U\cap\Omega_1E_D(4)$ of $(X_0,\omega_0)$ by the parameters $(\theta, w) \in \S^1\times\Hbb$ (here $\S^1\simeq \R/(2\pi\Z)$) where
$$
\theta(X_0,\omega_0)= \arg(\omega_0(\alpha_1)) \quad \text{ and } \quad w(X_0,\omega_0)=e^{-\imath \theta(X_0,\omega_0)}\omega_0(\beta_1).
$$
Let us write $w=x+\imath y$, with $x,y\in \R, y >0$.
Then the condition $\Aa(X_0,\omega_0)=1$ implies that $\omega_0(\alpha_1)=e^{\imath\theta}/(K y)$.

\begin{Lemma}\label{lm:express:dvol:1:OmED:4}
In the system of coordinates $(\theta,w)$, we have
\begin{equation}\label{eq:express:dvol:1:OmED:4}
d\vol'_1=\frac{-d\theta dx dy}{8y^2}.
\end{equation}
\end{Lemma}
\begin{proof}
Since $C_1(B)$ is an open subset of $\Vb$, we have a system of local coordinates on $C_1(B)$ given by  $(z_1,w_1)$, where $z_1$ is the period of $\alpha_1$ and $w_1$ is the period of $\beta_1$. In these coordinates, the intersection form is given by
$$
\hh=\frac{\imath K}{2}\cdot( dz_1\otimes d\bar{w_1} - dw_1\otimes d\bar{z}_1).
$$
Therefore
$$
\Omega=\frac{K}{4}\cdot(dz_1\wedge d\bar{w}_1 -dw_1\wedge d\bar{z}_1).
$$
and
$$
\frac{\Omega^2}{2} =\frac{K^2}{16}dz_1d\bar{z}_1dw_1d\bar{w}_1.
$$
Since $C_1(B) \simeq (0;1]\times B$, we can also parametrize $C_1(B)$ by the parameters $(r,\theta,w) \in (0;1]\times\S^1\times \Hbb$  such that
$z_1=\frac{re^{\imath\theta}}{Ky}$ an $w_1=re^{\imath\theta}w_1$.
Let $\zeta=re^{\imath\theta}$, a quick calculation shows
$$
\frac{\Omega^2}{2}=\frac{|\zeta|^2d\zeta d\bar{\zeta}dw d\bar{w}}{8y^2} = \frac{-r^3drd\theta dx dy}{2y^2}
$$
By definition
$$
\vol'_1(B)=-\int_0^1 r^3dr\int_B\frac{d\theta dx dy}{2y^2}=-\int_B\frac{d\theta dx dy}{8y^2}.
$$
which means that
$$
d\vol'_1=-\frac{d\theta dx dy}{8y^2}.
$$
\end{proof}

\begin{Lemma}\label{lm:express:dvol:1:in:tub:neigh:OmED:4}
Let $(s,\phi)\in \R_{>0}\times\S^1$ be the polar coordinates on $\Delta^*_{\sqrt[5]{\eps_0}}$. Then we have
\begin{equation}\label{eq:express:dvol:1:in:tub:neigh:OmED:4}
F_1^*d\vol_1=\frac{50}{3} \cdot s^9 \cdot d\vol'_1\wedge(ds\wedge d\phi)
\end{equation}
on  $B\times\Delta^*_{\sqrt[5]{\eps_0}}\cap \cU_{\eps_0}$.
\end{Lemma}
\begin{proof}
Let $t = se^{\imath\phi}\in \Delta^*_{\sqrt[5]{\eps_0}}$ be a number such that $((X_0,\omega_0),t) \in \cU_{\eps_0}$ and  $(X_t,\omega_t,\sigma_t):=F_1(X_0,\omega_0,t)$.
By construction, we can consider $\alpha_1,\beta_1$ as elements of $H_1(X_t,\Z)$. 
We have $(\omega_t(\alpha_1),\omega_t(\beta_1))=(\omega_0(\alpha_1),\omega_0(\beta_1)$, and $\omega_t(\sigma_t)=t^5$.
We have a  local system coordinates $(z_1,w_1,z)$  in a neighborhood of $(X_t,\omega_t,\sigma_t)$ in $\Omega\tilde{E}^{(1)}_D(2,2)^\odd$, where for all $(X,\omega,\sigma)$, $z_1=\omega(\alpha_1), w_1=\omega(\beta_1), z=\omega(\sigma)$. In this system of coordinates, we have
$$
d\Vol = \frac{\Omega^2}{2}\wedge\left(\frac{\imath}{2}dz\wedge d\bar{z}\right) = \frac{K^2}{16}dz_1d\bar{z}_1dw_1d\bar{w}_1 \wedge\left(\frac{\imath}{2}dz\wedge d\bar{z}\right).
$$
Let $\tilde{B}$ be an open neighborhood of $(X_t,\omega_t,\sigma_t)$ in $\Omega_1 \tilde{E}^{(1)}_D(2,2)^\odd$.
By definition, $\vol_1(\tilde{B})=\Vol(C_1(\tilde{B}))$.
Since $F_1$ is a covering, we can use $(r,\theta,w,t)\in (0;1]\times\S^1\times\Hbb\times\Delta^*_{\sqrt[5]{\eps_0}}$ as a   local system of coordinates on $C_1(\tilde{B})$.
By the same calculations as in Lemma~\ref{lm:express:dvol:1:OmED:4} we get
\begin{align*}
d\Vol & = \frac{K^2}{16} dz_1d\bar{z}_1 dw_1d\bar{w}_1 \wedge\left(\frac{\imath}{2}dz\wedge d\bar{z}\right)\\
& = \frac{-r^3}{2y^2} dr d\theta dxdy \wedge \left(25r^2|t|^{8}\cdot\frac{\imath}{2}\cdot dt\wedge d\bar{t}\right)\\
& =\frac{-25r^5s^9}{2y^2}  dr d\theta dx dy ds d\phi.
\end{align*}
Therefore
\[
\Vol(C_1(\tilde{B}))  = \frac{-25}{2}\int_0^1 r^5dr \cdot \int_{\tilde{B}} \frac{s^9d\theta dx dy ds d\phi}{y^2}  = \frac{-25}{12}\int_{\tilde{B}} \frac{s^9 d\theta dx dy ds d\phi}{y^2}
\]
which means that
\[
d\vol_1=\frac{-25}{12}\cdot \frac{s^9 d\theta dx dy ds d\phi}{y^2}= \frac{50}{3}\cdot s^9 \cdot d\vol'_1 \wedge (ds\wedge d\phi).
\]
\end{proof}

\subsection*{Proof of Proposition~\ref{prop:vol:neigh:Omg:E:D:4}}
\begin{proof}
From Lemma~\ref{lm:express:dvol:1:in:tub:neigh:OmED:4}, we have
\[
\int_{\cU_{\eps_0}} F^*_1d\vol_1 =  \frac{50}{3}\int_{0}^{2\pi}d\phi \int_0^{\sqrt[5]{\eps_0}} s^9ds \int_{\Omega_1 E_D(4)}d\vol_1' = \frac{10\pi\eps_0^2}{3}\int_{\Omega_1 E_D(4)}d\vol'_1. 
\]	
By \cite[Th. 1.4]{Ng25} and Proposition~\ref{prop:int:Hodge:curv:on:Teich:curves}
$$
\int_{\Omega_1 E_D(4)}d\vol'_1 = -\frac{\pi^2}{2}\cdot c_1(\OO(-1))\cdot [W_D(4)] =\frac{\pi^2}{4}\chi(W_D(4)).
$$
Thus we have
$$
\int_{\cU_{\eps_0}} F^*_1d\vol_1 = \frac{5\pi^3\eps_0^2}{6}\chi(W_D(4))
$$	
as desired.
\end{proof}

\subsection{Proof of Theorem~\ref{th:Siegel:Veech}}
\begin{proof}
Assume that $4 \, | \, D$. For each $(X,\omega) \in \Omega_1 E_D(2,2)^\odd$, let $\Lambda_\omega^{(k)} \subset \C, \; k=1,2,3$, denote the set of periods of saddle connections in $X$ connecting the two zeros of  $\omega$ with multiplicity $k$.
Let $f_{\eps_0}: \C \to \R$ be the indicator function of the disc $\Delta(\eps_0)$, and $\hat{f}^{(k)}_{\eps_0}$ its Siegel-Veech transform with respect to the sets $\Lambda^{(k)}_\omega$.
By definition for all $(X,\omega) \in \Omega_1 E_D(2,2)^\odd$, $\hat{f}^{(k)}_{\eps_0}(X,\omega)$ counts the number of saddle connections with multiplicity $k$ of length at most $\eps_0$. We have
$$
\frac{1}{\vol_1(\Omega_1 E_D(2,2)^\odd)}\cdot\int_{\Omega_1 E_D(2,2)^\odd}\hat{f}^{(k)}_{\eps_0}d\vol_1 = c^{SV}_k(D)\pi\eps_0^2.
$$

Let $\sigma$ be a saddle connection of multiplicity one on $(X,\omega)$ such that $|\sigma| < \eps_0$. By Corollary~\ref{cor:collapse:s:c:non:par:abs:per}, if $\omega(\sigma)$ is not parallel to any vector in ${\rm Per}(\omega)$ 
then $\sigma$ can be collapsed and we  get a surface in $\Omega_1 E_D(4)$.  This means that $(X,\omega,\sigma) \in \Upsilon_1\circ F_1(\cU_{\eps_0})$, where $\Upsilon_1: \Omega_1 \tilde{E}^{(1)}_D(2,2)^\odd \to \Omega_1 E_D(2,2)^\odd$ is the map consisting of forgetting the marked saddle connection.
Thus $F_1(\cU_{\eps_0})$ contains a full measure subset of ${\rm supp}(\hat{f}^{(1)}_{\eps_0})$.   For all $(X,\omega)$ in this subset  $\hat{f}^{(1)}_{\eps_0}(X,\omega)$ counts the preimages of $(X,\omega)$  by $\Upsilon_k$ in $F_1(\cU_{\eps_0})$. Since $\deg F_1=2$, it follows
$$
\int_{\Omega_1E_D(2,2)^\odd}\hat{f}_{\eps_0}^{(1)}d\vol_1=\int_{F_1(\cU_{\eps_0})}d\vol_1=\frac{1}{2}\int_{\cU_{\eps_0}} F_1^*d\vol_1.
$$
It follows from Proposition~\ref{prop:vol:neigh:Omg:E:D:4} that
\begin{align*}
\int_{\Omega_1E_D(2,2)^\odd}\hat{f}_{\eps_0}^{(1)}d\vol_1 
&= \frac{5\pi^3\eps_0^2}{12}\chi(W_D(4)).
\end{align*}
As a consequence, we get
$$
c^{SV}_1(D)=\frac{5\pi^2\chi(W_D(4))}{12\vol_1(\Omega_1E_D(2,2)^\odd)}= \frac{5\pi^2\chi(W_D(4))}{12\mu(\Pb\Omega E_D(2,2)^\odd)}.
$$
By Theorem~\ref{th:vol:Omg:E:D:2:2}, we know that
$$
\mu(\Pb\Omega E_D(2,2)^\odd) =\frac{\pi^2}{36}(\chi(W_D(2))+b_D\chi(W_{D/4}(2))+9\chi(W_D(0^3))).
$$
Therefore
$$
c^{SV}_1(D)=\frac{15\chi(W_D(4))}{\chi(W_D(2))+b_D\chi(W_{D/4}(2))+9\chi(W_D(0^3))}.
$$
The proofs for $c^{SV}_2(D)$ and $c^{SV}_3(D)$ are similar. 

\medskip 

In the case $D \equiv 1 \; [8]$, one needs to distinguish the components $\Omega E_{D+}(2,2)^\odd$ and $\Omega E_{D-}(2,2)^\odd$.
By definition the closure of $\Pb\Omega E_{D+}(2,2)^\odd$ contains the curves $W_{D-}(4)$, $W^*_{D+}(2)$ and $W_{D+}(0^3)$. It is shown in \textsection \ref{sec:W:Teich:curves} that $W^*_{D+}(2)$ is a double cover of $W_{D+}(2)$. Therefore we have $\chi(W^*_{D+}(2))=2\chi(W_{D+}(2))$.
Similarly, the closure of $\Pb\Omega E_{D-}(2,2)^\odd$ contains the curves $W_{D+}(4)$, $W^*_{D-}(2)$ and $W_{D-}(0^3)$, and we have $\chi(W^*_{D-}(2))=2\chi(W_{D-}(2))$.
By the results of Bainbridge \cite{Bai:GT}, M\"oller \cite{Mo14}, and Corollary~\ref{cor:euler:char:W:Dpm:0:equal}, we know that
\[
\chi(W_{D+}(\kappa))=\chi(W_{D-}(\kappa))=\frac{\chi(W_D(\kappa))}{2}
\]
for all $\kappa \in \{4,2,0^3\}$.
Thus the desired conclusions follow from Theorem~\ref{th:vol:Omg:E:D:2:2}. The cases $k\in\{2,3\}$ follow from similar arguments.
\end{proof}
\appendix

\section{Degenerate Prym eigenforms}\label{sec:degenerate:eigen:form}
\subsection{Level structure and twisted differentials}\label{subsec:twisted:diff}
By definition, $\Omega\cX_D$ is contained in the stratum $\Omega'\cB_{4,1}(2,2) \subset \Omega'\cB_{4,1}$ which consists of tuples $(C,p_1,\dots,p_5,p'_5,\tau,\xi)$ such that $\div(\xi)=2 p_5+2p'_5$. Therefore, $\partial\ol{\cX}_D$ is contained in the closure  $\Pb \Omega'\ol{\cB}_{4,1}(2,2)$ in $\Pb\Omega'\ol{\cB}_{4,1}$. 
An important tool for our classification of the points in $\partial \ol{\cX}_D$ is the following  result
\begin{Theorem}[Bainbridge-Chen-Gendron-Grushevsky-M\"oller \cite{BCGGM1, BCGGM2}]\label{th:twisted:diff}
	Let $(C,p_1,\dots,p_5,p'_5,\tau,\xi)$ be an element of $\Omega' \ol{\cB}_{4,1}(2,2)$.  
	Denote the irreducible components of $C$ by $C_j, \; j \in J$.
	Then there exists on each  $C_j$ a meromorphic Abelian differential $\xi_j$, and there is a level structure on the set of components of $C$, that is an assignment to each $C_j$ a level $\ell_j \in \Z_{\leq 0}$,  such that
	\begin{itemize}
		\item[(a)] $\xi$ vanishes identically on all components of level $\leq -1$, and if $C_j$ is a component of level $0$ then $\xi_j=\xi_{\left| C_j\right.}$
		
		\item[(b)] For all $j\in J$, if $p_5 \in C_j$ (resp. $p'_5 \in C_j$) then $p_5$ (resp. $p'_5$) is a double zero of $\xi_j$, all the other zeros and poles of $\xi_j$ are located at the nodes incident to $C_j$.
		
		\item[(c)] If $\tau(C_j)=C_{j'}$ a then $C_j$ and $C_{j'}$ have the same level and we have $\tau^*\xi_{j'}=-\xi_j$.
		
		\item[(d)] The family $\{(C_j,\xi_j), \; j\in J\}$, which is called a {\em twisted differential}, is compatible with the level structure $\{\ell_j, \, j\in J\}$ which means the following: let $q$ be a node of $C$ which is incident to the irreducible components $C_j$ and $C_{j'}$ (it is possible that $j=j'$). Let $k_j$ (resp. $k_{j'}$) be the order of $\xi_j$ (resp. of $\xi_{j'}$) at $q$.
		Then we must have $k_j+k_{j'}=-2$,  $\ell_j >  \ell_{j'}$ implies $k_j > k_{j'}$, and if $\ell_j=\ell_{j'}$ then $k_j=k_{j'}=-1$ and
		$$
		\res_{q}(\xi_j)+\res_{q}(\xi_{j'})=0.
		$$
		
		\item[(e)] For any negative integer $L$, let $C^0_{>L}$ be  a connected component of  the union of all irreducible components with level $>L$. Let $q_1,\dots,q_r$  the nodes between $C^0_{>L}$ and the components of level $L$.
		For each $q_i$, let $C_{\sigma(i)}$ be the component of level $L$ that contains $q_i$. Note that by (d) $q_i$ is a pole of order at least two of $\xi_{\sigma(i)}$. Then  we must have
		\begin{equation}\label{eq:GRC}
			\sum_{i=1}^r\res_{q_i}\xi_{\sigma(i)}=0.
		\end{equation}
	\end{itemize}
\end{Theorem}
\begin{Remark}\label{rk:GRC}
	The data of $\{(C_j, \xi_j), \; j \in J\}$  is called a {\em twisted Abelian differential} and  property (e) is called the {\em Global Residue Condition}.
\end{Remark}

\subsection{Characterizing differentials in the boundary of Prym eigenform loci}\label{subsec:char:bdry:form}\hfill \\
We now prove a series of results providing characterizing properties of Abelian differentials in the boundary of $\ol{\cX}_D$. These characterizations will be used in the proof of Theorem~\ref{th:bdry:eigen:form:H22}.

Let $\pp:=(C,\underline{p},\tau, [\xi])$ be a point in $\partial\Pb\Omega'\ol{\cB}_{4,1}(2,2)$, where $\ul{p}=\{p_1,\dots,p_5,p'_5\}$. Recall that by definition
\begin{itemize}
	\item[$\bullet$] $(C,\ul{p})$ is a pointed nodal stable curve,
	
	\item[$\bullet$] $\tau$ is an  involution of $C$ that fixes each of the points in $\{p_1,\dots,p_4\}$, and exchanges $p_5$ and $p'_5$,
	
	\item[$\bullet$] $\xi$ is a non-trivial holomorphic section of the dualizing sheaf $\omega_C$ satisfying  $\tau^*\xi=-\xi$
\end{itemize}
Denote by $C_j, \; j\in J$ the irreducible components of $C$.
In what follows, by a subcurve of $C$ we mean a union of some of its irreducible components. 
Let $\{(C_j, \xi_j), \, j \in J\}$ be a twisted differential on $C$ (cf. Theorem~\ref{th:twisted:diff}). Consider a node $q$ of $C$. If $q$ is a self-node of an irreducible component $C_j$, then $\xi_j$ must have simple pole at $q$.
In particular, if $C_j$ has level zero, since $\xi_j=\xi_{\left|C_j\right.}$, $q$ must be a pair of simple poles with opposite residues of $\xi$.
In the case $q$ is incident to two distinct irreducible components, condition (d) of Theorem~\ref{th:twisted:diff} implies that either $\xi$ has simple poles at $q$, or at least one of the two components has negative level.

\begin{Proposition}\label{prop:node:fixed:by:invol}
	Let $p$ be a node of $C$ which is fixed by the Prym involution. Then $\xi$ cannot have simple pole at $p$. As a consequence, $\xi$ must vanish identically on at least one of the two irreducible components meeting at this node.
\end{Proposition}
\begin{proof}
	A neighborhood of this node in $C$ is isomorphic to $\{xy=0, \, (x,y) \in \C^2, \, |x| < \eps, \, |y| < \eps\}$, for some real positive number $\eps$. Since the Prym involution preserves this node, its action is given by $\tau: (x,y) \mapsto (-x,-y)$. Now, $\xi$ is given by $f(x)dx/x$ in the disc $\Delta_\eps\times\{0\}$, where $f$ is a holomorphic function. By assumption,  $\tau^*\xi=-\xi$. Thus we must have $f(-x)=-f(x)$, which implies that $f(0)=0$.
	Hence $\xi$ does not have a simple pole at $p$.
	It follows that at least  one of the components of $C$ containing $p$ has negative level. By Theorem~\ref{th:twisted:diff} (a), $\xi$ vanishes identically on this component.
\end{proof}

Let $S$ be a reference smooth curve  in $\cB_{4,1}$. Denote by $C^*$ the complement of the nodes in $C$. 
Note that $C^*$ is $\tau$-invariant.
There is an embedding $\varphi: C^* \hookrightarrow S$  conjugating the actions of the Prym involutions. The complement of $\varphi(C^*)$ in $S$ is a disjoint union of simple closed curves that correspond to the nodes of $C$. By Meyer-Vietoris, the induced morphism $\varphi_*: H_1(C,\Z)\to H_1(S,\Z)$ is surjective.  
Define $H_1(C^*,\Z)^-=\{c\in H_1(C^*,\Z), \; \tau_*c =-c\}$.
We have $\varphi_*(H_1(C^*,\Z)^-)=H_1(S,\Z)^-$.

\begin{Proposition}\label{prop:non:collapse:sub:curve}
	Let $\gamma$ be a cycle representing an element of $H_1(C^*,\Z)^-$ such that $\varphi_* \gamma \neq 0 \in H_1(X,\Z)^-$. If $\pp\in \ol{\cX}_D$ then $\int_\gamma\xi\neq 0$.
\end{Proposition}
\begin{proof}
	Since $\xi\neq 0$, there exists an element $\alpha\in H_1(C^*,\Z)^-$ such that $\int_{\alpha}\xi \neq 0$. 
	Note that we must have $\varphi_*\alpha\neq 0\in H_1(X,\Z)^-$.
	There is a symplectic basis $\{a_1,b_1,a_2,b_2\}$ of $H_1(X,\Q)^-$, where $a_1=\varphi_*\alpha$, and $\langle a_i,b_i\rangle=1$. For all $\xx=(X,\ul{x},\tau_X,[\omega]) \in \Pb\Omega'\ol{\cB}_{4,1}(2,2)$ close enough to $\pp$, there is a collapsing map $\phi: X\to C$ which contracts some simple closed  curves on $X$ to the nodes in $C$ such that $\phi$ restricts to a homemorphism from $\phi^{-1}(C^*)$ onto $C^*$.
	There is a homeomorphism $f: X \to S$ whose restriction to $\phi^{-1}(C^*)$ equals $\varphi\circ\phi_{|\varphi^{-1}(C^*)}$. 
	Note that the homotopy equivalence class of $f$ is only defined up to Dehn twists about curves that are contracted to the nodes of $C$.   
	
	Assume that $\pp \in \ol{\cX}_D$. Then we can find a sequence $\{\xx_n\}_{n\in \N} \subset \cX_D$, where $\xx_n=(X_n,\ul{x}_n, \tau_{X_n}, [\omega_n])$, converging to $\pp$ such that for all $n\in \N$, there is a distinguished homeomorphism $f_n: X_n \to S$ as above. 
	We can identify $H_1(X_n,\Q)^-$ with $H_1(S,\Q)^-$ using $f_n$.
	In particular, we can consider $\varphi_*\gamma$ as an element of  $H_1(X_n,\Z)^-$. 
	By Lemma~\ref{lm:D:no:square:hol:map:inj} there exists $(x,y) \in K_D^2$, where $K_D=\Q(\sqrt{D})$ such that
	\begin{equation}\label{eq:period:vanishing:cycle:1}
		\omega_n(\varphi_*\gamma)=x\cdot \omega_n(a_1)+y\cdot\omega_n(b_1), \quad \hbox{for all $n\in \N$.}
	\end{equation}
	Note that we also have
	$$
	\Aa(X,|\omega_n|) =M\cdot\omega_n(a_1) \wedge \omega_n(b_1)
	$$
	for some constant $M \in \R^*$ independent of $n$. 
	
	One can define a local section for the tautological line bundle $\OO(-1)$ in a neighborhood of $\pp$ by the condition $\omega(a_1)=1$ for all $\xx=(X,\ul{x},\tau_X,\omega)$ close enough to $\pp$. Thus we can suppose that $\omega_n(a_1)=1$ for all $n\in\N$.
	
	If $\int_\gamma\xi=0$, then as $\xx_n$ converges to $\pp$, we get that $\omega_n(\varphi_*\gamma) \overset{n \to \infty}{\rightarrow} 0$.   
	It follows from \eqref{eq:period:vanishing:cycle:1} that we have
	\begin{equation*}
		\Im(\omega_n(\varphi_*\gamma)) = \omega_n(a_1)\wedge \omega_n(\varphi_*\gamma) = y\omega_n(a_1)\wedge \omega_n(b_1) = \frac{y}{M}\Aa(X_n,|\omega_n|)=\frac{y}{M}\cdot||\omega_n||^2
	\end{equation*}
	where $||\omega_n||$ is the Hodge norm of $\omega_n$.
	Since $||\omega_n|| \overset{n\to \infty}{\to} ||\xi|| >0$ while $\omega_n(\varphi_*\gamma)\overset{n\to \infty}{\to} 0$, we must have $y=0$, which means that
	$$
	\omega_n(\varphi_*\gamma)= x\cdot\omega_n(a_1).
	$$
	Again, since $\omega_n(a_1)=1$, we also have $x=0$, which means that $\omega_n(\varphi_*\gamma)=0$ for all $n$. But by Lemma~\ref{lm:D:no:square:hol:map:inj}  we must have $\omega_n(\varphi_*\gamma)\neq 0$. Thus we have a contradiction which proves the proposition. 
\end{proof}

\begin{Proposition}\label{prop:s:poles:at:nodes:exchanged}
	Assume that $C$ is the union of two {\em connected} subcurves $C'$ and $C''$ invariant by $\tau$,  which intersect each other at a pair of  permuted nodes. If $\pp$ is  contained in $\ol{\cX}_D$, then $\xi$ must have simple poles at these two nodes.
\end{Proposition}
\begin{proof}
	Let $q$ and $q'$ be the nodes between $C'$ and $C''$. 
	Consider a point $\xx=(X,\ul{x},\tau_X,[\omega]) \in \cX_D$ close enough to $\pp$. Let $\gamma$ and $\gamma'$  be the simple closed curves on $X$ that are contracted to the nodes $q$  and $q'$ respectively. We choose the orientation of $\gamma$ and $\gamma'$ such that $\tau_{X*}\gamma=\gamma'$.
	Note that we have $\gamma+\gamma'=0\in H_1(X,\Z)$, therefore   $\gamma\in H_1(X,\Z)^-$. 
	
	If $\xi$ does not have simple poles at $q$, then $\xi(\gamma)=0$. We then get a contradiction by   Proposition~\ref{prop:non:collapse:sub:curve}. Therefore, $\xi$ must have simple poles at $q$ and $q'$.
\end{proof}

\begin{Proposition}\label{prop:bdry:form:dec:not:eigen:form}
	Assume that $C$ is the union of two subcurves $C',C''$ (not necessarily connected) both of which have (arithmetic) genus $\geq 1$ and are invariant under $\tau$. If $\xi$ vanishes identically  on either $C'$ or $C''$ then  $\pp$ is not contained in $\partial\ol{\cX}_D$.
\end{Proposition}
\begin{proof}
	Assume that $\pp \in \ol{\cX}_D$.
	The assumption that both $C'$ and $C''$ have genus at least one implies that there are at most two nodes between $C'$ and $C''$. As a consequence, up to a relabeling we have the following configurations
	\begin{itemize}
		\item[(i)]  $C'$ is a genus $1$ curve,  $C''$ is a genus $2$ curve,  and there is a unique node between $C'$ and $C''$.
		
		\item[(ii)] $C'$ is a genus $1$ curve, $C''$ is a disjoint union of two isomorphic genus $1$ curves, $C'$ and $C''$ intersect at two nodes that are     permuted by $\tau$.
		
		\item[(iii)] Both $C'$ and $C''$ are genus one curves, and $C'$ intersects $C''$ at two nodes, both of which are fixed by $\tau$.
		
		\item[(iv)]  Both $C'$ and $C''$ are genus one curves, and $C'$ intersects $C''$ at two nodes that are exchanged by $\tau$.

	\end{itemize}
	Let $\xi':=\xi_{\left|C'\right.}$ and $\xi'':=\xi_{\left|C''\right.}$. By assumption either $\xi'\equiv 0$ or $\xi'' \equiv 0$.
	Suppose that we are in cases (i), (ii), or (iii).
	Consider a point $\xx=(X,\ul{x},\tau_X, [\omega]) \in \cX_D$ close enough to $\pp$.
	Let $c$ denote the union of the simple closed curves on $X$  that are contracted to the nodes between $C'$ and $C''$.
	Let $X'$ (resp. $X''$) be the component of $X-c$ that corresponds to $C'$ (resp. to $C''$).
	One can specify a symplectic basis $\{\alpha',\beta',\alpha'',\beta''\}$ of $H_1(X,\Z)^-$ with $\alpha', \beta'$ represented by cycles on $X'$ and $\alpha'', \beta''$  represented by cycles in $X''$.
	By Proposition~\ref{prop:eigen:form:per:eq} there is an invertible matrix $B\in \Mb_2(\Q(\sqrt{D}))$ such that
	$$
	(\omega(\alpha''), \; \omega(\beta''))=(\omega(\alpha'), \; \omega(\beta'))\cdot B
	$$
	We have
	$$
	\Aa(X,\omega)=||\omega||^2=\frac{\omega(\alpha')\wedge\omega(\beta')}{m'}+\frac{\omega(\alpha'')\wedge\omega(\beta'')}{m''},
	$$
	where $m'=\langle \alpha',\beta'\rangle, \; m''= \langle \alpha'',\beta''\rangle$. Note that we have $|\det(\omega(\alpha''),\omega(\beta''))|=|\det(\omega(\alpha'),\omega(\beta'))|\cdot|\det B|$. Therefore
	$$
	||\omega||^2=K\cdot |\omega(\alpha')\wedge\omega(\beta')|,
	$$
	where $K$ is a positive real constant.	
	If $\xi''\equiv 0$ then  as $\xx$ converges to $\pp$, $(\omega(\alpha''),\omega(\beta''))$ converges to $(0,0)$, while
	$$
	|\omega(\alpha') \wedge\omega(\beta')| \overset{\xx\to \pp}{\longrightarrow} \frac{||\xi||^2}{K} >0.
	$$
	Therefore we get a contradiction. By the same argument, we also get a contradiction if $\xi'\equiv 0$. Thus the proposition is proved for the first three cases.
	
	\medskip
	
	In the case (iv), by Proposition~\ref{prop:s:poles:at:nodes:exchanged}, if $\pp\in \ol{\cX}_D$ then $\xi$ must have simple poles at the nodes between $C'$ and $C''$, which means that  $\xi'\not\equiv 0$ and $\xi''\not\equiv 0$. We thus have a contradiction and the proposition follows.
\end{proof}

\begin{Corollary}\label{cor:dec:two:tori:no:eigenform}
	Assume that $C$ has two connected subcurves of genus $1$ intersecting at two nodes both are fixed  by the Prym involution. Then $\pp \not\in\ol{\cX}_D$.
\end{Corollary}
\begin{proof}
	By Proposition~\ref{prop:node:fixed:by:invol}, $\xi$ must vanish identically in one of the two irreducible components. We then conclude by Proposition~\ref{prop:bdry:form:dec:not:eigen:form}.
\end{proof}

\begin{Proposition}\label{prop:s:poles:at:two:nodes:not:in:XD}
	If $\xi$ has simple poles at one pair of nodes that are exchanged by the involution and is holomorphic at all the other nodes, then $\pp\not\in\ol{\cX}_D$.
\end{Proposition}
\begin{proof}
	Consider a point $\xx:=(X,\ul{x}, \tau_X, [\omega])$ in $\cX_D$ close enough to $\pp$.
	Assume that $\xi$ has simple poles at the pair of nodes $p',p''$ permuted by $\tau_X$.
	It is a well known fact (see for instance \cite[Th. 5.5]{Bai:GT}) that for each node of $C$, we have a corresponding cylinder with large height on $(X,\omega)$. 
	Denote by $A'$ (resp.  $A''$)  the cylinder that corresponds to $p'$ (resp. to $p''$) in $X$.
	Since these two cylinders are permuted by the Prym involution $\tau_X$, they are parallel and have the same height.
	It may happen that there is a cylinder $A$ that contains both $A'$ and $A''$.  This happens when $p'$ and $p''$ are contained in an irreducible component isomorphic to $\Pb^1$ invariant by $\tau_X$.  
	
	We can suppose that $A'$ and $A''$ are both horizontal.
	Since $(X,\omega)$ is completely periodic, it is decomposed into a union of horizontal cylinders. Using Proposition~\ref{prop:stable:cyl:dec}, one can assume that the corresponding cylinder decomposition is stable.
	Thus, the associated cylinder diagram of $(X,\omega)$ is given by one of the four cases in Proposition~\ref{prop:stable:cyl:dec:models:H22}.
	Recall that $h_1,\dots,h_4$ are respectively the heights of $C_1,\dots,C_4$ in all the diagrams. By convention $C_3$ and $C_4$ are permuted by $\tau_X$, while $C_1$ are $C_2$ are invariant. In particular, we have $h_3=h_4$. 
	
	Since set of cylinder diagrams and the set of prototypes $\Pcal_{D,\cyl}$ is finite, one can find a sequence $\{\xx_n\}_{n\in \N} \subset \cX_D$, where $\xx_n=(X_n,\ul{x_n},\tau_{X_n},[\omega_n])$, converging to $\pp$  such that for all $n\in \N$, the surface $(X_n,\omega_n)$ is horizontally periodic with a fixed stable cylinder diagram and the same associated prototypes $\frakp =(a,b,d,e)\in \Pcal_{D,\cyl}$.   
	
	For concreteness, let us suppose that the cylinder decomposition of $(X_n,\omega_n)$ in the horizontal direction is given by Case I.A. 
	Denote by $C_{i,n}$, $i=1,\dots,4$, the horizontal cylinders in $X_n$, where $C_{i,n}$ corresponds to $C_i$ in Proposition~\ref{prop:cyl:dec:length:ratios:22}. Let $h_{i,n}$ be the height of $C_{i,n}$. It follows from Proposition~\ref{prop:cyl:dec:length:ratios:22} (i), that we have
	$$
	\frac{h_{2,n}+h_{3,n}}{h_{1,n}+h_{2,n}} =\frac{h_{2,n}+h_{4,n}}{h_{1,n}+h_{2,n}} =\frac{a}{\lambda}  
	$$
	In particular, the ratio $(h_{2,n}+h_{3,n})/(h_{1,n}+h_{2,n})$ is independent of $n$. 
	The assumption implies that one of the sequences $\{h_{1,n}\}, \{h_{2,n}\}, \{h_{3,n}\}$ tends to $+\infty$, while the other two are bounded. In all cases we have 
	$$
	\lim_{n\to \infty}\frac{h_{2,n}+h_{3,n}}{h_{1,n}+h_{2,n}} \in \{0,1,\infty\}
	$$
	Thus we must have $a/\lambda \in \{0,1,\infty\}$. But since $D$ is not a square $\lambda \not\in \Q$. Thus $a/\lambda \not\in \{0,1,\infty\}$, and we have a contradiction, which proves the proposition in this case.
	
	The proof of the proposition for the other cylinder diagrams follows the same line.
\end{proof}

We will also need the following 

\begin{Proposition}\label{prop:ab:diff:dble:zero:Prym:inv}
	Let $(X,\omega)$ be a holomorphic Abelian differential where $X$ is a Riemann surface of genus two. Assume that $X$ admits an involution $\tau$ with $2$ fixed points such that $\tau^*\omega=-\omega$. Then $\omega$ must have two simple zeros.
\end{Proposition}
\begin{proof}
	If $\omega$ has a double zero, denoted by $x_0$, then this zero must be a fixed point of $\tau$. Note that $x_0$ is also a Weierstrass point of $X$. Therefore, the hyperelliptic involution $\iota$ of $X$ also fixes $x_0$. It follows that $\iota\circ\tau$ is identity in a neighborhood of $x_0$. As a consequence $\iota\circ\tau=\id_X$, and hence $\tau=\iota$. But $\iota$ has $6$ fixed points, while $\tau$ only has two. Therefore we get a contradiction.
	
	Here is an alternative argument. Let $Y:=X/\langle\tau\rangle$. Then $Y$ is a torus. Since $\tau^*\omega^2=\omega^2$, $\omega^2$ is the pullback of a quadratic differential $\eta$ on $Y$ which has one simple pole and one simple zero.
	Since the canonical line bundle of $Y$ is trivial, this means that there is a holomorphic map of degree $1 $ from $Y$ onto $\Pb^1$, which is impossible.
\end{proof}

%

\section{Proof of Theorem~\ref{th:bdry:eigen:form:H22}}
\label{sec:prf:bdry:egein:form:H22}
Our goal in this section is to prove Theorem~\ref{th:bdry:eigen:form:H22} which classifies the strata of $\partial\ol{\cX}_D$.
Throughout this section, $\pp:=(C,p_1,\dots,p_5, p'_5, \tau, [\xi])$ will be an element of $\Pb\Omega'\ol{\cB}_{4,1}(2,2)$.
Let $(E,q_1,\dots,q_5)$ be the image of $\pp$  in $\ol{\cM}_{1,5}$, that is $E:=C/\langle\tau\rangle$,  $q_i$ is the image of $p_i$ for $i=1,\dots,4$, and $q_5$ is the image of $\{p_5, p'_5\}$ under the natural projection $C \to E$.
We will analyze the properties of $\pp$ following the stratum of $\partial\ol{\cM}_{1,5}$ to which $(E,q_1,\dots,q_5)$ belongs.

\subsection{Generalities on topology of the stable curves in the boundary of $\ol{\cX}_D$}\label{subsec:generalities:bdry}
By definition, every point in $\partial\ol{\cX}_D$ is mapped to a point in the boundary $\partial \ol{\cM}_{1,5}:=\ol{\cM}_{1,5}-\cM_{1,5}$ of $\ol{\cM}_{1,5}$.  We have a stratification of $\partial\ol{\cM}_{1,5}$, where each stratum parametrizes the set of stable curves having the same topological characteristics ({\em i.e} the same dual graph).

The topological properties of a stable curve $(E,q_1,\dots,q_5)\in \ol{\cM}_{1,5}$ are however not enough to determine the topology of its admissible double cover. The reason is that the preimage of a node of $E$ may contains one or two nodes of the double cover. To determine the numbers of nodes in the preimages of the nodes of $E$, one needs extra data coming from a realization of $E$ as a degeneration of a reference torus $E_0$  with five  marked points denoted by $e_1,\dots,e_5$.

Fix a group morphism $\varrho: \pi_1(E_0-\{e_1,\dots,e_4\})\to \Z/2\Z$ that maps a loop homotopic to the boundary of a small disc about $e_i$ to $1\in \Z/2\Z$, for all $i=1,\dots,4$.
Let $C_0^*$ denote the double cover of $E^*_0:=E_0-\{e_1,\dots,e_4\}$ associated  to the kernel of $\varrho$.
Then $C^*_0$ can be identified with $C_0-\{p_1,\dots,p_4\}$, where $C_0$ is a compact genus $3$ surface, and  $p_1,\dots,p_4$ are $4$ distinct points on $C_0$. The covering map $f: C^*_0 \to E^*_0$ extends to a ramified covering from $C_0$ onto $E_0$ branched over $e_1,\dots,e_4$.

Since $\Z/2\Z$ is Abelian, the image of a loop in $\pi_1(E^*_0)$ by $\varrho$ depends only  on its conjugacy class. This means that $\varrho$ factors through a morphism $\bar{\varrho}: H_1(E^*_0,\Z) \to \Z/2\Z$. The preimage of a simple closed curve $c$ on $E_0^*$ has one component if $\bar{\varrho}(c)=1$, and two components if $\bar{\varrho}(c)=0$.

It is a well known fact that topologically $(E,q_1,\dots,q_5)$ can be obtained from $(E_0,e_1,\dots,e_5)$  by pinching some simple closed curves in $E_0-\{e_1,\dots,e_5\}$ that become  nodes in $E$.
The number of points in the preimage of a node in $E$ is equal to the number of components of the preimage of the corresponding closed curve in $E_0$.

We will call a node of $E$ {\em separating} (resp. {\em non-separating}) if the corresponding curve on $E_0^*$ is separating (resp. non-separating).
Consider an essential simple closed curve $c$ on $E^*_0$. Since $E_0$ is a torus,  if $c$ is separating then it must bound a disc in $E_0$.
In this case we have $\bar{\varrho}(c) = r \mod 2$, where $r$ is the number of points in $\{e_1,\dots,e_4\}$ that are contained this disc.
Let $n_c$ be the node of $E$ corresponding to $c$. Then $n_c$ is the intersection of two subcurves of $E$, one of which has genus $0$, the other one has genus $1$.  The number of nodes in the preimage of $x_c$ is then determined by the number of points in $\{q_1,\dots,q_4\}$ that are contained in the genus $0$ component.

In the case $c$ is  non-separating, $\bar{\varrho}(c)$ can be $0$ or $1$. However, if we have a family $\{c_1,\dots,c_k\}$ of pairwise disjoint non-separating curves on $E_0-\{e_1,\dots,e_4\}$, then all the values $\bar{\varrho}(c_i), \; i=1,\dots,k$,  can be computed from a single value, say $\bar{\varrho}(c_1)$. This is because the complement of the union $c_1\cup\dots\cup c_k$ in $E_0-\{e_1,\dots,e_4\}$ is a union of annuli  with punctures. This means that the numbers of nodes in the preimages of all non-separating nodes of $E$ are determined once this number is known for a chosen one.

It turns out that in most cases, the numbers of points in the preimages of the nodes of $E$ are enough for us to get the topological type of the admissible double cover of $E$.

\subsection{Case $E$ has one node}\label{sec:str:1:node}
We will show
\begin{Proposition}
	Assume that $\pp\in \ol{\cX}_D$ and the curve $E$ has only one node. Then $C$ has two irreducible components, denoted by $C_0$ and $C_1$, meeting at one node such that
	\begin{itemize}
		\item[$\bullet$] $C_0$ is isomorphic to $\Pb^1$, contains $\{p_5,p'_5\}$  and one point in $\{p_1,\dots,p_4\}$.
		
		\item[$\bullet$] $C_1$ is a Riemann surface of genus three, and contains three points in $\{p_1,\dots,p_4\}$.
		
		\item[$\bullet$] $\xi$ vanishes identically  on $C_0$, and $(C_1, \xi_{\left|C_1\right.}) \in \Omega' \cB_{4,1}(4)$.
	\end{itemize}
\end{Proposition}
\begin{proof}
	Let $q$ be the unique node of $E$. This node can be separating or not. Assume first that $p$ is separating. In this case $E$ has two irreducible components denoted by $E_0$ and $E_1$, where $E_0$ is isomorphic to $\Pb^1$ and $E_1$ is an elliptic curve.
	Let $C_0$ and $C_1$ be respectively the preimages of $E_0$ and $E_1$ in $C$.
	Let $r:=\# E_0\cap \{q_1,\dots,q_4\}$
	\begin{itemize}
		\item[$\bullet$] If $r=1$, then $q_5\in E_0$,  $C_0$ is also isomorphic to $\Pb^1$, $C_1$ is a smooth curve of genus $3$, and $C_0$ meets $C_1$ at a node fixed by $\tau$. The conclusions of the proposition the follows from Theorem~\ref{th:twisted:diff}.
		
		\item[$\bullet$] If $r=2$ then $C_0$ is also isomorphic to $\Pb^1$, $C_1$ is a smooth curve of genus $2$, and $C_0$ meets $C_1$ at two nodes exchanged by $\tau$. It follows from Proposition~\ref{prop:s:poles:at:nodes:exchanged} that $\xi$ has simple poles at these two nodes. But since $\xi$ is holomorphic outside of these nodes this case is excluded by Proposition~\ref{prop:s:poles:at:two:nodes:not:in:XD}.
		
		\item[$\bullet$] If $r=3$, then $C_0$ is a smooth curve of genus one, $C_1$ is a smooth curve of genus two, and $C_0$ meets $C_1$ at a node fixed by $\tau$. In this case either $\xi_{\left|C_0\right.} \equiv 0$ or $\xi_{\left|C_1\right.} \equiv 0$. By Proposition~\ref{prop:bdry:form:dec:not:eigen:form} this is impossible.
		
		\item[$\bullet$] If $r=4$, then $C_0$ is a smooth curve of genus one, $C_1$ is either a smooth curve of genus one, or a disjoint union of two isomorphic curves of genus one. In both cases, $C_0$ meets $C_1$ at two nodes exchanged by $\tau$. In the former case, $\xi$ has simple poles at the nodes (by Proposition~\ref{prop:s:poles:at:nodes:exchanged}). But since $\xi$ is holomorphic elsewhere this contradicts Proposition~\ref{prop:s:poles:at:two:nodes:not:in:XD}. In the latter case, either $\xi_{\left|C_0\right.}\equiv 0$ or $\xi_{\left|C_1\right.}\equiv 0$. Thus this case is ruled out by Proposition~\ref{prop:bdry:form:dec:not:eigen:form}.
	\end{itemize}
	
	In the case $q$ is a non-separating node, the preimage of $q$ in $C$ must consist of two nodes exchanged by $\tau$. By Theorem~\ref{th:twisted:diff}, $\xi$ must have simple poles at those nodes. But this is again ruled out by Proposition~\ref{prop:s:poles:at:two:nodes:not:in:XD}.
	This completes the proof of the proposition.
\end{proof}

\subsection{ Case $E$ has two nodes}\label{subsec:bdry:str:2:nodes}
Suppose now that the curve $E$ has two nodes. We have several configurations
\subsubsection{Case two separating nodes}\label{subsec:2:sep:nodes}
In this case $E$ has three irreducible components, two of which are isomorphic to $\Pb^1$, and the third one is an elliptic curve. We denote the $\Pb^1$ components by $E'_1$ and $E'_2$, and the elliptic component by $E''$. We also denote the union of $E'_1$ and $E'_2$ by $E'$.
Let $n_i:=|E'_i\cap\{q_1,\dots,q_4\}|, \; i=1,2$, and $n':=n_1+n_2$.
Denote the preimages of $E'_1, E'_2, E', E''$ in $C$ by $C'_1, C'_2, C', C''$ respectively. Note that $C'_1, C'_2, C''$ are not necessarily irreducible.
\begin{Proposition}\label{prop:bdry:form:2:sep:nodes}
	Assume that $E$ has two nodes all of which are separating. If $\pp\in \ol{\cX}_D$, then $C$ and $\xi$ satisfy one of the following
	\begin{itemize}
		\item[(a)] Up to a renumbering of $E'_1, E'_2$, $n_1=3, n_2=1$, $C'_1$ is an elliptic curve, $C'_2$ is isomorphic to $\Pb^1$ and contains $\{p_5,p'_5\}$, $C''$ is a disjoint union of two isomorphic elliptic curves, $C'_2$ intersects $C'_1$ at one node fixed by $\tau$ and intersects $C''$ at two nodes permuted by $\tau$. The differential $\xi$ vanishes identically on $C'_2$, and restricts to non-trivial holomorphic $1$-forms on the other components.
		\item[(b)] Both $C'_1, C'_2$ are isomorphic to $\Pb^1$, $n_1=n_2=2$, $C''$ is an elliptic curve which contains $\{p_5, p'_5\}$ and intersects each of $C'_1, C'_2$ at two nodes permuted by $\tau$. The differential $\xi$ has simple poles at all the nodes of $C$.
	\end{itemize}
\end{Proposition}
\begin{proof}
	Assume first that $E'=E'_1\cup E'_2$ is connected. Note that $E'$ contains at least three points in $\{q_1,\dots,q_5\}$. Therefore $n'\geq 2$.
	\begin{itemize}
		\item[$\bullet$] If $n'=2$ then $C'$ is a genus zero curve which intersects $C''$ at two nodes (in particular $C'$ is connected). Since in this case $\{p_5,p'_5\} \subset C'$, we must have $\xi_{\left|C'\right.} \equiv 0$. It follows that $\xi$ is holomorphic at the nodes between $C'$ and $C''$. By Proposition~\ref{prop:s:poles:at:nodes:exchanged} this is impossible.
		
		\item[$\bullet$] If $n'=3$ then $C'$ is genus one curve, $C''$ is a smooth genus two curve, and $C'$ and $C''$ intersect at one node fixed by $\tau$. One readily checks that either $\xi_{\left|C'\right.} \equiv 0$ or $\xi_{\left|C''\right.} \equiv 0$. Thus this case is excluded by Proposition~\ref{prop:bdry:form:dec:not:eigen:form}.
		
		\item[$\bullet$] If $n'=4$ then $C'$ is an elliptic curve, $C''$ is either a smooth elliptic curve, or a disjoint union of two isomorphic elliptic curves. In both cases, $C'$ meets $C''$ at two nodes permuted by $\tau$. If $C''$ is a smooth elliptic curves then by Proposition~\ref{prop:s:poles:at:nodes:exchanged}, $\xi$ must have simple poles at the nodes between $C'$ and $C''$. This implies that $\xi$ must have some zeros in $C''$. Since $C''$ in invariant under $\tau$, we have $\{p_5,p'_5\} \subset C''$. But $\xi$ must have double zeros at $p_5$ and $p'_5$ (cf. Theorem~\ref{th:twisted:diff}). Thus this case cannot occur.
		
		In the case $C''$ is a disjoint union of two elliptic curves, we first observe that $\xi_{\left|C''\right.} \not\equiv 0$ by Proposition~\ref{prop:bdry:form:dec:not:eigen:form}. Since each component of $C''$ has only one node, $\xi_{\left|C''\right.}$ is holomorphic.
		Without loss of generality, we can assume that $E'_2$ is the component of $E'$ that meets $E''$. 
		Since $\xi$ does not have poles at the nodes between $C'_2$ and $C''$, we must have $\xi_{\left|C'_2\right.} \equiv 0$.
		By Proposition~\ref{prop:bdry:form:dec:not:eigen:form}, $\xi_{\left|C'_1\right.} \not\equiv 0$. This means that $C'_1$ must be an elliptic curve, which implies that either $n_1=3$ or $n_1=4$.
		\begin{itemize}
			\item[.] If $n_1=3$, $C'_2$ is isomorphic to $\Pb^1$, then $C'_1$ intersects $C'_2$ at one node fixed by $\tau$. Since $\xi$ cannot have zero in $C'_1$, we must have $\{p_5,p'_5\}\subset C'_2$. One readily checks that all the conditions in Case (a) are satisfied.
			
			\item[.] If $n_1=4$, then $C'_2$ is a disjoint union of two copies of $\Pb^1$, each of which contains one point in $\{p_5, p'_5\}$. By Theorem~\ref{th:twisted:diff}, on each component of $C'_2$ there is an Abelian differential $\nu$ which has a double zero and two double poles such that the residue of $\nu$ at either pole is zero. Since such a differential does not exist, this case is excluded. 
		\end{itemize}
	\end{itemize}    
	
	Assume now that $E'_1$ and $E'_2$ are disjoint. We can suppose that $n_1 \leq n_2$. We have $1 \leq n_1 \leq n_2 \leq 3$.
	\begin{itemize}
		\item[$\bullet$] If $n_2=3$ then $C'_2$ is an elliptic curve which intersects $C''$ at one node fixed by $\tau$. It follows that either $\xi_{\left|C'_2\right.} \equiv 0$ or $\xi_{\left|C''\right.} \equiv 0$. Note  that we must have $\xi_{\left|C'_1\right.} \equiv 0$ since $\{p_5, p'_5\} \subset C'_1$. Therefore we would have a contradiction to Proposition~\ref{prop:bdry:form:dec:not:eigen:form} in either case.

		\item[$\bullet$] If $n_2=2$ then we also have $n_1=2$. As a consequence both $C'_1$ and $C'_2$ are isomorphic to $\Pb^1$ and meet $C''$ at two nodes permuted by $\tau$. By Proposition~\ref{prop:s:poles:at:nodes:exchanged}, $\xi$ must have simple poles at all these nodes. This implies that $\xi$ is non-trivial on both $C'_1$ and $C'_2$, and therefore $\{p_5, p'_5\} \subset C''$. It follows that $C$ and $\xi$  satisfy the condition in Case (b).
	\end{itemize}
\end{proof}

\subsubsection{Case one separating node and one non-separating node}\label{subsec:str:1:sep:1:n:sep:nodes}
We now suppose that the curve $E$ has one separating node and one non-separating node. This means that $E$ has two irreducible components denoted by $E'$ and $E''$, where $E'$ has genus $0$ and $E''$ has genus $1$, and there is a node between $E'$ and $E''$.  Note that $E''$ has a self-node and its normalization has genus $0$.
Denote by $C'$ and $C''$ the preimages of $E'$ and $E''$ in $C$. Note that $C'$ is smooth, while $C''$ is a nodal curve.

\begin{Proposition}\label{prop:st:1:sep:1:n:sep:node}
	Assume that $E$ has one separating node and one non-separating node. If  $D$  is not a square,  then $\pp\in\ol{\cX}_D$ only if
	\begin{itemize}
		\item[.] $C'$ is isomorphic to $\Pb^1$ and contains two of the points $\{p_1,\dots,p_4\}$,
		
		\item[.]  $C''$ is a genus two curve with two nodes (that are exchanged by the Prym involution) containing $\{p'_5, p''_5\}$ and two points in $\{p_1,\dots,p_4\}$,
		
		\item[.] there are two nodes between $C'$ and $C''$, and
		
		\item[.] $\xi$ has simple poles at all of the nodes of $C$.
	\end{itemize}
\end{Proposition}
\begin{proof}
	Let $n':=|E'\cap\{q_1,\dots,q_4\}|$.

	\begin{itemize}
		\item[$\bullet$] Case $n'=1$. In this case we must have $q_5\in E'$, and therefore $p_5,p'_5\subset C'$, and there is one node between $C'$ and $C''$.  It follows that $\xi_{\left|C'\right.}\equiv 0$. Hence the restriction of $\xi$ on $C''$ is non-trivial. Note that $C''$ is connected.
		The preimage of the self-node of $E''$ consists of one or two nodes of $C''$.
		If $C''$ has one node, it must be fixed by the Prym involution, which implies that $\xi_{\left|C''\right.}\equiv 0$. But this is impossible since $\xi\not\equiv 0$. Thus $C''$ must have two nodes that are exchanged by $\tau$. By Proposition~\ref{prop:s:poles:at:nodes:exchanged}, $\xi$ must have simple poles at those nodes. However, since $\xi$ does not have any other poles this is excluded by Proposition~\ref{prop:s:poles:at:two:nodes:not:in:XD}.

		\item[$\bullet$] Case $n'=2$. In this case, both $C'$ and $C''$ are connected and $C'$ meets $C''$ at two nodes exchanged by $\tau$. By Proposition~\ref{prop:s:poles:at:nodes:exchanged}, $\xi$ has simple poles at those node. As a consequence $\xi_{\left|C'\right.}\not\equiv 0$, which implies that $\{p_5,p'_5\} \subset C''$.
		
		Since $\xi_{\left|C''\right.}\not\equiv 0$, $C''$ must have two nodes exchanged  by $\tau$ and $\xi$ must have simple poles at these nodes, and we get the desired conclusion. 
		
		\item[$\bullet$] Case $n'=3$. In this case $C'$ is an elliptic curve which intersects $C''$ at one node. It follows that either $\xi_{\left|C'\right.}\equiv 0$ or $\xi_{\left| C''\right.}\equiv 0$. By Proposition~\ref{prop:bdry:form:dec:not:eigen:form}, this case cannot occur.
		
		\item[$\bullet$] Case $n'=4$. In this case $C'$ is an elliptic curve which meets $C''$ at two nodes. If $C''$ is connected, $\xi$ must have two nodes at the nodes between $C'$ and $C''$ by Proposition~\ref{prop:s:poles:at:nodes:exchanged}. This implies that $\xi$ must have a double zero in the smooth part of $C'$. But since $C'$ is invariant by $\tau$, this is impossible. In the case $C''$ is disconnected, it must have two components, each of which is a genus $1$ curve with one node. The two nodes between $C'$ and $C''$ are separating. Therefore, $\xi$ cannot have simple poles at those nodes. As a consequence, either $\xi_{\left|C'\right.}\equiv 0$ or $\xi_{\left|C''\right.}\equiv 0$. In either case, we would have a contradiction to Proposition~\ref{prop:bdry:form:dec:not:eigen:form}. Thus this case cannot occur.
	\end{itemize}
\end{proof}

\subsubsection{Case two non-separating nodes}\label{subsec:str:2:n:sep:nodes}
In this case $E$ has two irreducible components denoted by $E_1$ and $E_2$, both of which are isomorphic to $\Pb^1$. Set $n_i:=|E_i\cap\{q_1,\dots,q_4\}|,  \; i=1,2$. We must have $n_1+n_2=4$. By convention, we always suppose that $n_1 \geq n_2$.  Let $C_1$ and $C_2$ be respectively the preimages of $E_1$ and $E_2$ in $C$.

\begin{Proposition}\label{prop:2:n:sep:nodes}
	Assume that $E$ has two non-separating nodes, and that $D$ is not a square. If $\pp\in \ol{\cX}_D$ then we have $(n_1, n_2)=(4,0)$ and
	\begin{itemize}
		\item[.] $C_1$ is a smooth curve of genus $2$,
		
		\item[.] $C_2$ is isomorphic to $\Pb^1$ and contains $\{p'_5, p''_5\}$,
		
		\item[.] $\xi$ vanishes identically on $C_2$, and $(C_1,\xi_{\left|C_1\right.}) \in \Omega\cM_2(2)$,
	\end{itemize}
\end{Proposition}

\begin{proof}
	We have three cases $(n_1,n_2)=(4,0)$, $(n_1,n_2)=(3,1)$, and $(n_1,n_2)=(2,2)$.
	\begin{itemize}
		\item[$\bullet$] Case $(n_1,n_2)=(4,0)$. In this case equivalently $\{q_1,\dots,q_4\}\subset E_1$ and $q_5\in E_2$.
		We claim that  the preimages of the two nodes of $E$ have the same cardinality. This is because the closed curves $c',c''$ on the reference torus $E_0$ that correspond to these nodes have the same image in $\Z/2\Z$ under the group morphism $\bar{\varrho}$. We have two subcases
		
		\begin{itemize}
			\item[-] Case 1: each node of $E$ gives two nodes in $C$ (that is $\bar{\varrho}(c')=\bar{\varrho}(c'')=0\in \Z/2\Z$). In this case $C_1$ is an elliptic curve, $C_2$ is a disjoint union of two copies of $\Pb_1$, each of which meets $C_1$ at two nodes.
			Since $\{p_5,p'_5\} \subset C_2$, the differential  $\xi$ vanishes identically on $C_2$ and is nowhere zero on $C_1$. 
			The smooth part $C_2^*$ of $C_2$ is the disjoint union of two open annuli denoted by $A'$ and $A''$. 
			Let $\gamma'$ and $\gamma''$ be respectively some core curves of $A'$ and $A''$. We endow these curves with the orientations such that $\gamma''=\tau_*\gamma'$. Thus $\gamma'-\gamma''\subset H_1(X,\Z)^- - \{0\}$, where $X$ is a reference smooth curve in $\cB_{4,1}$. 
			By Proposition~\ref{prop:non:collapse:sub:curve},  $\xi$ cannot vanish identically  on $C_2$. We thus have a contradiction showing that this case cannot occur.

			\item[-] Case 2: each node of $E$ gives rise to a node of $C$. In this case $C$ has two nodes, both are fixed by the Prym involution. The curve $C_1$ is a Riemann surface of genus $2$, while $C_2$ is a copy of $\Pb^1$ meeting $C_1$ at two nodes. Note that the restriction of $\tau$ to $C_1$ has $6$ fixed points: namely, $p_1,\dots,p_4$ and the two nodes of $C$. In particular, these nodes  are the Weierstrass points of $C_1$.
			It follows from Theorem~\ref{th:twisted:diff} that $\xi_{\left|C_1\right.}$ must have a double zero at one of the nodes, that is $(C_1,\xi_{\left|C_1\right.}) \in \Omega\cM_2(2)$, while $\xi_{\left|C_2\right.}\equiv 0$. We thus get the desired conclusion.
		\end{itemize}
		
		\item[$\bullet$] Case $(n_1,n_2)= (3,1)$  In this case $C_1$ is an elliptic curve, $C_2$ is isomorphic to $\Pb^1$, and there are 3 nodes between $C_1$ and $C_2$, one of the nodes is fixed by $\tau$, the other two are permuted.  Proposition~\ref{prop:node:fixed:by:invol} then implies that either $\xi_{\left|C_1\right.}\equiv 0$ or $\xi_{\left|C_2\right.} \equiv 0$. Therefore,  $\xi$ cannot have simple poles at the nodes permuted by $\tau$ which contradicts Proposition~\ref{prop:s:poles:at:nodes:exchanged}. Thus this case does not occur.
		
		\item[$\bullet$] Case $(n_1,n_2)=(2,2)$. In this case, both $C_1$ and $C_2$ are connected.  Either (a) both $C_1$ and $C_2$ are elliptic curves  intersecting each other at 2 nodes fixed by $\tau$, or (b) $C_1$ and $C_2$ are both isomorphic to $\Pb^1$ and intersect each other at $4$ nodes. By Corollary~\ref{cor:dec:two:tori:no:eigenform} (a) cannot happen. Suppose that $C$ satisfies (b). Then   $\xi$ has simple poles at all the nodes of $C$ by Proposition~\ref{prop:s:poles:at:nodes:exchanged}. This can only happen if each of $C_1, C_2$ contains a double zero of $\xi$. But since $C_1,C_2$ are both invariant by $\tau$, this cannot be the case. Thus  this case can not happen either.
	\end{itemize}
\end{proof}

\subsection{Case $E$ has three nodes}\label{subsec:st:3:nodes}
We now consider the case  $E$ has 3 nodes.

\subsubsection{Three separating nodes}\label{subsec:st:3:nodes:sep}
We first consider the case all the nodes of $E$ are separating. In this case, $E$ has $4$ irreducible components, three of which are isomorphic to $\Pb^1$,  the remaining one is an elliptic curve. We denote the $\Pb^1$ components by $E'_1, E'_2, E'_3$, and the elliptic one by $E''$.
Let $E':=E'_1\cup E'_2 \cup E'_3$.  Let $C'_i, i\in\{1,2,3\}, C'$, and $C''$ be respectively the preimages of $E'_i, E'$, and $E''$ in $C$.
Let $n':=|E'\cap\{q_1,\dots,q_4\}|$.
Define $\xi':=\xi_{\left|C'\right.}$ and $\xi'':=\xi_{\left|C''\right.}$

\begin{Proposition}\label{prop:bdry:str:3:sep:nodes}
	If $E$ has three nodes all of which are separating then $\pp\not\in\ol{\cX}_D$.
\end{Proposition}
\begin{proof}
	Let us suppose that $\pp\in \ol{\cX}_D$.
	Note that  $E'$ has at most $2$ connected components. We thus have two cases
	\begin{itemize}
		\item[(a)] Case $E'$ is connected. We have two subcases
		\begin{itemize}
			\item[$\bullet$] Case $n'=3$. In this case we must have $q_5\in E'$, $C'$ is a nodal curve of genus $1$, $C''$ a smooth curve of genus two, and $C'$ intersects $C''$ at a node fixed by $\tau$. It follows from Theorem~\ref{th:twisted:diff} that  $\xi'' \not\equiv 0$ and $\xi''$ must have a double zero at the  node between $C'$ and $C''$.
			Note that $\tau$ has two fixed points on $C''$ and satisfies $\tau^*\xi''=-\xi''$. But by Proposition~\ref{prop:ab:diff:dble:zero:Prym:inv} $\xi''$ must have two simple zeros. We thus have a contradiction, which means that this case cannot occur.\\
			
			\item[$\bullet$] Case $n'=4$. In this case $C'$ is a nodal curve of genus one, $C''$ is either an elliptic curve, or a disjoint union of two isomorpĥic elliptic curves, and there are two nodes between $C'$ and $C''$. In the former case, $\xi$ must have simple poles at the nodes between $C'$ and $C''$. This implies that $\xi''\not\equiv 0$. Since $\xi''$ has either no zero, or two double zeros in the smooth part of $C''$, this is impossible.
			In the latter case, we have $\xi'\not\equiv 0$ and $\xi''\not\equiv 0$ by Proposition~\ref{prop:bdry:form:dec:not:eigen:form}. Since $\xi''$ must be holomorphic on $C''$, we have $\{p_5,p'_5\} \subset C'$. Since $\xi'$ must have double zeros at $p_5,p'_5$, or vanish identically on the component(s) that contain $p_5$ and $p'_5$, the only admissible configuration is that $C'$ has $3$ irreducible components $C'_1,C'_2,C'_3$, where
			\begin{itemize}
				\item[-] $C'_1$ contains two points in $\{p_1,\dots,p_4\}$, intersects $C'_2$ at two nodes, and is disjoint from $C'_3$,
				
				\item[-] $C'_2$ contains one point in $\{p_1,\dots,p_4\}$, and intersects both $C'_1$ and $C'_3$,
				
				\item[-] $C'_3$ contains $\{p_5,p'_5\}$ and one point in $\{p_1,\dots,p_4\}$, intersects $C'_2$ at one node, and $C''$ at two nodes.  
			\end{itemize}
			The differential $\xi'$  vanishes identically on $C'_3$ and has simple poles at the nodes between $C'_1$ and $C'_2$. However, since these are the only pair of nodes at which $\xi$ has simple poles, we have a contradiction to Proposition~\ref{prop:s:poles:at:two:nodes:not:in:XD}. Thus this case cannot occur.   
		\end{itemize}
		
		\item[(b)] Case $E'$ is not connected. In this case $E'$ has two connected components. Without loss of generality, we will assume that $E'_1$ and $E'_3$ are in the same connected component of $E$. Let $n'_1:=|(E'_1\cup E'_3)\cap\{p_1,\dots,p_4\}|$ and  $n'_2:=|E'_2\cap\{p_1,\dots,p_4\}|$.  Note that we must have $n'_1\geq 2$, $n'_2\geq 1$, and $n'_1+n'_2=4$.

		\begin{itemize}
			\item[$\bullet$] Case $(n'_1,n'_2)=(2,2)$. In this case $\{p_5,p'_5\} \subset C'_1\cup C'_3$. By considering the compatible twisted differentials (cf. Theorem~\ref{th:twisted:diff}, we see that $\xi$ must vanish identically on $C'_1\cup C'_3$. Observe that $C'_2$ intersects $C''$ at two nodes. By  Proposition~\ref{prop:s:poles:at:nodes:exchanged}, $\xi$ must have simple poles at these two nodes. But since these are the only nodes at which $\xi$ has simple poles, we would have a contradiction to Proposition~\ref{prop:s:poles:at:two:nodes:not:in:XD}. Thus this case cannot occur.
			
			\item[$\bullet$] Case $(n'_1,n'_2)=(3,1)$. In this case $C'_1\cup C'_3$ is a nodal curve of genus one intersecting $C''$ at one node, while $C'_2$ is isomorphic to $\Pb^1$, contains $p_5,p'_5$, and intersects $C''$ also at one node. This implies that $C''$ is a smooth curve of genus two. Note that $\xi_{\left|C'_2\right.} \equiv 0$. 
			Since the node between $C'_1\cup C'_3$ and $C''$ is separating, either $\xi_{\left|C'_1\cup C'_3\right.} \equiv 0$ or $\xi_{\left|C''\right.} \equiv 0$. In either case, we would have a contradiction to Proposition~\ref{prop:bdry:form:dec:not:eigen:form}. The proposition is then proved.
		\end{itemize}
	\end{itemize}
\end{proof}

\subsubsection{Two separating nodes and one non-separating node}\label{subsec:st:2:nodes:sep:1:node:n:sep}
Assume now that $E$ has 2 separating nodes and one non-separating one. Then $E$ has $3$ irreducible components, two of which, denoted by $E'_1, E'_2$, are isomorphic to $\Pb^1$, the remaining one, denoted by $E''$, is a genus $1$ nodal curve. Let $E':=E'_1\cup E'_2$, and $n':=|E'\cap\{q_1,\dots,q_4\}|$. Let $C'_1, C'_2, C', C''$ be respectively the preimages of $E'_1, E'_2, E', E''$ in $C$.

\begin{Proposition}~\label{prop:st:2:nodes:sep:1:node:n:sep:Ep:connect}
	If $E$ has two separating nodes and one non-separating node, and $E'$ is connected, then $\pp\not\in\ol{\cX}_D$.
\end{Proposition}
\begin{proof}
	Assume that $\pp\in \ol{\cX}_D$. 
	Without loss of generality, we can assume that $E'_2$ intersects both $E'_1$ and $E''$. Note that we have $2 \leq n' \leq 4$.
	\begin{itemize}
		\item[$\bullet$] Case $n'=2$. In this case $C'$ is a genus zero curve which contains $\{p_5,p'_5\}$ and intersects $C''$ at two nodes. It follows from Proposition~\ref{prop:s:poles:at:nodes:exchanged} that $\xi$ must have simple poles at these nodes. This means that $\xi_{\left|C'_1\cup C'_2\right.} \not\equiv 0$. Since $\xi$ must have double zeros at $p_5,p'_5$ we would have a contradiction to Theorem~\ref{th:twisted:diff}. Thus this case does not occur. 
		
		\item[$\bullet$] Case $n'=3$. In this case $C'$ is a genus $1$ nodal curve, $C''$ is a genus $2$ nodal curve which intersects $C'$ at one node. 
		By  Proposition~\ref{prop:bdry:form:dec:not:eigen:form}, $\xi_{\left| C''\right.} \not\equiv 0$ and $\xi_{\left| C'\right.} \not\equiv 0$. This implies that $\xi$ vanishes identically on $C'_2$, and $\xi_{\left| C'_1\right.} \not\equiv 0$. One readily checks that this happens only if  $C'_1$ is an elliptic curve containing $3$ points in  $\{p_1,\dots,p_4\}$, $C'_2$ is isomorphic to $\Pb^1$, contains $\{p'_5, p''_5\}$ and intersects each of $C'_1$ and $C''$ at one node.
		
		Note that $C''$ has two self-nodes, and $\xi''$ must have simples simple poles that these nodes. Since these are the only nodes of $C$ at which $\xi$ has simple poles, we have a contradiction to Proposition~\ref{prop:s:poles:at:two:nodes:not:in:XD}. We can then conclude that this case cannot occur.

		\item[$\bullet$] Case $n'=4$. In this case, $C'$ is of  genus $1$, $C''$ is either (a) a connected genus $1$ curve or (b) a disjoint union of two isomorphic genus one curves, and $C'$ intersects $C''$ at two nodes. In case (a), $C''$ can have either one or two self-nodes. If $C''$ has one self-node, since this node is fixed by $\tau$, we must have $\xi_{\left|C''\right.}\equiv 0$, but this is a contradiction to Proposition~\ref{prop:bdry:form:dec:not:eigen:form}. Thus $C''$ must have two self-nodes.  By Proposition~\ref{prop:s:poles:at:nodes:exchanged}, $\xi$ has simple poles at the nodes between $C'$ and $C''$. It follows that $\xi$ has three simple poles in each irreducible component of $C''$ (which is isomorphic to $\Pb^1$). But as $\xi$ has either no zero or a double zero on an irreducible component of $C''$, this case cannot occur.
		
		In case (b) the nodes between $C''$ and $C'$ are separating.
		Since $\xi_{\left|C''\right.}\not\equiv 0$ by Proposition~\ref{prop:bdry:form:dec:not:eigen:form}, we must have 
		$\xi_{\left|C'_2\right.}\equiv 0$. Note that we also have $\xi_{\left|C'\right.}\not\equiv 0$, which means that $\xi_{\left|C'_1\right.}\not\equiv 0$. It follows that $C'_1$ is an elliptic curve. In particular, $\xi$ does not have simple pole on $C'$. Since $\xi$ must have simple poles at the self-nodes of $C''$, we get a contradiction to Proposition~\ref{prop:s:poles:at:two:nodes:not:in:XD}, which means that this case cannot occur either.
	\end{itemize}
\end{proof}

\begin{Proposition}~\label{prop:2:nodes:sep:1:node:n:sep:Ep:no:connect}
	Suppose that $E$ has two separating nodes and one non-separating node,  and that $E'$ is not connected. Then $\pp\in \ol{\cX}_D$ only if
	\begin{itemize}
		\item[$\bullet$] $C'_1$ and $C'_2$ are both isomorphic to $\Pb^1$,
		
		\item[$\bullet$] $C''$ is either
		\begin{itemize}
			\item[(a)] a genus two curve with two nodes,
			
			\item[(b)] a genus one curve with two nodes, or
			
			\item[(c)] a disjoint union of two genus $1$ curves with one node,
		\end{itemize}
		
		\item[$\bullet$] $\xi$ has simple poles at all the non-separating nodes of $C$.
	\end{itemize}
\end{Proposition}
\begin{proof}
	Let $n_1:=|E'_1\cap\{q_1,\dots,q_4\}|$ and $n_2:=|E'_2\cap\{q_1,\dots,q_4\}|$.
	Without loss of generality, we can assume that $n_1\geq n_2$.
	Since $n'=n_1+n_2 \leq 4$, we have $1\leq n_2 \leq n_1 \leq 3$.
	
	\begin{itemize}
		\item[(i)] Case $(n_1,n_2)=(2,1)$. In this case $q_5$ must be contained in $E'_2$, and  each of  $C'_1, C'_2$ is isomorphic to $\Pb^1$, $C'_1$ intersects $C''$ at two nodes, $C'_2$ intersects $C''$ at one node. 
		It follows that $\xi$ vanishes identically on $c'_2$ and $\xi'':=\xi_{\left| C''\right.}$ has a zero of order four at the node between $C''$ and $C'_2$.
		Note that $C''$ is a genus two curve with two self-nodes.
		By Proposition~\ref{prop:s:poles:at:nodes:exchanged}, $\xi$ has simple poles at those nodes between $C''$ and $C'_1$.  Since $\xi$ must have simple poles at the self-nodes of $C''$, we get the desired conclusion with $C''$ in case (a).

		\item[(ii)] Case $(n_1,n_2)=(3,1)$. In this case $C'_1$ is an elliptic curve, $C'_2$ is isomorphic to $\Pb^1$ and contains $\{p_5, p'_5\}$,  $C''$ is a nodal curve of genus two intersecting each of $C'_1, C'_2$ at one node. One readily checks that $\xi$ must vanishes identically on $C''$. Thus we have a contradiction to Proposition~\ref{prop:bdry:form:dec:not:eigen:form}.
		
		\item[(iii)] Case $(n_1,n_2)=(2,2)$. In this case both $C'_1, C'_2$ are isomorphic to $\Pb^1$, while $C''$ can be either a genus one curve with one node, a genus one curve with two nodes, or a union of two nodal genus one curves. Note that $C''$  and intersects each of $C'_1, C'_2$ at two nodes.  
		
		In the first case $\xi_{\left|C''\right.}\equiv 0$, which implies that $\xi_{\left|C'_1\right.}\equiv 0$ and $\xi_{\left|C'_2\right.}\equiv 0$, that is $\xi\equiv 0$. Thus this case is excluded.
		
		In the second case, $\xi$ must have simple poles at all the nodes by Proposition~\ref{prop:s:poles:at:nodes:exchanged}, and $\pp$ has all the desired properties with $C''$ in case (b).
		
		In the last case, one readily checks that $\pp$ has all the desired properties with $C''$ in case (c).
	\end{itemize}
\end{proof}

\subsubsection{One separating node and two non-separating ones}\label{subsec:str:1:sep:2:n:sep:nodes}
Assume now that $E$ has one separating nodes and two non-separating ones. In this case, $E$ has 3 irreducible components, all of which are isomorphic to $\Pb^1$. One of the component, that will be denoted by $E''_1$, intersects the other two. We denote by $E'$ the component that intersects $E''_1$ at one node, and by $E''_2$ the one that intersects  $E''_1$ at two nodes. Let $E'':=E''_1\cup E''_2$. We denote by $C',C''_1, C''_2,C''$ the preimages of $E', E''_1, E''_2, E''$ in $C$. Let $\xi':=\xi_{\left|C'\right.}, \; \xi''_1:=\xi_{\left|C''_1\right.}, \; \xi''_2:=\xi_{\left|C''_1\right.}$.
\begin{Proposition}\label{prop:str:1:sep:node:2:n:sep:nodes}
	If $E$ has one separating node and two non-separating nodes, then $\pp\not\in \ol{\cX}_D$.
\end{Proposition}
\begin{proof}
	Suppose that $\pp\in\ol{\cX}_D$. 
	Let $n':=|E'\cap\{q_1,\dots,q_4\}|, \; n''_1:=|E''_1\cap\{q_1,\dots,q_4\}|, \; n''_2:=|E''_2\cap\{q_1,\dots,q_4\}|$.   We must have $1 \leq n' \leq 4$ and $n'+n''_1+n''_2=4$.
	
	\begin{itemize}
		\item[(a)] Case $n'=1$. In this case $E'$ must contain $q_5$ and one point in $\{q_1,\dots,q_4\}$. Therefore, $C'_1$ is isomorphic to $\Pb^1$, and $\xi'\equiv 0$. We have the following subcases. 
		\begin{itemize}
			\item[(a.1)] $(n''_1,n''_2)=(0,3)$. In this case $C''_1$ is also isomorphic to $\Pb^1$, $C''_2$ is an elliptic curve, and $C''_1$ intersects $C''_2$ at three nodes. Since one of the nodes between $C''_1$ and $C''_2$ is fixed by $\tau$, 
			either $\xi''_1\equiv 0$, or $\xi''_2\equiv 0$. If $\xi''_2\equiv 0$, then since $C''_1$ is isomorphic to $\Pb^1$ we also have $\xi''_1\equiv 0$. Hence $\xi\equiv 0$ which is impossible. Thus, we must have $\xi''_1\equiv 0$. Note that two of the nodes between $C''_1$ and $C''_2$ are permuted by $\tau$. By considering the cycle supported in $C''_1$ consisting of two small circles bordering two disjoint small discs containing these nodes in the interior, we get a contradiction to Proposition~\ref{prop:non:collapse:sub:curve}. Thus this case cannot occur.
			
			\item[(a.2)] $(n''_1,n''_2)=(1,2)$. In this case either both $C''_1$ and $C''_2$ are elliptic curves that intersect each other at  two nodes fixed by $\tau$, or both $C''_1$ and $C''_2$ are isomorphic to $\Pb^1$  and intersect each other at  two pairs of nodes permuted by $\tau$. The former case is ruled out by Corollary~\ref{cor:dec:two:tori:no:eigenform}, while the latter cannot occur since there does not exist any  compatible twisted differential on $C$ (cf. Theorem~\ref{th:twisted:diff}).
			
			\item[(a.3)] $(n''_1, n''_2)=(2,1)$. In this case $C''_1$ is an elliptic curve, $C''_2$ is isomorphic to $\Pb^1$, and $C''_1,C''_2$ meet at three nodes. One readily checks that there cannot exists any compatible twisted differential on $C$. Therefore, this case does not occur. 
		\end{itemize}
		\item[(b)] Case $n'=2$. In this case $C'$ is isomorphic to $\Pb^1$ and intersects $C''_1$ at two nodes.  We have two subcases
		\begin{itemize}
			\item[(b1)] $(n''_1,n''_2)=(0,2)$. Either $C''_1$ is isomorphic to $\Pb^1$, $C''_2$ is an elliptic curve, and $C''_1$ intersects $C''_2$ at two nodes fixed by $\tau$, or $C''_1$ is a disjoint union of two copies of $\Pb^1$, $C''_2$ is isomorphic to $\Pb^1$ and intersects $C''_1$ at fours nodes. The former case is ruled out by Corollary~\ref{cor:dec:two:tori:no:eigenform}, while the latter is ruled out since there is no compatible twisted differential.
			
			\item[(b2)] $(n''_1,n''_2)=(1,1)$. In this case, both $C''_1, C''_2$ are isomorphic to $\Pb^1$ and intersect each other at three nodes. Since one of the nodes between $C''_1$ and $C''_2$ is fixed by $\tau$, $\xi$ mush vanish identically on $C''_1$ or on $C''_2$. In either case, by considering the pair of simple closed curves bordering two small discs containing the other two nodes between $C''_1$ and $C''_2$, we get a contradiction to Proposition~\ref{prop:non:collapse:sub:curve}. It follows that this case cannot occur.
			
			\item[(b3)] $(n''_1,n''_2)=(2,0)$. In this case, we must have $q_5\in E''_2$. Either $C''_1$ is an elliptic curve, $C''_2$ is isomorphic to $\Pb^1$, and $C''_1$ intersects $C''_2$ at two nodes fixed by $\tau$, or $C''_1$ is isomorphic to $\Pb^1$, $C''_2$ is a disjoint union of two copies of $\Pb^1$, and $C''_1$ intersects $C''_2$ at four nodes. In both cases, $\xi$ only has simple poles at the nodes between $C'$ and $C''_2$. Thus the two cases is ruled out by Proposition~\ref{prop:s:poles:at:two:nodes:not:in:XD}.   
		\end{itemize}
		\item[(c)] Case $n'=3$. In this case $C'$ is an elliptic curve which intersects $C''_1$ at one node, $C''_1$ is isomorphic to $\Pb^1$, $C''_2$ is either isomorphic to $\Pb^1$ or a disjoint union of two copies of $\Pb^1$. Since $C''=C''_1\cup C''_2$ is a genus two curve, by Proposition~\ref{prop:bdry:form:dec:not:eigen:form}, we must have $\xi'\not\equiv 0$ and $\xi'':=\xi_{\left|C''\right.} \not\equiv 0$. Since the node between $C'$ and $C''_1$ is fixed by $\tau$, we must have $\xi''_1=\xi_{\left|C''_1\right.} \equiv 0$. As a consequence $ \xi''_2=\xi_{\left|C''_2\right.} \not\equiv 0$. But since $C''_2$ is either isomorphic to $\Pb^1$ or a disjoint union of two copies of $\Pb^1$, $\xi$ must vanish identically on $C''_2$. We thus have a contradiction, which means that this case cannot occur.
		
		\item[(d)] Case $n'=4$. We must have $q_5\in E''_2$. In this case $C'$ is an elliptic curve which intersects $C''_1$ at two nodes. Either $C''_1$ and $C''_2$ are both isomorphic to $\Pb^1$ and intersect each other at two nodes fixed by $\tau$, or each of $C''_1$ and $C''_2$ is a disjoint union of two copies of $\Pb^1$. In the former case $\xi$ only has simple poles at the nodes between $C'$ and $C''_1$. Thus this case is ruled out by Proposition~\ref{prop:s:poles:at:two:nodes:not:in:XD}. In the latter, since $\{p_5,p'_5\} \subset C''_2$, we must have $\xi''_2\equiv 0$. It follows that $\xi''_1\equiv 0$, and we have a contradiction to Proposition~\ref{prop:bdry:form:dec:not:eigen:form}. This completes the proof of the proposition.    
	\end{itemize} 	
\end{proof}

\subsubsection{Three non-separating  nodes} \label{subsec:st:3:n:sep:nodes}
In this case $E$ has $3$ irreducible components,  denoted by $E_1, E_2, E_3$, all of which are isomorphic to $\Pb^1$.
For $i=1,2,3$, let $C_i$ be the preimage of $E_i$ in $C$, and $\xi_i:=\xi_{\left|C_i\right.}$.
\begin{Proposition}\label{prop:st:3:n:sep:nodes}
	If $E$ has three non-separating nodes then $\pp\not\in\ol{\cX}_D$. 
\end{Proposition}
\begin{proof}
	We assume that $\pp \in \ol{\cX}_D$.
	We have a partition of $\{q_1,\dots,q_4\}$ associated with the decomposition $E=E_1\cup E_2 \cup E_3$.
	Let $n_i:=|E_i\cap\{q_1,\dots,q_4\}|, \; i=1,2,3$.
	By convention, we always assume that $n_1\geq n_2\geq n_3$.
	Since $n_1+n_2+n_3=4$, we have $(n_1,n_2,n_3)\in \{(3,1,0), (2,2,0), (2,1,1)\}$.
	
	\begin{itemize}
		\item[(a)] Case $(n_1,n_2,n_3)=(3,1,0)$. In this case $q_5 \in E_3$, $C_1$ is an elliptic curve, $C_2$ is isomorphic to $\Pb^1$, and $C_3$ is either isomorphic to $\Pb^1$ or a disjoint union of two copies of $\Pb^1$. In all cases, since $\{p_5,p'_5\} \subset C_3$, we must have $\xi_3\equiv 0$. If $C_3$ is isomorphic to $\Pb^1$, then $\xi$ must have simple poles at the nodes between $C_1$ and $C_2$. Since these are the only nodes where $\xi$ has simple poles, this case is excluded by Proposition~\ref{prop:s:poles:at:nodes:exchanged}.  If $C_3$ is a disjoint union of two copies of $\Pb^1$ then each component of $C_3$ meets both $C_1$ and $C_3$.  The smooth part $C_3^*$ of $C_3$ consists of two open annuli. Let $\gamma'_3$ and $\gamma''_3$ be the core curves of those annuli. Then $\gamma'-\gamma''$ corresponds to non-trivial  element of $H_1(X,\Z)^-$, where $X$ is a reference smooth curve in $\cB_{4,1}$. It follows that we have a contradiction to Proposition~\ref{prop:non:collapse:sub:curve}. Therefore, this case is also excluded.
		
		\item[(b)] Case $(n_1,n_2,n_3)=(2,2,0)$.   Again, we must have $q_5 \in E_3$, or equivalently  $\{p_5,p'_5\}\subset C_3$. We have two possible configurations
		\begin{itemize}
			\item[(b.1)] $C_1$ and $C_2$ are elliptic curves intersecting each other at one node, $C_3$ is isomorphic to $\Pb^1$ and intersects each of $C_1, C_2$ at one node. Note that all the nodes are fixed by $\tau$. It follows from Proposition~\ref{prop:node:fixed:by:invol} that $\xi$ vanishes identically on  $C_1$ or on $C_2$. Since the restrictions of $\tau$ to both $C_1, C_2$ are involutions with  four fixed points, there are non-trivial cycles anti-invariant by $\tau$ on both $C_1, C_2$. We thus have a contradiction to Proposition~\ref{prop:non:collapse:sub:curve}. Therefore this case cannot occur. 
			
			\item[(b.2)] $C_1,C_2$ are both isomorphic to $\Pb^1$ and intersect each other at two nodes permuted by $\tau$, $C_3$ is a disjoint union of two copies of $\Pb^1$ each of which meets both $C_1,C_2$. Since $\{p_5,p'_5\} \subset C_3$, $\xi$ vanishes identically on $C_3$. It follows that $\xi$ only has simple poles at the nodes between $C_1$ and $C_2$. By Proposition~\ref{prop:s:poles:at:two:nodes:not:in:XD} this impossible.  
		\end{itemize} 
		
		\item[(c)] Case $(n_1,n_2,n_3)=(2,1,1)$. We have two  configurations
		\begin{itemize}
			\item[(c.1)] $C_1$ is an elliptic curve which meets each of $C_2,C_3$ at one node fixed by $\tau$, $C_2,C_3$ are both isomorphic to $\Pb^1$ and intersect each other at two nodes. If $\xi_1\not\equiv 0$ then $\xi_2\equiv 0$ and $\xi_3\equiv 0$. By considering the simple closed curves bordering small discs containing the nodes between $C_2$ and $C_3$, we get a contradiction to Proposition~\ref{prop:non:collapse:sub:curve}. If $\xi_1\equiv 0$, then $\xi_2\not\equiv 0$ and $\xi_3\not\equiv 0$. It follows that $\xi$ has simple poles at the nodes between $C_2$ and $C_3$, and we get a contradiction to Proposition~\ref{prop:s:poles:at:nodes:exchanged}. 
			
			\item[(c.2)] $C_1$ is isomorphic to $\Pb^1$ and intersects each of $C_2,C_3$ at two nodes permuted by $\tau$, $C_2,C_3$ are both isomorphic to $\Pb^1$ and intersect each other at one node. One readily checks that a compatible twisted differential  exists only if $\{p_5,p'_5\}\subset C_2$ or $\{p_5,p'_5\}\subset C_3$. In the former case $\xi$ vanishes identically on $C_2$ and has simple poles at the nodes between $C_1$ and $C_3$. We thus have a contradiction to Proposition~\ref{prop:s:poles:at:two:nodes:not:in:XD}. The latter case is also excluded by the same argument.  This completes the proof of the proposition.  
		\end{itemize}
	\end{itemize}
\end{proof}

\subsection{Case $E$ has four nodes}\label{subsec:str:4:nodes}
\subsubsection{Case four separating nodes}\label{subsec:str:4:sep:nodes}
In this case, $E$ has $5$ irreducible components, 4 of which are isomorphic to $\Pb^1$, the remaining one is an elliptic curve.
Denote by $E'_1,\dots,E'_4$ the $\Pb^1$-components, and by $E''$ the elliptic one. The union  $E'_1\cup\dots\cup E'_4$ is denoted by $E'$.
Let $n'_i:=|E'_i\cap\{q_1,\dots,q_4\}|$.
The preimages of $E'_1,\dots, E'_4, E', E''$ in $C$ are denoted by $C'_1,\dots,C'_4, C'',C'$ respectively.

\begin{Proposition}\label{prop:4:sep:nodes}
	If $E$ has $4$ separating nodes then $\pp\not\in\ol{\cX}_D$.
\end{Proposition}
\begin{proof}
	Since there are $5$ marked points on $E$, $E'$ must be a connected curve and contains all the points in $\{q_1,\dots,q_5\}$.
	We can consider $E'$ as a stable genus $0$ curve with $6$ marked points,  with the 6th marked point being the node between $E'$ and $E''$.
	We call a component of $E'$ that intersects only one other component an {\em end component}.
	There are $2$ possible configurations for $E'$:  we denote by $(a)$  the configuration where $E'$ has two end components, and by $(b)$ the configuration where $E'$ has three end components. If $E'$ has configuration (a), we will denote its components such that $E'_i$ is adjacent to $E'_{i+1}$, for $i=1,2,3$. If $E'$ has configuration (b) then we denote its end components by $E'_1, E'_2, E'_3$, and the remaining component by $E'_4$.
	Each choice for the 6th marked point of $E'$ gives us an admissible configuration for $E$.
	By symmetry, we only need to consider 3 configurations, which will be denoted by $(a1), (a2)$ and $(b)$ as follows
	\begin{itemize}
		\item[(a1)] $E'$ has two end components, one of which intersects $E''$. 
		
		\item[(a2)] $E'$ has two end components, one of the remaining two intersects $E''$. 
		
		\item[(b)] $E'$ has 3 end components, one of which intersects $E''$.
	\end{itemize}
	In all cases, $C''$ can be  either an elliptic curve or a disjoint union of two elliptic curves, and there are two nodes between $C''$ and $C'$. 
	In the former case, $C''$ must have negative level in any compatible twisted differential  on $C$ (cf. Theorem~\ref{th:twisted:diff}).  This means that $\xi$ vanishes identically on $C''$, and hence $\xi$ does not have simple poles at the nodes between $C''$ and $C'$. We thus get a contradiction to Proposition~\ref{prop:s:poles:at:nodes:exchanged}, which shows that this case cannot occur.
	From now on, we suppose that $\pp\in \ol{\cX}_D$, and  that $C''$ consists of two elliptic curves permuted by $\tau$. Our goal is to obtain  a contradiction for each of the admissible configurations of $E$.

	\begin{itemize}
		\item[$\bullet$] Case (a1):  we can suppose that $E''$ intersects $E'_4$. Since $E'_1$ contains two points in $\{q_1,\dots,q_5\}$ and  for $i=2,\dots,4$, $E'_i$ contains one point in $\{q_1,\dots,q_5\}$, at least one of the following holds $n'_1+n'_2=3$ or $n'_1+n'_2+n'_3=3$. In the former case,  let $q$ denote the node between $E'_2$ and $E'_3$,  and in the latter  let $q$ denote node between $E'_3$ and $E'_4$.  The preimage of $q$ is a node fixed by $\tau$ which decomposes $C$ into a union of a genus 1 nodal curve, denoted by $C_1$, and a genus two nodal curve, denoted by $C_2$. Note that $C_1$ contains $C'_1$ and $C'_2$, while $C_2$ contains $C''$.
		If either $\xi_{\left|C_1\right.}\equiv 0$, or $\xi_{\left|C_2\right.}\equiv 0$, then $\pp\not\in\ol{\cX}_D$ by Proposition~\ref{prop:bdry:form:dec:not:eigen:form}. Thus we must have $\xi_{\left|C_1\right.}\not\equiv 0$.     One can readily check that $\xi_{\left|C_1\right.}\not\equiv 0$ only in the case $C_1=C'_1\cup C'_2$, and $\xi$ has simple poles at the nodes between $C'_1$ and $C'_2$.     It follows that $\xi$ vanishes identically on $C'_3$ and $C'_4$, and holomorphic on $C''$. 
		But since $\xi$ only has simple poles at the nodes between $C'_1$ and $C'_2$, we get a contradiction to Proposition~\ref{prop:s:poles:at:two:nodes:not:in:XD} which means that this case cannot occur.

		\item[$\bullet$] Case (a2): without loss of generality we can assume that $E''$ intersects $E'_3$. Note that $E'_3$ does not contain any point in $\{q_1,\dots, q_5\}$. In particular, $n'_3=0$. Assume first that $n'_1+n'_2=3$. Then the preimage  of the node between $E'_2$ and $E'_3$ is a node $p$ of $C$ that is fixed by $\tau$. The node $p$ decomposes $C$ into a union of two subcurves: $C_1=C'_1\cup C'_2$ is a nodal genus $1$ curve, and $C_2:=C'_3\cup C'_4\cup C''$ is a nodal genus 2 curve. It is not difficult to see that $\xi$ vanishes identically on $C'_4$ and $C'_3$. By Proposition~\ref{prop:bdry:form:dec:not:eigen:form}, $\xi_{\left| C_1 \right.}\not\equiv 0$, and $\xi_{\left|C''\right.}\not\equiv 0$. 
		Since $C_1$ is a union of two copies of  $\Pb^1$ meeting at two points, $\xi$ has simple poles at the nodes between $C'_1$ and $C'_2$. Since $\xi$ is holomorphic at all the other nodes of $C$, we get a contradiction to Proposition~\ref{prop:s:poles:at:two:nodes:not:in:XD}, which means that this case cannot occur.
		
		Suppose now that $n'_1+n'_2=2$ (that is $q_5\in E'_1\cup E'_2)$. In this case $C'_3$ (which is the preimage of $E'_3$) is a disjoint union of two copies of $\Pb^1$. On can readily check that we always have $\xi_{\left|C'\right.}\equiv 0$ (recall that $C'=C'_1\cup\dots\cup C'_4$). Thus $\xi_{\left|C''\right.}\not\equiv 0$. But since $C'$ is a nodal curve of genus one, we then get again a contradiction to Proposition~\ref{prop:bdry:form:dec:not:eigen:form}. Thus this case does not occur either.
		
		\item[$\bullet$] Case (b): we can assume that $E''$ intersects $E'_3$. This means that each of $E'_1$ and $E'_2$ contains two points in $\{q_1,\dots,q_5\}$, while  $E'_3$ contains one point in $\{q_1,\dots,q_5\}$.
		If $n'_1+n'_2=3$, then the preimage of $E'_1\cup E'_2\cup E'_4$ in $C$ is a nodal curve of genus $1$, denoted by $C_1$, and the preimage of $E'_3\cup E''$ is a genus two nodal curve, denoted by $C_2$. The subcurves $C_1$ and $C_2$ intersect at one node fixed by $\tau$.
		Since $\{p_5, p'_5\} \subset C_1$, we  have $\xi_{\left|C_1\right.}\equiv 0$. 
		Proposition~\ref{prop:bdry:form:dec:not:eigen:form} then implies that $\pp\not\in \ol{\cX}_D$. Thus this case does not occur.
		
		Assume now that $n'_1+n'_2=4$ (that is $q_5\in E_3$). Then each of $C'_1, C'_2$ is a copy of $\Pb^1$, while $C'_3$ (resp. $C'_4$) is a disjoint union of two copies of $\Pb^1$. We have $C_1=C'_1\cup C'_2\cup C'_4$ is a nodal genus $1$ curve, which has 4 self-nodes, and $C_2=C'_3\cup C''$ consists  of two copies of a genus $1$ curve.  
		Since $\{p_5, p'_5\}\subset C'_3$, we have $\xi_{\left|C'_3\right.}\equiv 0$.
		By Proposition~\ref{prop:bdry:form:dec:not:eigen:form}, we must have $\xi_1:=\xi_{\left| C_1\right.} \not\equiv 0$ and $\xi'':=\xi_{\left|C''\right.} \not\equiv 0$. Note that $\xi_1$ and $\xi''$ are nowhere vanishing on $C_1$ and $C''$ respectively.
		
		By Theorem~\ref{th:twisted:diff}, on each component of $C'_3$ (which is a copy of $\Pb^1$) there is a meromorphic Abelian differential $\nu$ which has two double poles and a double zero such that the residues of $\nu$ at the poles are both zero. Since such a differential cannot exist, we get a contradiction which completes the proof of the proposition.
	\end{itemize}
\end{proof}

\subsubsection{Case three separating and one non-separating nodes}\label{subsec:3:sep:1:non-sep:nodes}
In this case, $E$ has 4 irreducible components, $3$ of which are isomorphic to $\Pb^1$, denoted by $E'_1,E'_2,E'_3$, the remaining one is a nodal genus 1 curve denoted by $E''$. Let $E'=E'_1 \cup E'_2\cup E'_3$. Set $n'_i:=|E'_i\cap\{q_1,\dots,q_4\}|, \; i=1,2,3$, and $n'=n_1+n'_2+n'_3$. 
Denote by $C', C'_1, C'_2, C'_3, C''$  the preimages of $E', E'_1, E'_2, E'_3, E''$ in $C$ respectively.

\begin{Proposition}\label{prop:str:3:sep:1:n:sep:nodes:Ep:disc}
	If $E'$ is disconnected, then $\pp\not\in\ol{\cX}_D$.
\end{Proposition}
\begin{proof}
	If the subcurve $E'$ is disconnected, then it must have two connected components and contains all the points in $\{q_1,\dots,q_5\}$. 
	We  suppose one component of $E'$ is the union of $E'_1$ and $E'_2$, and the other one consists of $E'_3$. 
	We can also assume that $E''$ intersects each of $E'_2$ and  $E'_3$ at one node.
	There are two cases:
	\begin{itemize}
		\item[$\bullet$] Case $n'_3=1$. This means that $q_5\in E'_3$ and $E'_1\cup E'_2$ contains three points in $\{q_1,\dots,q_4\}$.  Hence $C'_1\cup C'_2$ is a nodal curve of genus $1$, while $C'_3$ is isomorphic to $\Pb^1$, and $C''$ is a (connected) nodal curve of genus $2$. The differential $\xi$ vanishes identically on $C'_3$.
		Since  $C'_2$ and $C''$ intersect at a separating node, either $\xi_{\left|C'_1\cup C'_2\right.}\equiv 0$ or $\xi_{\left|C''\right.}\equiv 0$. In either case, we will have a contradiction by Proposition~\ref{prop:bdry:form:dec:not:eigen:form}. Thus this case does not occur.
		
		\item[$\bullet$] Case $n'_3=2$. In this case $C'_1\cup C'_2$ is a genus $0$ nodal curve which contains $\{p_5,p'_5\}$ and intersects $C''$ at two nodes permuted by $\tau$.  One readily checks that $\xi$ must vanish identically on $C'_1\cup C'_2$. This implies that $\xi$ is holomorphic at the nodes between $C'_2$ and $C''$.  Remark that both  $C'_1\cup C'_2$ and $C'_3\cup C''$ are connected. Therefore, we would have a contradiction to Proposition~\ref{prop:s:poles:at:nodes:exchanged}, which means that this case cannot occur either. The proposition is then proved.
	\end{itemize}
\end{proof}

We can now show 

\begin{Proposition}\label{prop:str:3:sep:1:n:sep:nodes}
	Assume that $E$ has $3$ separating nodes and one non-separating node. Then $\pp\in \ol{\cX}_D$ only if
	\begin{itemize}
		\item[.] $C'_1, C'_2, C'_3$ are all isomorphic to $\Pb^1$, and $C'=C'_1\cup C'_2 \cup C'_3$ is connected.
		
		\item[.] Up to a relabeling of the components of $C'$, $C'_2$ is adjacent to both $C'_1$ and $C'_3$, $C'_3$ is adjacent to $C''$, and we have $n'_1=2$, $n'_2=n'_1=1$,  $\{p_5,p'_5\} \subset C'_3$. 
		
		\item[.]  $C''$ is a disjoint union of two nodal curves of genus $1$ each of which intersects $C'_3$ at one node.

		\item[.] The differential $\xi$ vanishes identically on $C'_3$ and has simple poles at the nodes between $C'_1$ and $C'_2$, and the self-nodes of $C''$.
	\end{itemize}
\end{Proposition}
\begin{proof}
	By Proposition~\ref{prop:str:3:sep:1:n:sep:nodes:Ep:disc}, we know that $E'$ must be connected. We can label the $\Pb^1$ components of $E$ such that $E'_2$ is adjacent to $E'_1$ and $E'_3$, and $E'_3$ is adjacent to $E''$.
	
	Remark that we have $3 \leq  n' \leq 4$.
	We first consider the case $n'=3$. In this case $C'$ is a nodal curve of genus one,  $C''$ is a nodal curve of genus two, and $C''$ intersects $C'$ at one node which is fixed by $\tau$. It follows from Proposition~\ref{prop:bdry:form:dec:not:eigen:form} that we must have $\xi':=\xi_{\left|C'\right.}\not\equiv 0$ and $\xi'':=\xi_{\left|C''\right.}\not\equiv 0$. This can only happen if $n'_3=0$, and $C'_3$ contains $\{p_5,p'_5\}$. It follows from Theorem~\ref{th:twisted:diff} that $\xi''$ has a double zero at the node between $C''$ and $C'_3$ and simple poles at the self-nodes of $C''$. One can simultaneously smoothen the self-nodes of $C''$ to obtain a genus two Riemann surface $X''$ together with a holomorphic Abelian differential $\omega''$ such that
	\begin{itemize}
		\item[$\bullet$]  $X''$ admits an involution $\tau''$ with two fixed points satisfying $\tau''{}^*\omega''=-\omega''$,
		
		\item[$\bullet$] $\omega''$ has a double zero at one fixed point of $\tau''$. 
	\end{itemize}
	But by Proposition~\ref{prop:ab:diff:dble:zero:Prym:inv}, the pair $(X'',\omega'')$ cannot exist. We thus have a contradiction proving that we must have $n'=4$
	
	Suppose from now on that $n'=4$. Then we must have $n'_1=2, n'_2=n'_3=1$. In this case $C'_1\cup C'_2$ is a nodal genus one curve, $C'_3$ meets $C'_2$ at one node and meets $C''$ at two nodes, and $C''$ is either a genus one nodal curve of a disjoint union of two genus one nodal curve. If $C''$ is a genus one nodal curve, then by Proposition~\ref{prop:s:poles:at:nodes:exchanged}, $\xi$ must have simple poles at the nodes between $C''$ and $C'_3$. This means that $\xi'_3:=\xi_{\left|C'_3\right.} \not\equiv 0$. Since $C'_3$ meets $C'_1\cup C'_2$ at one node, it follows that $\xi_{\left|C'_1\cup C'_2\right.} \equiv 0$, and we get a contradiction to  Proposition~\ref{prop:bdry:form:dec:not:eigen:form}. Thus $C''$ must be a disjoint union of two nodal genus one curves. 
	
	Note that each component of $C''$ has one node. 
	By Proposition~\ref{prop:bdry:form:dec:not:eigen:form}, we must have  $\xi''\not\equiv 0$. As a consequence $\xi'_3\equiv 0$, and  $\xi_{\left| c'_1\cup C'_2\right.}\not\equiv 0$. One readily checks that these conditions can be realized only if $\{p_5,p'_5\} \subset C'_3$, and in which case, by Theorem~\ref{th:twisted:diff} $\xi$ has simple poles at the nodes between $C'_1$ and $C'_2$ and at the self-nodes of $C''$, and $\xi$ is holomorphic elsewhere. This completes the proof of the proposition.
\end{proof}

\subsubsection{Case two separating nodes and two non-separating nodes}\label{subsec:2:sep:2:n:sep:nodes}
In this case, $E$ has $4$ irreducible components, all of which are isomorphic to $\Pb^1$. Note that two non-separating nodes correspond to two simple closed  curves on the reference torus $E_0$ which decompose $E_0$ into two cylinders. There are two components of $E$ that contain only separating node, we denote those components by $E'_1$ and $E'_2$.
The remaining two components intersect each other at two non-separating nodes, we denote those components by $E''_1, E''_2$.
Define $E':=E'_1\cup E'_2$ and $E'':=E''_1\cup E''_2$. 
Let $n'_i:=|E'_i\cap\{q_1,\dots,q_4\}|, \; i=1,2$, and $n'=n_1+n_2$.
Let $C'_1,C'_2,C''_1,C''_2,C',C''$ be respectively the preimages of $E'_1, E'_2, E''_1, E''_2, E',E''$ in $C$.

\begin{Proposition}\label{prop:str:2:sep:2:n:sep:nodes:Ep:connect}
	If $E'$ is connected then $\pp\not\in \ol{\cX}_D$.
\end{Proposition} 
\begin{proof}
	Let us suppose that $\pp \in \ol{\cX}_D$.  
	We have $n'\in \{2, 3,4\}$.
	\begin{itemize}
		\item[$\bullet$] Case $n'=2$. In this case, $E'$ is a nodal curve of genus zero, and $\xi':=\xi_{\left|C'\right.}\equiv 0$ (by Theorem~\ref{th:twisted:diff}). There are two nodes between $C'$ and $C''$. Since $\xi'\equiv 0$, $\xi$  does not have simple poles at these nodes and we have a contradiction to Proposition~\ref{prop:s:poles:at:nodes:exchanged}. Thus this case cannot occur.
		
		\item[$\bullet$] Case $n'=3$. In this case $C'$ is a genus one nodal curve, $C''$ a genus two nodal curve, and $C'$ intersects $C''$ at a separating node. Using Theorem~\ref{th:twisted:diff}, one readily shows that  we always have either $\xi'=\xi_{\left| C'\right.} \equiv 0$ or $\xi'':=\xi_{\left|C''\right.}\equiv 0$. In both cases we get a contradiction to Proposition~\ref{prop:bdry:form:dec:not:eigen:form}. Thus this case cannot occur.
		
		\item[$\bullet$] Case $n'=4$. In this case, $C''$ is either a nodal genus one curve with two irreducible components or a disjoint union of two isomorphic nodal genus one curves. It contains  $\{p_5,p'_5\}$ $C'$ and intersects $C'$ at two nodes. One readily checks that in all cases, $\xi$ vanishes identically  on $C''$, and we have a contradiction to either Proposition~\ref{prop:s:poles:at:nodes:exchanged} or Proposition~\ref{prop:bdry:form:dec:not:eigen:form}. This completes the proof of the proposition.
	\end{itemize} 
\end{proof}

\begin{Proposition}\label{prop:2:sep:2:n:sep:nodes:np:4:Ep:disconn}
	Assume that $E'$ is disconnected. Then $\pp\in\ol{\cX}_D$ only if up to a relabeling of $C''_1, C''_2$ 
	\begin{itemize}
		\item[.] $C''_1$ intersects each of $C'_1$ and $C'_2$ at two nodes,
		
		\item[.] there are two nodes between $C''_1$ and $C''_2$, both of which are fixed by $\tau$,
		
		\item[.] $\{p'_5, p''_5\}\subset C''_2$, and $\xi_{\left|C''_2\right.}\equiv 0$,
		
		\item[.] $\xi_{\left|C''_1\right.}$ has a double zero at a node between $C''_1$ and $C''_2$, and simple poles at all the nodes between $C''_1$ and $C'_1\cup C'_2$.
	\end{itemize}
\end{Proposition}
\begin{proof}
	Suppose that $\pp \in \ol{\cX}_D$. 
	We first consider the case $E'_1$ and $E'_2$ intersect two different components of $E''$. Up to a relabeling, we can always assume that $E'_1$ intersects $E''_1$, $E'_2$ intersects $E''_2$, and that $n'_1\geq n'_2$. Note that $(n'_1,n'_2) \in \{(2,1), (3,1), (2,2)\}$. 
	
	\begin{itemize}
		\item[$\bullet$] Case $(n'_1,n'_2)=(2,1)$.  We must have $q_5\in E'_2$, or equivalently $\{p_5,p'_5\} \subset C'_2$. There are two nodes between $C'_1$ and $C''_1$, and one node between $C'_2$ and $C''_2$. There is one point in $\{q_1,\dots,q_4\}$, say $q_1$, which is contained in $E''$. 
		If $q_1\in E''_1$ then  $C''_1, C''_2$ are both isomorphic to $\Pb^1$ and intersect each other at three nodes.  In this case, there would be no compatible twisted differential on $C$.
		
		If $q_1 \in E''_2$ then there are either two nodes (both fixed by $\tau$), or four nodes between $C''_1$ and $C''_2$. The former case case $C''_1$ is an isomorphic to $\Pb^1$, while $C''_2$ is an elliptic curve. It follows from Theorem~\ref{th:twisted:diff} that $\xi$ must vanish identically on $C''_2\cup C''_1$. We then have a contradiction to Proposition~\ref{prop:bdry:form:dec:not:eigen:form}. 
		In the latter case $C''_1$ is a disjoint union of two copies of $\Pb^1$, while $C''_2$ is isomorphic to $\Pb^1$. One readily checks that in this case $\xi$ must vanish identically on all the components of $C$. Thus this case is excluded as well.
		
		\item[$\bullet$] Case $(n'_1,n'_2)=(3,1)$. In this case $C'_1$ is an elliptic curve,  $C''$ is a nodal genus two curve, and $C'_1$ intersects $C''$ at one node. We thus have a contradiction to Proposition~\ref{prop:bdry:form:dec:not:eigen:form}.
		
		\item[$\bullet$]  Case $(n'_1,n'_2)=(2,2)$. There are either two nodes or four nodes between $C''_1$ in $C''_2$. In the former case let $C_1:=C'_1\cup C''_1$ $C_2=C'_2\cup C''_2$. Then $C_1$ and $C_2$ are both nodal curves of of genus one intersecting each other at two nodes fixed by $\tau$. By Corollary~\ref{cor:dec:two:tori:no:eigenform}, this case cannot occur. In the latter case, each of $C''_1, C''_2$ is a disjoint union of two copies of $\Pb^1$, and it follows from Theorem~\ref{th:twisted:diff} that $\xi$ must vanish identically on $C$. Therefore this case is also excluded.
	\end{itemize}
	
	We now turn to the case $E'_1$ and $E'_2$ intersect the same component of $E''$. Without loss of generality we can suppose that both $E'_1$ and $E'_2$ intersect $E''_1$. In this case, $E''_2$  contains exactly one point in $\{q_1,\dots,q_4\}$. If $E''_2$ contains one point in $\{q_1,\dots,q_4\}$ then both $C'_1$ and $C'_2$ are isomorphic to $\Pb^1$, $C'_1$ intersects $C''_1$ at two nodes, $C'_2$ intersects $C''_1$ at one node, and there are three nodes between  $C''_1$ and $C''_2$. One readily checks that there is no non-trivial compatible twisted differential on $C$. Therefore $E''_2$ must contain $q_5$, and  each of $E'_1$ and $E'_2$ contains two points in $\{Q_1,\dots,q_4\}$. This means that both of $C'_1$ and $C'_2$ are isomorphic to $\Pb^1$ and intersect $C'_1$ at two nodes. There can be two nodes or four nodes between $C''_1$ and $C''_2$. If there are four nodes, the each of $C''_1$ and $C''_2$ is a disjoint union of two copies of $\Pb^1$. Since $\{p_5, p'_5\} \subset C''_2$, $\xi$ must vanish identically  on $C''_2$. This means that $\xi$ is holomorphic at the nodes between $C''_1$ and $C''_2$. But since these nodes are all non-separating, we get a contradiction to Proposition~\ref{prop:non:collapse:sub:curve}. Thus we conclude that this case cannot occur.
	
	Finally, assume that $q_5\in E'_2$ and that there are two nodes between $C''_1$ and $C''_2$. Note that both of the nodes between $C''_1$ and $C''_2$ are fixed by $\tau$. We must have $\xi_{\left|C''_2\right.}\equiv 0$. 
	Let $C_0:=C'_1\cup C'_2\cup C''_1$. Since $\xi$ must have simple poles between $C''_1$ and $C'_1\cup C'_2$, $C_0$ is the level zero subcurve in a compatible twisted differential on $C$. Let $\xi_0:=\xi_{\left|C_0\right.}$. 
	Let $p$ and $p'$ be the nodes between $C''_1$ and $C''_2$. Then $\xi_0$ may have two simple zeros at both $p$ and $p'$, or one of $\{p,p'\}$ is a double zero, and the other one is a regular point of $\xi_0$. In the former case, by smoothening simultaneously the nodes between $C''_1$ and $C'_1\cup C'_2$, we would get a Riemann surface of genus two together with a holomorphic $1$-form having two simple zeros at two Weierstrass points. Since such a $1$-form does not exist, the this case cannot occur.  Thus we then conclude that $\xi_0$ has a double zero at one of the nodes between $C''_1$ and $C''_2$, and the other node is a regular point for $\xi_0$. All the conditions in the statement of the proposition are now fulfilled. The proposition is then proved.
\end{proof} 

\subsubsection{Case one separating node and three non-separating nodes}\label{subsec:1:sep:3:n:sep:nodes}
In this case $E$ has four irreducible components. One of the components has only one node, we will denote this one by $E'$. Each of the other three components  has two non-separating nodes, we denote these components by $E''_1, E''_2$, and  $E''_3$, where by convention, $E''_1$ intersects $E'_1$ at one node. Note that the stability condition mean that each of $E''_2, E''_3$ contains at least one point in $\{q_1,\dots,q_5\}$.
Let $E'':=E''_1\cup E''_2\cup E''_3$.
Denote by $C', C'', C''_i, \; i=1,2,3$, the preimages of $E', E'', E''_i,  i=1,2,3$, in $C$ respectively.

\begin{Proposition}\label{prop:1:sep:3:n:sep:nodes}
	Assume that $E$ has one separating node, and three non-separating nodes. Then $\pp\in\ol{\cX}_D$ only if
	\begin{itemize}
		\item[.] $C'$ contains two points in $\{p_1,\dots,p_4\}$, each of $C''_1, C''_2$ contains one point in $\{p_1,\dots,p_4\}$, and $\{p'_5, p''_5\}\subset C''_3$.
		
		\item[.] $C''_1$ intersects $C''_2$ at two nodes, and $C''_3$ at one nodes.
		
		\item[.] $\xi_{\left|C''_1\right.}\equiv 0$ has a zero of order $2$ at the node between $C''_1$ and $C''_3$, and has simple poles at all the nodes between $C''_1$ and $C'_1\cup C''_2$.
	\end{itemize}
	
\end{Proposition}
\begin{proof}
	Let $n_1:=|E'\cup E''_1)\cap\{q_1,\dots,q_4\}|$, and for $i=2,3$, $n_i:=|C''_i\cap\{q_1,\dots,q_4\}|$. Up to a relabeling of $E''_1, E''_2$, we have $(n_1,n_2,n_3)\in \{(1,2,1), (2,2,0), (2,1,1), (3,1,0)\}$. 
	\begin{itemize}
		\item[$\bullet$] Case $(n_1,n_2,n_3)=(1,2,1)$. In this case $E'$ contains $q_5$ and one point in $\{q_1,\dots,q_4\}$. 
		Therefore both $C'$ and $C''_1$ are isomorphic to $\Pb^1$ and there is a node between $C'$ and $C''_1$. We must have $\xi_{\left|C'_1\cup C''_1\right.} \equiv 0$. 
		Either there are two nodes between $C''_1$ and $C''_2$, or two nodes between $C''_1$ and $C''_3$. Since $\xi$ does not have simple poles at these nodes, we get a contradiction to Proposition~\ref{prop:non:collapse:sub:curve}. Thus this case does not occur.
		
		\item[$\bullet$] Case $(n_1,n_2,n_3)=(2,2,0)$. We must have $q_5\in E''_3$ and $E''_1\cap\{q_1,\dots,q_5\}=\varnothing$. It follows that  $C'$   is  isomorphic to $\Pb^1$ and intersect $C''_1$ at two nodes. 
		Both $C''_2$ and $C''_3$ intersect $C''_1$ at either one or two nodes. In the former case, we have a decomposition of $C$ into two subcurves of genus one, namely $C_1:=C'\cup C''_1$ and $C_2:=C''_2\cup C''_3$, which intersect each other at two nodes fixed by $\tau$. By Corollary~\ref{cor:dec:two:tori:no:eigenform} this is impossible. 
		If there are two nodes between $C''_1$ and $C''_i$, $i=2,3$, then these nodes are non-separating. Note that $C''_3$ is a disjoint union of two copies of $\Pb^1$. Since $\{p_5,p'_5\} \subset C''_3$, we must have $\xi_{\left| C''_3\right.} \equiv 0$. But this implies a contradiction to Proposition~\ref{prop:non:collapse:sub:curve}. Thus this case does not occur either. 
		
		\item[$\bullet$] Case $(n_1,n_2,n_3)=(2,1,1)$. We have two subcases: either $C''_1$ intersects each of $C''_2$ and $C''_3$ at one node, or $C''_1$ intersects each of $C''_2$ and $C''_3$ at two nodes.  In the first case both $C_1:=C'_1\cup C''_1$ and $C_2=C''_2\cup C''_3$ are nodal genus one curves which meet each other at two nodes fixed by $\tau$. By Corollary~\ref{cor:dec:two:tori:no:eigenform}, this case does not occur. 
		In the second case both $C''_2$ and $C''_3$ are isomorphic to $\Pb^1$ and intersect each other at one node fixed by $\tau$. We must have either $\xi_{\left| C''_2\right.} \equiv 0$ or $\xi_{\left|C''_3\right.}\equiv 0$. In both case, since the nodes between $C''_1$ and $C''_2\cup C''_3$ are non-separating, we get a contradiction to Proposition~\ref{prop:non:collapse:sub:curve}. Thus this case is also excluded. 
		
		\item[$\bullet$] Case $(n_1,n_2,n_3)=(3,1,0)$. If $E'$ contains three points in $\{q_1,\dots,q_4\}$, then $C'$ is an elliptic curve, $C''$ is an nodal genus two curve, and $C'$ intersects $C''$ at one node. This case is excluded by Proposition~\ref{prop:bdry:form:dec:not:eigen:form}. 
		This $E'$ must contains two points in $\{q_1,\dots,q_4\}$, and $E''_1$ contains one point in $\{q_1,\dots,q_4\}$. Both $C'$ and $C''_1$ are isomorphic to $\Pb^1$ and intersect each other at two nodes. 
		There are either one node of two nodes between $C''_1$ and $C''_2$. Tn the former case, $C''_3$ consists of two copies of $\Pb^1$ each of which intersects both $C''_1$ and $C''_2$. Since $\{p_5, p'_5\} \subset C''_3$, we must have $\xi_{\left|C''_3\right.}\equiv 0$. Since the nodes between $C''_3$ and $C''_1$ are non-separating, we would get a contradiction to Proposition~\ref{prop:non:collapse:sub:curve}, which means that this case does not occur. 
		
		Finally, let us assume that there are two nodes between $C''_1$ and $C''_2$. This implies that $C''_3$ is isomorphic to  $\Pb^1$ and intersects each of $C''_1, C''_2$ at one node. We must have $\xi_{\left|C''_3\right.}\equiv 0$. By Proposition~\ref{prop:s:poles:at:nodes:exchanged}, $\xi$ must have simple poles at the nodes between $C''_1$ and $C'\cup C''_2$. Let $\xi''_1:=\xi_{\left|C''_1\right.}$ and $\xi''_2:=\xi_{\left|C''_2\right.}$. By Theorem~\ref{th:twisted:diff},  $\xi''_1$ has a double zero , while $\xi''_2$ is nowhere vanishing on $C''_2$. In particular, the node between $C''_2$ and $C''_3$ is a regular point for $\xi''_2$. This complete the proof of the proposition.   
	\end{itemize}
\end{proof}

\subsubsection{Case four non-separating nodes} \label{subsec:str:4:n:sep:nodes}
In this case $E$ has four irreducible components that we will denote by $E_i, \; i=1,\dots,4$, in the cyclic order.  Let $n_i:=|E'_i\cap\{q_1,\dots,q_4\}, \; i=1,\dots,4$. Up to a renumbering of the irreducible components, we can always suppose that $n_1=\max\{n_i, \; i=1,\dots,4\}$. 
By the stability condition, we have $(n_1,\dots,n_4) \in \{(2,1,1,0), (2,1,0,1), (2,0,1,1), (1,1,1,1)\}$.
Let $C_i$ be the preimage of $E_i$ in $C$.

\begin{Proposition}~\label{prop:4:n:sep:nodes}
	Assume that $E$ has four non-separating nodes. Then $\pp\not\in\ol{\cX}_D$. 
\end{Proposition}
\begin{proof}
	Suppose that $\pp \in \ol{\cX}_D$.  We have the following cases:
	\begin{itemize}
		\item[$\bullet$] Case $(n_1,\dots,n_4)=(2,1,1,0)$ or $(n_1,\dots,n_4)=(2,0,1,1)$. By symmetry, we only need to consider the case $(n_1,\dots,n_4)=(2,1,1,0)$. In this case $q_5 \in E_4$. There are either one node or two nodes between $C_1$ and $C_2$. Assume first that $C_1$ and $C_2$ intersects at two nodes. Then $C_1$ is an elliptic curve, $C_4$ is isomorphic to $\Pb^1$ and intersects each of $C_1$ and $C_3$ at one node, while $C_2$ and $C_3$ intersect each other at two nodes. Let $C':=C_1\cup C_4$ and $C'':=C_2\cup C_4$. Then $C'$ and $C''$ are both nodal curve of genus one and intersect each other at two nodes fixed by $\tau$. By Corollary~\ref{cor:dec:two:tori:no:eigenform}, this case cannot occur. 
		
		Assume now that there are two nodes between $C_1$ and $C_2$. Then $C_1,C_2,C_3$ are all isomorphic to $\Pb^1$, $C_2$ meets $C_3$ at one node, $C_4$ is a disjoint union of two copies of $\Pb^1$ each of which intersects both $C_1$ and $C_3$. Since $\{p_5,p'_5\} \subset C_4$, we must have $\xi_{\left|C_4\right.}\equiv 0$. But since the nodes between $C_1$ and $C_4$ are non-separating, we have a contradiction to Proposition~\ref{prop:non:collapse:sub:curve}. Thus this case does not occur either.
		
		\item[$\bullet$] Case $(n_1,\dots,n_4)=(2,1,0,1)$. Again, we have two subcases, either $C_1$ intersects $C_2$ at one node, or $C_1$ intersects $C_2$ at two nodes. In the former case, $C_1$ is an elliptic curve which intersects both $C_2$ and $C_4$ at one node, $C_2$ and $C_4$ are isomorphic to $\Pb^1$, $C_3$ is a disjoint union of two copies of $\Pb^1$, each of which intersects both $C_2$ and $C_4$. Note that $C':=C_2\cup C_3\cup C_4$ is a nodal curve of genus one. Since we must have  $\xi_{\left|C_3\right.}\equiv 0$, it follows that $\xi_{\left|C'\right.}\equiv 0$. Therefore we get a contradiction to Proposition~\ref{prop:bdry:form:dec:not:eigen:form}, which shows that this case does not occur. In the latter case, all the irreducible components of $C$ are isomorphic to $\Pb^1$, $C_1$ intersects both of  $C_2, C_4$ at two nodes, while $C_3$ intersects both of $C_2,C_4$ at one node. One readily checks that there cannot a compatible twisted differential on $C$. Thus is case is also excluded.
		
		\item[$\bullet$] Case $(n_1,\dots,n_4)=(1,\dots,1)$. In this case one readily checks that $C$ is always a union of two nodal curves of genus one intersecting each other at two nodes fixed by $\tau$. Thus this case is excluded by Corollary~\ref{cor:dec:two:tori:no:eigenform}. This completes the proof of the proposition. 
	\end{itemize} 
\end{proof}

\subsection{Case $E$ has five nodes}\label{subsec:str:5:nodes}
Assume now that the curve $E$ has $5$ nodes. We first remark that at least one of the nodes of $E$ is non-separating (otherwise, the stability condition  cannot be satisfied).

\subsubsection{Four separating  and one non-separating nodes}\label{subsec:4:sep:1:n:sep:node}
\begin{Proposition}\label{prop:4:sep:1:n:sep:nodes}
	If $E$ has 4 separating nodes and one non-separating node, then $\pp\not\in\ol{\cX}_D$.
\end{Proposition}
\begin{proof}
	In this case we will use the same notation and convention as in Proposition~\ref{prop:4:sep:nodes}. 
	By the same arguments as in the proof of Proposition~\ref{prop:4:sep:nodes}, we get that $C''$ consists of two copies of nodal genus one curve, and $\xi'':=\xi_{\left|C''\right.} \not\equiv 0$. As usual we suppose that $\pp \in \ol{\cX}_D$ in order to get a contradiction.
	\begin{itemize}
		\item[(a1)] Remark that in this case $E'_1$ contains two points in $\{q_1,\dots,q_5\}$, and each of $E_2, E_3, E_4$ contains one point in $\{q_1\dots,q_5\}$. If $q_5\in E'_1$ or $q_5\in E'_2$, we would get a contradiction to Proposition~\ref{prop:bdry:form:dec:not:eigen:form}. 
		If $q_5 \in E'_3$, then $C'_3$ and $C'_4$ are both isomorphic to $\Pb^1$ and intersect each other one node. The differential $\xi$ mush vanish identically  on $C'_3\cup C'_4$. By Theorem~\ref{th:twisted:diff}, there is a meromorphic Abelian differential $\nu$ on $C'_4$ that has a double zero at the node between $C'_4$ and $C'_3$ and double poles at the nodes between $C_4$ and $C''$. More over the residues of $\nu$ at the poles must be zeros. 
		We can   identify $C'_3$ with $\Pb^1$ such that the restriction of $\tau$ to $C'_4$ is given by $x \mapsto 1/x$, the node between $C'_4$ and $C'_3$ corresponds to $x=1$, while the nodes between $C'_4$ and $C''$ correspond to $x=0$ and $x=\infty$. It follows that up to a constant, we must have $\nu = \frac{(x-1)^2dx}{x^2}$. One readily checks that the residues of $\nu$ at the poles cannot be zero. Therefore, this case is excluded.
		
		If $q_5 \in E'_4$ then $C'_4$ consists of two copies of $\Pb^1$, and by Theorem~\ref{th:twisted:diff}, each component of $C'_4$ must carry a meromorphic Abelian differential with the same property as $\nu$. Therefore this case is excluded as well. 
		
		\item[(a2)] In this case $q_5$ is contained in one of the components $E'_1, E'_2, E'_4$. If $q_5\in E'_1$ or $q_5\in E'_3$ then $\xi$ vanish identically on $C'_1\cup\dots\cup C'_4$, and we have a contradiction by Proposition~\ref{prop:bdry:form:dec:not:eigen:form}. If $q_5\in E'_4$, then $C'_3$ is isomorphic to $\Pb^1$ and intersects each of $C'_2, C'_4$ at one node. By Theorem~\ref{th:twisted:diff}, $C'_3$ carries a meromorphic Abelian differential $\nu$ with the following properties
		\begin{itemize}
			\item[.] $\nu$ has a zero of order $4$ at the node between $C'_3$ and $C'_4$,
			\item[.] $\nu$ has poles of order two at the nodes between $C'_3$ and $C'_2\cup C''$.
			\item[.] the residues of $\nu$ at all the poles are zero. 
		\end{itemize} 
		We can identify $C'_3$ with $\Pb^1$ such that the restriction of $\tau$ to $C'_3$ is given by $x \mapsto -x$, $0$ corresponds to the node between $C'_3$ and $C'_4$, $\infty$ the node between $C'_3$ and $C'_2$, and $\pm 1$ the nodes between $C'_3$ and $C''$. Up to a constant, we have 
		$$
		\nu= \frac{x^4dx}{(x-1)^2(x+1)^2}.	
		$$
		One readily checks that $\res_{\pm 1}(\nu) \neq 0$, which means that this case does not occur.
		
		\item[(b)] Recall that in this case  $E'_4$ is adjacent to all of $E'_1, E'_2, E'_3$, and contains no point in $\{q_1,\dots,q_5\}$. Without loss of generality, we can assume that $E'_3$ is adjacent to $E''$. 
		Let $C_1$ denote the subcurve $C'_1\cup C'_2\cup C'_4$. 
		If $q_5$ is contained in either $C'_1$ or $C'_2$ then $C_1$ is a nodal curve of genus one, on which $\xi$ vanishes identically. Thus we have a contradiction to Proposition~\ref{prop:bdry:form:dec:not:eigen:form}. If $q_5\in E'_3$ then $C'_3$ consists of two copies of $\Pb^1$, each of which carries a meromorphic Abelian differential which has one double zeros and two double poles such that the residues at the poles are zero. Since such a differential does not exist, this case cannot occur. The proposition is then proved.   
	\end{itemize} 
\end{proof} 
\subsubsection{Three separating nodes and two non-separating nodes}\label{subsec:3:sep:2:n:sep:nodes}
In this case $E$ has $5$ irreducible components, all of which are isomorphic to $\Pb^1$. Three of the components are not incident to non-separating nodes, we denote those component by $E'_1, E'_2, E'_3$ and their union by $E'$. The remaining two components intersect each other at two non-separating nodes, we denote those components by $C''_1, C''_2$, and their union by $E''$. The preimages of $E'_i, E''_j, E', E''$ in $C$ are denoted by $C'_i, C''_j, C', C''$ respectively.
Let $n'_i:=|E'_i\cap\{q_1,\dots,q_4\}|, \; i=1,2,3$, and $n'=n'_1+n'_2+n'_3$.
\begin{Proposition}\label{prop:3:sep:1:n:sep:nodes}
	If $E$ has $3$ separating nodes and $2$ non-separating ones then $\pp\not\in\ol{\cX}_D$.
\end{Proposition}
\begin{proof}
	We have two cases
	\begin{itemize}
		\item[(a)] $E'$ is connected. We label the components of $E'$ such that $E'_2$ is adjacent to both $E'_1$ and $E'_3$.	
		Without loss of generality, we can assume that $E'$ intersects $E''_1$ and disjoint from $E''_2$. Since $E''_2$ must contain one point in $\{q_1,\dots,q_5\}$,  we have $3\leq n' \leq 4$. 
		We have two subcases:
		\begin{itemize}
			\item[(a1)] $E'_3$ intersects $E''_1$. If $n'=3$, then $C'$ is a nodal curve of genus one, $C''$ is a genus two nodal curve having two irreducible components intersecting at three nodes. One readily checks that and $\xi$ must vanish identically on  $C''$. We thus have a contradiction by Proposition~\ref{prop:bdry:form:dec:not:eigen:form}. 
			
			If $n'=4$, then $q_5\in E''_2$. Let $C_1:=C'_1\cup C'_2$ and $C_2:=C'_3\cup C''$. Observe that $C_1$ is a genus one curve with two nodes, $C_2$ is a genus two curve, and $C_1$ intersects $C_2$ at one node. One then readily checks that since $\{p_5,p'_5\}\subset C''_2$, $\xi$ must vanish identically on $C_2$. We thus get a contradiction by Proposition~\ref{prop:bdry:form:dec:not:eigen:form}.
			
			\item[(a2)] $E'_2$ intersects $E''_1$. If $n'=3$, then $C'$ is a genus one curve, $C''$ is a genus two curve, and $C'$ intersects $C''$ at one node. One readily checks that $\xi_{\left|C'\right.}\equiv 0$. Thus we get a contradiction to  Proposition~\ref{prop:bdry:form:dec:not:eigen:form}.
			
			If $n'=4$, then $C'$ is a genus one curve, while $C''$ can be either a genus one curve or a disjoint union of two nodal curves of genus one. In the former case, $\xi$ must have simple poles at the nodes between $C'_2$ and $C''_1$ by Proposition~\ref{prop:s:poles:at:nodes:exchanged}. But there does not exist a compatible twisted differential on $C$ satisfying this property. Therefore, this case is excluded.   
			If there are four nodes between $C''_1$ and $C'_2$, then $C''$ has two connected components, each of which is a genus one nodal curve on which $\xi$ vanish identically. We thus have a contradiction by Proposition~\ref{prop:bdry:form:dec:not:eigen:form}.  
		\end{itemize}
		
		\item[(b)] $E'$ is disconnected. Note that $E'$ can not have more than two connected components because of the stability condition. We can always suppose that the two connected components of $E'$ are $E'_1\cup E'_2$ and $E'_3$. We can also suppose that $E'_2$ intersects $E''_1$ and $E'_3$ intersects $E''_2$. 
		
		We have $2 \leq n'_1+n'_2\leq 3$. If $n'_1+n'_2=3$ then  $C'_1\cup C'_2$ is a nodal genus one curve, $C''\cup C'_3$ is a genus two curve intersecting $C'_1\cup C'_2$ at one node. Since $\{p_5,p'_5\}\subset C'_3$, we must have $\xi_{\left|C''\cup C'_3\right.} \equiv 0$. But this is a contradiction to Proposition~\ref{prop:bdry:form:dec:not:eigen:form}. 
		
		If $n'_3=2$, then there are two nodes between $C'_2$ and $C''_1$. Since $\{p_5,p'_5\} \subset C'_1\cup C'_2$ we must have $\xi_{\left|C'_1\cup C'_2\right.} \equiv 0$. This means that $\xi$ does not have simple poles at the nodes between $C'_2$ and $C''_1$. We thus get a contradiction to Proposition~\ref{prop:s:poles:at:nodes:exchanged} and the proposition follows.	
	\end{itemize}	
\end{proof}

\subsubsection{Two separating nodes and three non-separating nodes}\label{subsec:2:sep:3:n:sep:nodes}
Two  irreducible components of $E$ contain only separating nodes, they will be denoted by $E'_1, E'_2$. The remaining components will be denoted by $E''_1,E''_2, E''_3$. Let $E':=E'_1\cup E'_2, E'':=E''_1\cup E''_2\cup E''_3$. The preimages of $E', E'', E'_i,  E''_j, i\in\{1,2\}, j\in\{1,2,3\}$ in $C$ are denoted by $C', C'', C'_i, C''_j$ respectively.

\begin{Proposition}\label{prop:2:sep:3:n:sep:nodes}
	If $E$ has two separating nodes and three non-separating ones, then $\pp\not\in\ol{\cX}_D$.
\end{Proposition}
\begin{proof}
	The subcurve $E'$ can be connected or not.
	\begin{itemize}
		\item[(a)] $E'$ is connected.  Without loss of generality, we can assume that $E'_2$  intersects $E''_1$ at a separating node. By the stability condition, $E''_1$ does not contain any point in $\{q_1,\dots,q_5\}$, while each of $E''_2,E''_3$ contains exactly one point in $\{q_1,\dots,q_5\}$.
		We have two subcases
		\begin{itemize}
			\item[(a1)] $q_5\in  E'$. In this case, $C'_2$ intersects $C'_1$ at two nodes. Since both $C'$ and $C''$ are connected, these nodes are non-separating in $C$. 
			Since $\{p_5,p'_5\} \subset C'$,  we have $\xi_{\left|C'\right.}\equiv 0$, which is a contradiction to Proposition~\ref{prop:s:poles:at:nodes:exchanged}. Thus this case cannot happen.

			\item[(a2)] $q_5\in E''_2\cup E''_3$. Without loss of generality we can assume that $q_2\in E''_2$. In this case $C'$ is a nodal curve of genus $1$, $C''$ is a nodal curve of genus $2$, and $C'$ intersects $C''$ at one node. 
			Note that  $C''_2$ is either isomorphic to $\Pb^1$, or a disjoint union of two copies of $\Pb^1$. Moreover we must have $\xi_{\left|C''_2\right.}\equiv 0$.
			Suppose that $C''_2$ is isomorphic to $\Pb^1$, then $C''_2$ intersects each of $C''_1, C''_3$ at one node, while $C''_1$ intersects $C''_3$ at two nodes. If either $\xi_{\left|C''_1\right.}\equiv 0$ or $\xi_{\left|C''_3\right.}\equiv 0$, then $\xi_{\left|C''\right.}\equiv 0$ and we have a contradiction to Proposition~\ref{prop:bdry:form:dec:not:eigen:form}.
			If $\xi_{\left|C''\right.}\not\equiv 0$, then we must have that $\xi_{\left|C''_1\cup C''_3\right.}$ is nowhere zero and has simple poles at the nodes between $C''_1$ and $C''_3$. It follows that $\xi_{\left|C'_2\right.}\equiv 0$, and therefore $\xi_{\left|C'\right.}\equiv 0$. But this contradicts Proposition~\ref{prop:bdry:form:dec:not:eigen:form}, hence this case cannot occur.

			In the  case $C''_2$ is a disjoint union of two copies of $\Pb^1$,  $C''_1$ intersects $C''_3$ at one node. Therefore, either $\xi_{\left|C''_1\right.}\equiv 0$ or $\xi_{\left|C''_3\right.}\equiv 0$. But in either case, we would get  $\xi_{\left|C''\right.}\equiv 0$, which is a contradiction to Proposition~\ref{prop:bdry:form:dec:not:eigen:form}. Thus we can conclude that this case cannot occur.
		\end{itemize}
		
		\item[(b)] $E'_1$ and $E'_2$ are disjoint. We can suppose that $E'_1$ (resp. $E'_2$) intersects $E''_1$ (resp. $E''_2$) at a separating node. Note that each of $E''_1, E''_2$ contains no point in $\{p_1,\dots,p_5\}$, while  $E''_3$ contains exactly  on point in $\{q_1,\dots,q_5\}|=1$.     We have two subcases
		\begin{itemize}
			\item[(b1)] $q_5\in E'$. Without loss of generality, we can assume that $q_5\in E'_1$. Then $C'_1$ intersects $C''_1$ at one node, and $C'_2$ intersects $C''_2$ at two nodes. Note that all of the irreducible components of $C$ are isomorphic to $\Pb^1$. Note that $C''_1$ intersects $C''_2\cup C''_3$ at three nodes.
			
			Let $\xi''_1:=\xi_{\left|C''_2\right.}$.
			Since $\{p_5,p'_5\}\subset C'_1$, if $\xi''_1 \not\equiv 0$ then by Theorem~\ref{th:twisted:diff}, if must have a zero of order four at the node between $C''_1$ and $C'_1$.  
			Since $\xi''_1$ has at worst simple poles at the nodes between $C''_1$ and $C''_2\cup C''_3$, this is impossible. Therefore, we must have $\xi''_1\equiv 0$. But this implies that $\xi$ does not have simple poles at two non-separating nodes permuted by $\tau$, which is a contradiction to Proposition~\ref{prop:s:poles:at:nodes:exchanged}. Thus this case cannot occur.
			
			\item[(b2)] $q_5\in E''_3$. Under this assumption, $C''$ is either a genus nodal one curve having three irreducible components, or a disjoint union of two genus one nodal curves each of which has three irreducible components,  while $C'_1, C'_2$ are both isomorphic to $\Pb^1$. If $C''$ is connected, all the nodes between components of $C''$ are fixed by $\tau$. This implies that either $\xi_{\left|C''_1\right.} \equiv 0$ or  $\xi_{\left|C''_2\right.} \equiv 0$. In either case, since there are two nodes between $C''_1$ and $C'_1$ and two nodes between $C''_2$ and $C'_2$, this would implies a contradiction to Proposition~\ref{prop:s:poles:at:nodes:exchanged}. Thus this case is excluded. 
			
			If $C''$ has two connected components, then so does $C''_3$. It follows from Theorem~\ref{th:twisted:diff} that $\xi$ must vanish identically on $C''_3$. But since the nodes between $C''_3$ and $C''_1\cup C''_2$ are non-separating, we get a contradiction to Proposition~\ref{prop:non:collapse:sub:curve}. This completes the proof of the proposition.
		\end{itemize}
	\end{itemize}
\end{proof}

\subsubsection{One separating node and four non-separating nodes}\label{subsec:1:sep:4:n:sep:nodes}
We now consider the case $E$ has one separating node and four non-separating ones. In this case, one of the irreducible components of $E$, denoted by $E'$, has only one node. The other components have two or three nodes, and are denoted by $E''_1,\dots, E''_4$ in the cyclic ordering. We will always assume that $E'$ intersects $E''_1$.
The component $E'$ must contain two points in $\{q_1,\dots,q_5\}$.
We have $E''_1\cap\{q_1,\dots,q_5\}=\vide$, and for $i=2,3,4$, $E''_i$ contains exactly one point in $\{q_1,\dots,q_5\}$.
Let $C', C''_i,\; i=1,\dots,4$, denote the  preimages of $E', E''_i$ respectively.
Let $\xi':=\xi_{\left|C'\right.}$, and $\xi''_i:=\xi_{\left|C''_i\right.}$, for $i=1,\dots,4$.

\begin{Proposition}\label{prop:1:sep:4:n:sep:nodes}
	If $E$ has one separating node and four non-separating ones, then $\pp\not\in\ol{\cX}_D$.
\end{Proposition}
\begin{proof}
	We suppose that $\pp\in \ol{\cX}_D$.
	\begin{itemize}
		\item[(a)] $q_5\in E'$. In this case $C'$ and $C''_i, \; i=1,\dots,4$, are all isomorphic to $\Pb^1$. Moreover, for each $i=1,\dots,4$, $C''_i$ intersects $C''_{i-1}\cup C''_{i+1}$ at $3$ nodes, with the convention $C''_0=C''_4$ and $C''_5=C''_1$.
		Without loss of generality, we can suppose that $C''_1$ intersects $C''_2$ at two nodes, and intersects $C''_4$ at one node.
		
		If $\xi''_1\not\equiv 0$ then from Theorem~\ref{th:twisted:diff} it must have a zero of order four at the node between $C''_1$ and $C'$.
		But since $\xi''_1$ cannot have poles of order greater than $1$ at the nodes between $C''_1$ and $C''_2\cup C''_4$,
		we then have a contradiction, which means that $\xi''_1\equiv 0$.
		It follows that $\xi$ is holomorphic at the nodes between $C''_1$ and $C''_2\cup C''_4$. But since the nodes between $C''_1$ and $C''_2$ are non-separating, we get a contradiction to Proposition~\ref{prop:s:poles:at:nodes:exchanged}. This case is therefore excluded.

		\item[(b)] $q_5\in E''_2\cup E''_4$. It is enough to consider the case $q_5\in E''_2$. We have two subcases
		\begin{itemize}
			\item[(b1)] There is one node between $C''_1$ and $C''_2$. In this case, there is also one node between $C''_1$ and $C''_4$, and two nodes between $C''_3$ and $C''_4$. Note that $C'$ intersects $C''_1$ at two nodes. If $\xi''_1\equiv 0$ then $\xi$ does not have simple poles at the nodes between $C'$ and $C''_1$, and we have a contradiction to Proposition~\ref{prop:s:poles:at:nodes:exchanged}. Thus we must have $\xi''_1\not\equiv 0$. This implies that $\xi''_4\equiv 0$ (since $C''_1$ and $C''_4$ intersects at a node fixed by $\tau$), which is a contradiction to Proposition~\ref{prop:non:collapse:sub:curve}. We conclude that this case cannot occur.

			\item[(b2)] There are  two nodes between $C''_1$ and $C''_2$. In this case  both $C''_1$ and $C''_2$ consist of two copies of $\Pb^1$, while $C''_3$ and  $C''_4$ are both isomorphic to $\Pb^1$. Note that $C''_3$ intersects $C''_4$ at one node. Since $\{p_5,p'_5\}\subset C''_2$, we must have $\xi''_2\equiv 0$. But since the nodes between $C''_2$ and $C''_1\cup C''_3$ are non-separating, this is a contradiction to Proposition~\ref{prop:non:collapse:sub:curve}. Hence this case is also excluded.  
		\end{itemize}
		
		\item[(c)] $q_5\in C''_3$. We also have two subcases
		\begin{itemize}
			\item[(c1)] There is one node between $C''_1$ and $C''_2$. In this case, there is also one node between $C''_1$ and $C''_4$. The subcurve $C''_3$ consists of two copies of $\Pb^1$ each of which intersects both $C''_2$ and $C''_4$. We must have $\xi''_3\equiv 0$, which is a contradiction to Proposition~\ref{prop:non:collapse:sub:curve} (since the nodes between $C''_3$ and $C''_2\cup C''_4$ are non-separating). Thus this case cannot occur.

			\item[(c2)] There are two nodes between $C''_1$ and $C''_2$. In this case, $C''_1$ is a disjoint union of two copies of $\Pb^1$, while all of $C''_2, C''_3, C''_4$ are isomorphic to $\Pb^1$. Note that each component of $C''_1$ intersects both $C''_2$ and $C''_4$, and $C''_3$ intersects each of $C''_2, C''_4$ at one node. We have $\xi''_3\equiv 0$. 
			All of the nodes that are not contained in $C''_3$ are non-separating and not fixed by $\tau$. By Proposition~\ref{prop:s:poles:at:nodes:exchanged}, $\xi$ must have simple poles at those nodes. But since each component of $C''_1$ contains three nodes, we get a contradiction which shows that this case cannot occur either. The proposition is then proved. 
		\end{itemize}
	\end{itemize}
\end{proof}

\subsubsection{Five non-separating nodes}\label{subsec:5:n:sep}
Suppose that $E$ has five non-separating nodes. Then $E$ has five irreducible components, denoted by $E_i, \; i=1,\dots,5$, in the cyclic order.
Each component of $E$ contains exactly one point in $\{q_1,\dots,q_5\}$.
We can suppose that $q_5\in E_1$.

For all $i=1,\dots,5$, let $C_i$ be the preimage of $E_i$ in $C$. Let $\xi_i:=\xi_{\left|C_i\right.}$, and  $(C_i,\nu_i)_{1\geq i \geq 5}$ be the twisted  differential  on the $C$, which is given by Theorem~\ref{th:twisted:diff}.
\begin{Proposition}
	Assume that $E$ has five non-separating nodes. Then $\pp\not\in\ol{\cX}_D$.
\end{Proposition}
\begin{proof}
	In this case, $C_i$ is isomorphic to $\Pb^1$ for $i=2,\dots,5$. 
	Suppose that $\pp\in \ol{\cX}_D$.
	We have two cases
	\begin{itemize}
		\item[(i)] $C_1$ is isomorphic to $\Pb^1$. In this case $C_1$ intersects each of $C_2, C_5$ at one node.
		Since $\{p_5,p'_5\}\subset C_1$, $\xi_1 \equiv 0$, and $\nu_1$ has two double zeros on $C_1$. Since $\nu_1$ has poles of even order at the nodes fixed by $\tau$,  $\nu_1$ must have a pole of order $2$ and a pole of order $4$ at the nodes between $C_1$ and $C_2\cup C_5$. Without loss of generality, suppose that the node between $C_1$ and $C_2$ is a pole of order $4$ of $\nu_1$.
		Then this node is a double zero of $\nu_2$. It follows that $\nu_2$ has double poles at the node between $C_2$ and $C_3$.
		This means that $\xi_2\equiv 0$, which implies that $\xi_3\equiv 0$.
		As a consequence $\xi_{\left|C_4\cup C_5\right.}$ is nowhere zero, and has simple poles at the nodes between $C_4$ and $C_5$.
		We now remark that  $C':=C_1\cup C_2 \cup C_3$ and $C'':=C_4\cup C_5$ are two curves of genus $1$ which intersect at two nodes fixed by $\tau$. Since $\xi$ vanishes identically on $C'$, we get a contradiction to Corollary~\ref{cor:dec:two:tori:no:eigenform}. Thus this case cannot occur.
		
		\item[(ii)] $C_1$ has is a disjoint union of two copies of $\Pb^1$. In this case both components of $C_1$ intersect $C_2$ and $C_5$. There is one node between $C_2$ and $C_3$, and one node between $C_4$ and $C_5$.
		Let $C':=C_1\cup C_2\cup C_5$ and $C'':=C_3\cup C_4$.
		Then  $C'$ and $C''$ are nodal curves of genus $1$, which intersect at two nodes fixed by $\tau$.
		Since $\xi_1\equiv 0$, we have $\xi_2\equiv 0$ and $\xi_5\equiv 0$, which means that $\xi_{\left|C'\right.}\equiv 0$.
		Therefore, we get a contradiction to Corollary~\ref{cor:dec:two:tori:no:eigenform}.
		This completes the proof of the proposition.
	\end{itemize}
\end{proof}
\section{Proof of Theorem~\ref{th:iden:triple:tor:e:times:deg:proj:1:e}}\label{subsec:prf:iden:triple:tor:D:odd}
We first prove the following
\begin{Proposition}\label{prop:iden:D:1:mod:8:sigma:1}
	For all $D\equiv 1 \; [8]$, $D$ not a square, we have
	\begin{equation}\label{eq:iden:D:1:mod:8:sigma:1}
		\sum_{\substack{0 < e < \sqrt{D} \\ e \; {\rm odd}}}(-1)^{\frac{e-1}{2}}\cdot e \cdot \sigma_1(\frac{D-e^2}{8})=0.
	\end{equation}
\end{Proposition}
\begin{proof}
	Let $\psi: \Z \to \{0, \pm 1\}$ be the Dirichlet character of conductor $4$ defined by
	$$
	\psi(n) = \left\{
	\begin{array}{cl}
		1 & \hbox{ if $n\equiv 1 \mod 4$}\\
		-1 & \hbox{ if $n\equiv 3 \mod 4$}\\
		0 & \hbox{ otherwise }.
	\end{array}
	\right.
	$$
	Consider the function
	$$
	\theta_\psi(z):= \sum_{n=0}^{+\infty} \psi(n)n \exp(2\pi\imath n^2 z)
	$$
	for all $z \in \Hbb$.  Define for all $\gamma= \left(\begin{smallmatrix} a & b \\ c & d \end{smallmatrix}\right) \in \Gamma_0(4)$ and $z\in \Hbb$
	$$
	J(\gamma,z):=\left(\frac{c}{d}\right)\varepsilon_d^{-1}(cz+d)^{1/2},
	$$
	where $\left(\frac{\bullet}{\bullet}\right)$ is the Kronecker symbol and
	$$
	\varepsilon_d=\left\{
	\begin{array}{cl}
		1 & \hbox{ if $d \equiv 1 \mod 4$} \\
		\imath & \hbox{ if $d \equiv 3 \mod 4$ }.
	\end{array}
	\right.
	$$
	Then for all $\gamma= \left(\begin{smallmatrix} a & b \\ c & d \end{smallmatrix}\right) \in \Gamma_0(64)$, we have
	\begin{equation}\label{eq:theta:psi:iden}
		\theta_\psi(\gamma\cdot z) =\psi(d)\cdot\left(\frac{-1}{d}\right)\cdot J(\gamma,z)^3\cdot\theta_\psi(z).
	\end{equation}
	In particular, $\theta_\psi$ is a {\em modular form of weight} $3/2$ (see \cite[\textsection 4.9]{Miyake}).  As a consequence of \eqref{eq:theta:psi:iden}, we get
	\begin{equation}\label{eq:theta:psi:iden:derived}
		\theta'_\psi(\gamma\cdot z) = \psi(d)\left(\frac{-1}{d}\right)\left(\frac{c}{d}\right)\varepsilon_d\cdot \left((cz+d)^{\frac{7}{2}}\cdot\theta'_\psi(z) +  \frac{3c}{2}\cdot(cz+d)^{\frac{5}{2}}\cdot\theta_\psi(z)\right).
	\end{equation}
	Recall that $G_2$ is the function on $\Hbb$ defined by
	$$
	G_2(z)=\frac{-1}{24}+\sum_{n=1}^\infty \sigma_1(n)\exp(2\pi\imath nz).
	$$
	It is well known that $G_2$ satisfies
	$$
	G_2(\gamma\cdot z)=(cz+d)^2G_2(z)-\frac{c(cz+d)}{4\pi\imath}.
	$$
	for all $\gamma= \left(\begin{smallmatrix} a & b \\ c & d \end{smallmatrix}\right) \in \SL(2,\Z)$.
	It is straightforward to check that the function
	$$
	f(z):=G_2(8z)\theta_\psi(z)+\frac{1}{48\pi\imath}\cdot \theta'_\psi(z)
	$$
	satisfies
	$$
	f(\gamma\cdot z)=\psi(d)\left(\frac{-1}{d}\right)\left(\frac{c}{d}\right)\varepsilon_d(cz+d)^{\frac{7}{2}}\cdot f(z),
	$$
	for all $\gamma= \left(\begin{smallmatrix} a & b \\ c & d \end{smallmatrix}\right) \in \Gamma_0(64)$.
	This means that $f$ is an {\em integral modular form of weight $7/2$} with respect to $\Gamma_0(64)$.
	Let $f(z)=\sum_{n=0}^\infty c_n\exp(2\pi\imath nz)$ be the Fourier expansion of $f$. A direct computation shows that $c_n=0$ if $n\not\equiv 1 \, [8]$, and for $n \equiv 1 \; [8]$ we have
	$$
	c_n =\left\{
	\begin{array}{cl}
		\displaystyle{\sum_{0 <e < \sqrt{n}, e \, \odd} \psi(e) \cdot e \cdot \sigma_1(\frac{n-e^2}{8})} & \hbox{ if $n$ is not a square}\\
		\displaystyle{\sum_{0 < e < d, e \, \odd} \psi(e)\cdot e \cdot \sigma_1(\frac{d^2-e^2}{8}) +\psi(d)\frac{d^3-d}{24} }& \hbox{ if $n=d^2$}
	\end{array}
	\right.
	$$
	We claim that $f \equiv 0$. To see this, we consider $f^4$ which is an integral modular form of weight $14$ with respect to $\Gamma_0(64)$.
	The Riemann surface $X_0(64):=\Hbb/\Gamma_0(64)$ has genus $3$, $12$ cusps and no elliptic points. Thus an integral modular form of weight $14$ on $X_0(64)$ which vanishes to the order at least $14\times(3-1)+14\times 12/2=112$ at $\infty$ must be zero (cf. \cite[Cor. 2.3.4]{Miyake}). One can easily check that $f$ vanishes at least to the order $30$ at $\infty$. Hence $f^4$ vanishes at least to the order $120$ at $\infty$. Therefore, we must have $f^4 \equiv 0$, which implies that $f\equiv 0$. As a consequence, $c_n=0$ for all $n\in \N$ and \eqref{eq:iden:D:1:mod:8:sigma:1} follows.
\end{proof}
\subsection*{Proof of Theorem~\ref{th:iden:triple:tor:e:times:deg:proj:1:e}}
\begin{proof}
	For all $D >9, \, D \equiv 1 \, [8]$ not a square. Let
	$$
	S_D:=\sum_{\substack{0 < e < \sqrt{D} \\ e \; {\rm odd}}}(-1)^{\frac{e-1}{2}}\cdot e \cdot m_D(e).
	$$
	It follows from Proposition~\ref{prop:iden:D:1:mod:8:sigma:1} and Corollary~\ref{cor:triple:tor:deg:proj:e:D:prim} that $S_D=0$  if $D$ is $(1,2)$-primitive.
	Assume now that $D=f^2D_0$, where $D_0$ is $(1,2)$-primitive discriminant and $f \in \Z_{>1}$.
	We claim that
	$$
	\sum_{\substack{0 < e < \sqrt{D} \\ e \; {\rm odd}}}(-1)^{\frac{e-1}{2}}\cdot e \cdot \sigma_1(\frac{D-e^2}{8}) = \sum_{r \, | \, f } (-1)^{\frac{r-1}{2}}\cdot r\cdot S_{D/r^2}.
	$$
	To see this, let us fix an odd integer $e$ such that  $0 < e < \sqrt{D}$. Then $\sigma_1(\frac{D-e^2}{8})$ is the cardinality of the set $\tilde{\Pcal}_{D,e}(0)$ of triples $(a,b,d) \in \Z^3$ such that
	$$
	a>0, \, d>0, ad=\frac{D-e^2}{8}, \;  0 \leq b < a.
	$$
	Let $r=\gcd(a,b,d,e)$ and $(a',b',d',e'):=1/r\cdot(a,b,d,e)$.
	Since we have $D=e^2+8ad=r^2(e'{}^2+8a'd')$, it follows that $r \, | \, f$, and by definition $(a',b',c',e') \in \Pcal_{D/r^2}(0)$. On the other hand, if $(a',b',d',e') \in \Pcal_{D/r^2}(0)$ then $(ra',rb', rc',re') \in \tilde{\Pcal}_{D,re'}(0)$. Thus we have
	$$
	\sigma_1(\frac{D-e^2}{8})=\#\tilde{\Pcal}_{D,e}=\sum_{r \, | \, \gcd(e,f)}\#\Pcal_{D/r^2,e/r}(0)=\sum_{r \, | \, \gcd(e,f)}m_{D/r^2}(e/r).
	$$
	Therefore
	$$
	(-1)^{\frac{e-1}{2}}\cdot e \cdot \sigma_1(\frac{D-e^2}{8}) = (-1)^{\frac{e-1}{2}}\cdot e \cdot \sum_{r \, | \, \gcd(e,f)}m_{D/r^2}(e/r).
	$$
	Using $(-1)^{(ab-1)/2}=(-1)^{(a-1)/2}(-1)^{(b-1)/2}$ if both $a,b$ are odd numbers, we get
	$$
	(-1)^{\frac{e-1}{2}}\cdot e \cdot \sigma_1(\frac{D-e^2}{8}) = \sum_{r \, | \, \gcd(e,f)}(-1)^{\frac{r-1}{2}}\cdot r \cdot (-1)^{(e/r-1)/2}\cdot(e/r)\cdot m_{D/r^2}(e/r)
	$$
	Since for any $r \, | \, f$, a prototype $(a',b',d',e') \in \Pcal_{D/r^2}(0)$ only appears in $\tilde{\Pcal}_{D, re'}(0)$, we have
	$$
	\sum_{\substack{0 < e < \sqrt{D} \\ e \; {\rm odd}}}(-1)^{\frac{e-1}{2}}\cdot e \cdot \sigma_1(\frac{D-e^2}{8}) = \sum_{r \, | \, f } (-1)^{\frac{r-1}{2}}\cdot r\cdot S_{D/r^2}.
	$$
	It follows from \eqref{eq:iden:D:1:mod:8:sigma:1} that
	$$
	\sum_{r \, | \, f } (-1)^{\frac{r-1}{2}}\cdot r\cdot S_{D/r^2}=0.
	$$
	Since $S_{D/f^2}=0$ by Proposition~\ref{prop:iden:D:1:mod:8:sigma:1}, one concludes that $S_D=0$ by induction.
\end{proof}


\end{document}